%% file: article-plain.tex
\documentclass[a4paper]{article}
\usepackage[hidelinks]{hyperref}
\usepackage[margin=3cm,a4paper]{geometry}
\usepackage{macros}
\usepackage{amssymb}

\title{Rewriting in Gray categories with applications to coherence}
\author{Simon Forest, Samuel Mimram}
\date{}

\begin{document}
\maketitle

\input{abstract}

\input{plan}
\end{document}

%% file: abstract.tex
\begin{abstract}
  Over the recent years, the theory of rewriting has been used and extended in
  order to provide systematic techniques to show coherence results for strict
  higher categories. Here, we investigate a further generalization to Gray
  categories, which are known to be equivalent to tricategories. This requires
  us to develop the theory of rewriting in the setting of precategories, which
  are adapted to mechanized computations and include Gray categories as
  particular cases. We show that a finite rewriting system in precategories
  admits a finite number of critical pairs, which can be efficiently computed.
  We also extend Squier's theorem to our context, showing that a convergent
  rewriting system is coherent, which means that any two parallel 3-cells are
  necessarily equal. This allows us to prove coherence results for several
  well-known structures in the context of Gray categories: monoids, adjunctions,
  Frobenius monoids.
\end{abstract}
%%% Local Variables:
%%% mode: latex
%%% TeX-master: "article"
%%% End:

%% file: plan.tex
% \todo{HACKS: des hacks sont présents. vérifier l'espacement des labels HACK plus
%   bas À CHAQUE CHANGEMENT}
% \tableofcontents

% \section{Coherent presentations of Gray categories}
% \input{notions}
% \section{Preliminaries}
% \input{prelim}

% \section*{Typographie}
% Cellules:
% \begin{itemize}
% \item $0$: $x,y,z$
% \item $1$: $f,g,h$ ($a,b,c$ générateur)
% \item $2$: $\phi,\psi$ ($\alpha,\beta,\gamma$ générateur)
% \item $3$: $F,G,H$ ($A,B,C$ générateur)
% \item $4$: $\Lambda,\Gamma$
% \end{itemize}

% \newpage

\section*{Introduction}
\input{intro}

\section{Precategories}
\label{sec:precat}
\input{precat}
\section{Gray categories}
\label{sec:gray}
\input{gray}
\section{Rewriting}
\label{sec:rewriting}
\input{rewriting}
\section{Applications}
\label{sec:applications}
\input{applications}

\clearpage
\bibliographystyle{plain}
\bibliography{article-plain}
\newpage
\appendix
% \section*{Appendix}
\input{appendix}

%%% Local Variables:
%%% mode: latex
%%% TeX-master: "article"
%%% End:

%% file: intro.tex
Algebraic structures, such as monoids, can be defined inside arbitrary
categories. In order to generalize their definition to higher categories, the
general principle is that one should look for a \emph{coherent} version of the
corresponding algebraic theory: this roughly means that we should add enough
higher cells to our algebraic theory so that ``all diagrams commute'' up to these cells. For
instance, when generalizing the notion of monoid from monoidal categories to
monoidal $2$\categories, associativity and unitality are now witnessed by
$2$\cells, and one should add new axioms in
order to ensure their coherence: in this case, those can be chosen to be MacLane's
unit and pentagon
equations, thus resulting in the notion of pseudomonoid. The fact that these are
indeed enough to make the structure coherent constitutes a reformulation of
MacLane's celebrated coherence theorem for monoidal
categories~\cite{maclane1963natural}. In this context, a natural question is: how
can we systematically find those higher coherence cells?

Rewriting theory~\cite{baader1999term,terese2003term} provides a satisfactory
answer to this question. Its starting point is the observation that it is often
useful to provide an orientation to equations in algebraic structures, which
dates back to the work of Dehn~\cite{dehn1911unendliche} and
Thue~\cite{thue1914probleme} on presentations of groups and monoids
respectively. Namely, when the resulting rewriting system is suitably behaved
(confluent and terminating) equality can be tested efficiently (by comparing
normal forms). Moreover, confluence of the whole rewriting system can be decided
algorithmically by computing critical branchings and testing for the confluence
of those only, which are always in finite number when the rewriting system is
finite.
In fact, this result can be reformulated as the fact that the confluence
diagrams for critical branchings provide us precisely with enough cells to make
the structure coherent. This was first observed by Squier for monoids, first
formulated in homological language~\cite{squier1987word} and then generalized as
a homotopical condition~\cite{squier1994finiteness,lafont1995new}. These results
were then extended to strict higher categories by Guiraud and
Malbos~\cite{guiraud2009higher,guiraud2012coherence,guiraud2016polygraphs} based
on a notion of rewriting system adapted to this setting, which is provided by
Burroni's \emph{polygraphs}~\cite{burroni1993higher} (also called
computads~\cite{street1976limits}). In particular, their work allows recovering
the coherence laws for pseudomonoids in this way.

Our aim is to generalize those techniques in order to be able to define coherent
algebraic structures in \emph{weak} higher categories. We actually handle here
the first non-trivial case, which is the one of dimension~$3$. Namely, it is
well-known that tricategories are not equivalent to strict $3$\categories: the
``best'' one can do is to show that they are equivalent to \emph{Gray
  categories}~\cite{gordon1995coherence, gurski2013coherence}, which is an
intermediate structure between weak and strict $3$\categories, roughly
consisting in $3$\categories in which the exchange law is not required to hold
strictly. This means that classical rewriting techniques cannot be used off the
shelf in this context and one has to adapt those to Gray categories, which is
the object of this article.

It turns out that a slightly more general notion than Gray categories is adapted
to rewriting: \emph{precategories}. The notion of precategory is a
generalization of the one of sesquicategory, whose use has already been
advocated by Street in the context of rewriting~\cite{street1996categorical}.
The interest in those has also been renewed recently, because they are at the
heart of the graphical proof-assistant Globular~\cite{bar2016globular,
  bar2017data}. Gray categories are particular $3$\precategories equipped with
exchange $3$\cells satisfying suitable axioms. We first work out in details the
definition of precategories and, based on the work of
Weber~\cite{weber2009free}, show that $(n{+}1)$\precategories can be defined as
categories enriched in $n$\precategories equipped with the so-called \emph{funny
  tensor product}, see \Cref{sec:precat}. This is analogous to the well-known
fact that Gray categories are categories enriched over 2-categories equipped
with the Gray tensor product~\cite{gordon1995coherence}, that we recall in
\Cref{sec:gray}. We then define in \Cref{sec:rewriting} a notion of polygraph
adapted to precategories, called \emph{prepolygraph}. It is amenable to computer
implementation: there is an efficient representation of the morphisms in free
precategories, which allows for mechanized computation of critical branchings.
Moreover, it can be used to present other precategories, in particular Gray
categories (\Ssecr{gray-presentation}). In order to study these presentations,
we adapt the theory of rewriting to the context of prepolygraphs in
\Secr{rewriting}, and we show that our notion of rewriting system retains the
classical properties. In particular, a finite rewriting system always has a
finite number of critical branchings, which contrasts with the case of strict
categories~\cite{lafont2003towards, guiraud2009higher, mimram2014towards}. It
moreover allows for a Squier-type coherence theorem (\Thmr{squier}). Finally, in
\Cref{sec:applications}, we apply our technology to several algebraic structures
of interest, which allows us to recover known coherence theorems and find new
ones, such as for pseudomonoids (\Cref{ssec:app-pseudomonoid}),
pseudoadjunctions (\Cref{ssec:app-adjunction}), self-dualities
(\Cref{ssec:app-untyped-adjunction}), and also Frobenius pseudomonoids but up to
a termination conjecture (\Cref{ssec:app-frobenius-monoid}).

\paragraph*{Acknowledgements}
\input{ack}

%% file: ack.tex
The authors would like to thank the reviewer for their high-quality
proofreading, and the associated numerous and useful comments.

%% file: precat.tex
In this work, we use a variant of the notion of $n$-category called
\emph{precategory} whose $2$\nbd-dimen\-sional version is better known as
\emph{sesquicategory}~\cite{street1996categorical}. Many definitions of
``semi-strict'' higher categories can be described as precategories with
additional structures and equations, and this is in particular the case for Gray
categories. Moreover,
% since the exchange law is not required for precategories
contrarily to strict higher categories, their cells can be easily described by
normal forms, making them amenable to computations.
%
% The lack of exchangeability \todo{c'est un mot ?} is witnessed by the
% inductive definition of $n$\precategories that uses the funny tensor product,
% which does not require any commutativity condition.
%
This notion was used to give several definitions of semi-strict higher
categories~\cite{bar2017data} and is the underlying structure of the Globular
tool for higher categories~\cite{bar2016globular}. Premises of it can be found
in the work of Street~\cite{street1996categorical} and
Makkai~\cite{makkai2005word}. In what follows, we give equational and enriched
definitions of precategories (\Ssecr{precat-def} and \Ssecr{funny-tensor}).
Then, we define prepolygraphs as a direct adaptation of the notion of polygraph
for strict categories (\Ssecr{pol}), and we show that the cells of such a
prepolygraph admit a normal form (\Ssecr{cell-nf}). Finally, we recall the usual
construction of localization, in the context of $3$-dimensional precategories only
(\Ssecr{(3,2)-precategory}), since our subsequent results will mostly target
$(3,2)$\precategories.

\subsection{Globular sets}
\label{ssec:globular-sets}
Given $n\in\N$, an \emph{$n$-globular set}~$C$ is a diagram of sets
\[
  \begin{tikzcd}
    C_0
    &
    \ar[l,shift right,"\csrc_0"']
    \ar[l,shift left,"\ctgt_0"]
    C_1
    &
    \ar[l,shift right,"\csrc_1"']
    \ar[l,shift left,"\ctgt_1"]
    C_2
    &
    \ar[l,shift right,"\csrc_2"']
    \ar[l,shift left,"\ctgt_2"]
    \ldots
    &
    \ar[l,shift right,"\csrc_{n-1}"']
    \ar[l,shift left,"\ctgt_{n-1}"]
    C_n
  \end{tikzcd}
\]
such that $\csrc_i\circ \csrc_{i+1}=\csrc_i\circ \ctgt_{i+1}$ and
$\ctgt_i\circ \csrc_{i+1}=\ctgt_i\circ \ctgt_{i+1}$ for $0\leq i<n-1$. An
element~$u$ of $C_i$ is called an \emph{$i$-globe} of~$C$ and, for $i>0$, the
globes $\csrc_{i-1}(u)$ and $\ctgt_{i-1}(u)$ are respectively called the
\emph{source} and \emph{target} and~$u$.
We write~$\nGlob n$ for the category of $n$-globular sets, a morphism $f\co C\to
D$ being a family of morphisms $f_i\co C_i\to D_i$, for $0\leq i\leq n$, such
that $\csrctgt\eps_i\circ f_{i+1}=f_i\circ \csrctgt\eps_i$ for $\eps \in
\set{-,+}$. Given $m \ge n$ and $C \in \nGlob m$, we denote by $\truncf m n(C)$
the $n$\globular set obtained from~$C$ by removing the $i$\globes for $n < i \le
m$. This operation extends to a functor $\truncf m n\co \nGlob m \to \nGlob n$.

For simplicity, we often implicitly suppose that, in an $n$-globular set~$C$,
the sets $C_i$ are pairwise disjoint and write $u\in C$ for
$u\in\bigsqcup_iC_i$.
For $\eps\in\set{-,+}$ and $k\geq 0$, we write
\[
  \partial_{i,k}^\eps=\partial_{i}^\eps\circ\partial_{i+1}^\eps\circ\cdots\circ\partial_{i+k-1}^\eps
\]
for the \emph{iterated source} (when $\eps=-$) and \emph{target} (when $\eps=+$)
maps. We generally omit the index~$k$ when it is clear from the context. Also,
we sometimes simply write $\partial^\eps(u)$ for $\partial_{i,1}^\eps(u)$.
Given $i,j,k\in \set{0,\ldots,n}$ with $k<i$ and $k<j$, we write $C_i\times_kC_j$ for the
pullback
\[
  \begin{tikzcd}[column sep=tiny,row sep=small,cramped]
    & C_i\times_kC_j \ar[rd,dotted] \ar[ld,dotted]
    \ar[dd,phantom,"\dcorner",very near start]& \\
    C_i\ar[rd,"\ctgt_k"']&&\ar[ld,"\csrc_k"]C_j \\
    & C_k
  \end{tikzcd}\pbox.
\]
A sequence of globes $u_1 \in C_{i_1}, \ldots, u_p \in C_{i_p}$ is said
\emph{$i$-composable}, for some $i < \min(i_1,\ldots,i_p)$, when
$\ctgt_{i}(u_j) = \csrc_i(u_{j+1})$ for $1 \le j < p$. Given $u,v \in C_{i+1}$
with $i < n$, $u$ and $v$ are said \emph{parallel} when $\csrctgt\eps(u) =
\csrctgt\eps(v)$ for $\eps \in \set{-,+}$.

For $u\in C_{i+1}$, we sometimes write $u\co v\to w$ to indicate that
$\csrc_i(u)=v$ and $\ctgt_i(u)=w$. In low dimension, we use $n$-arrows such as
$\To$, $\TO$, $\TOO$, \etc to indicate the sources and the targets of $n$-globes
in several dimensions. For example, given a $2$\globular set $C$ and $\phi \in
C_2$, we sometimes write
\[
  \phi\co f \To g \co x \to y
\]
to indicate that
$\csrc_1(\phi) = f$, $\ctgt_1(\phi) = g$, $\csrc_0(\phi) = x$ and $\ctgt_0(\phi)
= y$. We also use these arrows in graphical representations to picture the
elements of a globular set~$C$. For example, given an $n$\globular set~$C$ with
$n \ge 2$, the drawing
\begin{equation}
  \label{eq:some-globular-set}
  \begin{tikzcd}[sep=4em]
    x
    \ar[r,"f",bend left=50,""{auto=false,name=fst}] 
    \ar[r,"g"{description},""{auto=false,name=snd}]
    \ar[r,"h"',bend right=50,""{auto=false,name=trd}]
    &
    y
    \ar[from=fst,to=snd,phantom,"\phantom{\scriptstyle\phi}\Downarrow\scriptstyle\phi"]
    \ar[from=snd,to=trd,phantom,"\phantom{\scriptstyle\psi}\Downarrow\scriptstyle\psi"]
    \ar[r,"k"]
    &
    z
  \end{tikzcd}
\end{equation}
figures two $2$\globes $\phi,\psi \in C_2$, four $1$\globes $f,g,h,k \in C_1$ and
three $0$\globes $x,y,z \in C_0$ such that
\begin{gather*}
  \csrc_1(\phi) = f, \qquad \ctgt_1(\phi) = \csrc_1(\psi)= g, \qquad 
  \ctgt_1(\psi) = h, \\
  \csrc_0(f) = \csrc_0(g) = \csrc_0(h) = x, \qquad
  \ctgt_0(f) = \ctgt_0(g) = \ctgt_0(h) = \csrc_0(k) = y, \qquad \ctgt_0(k) = 0.
\end{gather*}

\subsection[\texorpdfstring{$n$}{n}-precategories]{\texorpdfstring{$\bm n$}{n}-precategories}
\label{ssec:precat-def}
Given $n\in\N$, an \emph{$n$-precategory}~$C$ is an $n$-globular set equipped
with
\begin{itemize}
\item identity functions $\unitp{k}{}\co C_{k-1}\to C_{k}$, for $0< i\le n$,
\item composition functions
  $\comp_{k,l}\co C_k\times_{\min(k,l)-1}C_l\to C_{\max(k,l)}$, for
  $0<k,l\leq n$,
\end{itemize}
satisfying the axioms below. In this context, the elements
of~$C_i$ are called \emph{$i$\cells}. Since the dimensions of the cells determine the
functions to be used, we often omit the indices of $\unit{}$ and, given $0 < k,l
\le n$ and $i = \min(k,l) - 1$, we often write~$\comp_{i}$ for $\comp_{k,l}$.
For
example, in a $2$\precategory which has a configuration of cells as
in~\eqref{eq:some-globular-set}, there are, among others, $1$-cells $f\comp_0
k$, $h\comp_0 k$ and $2$-cells $\phi\comp_1\psi$ and $\psi\comp_0k$ given by the
composition operations. The axioms of $n$\precategories are the following:
\begin{enumerate}[label=(\roman*),ref=(\roman*)]
\item \label{precat:first}\label{precat:src-tgt-unit}for~$k < n$ and~$u\in C_k$,
  \[
    \csrc_k(\unit {u})=u=\ctgt_k(\unit {u}),
  \]
%    \todo{introduire la notation $\csrctgt\eps$ sans indice}
\item \label{precat:csrc-tgt}for~$i,k,l \in \set{0,\ldots,n}$ such that~$i = \min(k,l) -1$,~$(u,v)\in C_k\times_iC_l$, and~${\eps \in \set{-,+}}$,
  \begin{align*}
    \csrctgt\eps(u \comp_i v)
    &=
    \begin{cases}
      u\comp_i \csrctgt\eps(v)& \text{if~$k < l$,}\\
      \csrc(u)&\text{if~$k=l$ and~$\eps = -$,}\\
      \ctgt(v)&\text{if~$k=l$ and~$\eps = +$,}\\
      \csrctgt\eps(u)\comp_i v&\text{if~$k>l$,}
    \end{cases}
  \end{align*}
\item \label{precat:compat-id-comp}for~$i,k,l \in \set{0,\ldots,n}$ with~$i = \min(k,l) -
  1$, given~$(u,v)\in C_{k-1}\times_{i}C_l$,
  \begin{align*}
    \unit u\comp_i v
    &=
    \begin{cases}
      v&\text{if~$k \le l$,}\\
      \unit{u\comp_i v}&\text{if~$k > l$,}
    \end{cases}
  \shortintertext{and, given~$(u,v) \in C_k \times_i C_{l-1}$,}
    u\comp_i\unit{v}
    &=
    \begin{cases}
      u&\text{if~$l\le k$,}\\
      \unit{u\comp_i v}&\text{if~$l > k$,}
    \end{cases}
  \end{align*}

\item \label{precat:before-last} \label{precat:assoc}for~$i,k,l,m \in
  \set{0,\ldots,n}$ with~$i = \min(k,l) - 1 = \min(l,m) - 1$, and~$u \in
  C_k$,~$v \in C_l$ and~$w \in C_w$ such that~$u,v,w$ are $i$\composable,
  \[
    (u\comp_iv)\comp_iw
    =
    u\comp_i(v\comp_iw),
  \]
\item \label{precat:distrib}for~$i,j,k,k' \in \set{0,\ldots,n}$ such that
  \[
    j = \min(k,k') - 1
    \qtand
    i < j\zbox,
  \]
  given~$u \in C_{i+1}$ and~$(v,v') \in C_k \times_j C_{k'}$ such that~$u,v$ are $i$\composable,
  \[
    u \comp_i (v \comp_j v') = (u \comp_i v) \comp_j (u \comp_i v')
  \]
  and, given~$(u,u') \in C_k \times_j C_{k'}$ and~$v \in C_{i+1}$ such that~$u,v$
  are $i$\composable,
  \[
    (u \comp_j u') \comp_i v = (u \comp_i v) \comp_j (u' \comp_i v).
  \]
\end{enumerate}
A morphism of $n$-precategories, called an \emph{$n$-prefunctor}, is a morphism
between the underlying globular sets which preserves identities and compositions
as expected. We write $\nPCat n$ for the category of $n$-precategories. The
above description exhibits $n$\precategories as an essentially algebraic theory.
Thus, $\nPCat n$ is a locally presentable
category~\cite[Thm.~3.36]{adamek1994locally}; consequently, it is
complete and cocomplete~\cite[Cor.~1.28]{adamek1994locally}. In the following, we
write $\termcat$ for the terminal $n$\precategory for $n\ge 0$.

In dimension $2$, string diagrams can be used as usual to represent compositions
of~$2$\cells. For example, given the $2$\precategory~$C$ freely generated on the
globular set
\[
  \begin{tikzcd}
    x
    \ar[r,bend left=70,"f"{myname=f}]
    \ar[r,bend right=70,"f'"'{myname=fp,pos=0.51}]
    \cphar[r,"\Downarrow\!\phi"]
    &
    y
    \ar[r,bend left=70,"g"{myname=g}]
    \ar[r,bend right=70,"g'"'{myname=gp,pos=0.49}]
    \cphar[r,"\Downarrow\!\psi"]
    &
    z
  \end{tikzcd}
\]
we can represent the two $2$\cells
\[
  (\phi \comp_0 g) \comp_1 (f' \comp_0 \psi)
  \qquad\qqtand\qquad
  (f \comp_0 \psi) \comp_1 (\phi \comp_0 g')
\]
respectively by
\[
  % \tikzset{3-cell/.style={draw,thick,rounded corners=2,text width=1cm,align=center,align=center}}
  % \begin{tikzpicture}[baseline=(current bounding box.center),xscale=1.5,yscale=0.8]
    % \node (TL) at (0,2) {$f\mathstrut$};
    % \node (TR) at (1,2) {$g\mathstrut$};
    % \node[3-cell] (L) at (0,1) {$\phi$};
    % \node[3-cell] (R) at (1,0) {$\psi$};
    % \node (BL) at (0,-1) {$f'\mathstrut$};
    % \node (BR) at (1,-1) {$g'\mathstrut$};
    % \draw[thick] (TL) -- (L) -- (BL);
    % \draw[thick] (TR) -- (R) -- (BR);
    % \end{tikzpicture}
  % \quad\qqtand\quad
  % \begin{tikzpicture}[baseline=(current bounding box.center),xscale=1.5,yscale=0.8]
    % \node (TL) at (0,2) {$f\mathstrut$};
    % \node (TR) at (1,2) {$g\mathstrut$};
    % \node[3-cell] (L) at (0,0) {$\phi$};
    % \node[3-cell] (R) at (1,1) {$\psi$};
    % \node (BL) at (0,-1) {$f'\mathstrut$};
    % \node (BR) at (1,-1) {$g'\mathstrut$};
    % \draw[thick] (TL) -- (L) -- (BL);
    % \draw[thick] (TR) -- (R) -- (BR);
  % \end{tikzpicture}
  \satex[scale=1.5]{phi-psi}
  \quad\qqtand\quad
  \satex[scale=1.5]{psi-phi}
\]
Note that these two $2$\cells are different, and that the diagram
\[
  % \tikzset{3-cell/.style={draw,thick,rounded corners=2,text width=1cm,align=center,align=center}}
  % \begin{tikzpicture}[baseline=(current bounding box.center),xscale=1.5,yscale=0.8]
    % \node (TL) at (0,1.5) {$f\mathstrut$};
    % \node (TR) at (1,1.5) {$g\mathstrut$};
    % \node[3-cell] (L) at (0,0.5) {$\phi$};
    % \node[3-cell] (R) at (1,0.5) {$\psi$};
    % \node (BL) at (0,-0.5) {$f'\mathstrut$};
    % \node (BR) at (1,-0.5) {$g'\mathstrut$};
    % \draw[thick] (TL) -- (L) -- (BL);
    % \draw[thick] (TR) -- (R) -- (BR);
    % \end{tikzpicture}
  \satex[scale=1.5]{psiphi}
\]
makes no sense in a generic $2$\precategory.

\subsection{Truncation functors}
\label{ssec:truncation-functors}
Similarly to strict categories~\cite{metayer2008cofibrant}\todo{Métayer ne fait
  que la troncation et l'inclusion. une meilleure idée de citation?}, the
categories $\nPCat n$ for $n \ge 0$ can be related by several functors. For
$m\geq n$, we have a truncation functor
\[
  \truncf{m}{n}\co \nPCat m \to\nPCat n
\]
where, given an $m$-precategory~$C$, $\truncf{m}{n}(C)$ is the $n$-precategory
obtained by forgetting all the $i$-cells for $n < i \le m$. This functor admits
a left adjoint
\[
  \incf{n}{m}\co\nPCat n\to\nPCat m
\]
which, to an $n$\precategory~$C$, associates the
$m$\precategory~$\incf{n}{m}(C)$ obtained by formally adding $i$\nbd-identities
for $n < i \le m$, \ie $\incf{n}{m} (C)_i=C_i$ for $i \le n$ and
$\incf{n}{m} (C)_i=C_n$ for $i > n$.

\begin{prop}
  \label{prop:trunc-incf-la-ra}
  For $m > n$, the functors $\truncf{m}{n}$ and $\incf{n}{m}$ admit both left
  and right adjoints, \ie we have a sequence of adjunctions
  \[
    \incfla{m}{n}\dashv\incf{n}{m}\dashv\truncf{m}{n}\dashv\truncfra n m.
  \]
  As a consequence, the functors $\truncf{m}{n}$ and $\incf{n}{m}$ preserve both
  limits and colimits.
\end{prop}
\begin{proof}
  Suppose given an $m$-precategory~$C$.
  The $n$-precategory $\incfla m n(C)$ has the same $i$-cells as $C$ for $i<n$
  and $\incfla m n(C)_n$ is obtained by quotienting $C_n$ under the smallest
  congruence $\sim$ such that $u\sim v$ whenever there exists an $(n+1)$-cell
  $\alpha\co u\to v$.
  The $n$-precategory $\truncfra mn(C)$ has the same $i$-cells as $C$ for
  $0\leq i\leq n$ and, for $n\leq i<m$, $\truncfra mn(C)_{i+1}$ is defined from
  $\truncfra mn(C)_i$ as the set of pairs
  $(u,v)\in \truncfra mn(C)_i\times\truncfra mn(C)_i$ with $\csrc(u)=\csrc(v)$ and
  $\ctgt(u)=\ctgt(v)$, with $\csrc(u,v)=u$ as source and $\ctgt(u,v)=v$ as
  target.
  Details are left to the reader.
\end{proof}

Given $n < m$, we write $\catsk n {(-)}$ for the functor~$\incf{n}{m} \circ
\truncf{m}{n}\co \nPCat m \to \nPCat m$ and, given an $m$\precategory~$C$, we
call~$\catsk n C$ the \emph{$n$\skeleton} of~$C$. It corresponds to the
$m$\precategory obtained from~$C$ by removing all non-trivial $i$-cells with
$i>n$. We write
\[
  \cucatsk{(-)}\co \catsk n
  {(-)} \to \id{\nPCat m}
\]
for the counit of the adjunction $\incf{n}{m} \dashv \truncf{m}{n}$.
% SF: plus utile car on ne fait plus de produit funny généralisé
% Since
% $\incf{n}{m} \truncf{m}{n} \incf{n+1}{m} \truncf{m}{n+1} = \incf{n}{m}
% \truncf{m}{n}$ for $n < m$, there is a natural transformation $\cucatsk{(-)}'\co
% \catsk{n}{(-)} \To \catsk{n+1}{(-)}$ defined by $\cucatsk{(-)}' = \cucatsk{(-)}
% \circ \incf{n+1}{m} \truncf{m}{n+1}$.
Since $\incf{n}{m}$ and $\truncf{m}{n}$ both preserve limits and colimits by
\Propr{trunc-incf-la-ra}, so does the functor $\catsk n {(-)}$.

\subsection{The funny tensor product}
\label{ssec:funny-tensor}
We now define the funny tensor product, that we will use to give an enriched
definition of precategories. It can be thought of as a variant of the cartesian
product of categories where we restrict to morphisms where one of the components
is the identity (or, more precisely, to formal composites of such morphisms).
We give a rather direct and concise definition, and
we refer the reader to the work of Weber~\cite{weber2009free} for a more
abstract definition. Given $n \ge 0$ and two $n$\precategories~$C$ and~$D$,
the \emph{funny tensor product of~$C$ and~$D$} is the pushout
\[
  \begin{tikzcd}[sep=large]%,baseline=(\tikzcdmatrixname-2-2.base)]
    \catsk 0 C\times \catsk 0 D
    \ar[d,"\cucatsk C \times \catsk 0 D"']
    \ar[r,"{\catsk 0 C \times \cucatsk{D}}"]
    &
    \ar[d,dotted,"\rincfun{C,D}"]\catsk 0 C\times D\\
    C\times \catsk 0 D\ar[r,dotted,"\lincfun{C,D}"']&C\funny D
  \end{tikzcd}
  \pbox.
\]
Since $\cucatsk{(-)}$ is a natural transformation, the funny tensor product can
be extended as a functor
\[
  (-)\funny(-)\co \nPCat n \times \nPCat n \to \nPCat n\pbox.
\]
% which is the following pushout in $\CAT(\nPCat n\times \nPCat n,
% \nPCat n)$:
% \[
%   \begin{tikzcd}[sep=large]
%     \catsk i {(-)}\times \catsk i {(-)}
%     \ar[d,"\cucatsk{} \times \catsk i {(-)}"']
%     \ar[r,"\catsk i {(-)} {\times} \cucatsk{}"]&
%     \ar[d,dotted,"\rincfun{i}"]\catsk i {(-)}\times {(-)}\\
%     (-)\times \catsk i {(-)}\ar[r,dotted,"\lincfun{i}"']&(-)\funny_i (-)
%   \end{tikzcd}.
% \]
% Given $i < n$ and two $n$\precategories~$C$ and $D$, by the universal property
% of pushouts, the morphisms $\cucatsk{C}' \times D\co \catsk i C \times D \to
% \catsk {i+1} C \times D$ and $C \times \cucatsk{D}'\co C \times \catsk i D \to C
% \times \catsk {i+1} D$ induce a morphism $\funup{i}\co C \funny_i D \to C
% \funny_{i+1} D$. In the following, we write $\funup{i,j}$ for $\funup{j-1}
% \circ \cdots \circ \funup{i}$.

% \subsection{Precategories are cartesian closed}
% \todo{vérifier si on a résolu le problème ou pas, sinon import dispo}
% \input{precat-closed}

\input{funny-mod}

\subsection{Prepolygraphs}
\label{ssec:pol}
In this section, we introduce the notion of \emph{prepolygraph} which
generalizes in arbitrary dimension the notion of rewriting system. This
definition is an adaptation to precategories of the notion of polygraph introduced by Burroni for
strict categories~\cite{burroni1993higher}. Polygraphs were also generalized by Batanin to
algebras of any finitary monad on globular sets~\cite{batanin1998computads}, and prepolygraphs
are a particular instance of this construction, for which we provide rather here an explicit construction.

For $n \ge 0$, writing $\fgf n$ for the canonical forgetful functor $\nPCat n
\to \nGlob n$, we define the category $\nPCat n^+$ as the pullback
\[
  \begin{tikzcd}[sep=3ex]
    \nPCat n^+
    \ar[d,"{\fgfv n}"',dotted]
    \ar[r,"{\fgf n^+}",dotted]&
    \nGlob{n+1}\ar[d,"\truncf {n+1} n"]\\
    \nPCat n\ar[r,"\fgf n"']&\nGlob n
  \end{tikzcd}
\]
and write $\fgf n^+\co \nPCat n^+ \to \nGlob{n+1}$ for the top arrow of the
pullback and $\fgfv n\co \nPCat n^+ \to \nPCat n$ for the left arrow. An object
$(C,C_{n+1})$ of $\nPCat n^+$ consists of an $n$-precategory~$C$ equipped with a
set~$C_{n+1}$ of $(n{+}1)$-cells and two maps $\gsrc_n,\gtgt_n\co C_{n+1}\to
C_n$ (note however that there is no notion of composition for $(n{+}1)$-cells).
There is a functor~$\fgfw n\co \nPCat {n+1} \to \nPCat n^+$ defined as the
universal arrow
\[
  \begin{tikzcd}[sep=3ex]
    \nPCat {n+1}
    \ar[rrd,bend left,"\fgf{n+1}"] \ar[ddr,bend right,"\truncf{n+1}{n}"']
    \ar[rd,dotted,"\fgfw n"'] \\
    & \nPCat n^+
    \ar[d,"\fgfv n"']
    \ar[r,"\fgf n^+"]&\nGlob{n+1}\ar[d,"\truncf {n+1} n"]\\
    & \nPCat n\ar[r,"\fgf n"']&\nGlob n
  \end{tikzcd}
\]
and, since categories and functors in the above diagram are induced by finite
limit sketches and morphisms of finite limit sketches, they are all right
adjoints (see~\cite[Thm.~4.1]{barr2005topos} for instance), so that $\fgfw n$
admits a left adjoint $\freefl n\co \nPCat n^+ \to \nPCat {n+1}$.

We define the category $\nPol n$ of $n$\polygraphs together with a functor
$\freef n\co \nPol n \to \nPCat n$ by induction on~$n$. We define $\nPol 0 = \Set$
and take $\freef 0$ to be the identity functor. Now suppose that $\nPol n$ and
$\freef n$ are defined for $n \ge 0$. We define $\nPol {n+1}$ as the pullback
\[
  \begin{tikzcd}[row sep=3ex]
    \nPol{n+1}\ar[d,"\truncf {n+1} n"',dotted]\ar[r,"{\freef n^+}",dotted]
    &\nPCat n^+\ar[d,"\fgfv n"]\\
    \nPol{n}\ar[r,"\freef n"']&\nPCat n
  \end{tikzcd}
\]
and write $\freef n^+\co \nPol{n+1} \to \nPCat n^+$ for the top arrow and
$\truncf {n+1} n$ for the left arrow of the diagram. Finally, we define
$\freef {n+1}$ as $\freefl n \circ \freef n^+$.

More explicitly, an $(n{+}1)$-prepolygraph~$\P$ consists in a diagram of sets
\[
  \begin{tikzcd}[column sep=10ex,labels={inner sep=0.5pt}]
    \P_0\ar[d,"\polinj0"{inner sep=2pt}]
    &\P_1
    \ar[dl,shift right,"\gsrc_0"',pos=0.3]
    \ar[dl,shift left,"\gsrc_0",pos=0.3]\ar[d,"\polinj1"{inner sep=2pt}]
    &\P_2\ar[dl,shift right,"\gsrc_1"',pos=0.3]\ar[dl,shift left,"\gsrc_1",pos=0.3]\ar[d,"\polinj2"{inner sep=2pt}]
    &\ldots
    &\P_{n}
    \ar[dl,shift right,"\gsrc_{n-1}"',pos=0.3]
    \ar[dl,shift
    left,"\gsrc_{n-1}",pos=0.3]\ar[d,"\polinj{n}"{inner sep=2pt}]
    &\P_{n+1}\ar[dl,shift right,"\gsrc_{n}"',pos=0.3]\ar[dl,shift left,"\gsrc_{n}",pos=0.3]\\
    \freecat{\P_0}
    &
    \freecat{\P_1}
    \ar[l,shift right,"{\csrc_0}"']
    \ar[l,shift
    left,"{\ctgt_0}"]
    &\ldots
    \ar[l,shift
    right,"{\csrc_1}"']\ar[l,shift
    left,"{\ctgt_1}"]
    &\freecat{\P_{n-1}}
    &\ar[l,shift
    right,"{\csrc_{n-1}}"']\ar[l,shift
    left,"{\ctgt_{n-1}}"]\freecat{\P_{n}}
  \end{tikzcd}
\]
such that ${\csrc_i}\circ\gsrc_{i+1}={\csrc_i}\circ\gtgt_{i+1}$ and
$\ctgt_i\circ\gsrc_{i+1}=\ctgt_i\circ\gtgt_{i+1}$, together with a
structure of $n$\precategory on the globular set on the bottom row: $\P_i$ is
the set of \emph{$i$-generators}, $\gsrc_i,\gtgt_i\co \P_{i+1}\to\freecat{\P_i}$
respectively associate to each $(i{+}1)$-generator its \emph{source} and
\emph{target}, and $\freecat{\P_i}$ is the set of \emph{$i$-cells}, \ie formal
compositions of $i$-generators. In line with the latter notation, we will often
write $\freecat\P$ for the image of an $n$\prepolygraph~$\P$ by $\freef n$.

By definition, an $(n{+}1)$\prepolygraph~$\P$ has an underlying
$n$\prepolygraph~$\truncf {n+1} n(\P)$, that we will often denote by $\restrict
n \P$. More generally, for $m \ge n$, an $m$\prepolygraph~$\P$ has an underlying
$n$\prepolygraph $\restrict n \P$ obtained by applying successively the
forgetful functors $\truncf {i+1} i$ for $m > i \ge n$.

\begin{example}
  \label{ex:pseudo-monoid-3-pol}
  We define the \emph{$3$\prepolygraph~$\P$ for pseudomonoids} as follows. We put
  \begin{align*}
    \P_0&= \set{x},
    &
    \P_1&= \set{a\co x \to x},
    &
    \P_2&= \set{\mu\co \bar 2 \To \bar 1, \eta\co \bar 0 \To \bar 1},
  \end{align*}
  where, given $n\in\N$, we write $\bar n$ for the composite
  $a \comp_0 \cdots \comp_0 a$ of $n$ copies of~$a$, and we define $\P_3$ as the
  set with the following three elements
  \[
    \setlength\arraycolsep{0pt}
    \begin{array}{r@{\ }c@{\ }r@{\ }c@{\ }l}
      \monA&\co &(\mu \comp_0 \bar 1) \comp_1 \mu &\TO& (\bar 1 \comp_0 \mu) \comp_1 \mu \\
      \monL&\co &(\eta \comp_0 \bar 1) \comp_1 \mu
                                                                 &\TO& \unit{\bar 1} \\
      \monR&\co &(\bar 1 \comp_0 \eta) \comp_1 \mu &\TO& \unit{\bar 1}
    \end{array}
    \ \zbox.
  \]
  Note that we make use of the arrows $\to$, $\To$ and $\TO$ to indicate
  the source and target of each $i$\generator for $i\in \set{1,2}$: $a$ is a
  $1$\generator such that $\gsrc_0(a) = \gtgt_0(a) = x$, $\mu$ is a
  $2$\generator such that $\gsrc_1(\mu) = a \comp_0 a$ and $\gtgt_1(\mu) = a$,
  and so on. In the following, we will keep using this notation to describe the
  generators of other prepolygraphs.
\end{example}

\subsection{Presentations}
\label{ssec:presentations}
\input{pres}

\subsection{Freely generated cells}
\label{ssec:freely-generated-cells}
Given $(C,C_{n+1})\in \nPCat n^+$, we
give an explicit description of the free $(n{+}1)$\precategory~$\freefl n(C,C_{n+1})$ it
generates, similar to the one given in~\cite{metayer2008cofibrant} in the case of polygraphs.
This $(n{+}1)$-precategory has~$C$ as underlying $n$-precategory so
that we focus on the description of the $(n{+}1)$-cells, which can be described
as equivalence classes of terms, called here \emph{expressions}, corresponding to formal composites of cells. These expressions are
defined inductively as follows:
\begin{itemize}
\item for every element $u\in C_{n+1}$, there is an expression, still noted $u$,
\item for every $n$-cell $u\in C_n$, there is an expression $\unit u$,
\item for every $0\leq i<n$, for every $u\in C_{i+1}$ and every expression $v$,
  there is an expression $u\comp_i v$,
\item for every $0\leq i<n$, for every expression $u$ and every $v\in C_{i+1}$,
  there is an expression $u\comp_i v$,
\item for every pair of expressions $u$ and $v$, there is an expression $u\comp_n v$.
\end{itemize}
We then define \emph{well-typed expressions} through typing rules in a sequent
calculus. We consider judgments of the form
\begin{itemize}
\item $\vdash t\co u\to v$, where $t$ is an expression and $u,v\in C_n$, with the
  intended meaning that the expression $t$ has $u$ as source and $v$ as target,
\item $\vdash t=t'\co u\to v$, where $t$ and $t'$ are expressions and $u,v\in C_n$,
  with the intended meaning that $t$ and $t'$ are equal expressions from~$u$
  to~$v$.
\end{itemize}
The associated typing rules are
\begin{itemize}
\item for every $t\in C_{n+1}$ with $\partial_n^-(t)=u$ and $\partial_n^+(t)=v$,
  \[
    \inferrule{ }{\vdash t\co u\to v}
  \]
\item for every $u\in C_n$,
  \[
    \inferrule{ }{\vdash\unit u\co u\to u}
  \]
\item for every $0\leq i<n$, every $u\in C_{i+1}$ and $v,v'\in C_n$ with $\ctgt_i(u)=\csrc_i(v)$,
  \[
    \inferrule{
      \vdash t\co v\to v'
    }{
      \vdash u\comp_i t\co (u\comp_iv)\to(u\comp_iv')
    }
  \]
\item for every $0\leq i<n$, every $u,u' \in C_n$ and $v\in C_{i+1}$ with
  $\ctgt_i(u)=\csrc_i(v)$
  \[
    \inferrule{
      \vdash t\co u\to u'
    }{
      \vdash t\comp_i v\co (u\comp_iv)\to(u'\comp_iv)
    }
  \]
\item and, for every $u,v,w \in C_n$
  \[
    \inferrule{
      \vdash t\co u\to v
      \\
      \vdash t'\co v\to w
    }{
      \vdash t\comp_n t'\co u\to w
    }
  \]
\end{itemize}
The equality rules, which express different desirable properties of the equality
relation, are introduced below. The first rules enforce that equality is an
equivalence relation:
\begin{equation}
  \label{eq:expr-equiv}
  \inferrule{\vdash t\co u\to v}{\vdash t=t\co u\to v}
  \qquad
  \inferrule{\vdash t=t'\co u\to v}{\vdash t'=t\co u\to v}
  \qquad
  \inferrule{\vdash t=t'\co u\to v\\\vdash t'=t''\co u\to v}{\vdash t=t''\co u\to v}
\end{equation}
The next ones express that that identities are neutral elements for composition:
\begin{gather*}
  \inferrule{\vdash t\co u\to v}{\vdash\unit u\comp_n t=t\co u\to v}
  \qquad
  \inferrule{\vdash t\co u\to v}{\vdash t\comp_n\unit v=t\co u\to v}
  \\
  \inferrule{
    \vdash t \co u \to u' \\
    i < n
  }
  {
    \vdash \unitp{i+1}{\csrc_i(u)} \comp_i t = t \co u \to u'
  }
  \qquad
  \inferrule{
    \vdash t \co u \to u'
  }
  {
    \vdash t \comp_i \unitp{i+1}{\ctgt_i(u)} = t \co u \to u'
  }
\end{gather*}
The next ones express that composition is associative:
{\displayskipforlongtable\begin{longtable}{c}
  \inferrule{
    \vdash t_1\co u_0\to u_1
    \\
    \vdash t_2\co u_1\to u_2
    \\
    \vdash t_3\co u_2\to u_3
  }{
    \vdash (t_1\comp_n t_2)\comp_n t_3=t_1\comp_n(t_2\comp_n t_3)\co u_0\to u_3
  }
  \crjot
  \inferrule{
    \vdash t \co v \to v' \\
    u_1,u_2 \in C_{i+1} \\
    \ctgt_i(u_1) = \csrc_i(u_2) \\
    \ctgt_i(u_2) = \csrc_i(v)
  }
  {\vdash u_1 \comp_i (u_2 \comp_i t) = (u_1 \comp_i u_2) \comp_i t \co u_1
    \comp_i u_2 \comp_i v \to u_1 \comp_i u_2 \comp_i v'}
  \crjot
  \inferrule{
    \vdash t \co u \to u' \\
    v_1,v_2 \in C_{i+1} \\
    \ctgt_i(u) = \csrc_i(v_1) \\
    \ctgt_i(v_1) = \csrc_i(v_2)
  }
  {\vdash (t \comp_i v_1) \comp_i v_2 = t \comp_i (v_1 \comp_i v_2)  \co u
    \comp_i v_1 \comp_i v_2 \to u' \comp_i v_1 \comp_i v_2}
  \crjot
  \inferrule{
    \vdash t \co v \to v' \\
    i < n \\
    u \in C_{i+1} \\
    w \in C_{i+1} \\
    \ctgt_i(u) = \csrc_i(v) \\
    \ctgt_i(v') = \csrc_i(w) } { \vdash (u \comp_i t) \comp_i w = u \comp_i (t
    \comp_i w) \co u \comp_i v \comp_i w \to u \comp_i v' \comp_i w }
\end{longtable}
}\noindent The next ones express that $(n{+}1)$-identities are compatible with low-dimensional compositions:
\begin{gather*}
  \inferrule{
    i < n \\
    u \in C_{i+1} \\
    v \in C_{n} \\
    \ctgt_i(u) = \csrc_i(v)
  }
  {
    \vdash u \comp_i \unit v = \unit {u \comp_i v}\co u \comp_i v \to u \comp_i v
  }
  \\
  \inferrule{
    u \in C_{n} \\
    i < n \\
    v \in C_{i} \\
    \ctgt_i(u) = \csrc_i(v)
  }
  {
    \vdash \unit u \comp_i v = \unit {u \comp_i v}\co u \comp_i v \to u \comp_i v
  }
\end{gather*}
The next ones express that $n$-compositions are compatible with low dimensional compositions:
\begin{gather*}
  \inferrule{
    \vdash t_1 \co v_1 \to v_2 \\
    \vdash t_2 \co v_2 \to v_3 \\
    u \in C_{i+1} \\
    \ctgt_i(u) = \csrc_i(v_1)
  }
  {
    \vdash u \comp_i (t_1 \comp_n t_2) = (u \comp_i t_1) \comp_n (u \comp_i t_2)
    \co u \comp_i v_1 \to u \comp_i v_3
  }
  \\
  \inferrule{
    \vdash t_1 \co u_1 \to u_2 \\
    \vdash t_2 \co u_2 \to u_3 \\
    v \in C_{i+1} \\
    \ctgt_i(u_1) = \csrc_i(v)
  }
  {
    \vdash (t_1 \comp_n t_2) \comp_i v = (t_1 \comp_i v) \comp_n (t_2 \comp_i v)
    \co u_1 \comp_i v \to u_3 \comp_i v
  }
\end{gather*}
The next ones express the distributivity properties between the different
low-dimensional compositions:
\begin{gather*}
  \inferrule{
    \vdash t \co w \to w' \\
    i < j < n \\
    u \in C_{i+1} \\
    \ctgt_i(u) = \csrc_i(w) \\
    v \in C_{j+1} \\
    \ctgt_j(v) = \csrc_j(w)
  }
  {
    \vdash u \comp_i (v \comp_j t) = (u \comp_i v) \comp_j (u \comp_i t) \co u
    \comp_i (v \comp_j w) \to u \comp_i (v \comp_j w')
  }
  \\
  \inferrule{
    \vdash t \co v \to v' \\
    i < j < n \\
    u \in C_{i+1} \\
    \ctgt_i(u) = \csrc_i(v) \\
    w \in C_{j+1} \\
    \ctgt_j(v) = \csrc_j(w)
  }
  {
    \vdash u \comp_i (t \comp_j w) = (u \comp_i t) \comp_j (u \comp_i w) \co u
    \comp_i (v \comp_j w) \to u \comp_i (v' \comp_j w)
  }
  \\
  \inferrule{
    \vdash t \co v \to v' \\
    i < j < n \\
    u \in C_{j+1} \\
    \ctgt_j(u) = \csrc_j(v) \\
    w \in C_{i+1} \\
    \ctgt_i(v) = \csrc_i(w)
  }
  {
    \vdash (u \comp_j t) \comp_i w = (u \comp_i w) \comp_j (t \comp_i w) \co (u
    \comp_j v) \comp_i w \to (u \comp_j v') \comp_i w
  }
  \\
  \inferrule{
    \vdash t \co u \to u' \\
    i < j < n \\
    v \in C_{j+1} \\
    \ctgt_j(u) = \csrc_j(v) \\
    w \in C_{i+1} \\
    \ctgt_i(v) = \csrc_i(w)
  }
  {
    \vdash (t \comp_j v) \comp_i w = (t \comp_i w) \comp_j (v \comp_i w) \co (u
    \comp_j v) \comp_i w \to (u' \comp_j v) \comp_i w
  }
\end{gather*}
Finally, the last ones express that equality is contextual:
{\displayskipforlongtable\begin{longtable}{c}
  \inferrule{
    \vdash t = t' \co v \to v' \\
    u \in C_{i+1} \\
    \ctgt_i(u) = \csrc_i(v)
  }
  {
    \vdash u \comp_i t = u \comp_i t'\co u \comp_i v \to u \comp_i v'
  }
  \crjot
  \inferrule{
    \vdash t = t' \co u \to u' \\
    v \in C_{i+1} \\
    \ctgt_i(u) = \csrc_i(v)
  }
  {
    \vdash t \comp_i v = t' \comp_i v\co u \comp_i v \to u' \comp_i v
  }
  \crjot
  \inferrule{
    \vdash t_1 = t_1' \co u_1 \to u_2 \\
    \vdash t_2  \co u_2 \to u_3
  }
  {
    \vdash t_1 \comp_n t_2 = t_1' \comp_n t_2 \co u_1 \to u_3
  }
  \crjot
  \inferrule{
    \vdash t_1 \co u_1 \to u_2 \\
    \vdash t_2 = t_2' \co u_2 \to u_3
  }
  {
    \vdash t_1 \comp_n t_2 = t_1 \comp_n t_2' \co u_1 \to u_3
  }
\end{longtable}
}\noindent The following lemmas show that typing is unique and well-behaved regarding
equality. They are easily shown by inductions on the derivations:
\begin{lem}[Uniqueness of typing]
  \label{lem:typing-unique}
  Given an expression $t$ such that the judgements~${\vdash t\co u\to v}$ and ${\vdash t\co u'\to v'}$
  are derivable, we have $u=u'$ and $v=v'$.
\end{lem}
\begin{lem}
  \label{lem:eq-implies-typing}
  If $\vdash t=t'\co u\to v$ is derivable then $\vdash t\co u\to v$ and
  $\vdash t'\co u\to v$ are derivable.
\end{lem}
% \begin{proof}
%   By induction on the derivations.
% \end{proof}
% Finally, the last lemma allows quotienting the terms by $=$ to obtain an
% $(n{+}1)$\precategory:
% \begin{lem}
%   \label{lem:eq-congurence}
%   The relation which identifies two terms $t$ and $t'$ such that
%   $\vdash t=t'\co u\to v$ is derivable is a congruence.
% \end{lem}
% \todo{supprimer ce lemme et juste faire une remarque}
% \begin{lem}
%   \label{lem:eq-congurence}
%   Let $t$ and $t'$ such that $\vdash t=t'\co u\to v$ is derivable. Then
%   \begin{enumerate}[label=(\roman*),ref=(\roman*)]
%   \item $\vdash a \comp_i t = a \comp_i t' \co a \comp_i u \to a \comp_i v$ for $i < n$ and $a \in C_{i+1}$ such
%     that $\ctgt_i(a) = \csrc_i(u)$,
%   \item $\vdash t \comp_i b = t' \comp_i b \co u \comp_i b \to v \comp_i b$ for $i < n$ and $b \in C_{i+1}$ such
%     that $\ctgt_i(u) = \csrc_i(b)$,
%   \item $\vdash \tilde t \comp_n t = \tilde t \comp_n t' \co \tilde u \to v$ for
%     $\tilde t$ and $\tilde u$ such that $\vdash \tilde t \co \tilde u \to u$ is derivable,
%   \item $\vdash t \comp_n \tilde t = t' \comp_n \tilde t \co u \to \tilde v$ for
%     $\tilde t$ and $\tilde v$ such that $\vdash \tilde t \co v \to \tilde v$ is derivable.
%   \end{enumerate}
% \end{lem}
% \begin{proof}
%   The axioms are already present above.
% \end{proof}
% \todo{ça veut dire quoi congruence ici ? clash avec l'autre
%   utilisation}
\noindent
A term $t$ is \emph{well-typed} if there are $u,v\in
C_n$ such that $\vdash t\co u\to v$ is derivable using the above rules. In this
case, by \Lemr{typing-unique}, the types $u$ and $v$ are uniquely determined by~$t$,
and we write $\csrc_n(t)=u$ and $\ctgt_n(t)=v$.
We define $C_{n+1}^*$ to be the set of equivalence classes under~$=$ of
well-typed expressions. By \Lemr{eq-implies-typing}, the operations $\csrc_n$
and $\ctgt_n$ are compatible with the relation~$=$. We finally define $\freefl
n(C,C_{n+1})$ as the $(n{+}1)$-precategory with~$C$ as underlying
$n$-precategory, $C_{n+1}^*$ as set of $(n{+}1)$-cells, with sources and targets
given by the maps $\csrc_n$ and $\ctgt_n$. The compositions and identities on
the $(n{+}1)$-cells are induced in the expected way by the corresponding
syntactic constructions (this is well-defined by the axioms of~$=$). It is
routine to verify that:
\begin{theo}
  The above construction defines a functor~$\freefl n$ which is left adjoint
  to~$\fgfw n$.
\end{theo}

\subsection{Normal form for cells}
\label{ssec:cell-nf}
Suppose given $(C,C_{n+1})\in\nPCat n^+$. The set $C_{n+1}^*$ of cells of
$\freefl n(C,C_{n+1})$ was described in the previous section as a quotient of
expressions modulo a congruence $=$. In order to conveniently work with its
equivalence classes, we introduce here a notion of normal form for those. From
now on, we adopt the convention that missing parenthesis in expressions are
implicitly bracketed on the right, \ie we write $u_1\comp_n
u_2\comp_n\cdots\comp_n u_k$ instead of $u_1\comp_n(u_2\comp_n(\cdots\comp_n
u_k))$.

By removing the relations~\eqref{eq:expr-equiv} in the definition of the
congruence $=$ and orienting from left to right the remaining equations, we
obtain a relation $\tred$ which can be interpreted as a rewriting relation on
expressions:
\begin{align*}
  \unit u \comp_n t &\tred t & t \comp_n \unit u &\tred t \\
  (t_1 \comp_n t_2) \comp_n t_3 & \tred t_1 \comp_n (t_2 \comp_n t_3) & (u_1 \comp_i t_n) \comp_i u_2 &\tred u_1 \comp_i (t_n \comp_i u_2) \\
  &\cdots & &\cdots
\end{align*}
We now study the properties of $\tred$. We recall that such a relation is said
to be \emph{terminating} when there is no infinite sequence $(t_i)_{i \ge 0}$
such that $t_i \tred t_{i+1}$ for $i \ge 0$. A \emph{normal form} is an
expression $t$ such that there exists no $t'$ with $t \tred t'$. Writing
$\tred^*$ for the reflexive transitive closure of $\tred$, the relation $\tred$
is said \emph{locally confluent} when for all expressions $t$, $t_1$ and $t_2$ such that
$t \tred t_1$ and $t \tred t_2$, we have $t_1 \tred^* t'$ and $t_2 \tred^* t'$
for some expression $t'$ (diagram on the left) and \emph{confluent} when for all expressions
$t$, $t_1$ and $t_2$ such that $t \tred^* t_1$ and $t \tred^* t_2$, we have $t_1 \tred^*
t'$ and $t_2 \tred^* t'$ for some expression $t'$ (diagram on the right):
\[
  \begin{tikzcd}[sep=small]
    &\ar[dl,Rightarrow]t\ar[dr,Rightarrow]&\\
    t_1\ar[dr,dotted,Rightarrow,"*"']&&\ar[dl,dotted,Rightarrow,"*"]t_2\\
    &t'&
  \end{tikzcd}
  \qquad\qquad\qquad\qquad
  \begin{tikzcd}[sep=small]
    &\ar[dl,Rightarrow,"*"']t\ar[dr,Rightarrow,"*"]&\\
    t_1\ar[dr,dotted,Rightarrow,"*"']&&\ar[dl,dotted,Rightarrow,"*"]t_2\\
    &t'&
  \end{tikzcd}
\]
Those notions are introduced in more details in~\cite{baader1999term}.

\begin{lem}
  \label{lem:precat-terminating}
  The relation $\tred$ is terminating.
\end{lem}
\begin{proof}
  In order to show termination, we define a measure on the terms that is
  decreased by each rewriting operation. To do so, we first define counting
  functions $c_n$ and $l_i,r_i$ for $0 \le i < n$ from expressions to $\N$ that
  take into account the three kinds of operations in the expression: top
  $n$-dimensional compositions, and lower $i$-dimensional left and right
  compositions. These functions count the numbers of potential reductions in an
  expression $t$ with the associated operations. Since reductions involving
  composition operations change value of counting functions of composition
  operations of lower dimension, we will use a lexicographical ordering of the
  counting functions to obtain the wanted measure. Given an expression $t$, we
  define $c_n(t) \in \N$ and $l_{i}(t),r_{i}(t) \in \N$ for $0 \le i < n$ by
  induction on $t$ as follows:
  \begin{itemize}
  \item if $g \in C_{n+1}$, we put $c_n(g) = l_{i}(g) = r_{i}(g) = 0$ for $0
    \le i < n$,
  \item if $u \in C_n$, we put $c_n(\unit u) = l_{i}(\unit u) =
    r_{i}(\unit u) = 1$,
  \item if $t = t_1 \comp_n t_2$, we put
    \begin{align*}
      c_n(t) &= 2c_n(t_1) +
               c_n(t_2) + 1,
      \\
      l_{i}(t) &= l_{i}(t_1) + l_{i}(t_2) + 2,  \\
      r_{i}(t) &= r_{i}(t_1) + r_{i}(t_2) + 2,
    \end{align*}
  \item if $t = u \comp_j t'$, we put $c_n(t) = c_n(t')$ and
    \begin{align*}
      l_{i}(t)&= 
                \begin{cases}
                  l_{i}(t') & \text{if $j < i$,} \\
                  2l_{i}(t') + 1 & \text{if $j = i$,} \\
                  l_{i}(t') + 1 & \text{if $j > i$,}
                \end{cases}
                            &
                              r_{i}(t)&=
                                        \begin{cases}
                                          r_{i}(t') & \text{if $j < i$,} \\
                                          r_{i}(t') + 1 & \text{if $j \ge i$,}
                                        \end{cases}    
    \end{align*}
  \item if $t = t' \comp_j v$, we put $c_{n}(t) = c_n(t')$ and
    \begin{align*}
      l_{i}(t)&=
                \begin{cases}
                  l_{i}(t') & \text{if $j \le i$,} \\
                  l_{i}(t') + 1 & \text{if $j > i$,}
                \end{cases}
                            &
                              r_{i}(t)&=
                                        \begin{cases}
                                          r_{i}(t') & \text{if $j < i$,} \\
                                          2r_{i}(t') + 1 & \text{if $j = i$,} \\
                                          r_{i}(t') + 1 & \text{if $j > i$.}
                                        \end{cases}    
    \end{align*}
  \end{itemize}
  For each expression $t$, we define
  \[
    N(t) = (c_n(t),l_{n-1}(t),r_{n-1}(t),\ldots,l_{0}(t),r_{0}(t)) \in \N^{2n+1}
  \]
  and consider the lexicographical ordering~$\ltlex$ on~$\N^{2n+1}$. For the
  inductive rules of $\tred$, we observe that
  \begin{itemize}
  \item if $t = t_1 \comp_n t_2$ and $t' = t'_1 \comp_n t_2$ with $N(t_1) \ltlex
    N(t'_1)$, then $N(t) \ltlex N(t')$,
    
  \item if $t = t_1 \comp_n t_2$ and $t' = t_1 \comp_n t'_2$ with $N(t'_2) \ltlex
    N(t_2)$, then $N(t') \ltlex N(t)$,
    
  \item if $t = u \comp_i \tilde t$ and $t' = u \comp_i \tilde t'$ with
    $N(\tilde t') \ltlex N(\tilde t)$, then $N(t') \ltlex N(t)$,
    
  \item if $t = \tilde t \comp_i v$ and $t' = \tilde t' \comp_i v$ with
    $N(\tilde t') \ltlex N(\tilde t)$, then $N(t') \ltlex N(t)$.
  \end{itemize}
  It is sufficient to prove that the other reduction rules decrease the norm
  $N(-)$. We only cover the most representative cases by computing the first
  component of $N(-)$ modified by the reduction rule and showing that it is
  strictly decreasing:
  \begin{align*}
    c_n(\unit u \comp_n t) &= c_n(t) + 3 > c_n(t), \\
    c_n((t_1 \comp_n t_2) \comp_n t_3) &= 4c_n(t_1) + 2c_n(t_2) + c_n(t_3) +
                                         3 \\
                           & > 2c_n(t_1) + 2c_n(t_2) + c_n(t_3) + 2 = c_n(t_1 \comp_n (t_2 \comp_n t_3)), \\
    l_{i}(u_1 \comp_i (u_2 \comp_i t)) &= 4l_{i}(t) + 3 > 2l_{i}(t) +
                                         1 = l_{i}((u_1 \comp_i u_2) \comp_i t), \\
    r_{i}((u_1 \comp_i t) \comp_i u_2) &= 2r_{i}(t) + 3 > 2r_{i}(t) + 2 = r_{i}(u_1 \comp_i (t \comp_i u_2)), \\
    l_{i}(u \comp_i (t_1 \comp_n t_2)) &= 2l_{i}(t_1) + 2l_{i}(t_2) +
                                         5 \\
                           &> 2l_{i}(t_1) + 2l_{i}(t_2) +
                                         4 = l_{i}((u \comp_i t_1) \comp_n (u \comp_i t_2)), \\
    l_{i}(u \comp_i (v \comp_j t)) &= 2l_{i}(t) + 3 > 2l_{i}(t) + 2 =
                                     l_{i}((u \comp_i v) \comp_j (u \comp_i t)) \text{ for $j > i$.}
  \end{align*}
  Thus, if $t \tred t'$, we have $N(t') \ltlex N(t)$. Since the lexicographical
  order~$\ltlex$ on~$\N^{2n+1}$ is well-founded, the reduction rule $\tred$ is
  terminating.
\end{proof}

\begin{lem}
  \label{lem:precat-locally-confluent}
  The relation $\tred$ is locally confluent.
\end{lem}

\begin{proof}
  By a direct adaptation of the critical pair lemma (for example
  \cite[Thm.~6.2.4]{baader1999term}), it is enough to show that all critical
  branchings are confluent, which can be checked by direct computation. For example, given
  $t_1$, $t_2$, $t_3$ and $t_4$ suitably typed, there is a critical branching given by the
  reductions
  \[
    (t_1 \comp_n (t_2 \comp_n t_3)) \comp_n t_4
    \tlred
    ((t_1 \comp_n t_2) \comp_n t_3) \comp_n t_4 \tred 
    (t_1 \comp_n t_2) \comp_n (t_3 \comp_n t_4).
  \]
  This branching is confluent since 
  \[
    (t_1 \comp_n (t_2 \comp_n t_3)) \comp_n t_4 \tred
    t_1 \comp_n ((t_2 \comp_n t_3) \comp_n t_4) \tred
    t_1 \comp_n (t_2 \comp_n (t_3 \comp_n t_4))
  \]
  and
  \[
    (t_1 \comp_n t_2) \comp_n (t_3 \comp_n t_4) \tred
    t_1 \comp_n (t_2 \comp_n (t_3 \comp_n t_4)).
  \]
  Another critical branching is given by the reductions
  \[
    (u_1 \comp_i u_2) \comp_i (t_1 \comp_n t_2) \tlred
    u_1 \comp_i (u_2 \comp_i (t_1 \comp_n t_2)) \tred
    u_1 \comp_i ((u_2 \comp_i t_1) \comp_n (u_2 \comp_i t_2))
  \]
  for $u_1,u_2 \in C_i$ with $i \le n$
  and $t_1,t_2$ suitably typed. This branching is confluent since 
  \[
    (u_1 \comp_i u_2) \comp_i (t_1 \comp_n t_2) \tred
    ((u_1 \comp_i u_2) \comp_i t_1) \comp_n ((u_1 \comp_i u_2) \comp_i t_2)
  \]
  and
  \begin{align*}
    u_1 \comp_i ((u_2 \comp_i t_1) \comp_n (u_2 \comp_i t_2)) &\tred
    (u_1 \comp_i(u_2 \comp_i t_1)) \comp_n (u_1 \comp_i (u_2 \comp_i t_2)) \\ &\tred
    ((u_1 \comp_iu_2) \comp_i t_1) \comp_n ((u_1 \comp_i u_2) \comp_i t_2).
  \end{align*}
  The other cases are similar.
\end{proof}

\begin{theo}
  \label{thm:precat-nf}
  Any cell in $u\in C_{n+1}^*$ admits a unique representative by an expression
  of the form
  \[
    u=u_1\comp_nu_2\comp_n\cdots\comp_n u_k
  \]
  where each $u_i$ decomposes as
  \begin{equation}
    \label{eq:precat-nf}
    u_i=v^i_{n}\comp_{n-1}(\cdots\comp_2(v^i_2\comp_1(v^i_1\comp_0 A^i\comp_0 w^i_1)\comp_1 w^i_2)\comp_2\cdots)\comp_{n-1} w^i_{n}
  \end{equation}
  where $A^i$ is an element of~$C_{n+1}$ and $v^i_j$ and $w^i_j$ are $j$-cells
  in $C_j$.
\end{theo}
\begin{proof}
  We have seen in \Lemr{precat-terminating} and \Lemr{precat-locally-confluent}
  that the relation $\tred$ is terminating and locally confluent. By Newman's
  lemma (see, for example, \cite[Lem.~2.7.2]{baader1999term}), it is thus confluent and every equivalence class of
  expressions contains a unique normal form, which can be obtained by reducing
  any expression to its normal form. It can be checked that those normal forms
  are in bijective correspondence with the expression of the form \eqref{eq:precat-nf}
  (essentially, those expressions are normal forms where identities have been
  suitably inserted).
  % we can verify \todo{développer ou pas ?} that they are essentially
  % expressions of the above form (excepting that identities are removed). By
  % adding the missing identities (that can be uniquely chosen by
  % well-typedness), we get the form of the statement.  \todo{discuter d'une
  % réorganisation de cette preuve}
\end{proof}
\noindent A cell of $\freecat C_{n+1}$ of the form~\eqref{eq:precat-nf} is
called a \emph{whisker}. By the inductive definition of prepolygraphs from
\Ssecr{pol} and \Thmr{precat-nf}, given an $m$\prepolygraph~$\P$ with $m > 0$,
an $(i{+}1)$-cell $u \in \freecat\P_{i+1}$ with $i \in \set{0,\ldots,m-1}$ can be
uniquely written as a composite of $(i{+}1)$\dimensional whiskers $u_1 \comp_i
\cdots \comp_i u_k$ for a unique $k \in \N$ that is called the \emph{length}
of~$u$ and denoted by~$\len{u}$. Moreover, each whisker $u_j$ admits a unique
decomposition of the form~\eqref{eq:precat-nf}. We will extensively use this
canonical form for cells of precategories freely generated by a prepolygraph in
the following, often omitting to invoke \Thmr{precat-nf}.
% \begin{coro}
  % \label{coro:precat-standard-cell}
  % Given $(C,C_{n+1}) \in \nPCat {n}^+$, $\freecat C$ is isomorphic to the
  % category~$D$ where $\restrictcat D n = C$ and $D_{n+1}$ are terms of the form
  % $\unit u\co u \to u$ for $u \in C_n$ and
  % \[
    % w_1 \comp_n \cdots \comp_n w_k
  % \]
  % for $k \ge 1$
  % with
  % \[
    % w_i = u_n \comp_{n-1} ( \cdots (u_1 \comp_0 g \comp_0 v_1)
    % \cdots) \comp_{n-1} v_n
  % \]
  % where $u_i,v_i \in C_i$ and $g \in C_{n+1}$ satisfying suitable source and
  % target conditions, and where compositions and identities are defined as expected.
% \end{coro}

\begin{example}
  Recall the $3$\prepolygraph of pseudomonoids~$\P$ from
  \Exr{pseudo-monoid-3-pol}. \Thmr{precat-nf} allows a canonical string diagram
  representation of the elements of~$\freecat\P_2$: first, we represent the
  $2$\generators $\mu$ and $\eta$ by $\satex{mu}$ and $\satex{eta}$
  respectively. Secondly, we represent the whiskers $\bar m \comp_0 \alpha
  \comp_0 \bar n$ for $m,n\in\N$ and $\alpha \in \P_2$ by adding $m$ wires on
  the left and $n$ wires on the right of the representation of $\alpha$. For
  example, $\bar 2 \comp_0 \mu \comp_0 \bar 3$ is represented by
  \[
    \satex{whisk}\pbox.
  \]
  Finally, a $2$-cell of~$\freecat\P_2$, which decomposes
  as a composite of whiskers $w_1 \comp_1 \cdots \comp_1 w_k$, is represented by
  stacking the representation the whiskers. For example, below are shown two
  $2$-cells with their associated graphical representation:
  \begin{align*}
    (0\comp_0\mu\comp_02)\comp_1(1\comp_0\mu\comp_00)\comp_1\mu&=\satex{mon-ex1}
    \\
    (2\comp_0\mu\comp_00)\comp_1(0\comp_0\mu\comp_01)\comp_1\mu&=\satex{mon-ex2}\pbox.
  \end{align*}
  Note that, contrary to $2$-cells of strict $2$-categories, these two $2$-cells
  are not equal in $\freecat\P_2$. The above graphical representation can be used
  in order to define unambiguously the source and target of $3$-cells. Here, the
  $3$-generators $\monA$, $\monL$, and $\monR$ can be described graphically by
  \[
    \setlength\myarrayintersep{\jot}% 
    \setlength\arraycolsep{0pt}
    \begin{array}{c@{\ }l@{\ }r@{\ }c@{\ }l}
      \monA&\co&\spacebelowsatex{mon-assoc-l} &\TO& \satex{mon-assoc-r} \\
      \monL&\co &\spacebelowsatex{mon-unit-l} &\TO& \satex{mon-unit-c} \\
      \monR&\co &\satex{mon-unit-r} &\TO& \satex{mon-unit-c}\pbox{\hspace*{1.5em}.}
    \end{array}
  \]
  \simon{centrer les targets de L et R}%
\end{example}
\subsection[\texorpdfstring{$(3,2)$}{(3,2)}-precategories]{$\bm{(3,2)}$-precategories}
\label{ssec:(3,2)-precategory}
In the following sections, we will mostly consider $3$\precategories that are
generated by $3$\prepolygraphs (as the one from \Exr{pseudo-monoid-3-pol}),
whose $3$\generators should moreover be thought as ``invertible operations''
(think of the $3$\generators $\monA,\monL,\monR$ of \Exr{pseudo-monoid-3-pol}). Thus, we
will in fact be dealing with $3$\precategories whose $3$\cells are all
invertible. Such $3$\precategories will usually be obtained by applying a
localization construction to the $3$\precategory~$\freecat\P$ for some
$3$\prepolygraph~$\P$, which is a direct
adaptation of the one for categories and described below.

Given a $3$\precategory~$C$, a $3$-cell~$F\co \phi\TO
\phi' \in C_3$ is \emph{invertible} when there exists $G\co \phi'\TO \phi$ such
that $F \comp_2 G = \unit{\phi}$ and $G \comp_2 F = \unit{\phi'}$. In this case,
$G$ is unique and we write it as $F^{-1}$. A \emph{$(3,2)$\precategory} is a
$3$\precategory where every $3$-cell is invertible. The $(3,2)$\precategories
form a full subcategory of~$\nPCat {3}$ denoted~$\nPCat {(3,2)}$.

There is a forgetful functor
\[
  \fginvf\co \nPCat{(3,2)} \to \nPCat {3}
\]
which admits a left adjoint $\freeinvf {(-)}$ also called \emph{localization
  functor} described as follows. Given a $3$\precategory~$C$, for every $F\co
\phi\To\phi' \in
C_3$, we write $F^+$ for a formal element of source $\phi$ and target
$\phi'$, and $F^-$ for a formal element of source $\phi'$ and target
$\phi$. A \emph{zigzag} of~$C$ is a list 
\begin{equation}
  \label{eq:zigzag-list}
  (F_1^{\eps_1},\ldots,F_k^{\eps_k})_{\phi,\phi'}
\end{equation}
for some $k \ge 0$, $F_1,\ldots,F_k \in C_{3}$ and $\eps_1,\ldots,\eps_k \in
\set{-,+}$ such that $\phi = \csrc(F_1^{\eps_1})$, $\phi' = \ctgt(F_k^{\eps_k})$
and $\ctgt(F_i^{\eps_i}) = \csrc(F_{i+1}^{\eps_{i+1}})$ for $1 \le i < k$ (there
is one empty list $()_{\phi,\phi}$ for each~$\phi \in \freecat\P_2$, by
convention). The source and the target of a zigzag as in~\eqref{eq:zigzag-list}
are $\phi$ and $\phi'$ respectively. Then, we define the truncation $\truncf 3
2(\freeinvf{C})$ as $\truncf 3 2(C)$ and $(\freeinvf{C})_{3}$ as the quotient of
the zigzags by the following equalities: for every
zigzag~$(F_1^{\eps_1},\ldots,F_k^{\eps_k})_{\phi,\phi'}$,
\begin{itemize}
\item if $F_i = \unit{\psi}$ for some $i \in \set{1,\ldots,k}$ and $\psi \in C_2$, then
  \[
    (F_1^{\eps_1},\ldots,F_k^{\eps_k})_{\phi,\phi'} =
    (F_1^{\eps_1},\ldots,F_{i-1}^{\eps_{i-1}},F_{i+1}^{\eps_{i+1}},\ldots,F_k^{\eps_k})_{\phi,\phi'},
  \]
\item if $\eps_i = \eps_{i+1} = +$ for some $i \in \set{1,\ldots,k-1}$, then
  \[
    (F_1^{\eps_1},\ldots,F_k^{\eps_k})_{\phi,\phi'} =
    (F_1^{\eps_1},\ldots,F_{i-1}^{\eps_{i-1}},(F_{i} \comp_2 F_{i+1})^{+},F_{i+2}^{\eps_{i+2}},\ldots,F_k^{\eps_k})_{\phi,\phi'},
  \]
\item if $\eps_i = \eps_{i+1} = -$ for some $i \in \set{1,\ldots,k-1}$, then
  \[
    (F_1^{\eps_1},\ldots,F_k^{\eps_k})_{\phi,\phi'} =
    (F_1^{\eps_1},\ldots,F_{i-1}^{\eps_{i-1}},(F_{i+1} \comp_2
    F_{i})^{-},F_{i+2}^{\eps_{i+2}},\ldots,F_k^{\eps_k})_{\phi,\phi'},
  \]
\item if $\set{\eps_i,\eps_{i+1}} = \set{-,+}$ and $F_i = F_{i+1}$ for some $i
  \in \set{1,\ldots,k-1}$, then
  \[
    (F_1^{\eps_1},\ldots,F_k^{\eps_k})_{\phi,\phi'} =
    (F_1^{\eps_1},\ldots,F_{i-1}^{\eps_{i-1}},F_{i+2}^{\eps_{i+2}},\ldots,F_k^{\eps_k})_{\phi,\phi'}.
  \]
\end{itemize}
Since the definitions of source and target of zigzags are compatible with the
above equalities, they induce source and target operations $\csrc,\ctgt\co
(\freeinvf{C})_{3} \to C_2$. Given
\[
  F = (F_1^{\eps_1},\ldots,F_k^{\eps_k})_{\phi_1,\phi_2} \in (\freeinvf{C})_{3}
  \qtand
  G = (G_1^{\delta_1},\ldots,G_l^{\delta_l})_{\phi_2,\phi_3} \in (\freeinvf{C})_{3},
\]
we define $F \comp_2 G$ as
\[
  F \comp_2 G = (F_1^{\eps_1},\ldots,F_k^{\eps_k},G_1^{\delta_1},\ldots,G_l^{\delta_l})_{\phi_1,\phi_3}
\]
and, given $i \in \set{0,1}$, $u \in C_{i+1}$ and $F =
(F_1^{\eps_1},\ldots,F_k^{\eps_k})_{\phi,\phi'}$ with $\ctgt_i(u) =
\csrc_i(\phi)$, we define $u \comp_i F$ as
\[
  u \comp_i F = ((u \comp_i F_1)^{\eps_1},\ldots,(u \comp_i F_k)^{\eps_k})_{u
    \comp_i \phi,u \comp_i \phi'}
\]
and, finally, given $\phi \in C_2$, we define $\unit \phi$ as $()_{\phi,\phi}$.
All these operations are compatible with the quotient equalities above, and they equip $\freeinvf{C}$ with a structure of $3$\precategory.

There is a canonical $3$\prefunctor $H\co C \to \freeinvf{C}$ sending $F\co\phi\TO\psi\in
C_3$ to $(F^+)_{\phi,\psi}$. Moreover, given a
$(3,2)$\precategory~$D$ and a $3$\prefunctor $G\co C \to D$, we can define
$G'\co \freeinvf{C} \to D$ by putting $G'(u) = G(u)$ for $u \in
C_i$ with $i \le 2$ and
\[
  G'((F_1^{\eps_1},\ldots,F_k^{\eps_k})_{\phi,\phi'}) = G'(F_1^{\eps_1}) \comp_2 \ldots
  \comp_2 G'(F_k^{\eps_k})
\]
for a zigzag $(F_1^{\eps_1},\ldots,F_k^{\eps_k})_{u,v}$ where
\[
  G'(F^\eps) =
  \begin{cases}
    G(F) & \text{if $\eps = +$} \\
    G(F)^{-1} & \text{if $\eps = -$}
  \end{cases}
\]
for $F \in C_{3}$ and $\eps \in \set{-,+}$. The definition of~$G'$ is compatible
with the quotient equalities above so that $G'$ is well-defined, and $G'$ can be shown to
uniquely factorize $G$ through $H$. Hence, $\freeinvf{(-)}$ is indeed a left
adjoint for $\fginvf$. In the following, given a $3$\precategory~$C$ and $F \in
C_3$, we often write $F$ for $H(F)$.

%% file: funny-mod.tex
\noindent We show that it equips $\nPCat n$ with a structure of monoidal
category. First, we prove several technical lemmas.
\begin{lem}
  \label{lem:precat-distrib}
  Given $n$\precategories $C$ and $(D^i)_{i \in I}$, the canonical morphism
  \[
    \coprod_{i\in I} (C \times D^i) \to C \times (\coprod_{i\in I} D^i) 
  \]
  is an isomorphism.
\end{lem}
\begin{proof}
  Write $F$ for this morphism. A morphism between $n$\precategories is an
  isomorphism if and only if the underlying morphism of globular sets is an
  isomorphism. Thus, it is sufficient to show that the isomorphism holds
  dimensionwise, \ie that the images of $F$ under the functors $(-)_j\co \nPCat
  n \to \Set$ are isomorphisms for $0 \le j \le n$. Products and coproducts are
  computed dimensionwise in $\nPCat n$, so that the functors $(-)_j$ preserve
  products and coproducts. Since coproducts distribute over products in $\Set$,
  $F_j$ is an isomorphism for $0 \le j \le n$, and so is~$F$.
\end{proof}
\begin{lem}
  \label{lem:times-units-preserve-colimits}
  Given an $n$\precategory $D$, the functor $(-) \times \catsk 0 D$ preserves colimits.
\end{lem}
\begin{proof}
  Since, by \Propr{trunc-incf-la-ra}, $\incf{0}{n}$ preserves limits and
  colimits, we have $\catsk 0 D \cong \coprod_{x \in D_0} \termcat$. Given
  a diagram $C(-)\co I \to \nPCat n$, by \Lemr{precat-distrib}, we have
  \[
    (\colim_{i\in I} C(i)) \times \catsk 0 D \cong \coprod_{x \in D_0}
    \colim_{i\in I} C(i)
    \cong \colim_{i\in I} \coprod_{x \in D_0} C(i)
    \cong \colim_{i\in I} (C(i) \times \catsk 0 D)\pbox.
    \qedhere
  \]
\end{proof}
\begin{lem}
  \label{lem:assoc-isom}
  Given $n$-precategories $C,D,E$, there is an isomorphism
  \[
    \funnyass_{C,D,E} \co (C \funny D) \funny E \xto\sim C \funny (D \funny E)
  \]
  natural in $C$, $D$ and $E$.
\end{lem}
\begin{proof}
  Given $n$\precategories $C$, $D$ and $E$, the precategory $(C \funny D) \funny E$ is defined by the pushout
  \[
    \begin{tikzcd}[column sep=6em,row sep=4em]
      (\catsk 0 C \times \catsk 0 D) \times \catsk 0 E
      \ar[r,"(\catsk 0 C \times
      \catsk 0 D) \times \cucatsk{E}"]
      \ar[d,"\cucatsk{C \funny D} \times \catsk
      0 E"'] &
      (\catsk 0 C \times \catsk 0 D) \times E \ar[d,dotted,"\rincfun{C \funny D,E}"]\\
      (C \funny D)
      \times \catsk 0 E \ar[r,dotted,"\lincfun{C \funny D,E}"'] & ( C
      \funny D) \funny E
    \end{tikzcd}
    \pbox.
  \]
  Since, by \Lemr{times-units-preserve-colimits}, $(-) \times \catsk 0 E$
  preserves colimits, the following diagram is also a pushout
  \[
    \begin{tikzcd}[column sep=6em,row sep=4em]
      (\catsk 0 C \times \catsk 0 D) \times \catsk 0 E
      \ar[r,"(\catsk 0 C \times \cucatsk{D}) \times \catsk 0 E"] \ar[d,"(\cucatsk{C} \times \catsk 0 D) \times \catsk 0 E"']
      & (\catsk 0 C \times D) \times \catsk 0 E
      \ar[d,dotted,"\rincfun{C,D} \times \catsk 0 E"]\\
      (C \times \catsk 0 D) \times \catsk 0 E
      \ar[r,dotted,"\lincfun{C,D} \times \catsk 0 E"'] & ( C \funny D)
      \times \catsk 0 E
    \end{tikzcd}
    \pbox.
  \]
  Thus, $(C \funny D) \funny E$ is the colimit of the diagram
  \begin{equation}
    \label{eq:funny-3-pushout}
    \begin{tikzcd}[column sep=9em,row sep=3em]
      & (C \times \catsk 0 D) \times \catsk 0 E \\
      (\catsk 0 C \times \catsk 0 D) \times \catsk 0 E
      \ar[r,"(\catsk 0 C \times \cucatsk D) \times \catsk 0 E"description]
      \ar[ur,bend left=15,"(\cucatsk C \times \catsk 0 D) \times \catsk 0 E"]
      \ar[dr,bend right=15,"(\catsk 0 C \times \catsk 0 D) \times \cucatsk E"'] & (\catsk 0 C \times D) \times \catsk 0 E \\
      & (\catsk 0 C \times \catsk 0 D)
      \times E
    \end{tikzcd}
  \end{equation}
  The precategory $C \funny (D \funny E)$ admits a similar diagram, and we
  deduce easily, using the associativity of~$\times$, a canonical morphism
  $\funnyass_{C,D,E}\co (C \funny D) \funny E \to C \funny (D \funny E)$, which
  admits an inverse defined symmetrically. The morphism~$\funnyass_{C,D,E}$ is easily checked to be natural in
  $C$, $D$ and $E$.
\end{proof}
Given an $n$\precategory, there are canonical morphisms
\[
  \funnyl_C\co \termcat \funny C \xto\sim C \qqtand \funnyr_C\co C \funny \termcat \xto\sim C
\]
where $\funnyl_C$ is defined by
\[
  \begin{tikzcd}[column sep=4em]
    \catsk 0 \termcat \times \catsk 0 C
    \ar[d,"\cucatsk \termcat \times \catsk 0 C"']
    \ar[r,"\catsk 0 \termcat \times \cucatsk C"]&
    \catsk 0 \termcat\times C\ar[d,"\rincfun{\termcat,C}"]
    \ar[ddr,bend left=15,"\pi_2"]\\
    \termcat\times \catsk 0 C
    \ar[r,"\lincfun{\termcat,C}"']
    \ar[rrd,bend right=12,"\cucatsk{C} \circ \pi_2"']&\termcat\boxtimes C\ar[dr,"\funnyl_C",dotted] \\
    & & C
  \end{tikzcd}
\]
and $\funnyr_C$ is defined similarly. Both are natural in~$C$. We can conclude
that:
% \begin{prop}
%   $(C,\funny,\termcat,\funnyass,\funnyl,\funnyr)$ is a structure of a monoidal category, that is,
%   the diagrams
%   \[
%     \begin{tikzcd}[column sep=small]
%       ((A \funny B) \funny C) \funny D \ar[r,"\funnyass"] \ar[dd,"\funnyass"']& (A \funny B) \funny (C
%       \funny D) \ar[d,"\funnyass"] \\
%       & A \funny (B \funny (C \funny D)) \\
%       (A \funny (B \funny C)) \funny D \ar[r,"\funnyass"']& A \funny ((B \funny C)
%       \funny D) \ar[u,"\funnyass"']
%     \end{tikzcd}
%     \begin{tikzcd}[sep=small]
%       (A \funny \termcat) \funny B \ar[rr,"\funnyass"] \ar[rd,"\funnyr"'] && A \funny (\termcat \funny B)\ar[ld,"\funnyl"] \\
%       & A \funny B
%     \end{tikzcd}
%   \]
%   are commutative.
% \end{prop}
\begin{prop}
  $(C,\funny,\termcat,\funnyass,\funnyl,\funnyr)$ is a monoidal category.
\end{prop}
\begin{proof}
  The axioms of monoidal categories follow from the pushout definition of the
  funny tensor product and the cartesian monoidal structure on
  $n$\precategories.
\end{proof}
\noindent In fact, the funny tensor product is a suitable product for an
inductive enriched definition of precategories, \ie
\begin{restatable}{prop}{precatequivalentdefsprop}
  \label{prop:enriched}
  There is an equivalence of categories between $(n{+}1)$\precategories and
  categories enriched in $n$\precategories with the funny tensor product.
\end{restatable}
\begin{proof}
  See \Appr{precat-equivalent-definitions}.
\end{proof}
%%% Local Variables:
%%% mode: latex
%%% TeX-master: "proxy-funny-mod"
%%% End:

%% file: pres.tex
Given an $n$\precategory~$C$ with $n > 0$, a \emph{congruence} for~$C$ is an
equivalence relation $\stdcong$ on~$C_n$ such that, for all $u,u' \in C_n$
satisfying $u \stdcong u'$,
\begin{itemize}
\item $\csrctgt\eps_{n-1}(u) = \csrctgt\eps_{n-1}(u')$ for $\eps \in \set{-,+}$,
  
\item for $v,w \in C_{i+1}$ with $0 \le i < n$ such that $v,u,w$ are
  $i$\composable, we have
  \[
    v \comp_i u \comp_i w \stdcong v \comp_i u' \comp_i w.
  \]
\end{itemize}
Given such a congruence for~$C$, there is an $n$\precategory~$C/{\stdcong}$
which is the $n$\precategory~$D$ such that $D_i = C_i$ for $i < n$ and $D_n =
C_n/{\stdcong}$ and where the identities and compositions are induced by the ones on~$C$.

Now, consider the composite functor
\[
  \begin{tikzcd}[cramped]
    \nPol{n+1}\ar[r,"\freef {n+1}"]&\nPCat{n+1}\ar[r,"\incfla {n+1}n"]&\nPCat{n}
  \end{tikzcd}
  \pbox.
\]
To an $(n{+}1)$-prepolygraph $\P$, it associates an $n$-precategory denoted by
$\prespcat\P$. Concretely, $\prespcat\P$ is isomorphic to
$\freecat{(\restrictcat\P n)}/{\stdcong^\P}$ where $\stdcong^\P$ is the
smallest congruence such that $\csrc_n(u) \stdcong^\P \ctgt_n(u)$ for $u \in
\P_{n+1}$. In the following, we say that an $(n{+}1)$\prepolygraph~$\P$ is a
\emph{presentation} of an $n$\precategory~$C$ when $C$ is isomorphic
to~$\prespcat\P$.
% Given $n \ge 1$, a \emph{presentation of $n$\precategory} is an
% $(n{+}1)$\prepolygraph~$\P$. Given such a presentation~$\P$, the
% \emph{$n$\precategory presented by $\P$} is the $n$\precategory~$C
% =\incfla{n+1}{n}(\P)$. More explicitly: let $\stdcong^\P$ be the smallest
% congruence\todo{définir congruence} on~$\freecat\P_n$ such that, for all $g\co s
% \to t \in \P_{n+1}$, $s \stdcong^\P t$. Then, $C$ can be defined as the $n$\precategory
% such that $\restrictcat{C}{n-1} = \restrictcat{(\freecat\P)}{n-1}$ and $C_n =
% (\freecat\P_n)_{/\stdcong^\P}$ with identity and composition operations as expected. We
% denote $\prespcat\P$ the presented $n$\precategory~$C$. \todo{exemples}
% \todo{déjà fait plus haut, fusionner}Since we are mainly concerned by
% $3$\precategories in what follows, we simply call \emph{presentations} the
% presentations of $3$\precategories.

%%% Local Variables:
%%% mode: latex
%%% TeX-master: "article"
%%% End:

%% file: gray.tex
Strict $3$\categories are categories enriched in the monoidal category $\nCat 2$
equipped with the cartesian product. Similarly, Gray categories are categories
enriched in the monoidal category $\nCat 2$ equipped with the Gray
tensor product. The latter can be seen as an ``asynchronous'' variant of the
cartesian product, similar to the funny tensor product, where two interleavings
of the same morphisms are related by ``exchange'' cells. Typically, consider the
$1$-categories $C$ and $D$ below
\begin{align*}
  C&=
  \begin{tikzcd}[ampersand replacement=\&]
    x\ar[r,"f"]\&x'
  \end{tikzcd}
  &
  D&=
  \begin{tikzcd}[ampersand replacement=\&]
    y\ar[r,"g"]\&y'
  \end{tikzcd}
\end{align*}
their funny and Gray tensor products are respectively
\begin{align*}
  C\funny D&=
  \begin{tikzcd}[sep=small,ampersand replacement=\&]
    (x,y)
    \ar[r,"{(f,y)}"]
    \ar[d,"{(x,g)}"']
    \&(x',y)
    \ar[d,"{(x',g)}"]
    \\
    (x,y')\ar[r,"{(f,y')}"']\&(x',y')
  \end{tikzcd}
  &
  C\grayy D&=
  \begin{tikzcd}[sep=small,ampersand replacement=\&]
    (x,y)
    \ar[r,"{(f,y)}"]
    \ar[d,"{(x,g)}"']
    \ar[rd,phantom,"\Downarrow\scriptstyle\chi"]
    \&(x',y)
    \ar[d,"{(x',g)}"]
    \\
    (x,y')\ar[r,"{(f,y')}"']\&(x',y')
  \end{tikzcd}
\end{align*}
where the exchange $2$-cell~$\chi$ can be invertible or not, depending on
whether we consider the pseudo or lax variant of the Gray tensor product. We
first recall quickly the definition of the Gray tensor product, both in
its lax and pseudo variants. We then give a more explicit description
in terms of generators and relations of categories enriched in $2$\categories
with the Gray tensor product. Then, we give a way for presenting canonically a
Gray category.

% \subsection{$2$-categories}
% % FASTMODE
% \input{gray-two-cats}

\subsection{The Gray tensor products}
\label{ssec:gray-tensor}
% FASTMODE
\input{gray-tens}

\subsection{Gray categories}
\label{ssec:gray-cat}
% FASTMODE
\input{gray-cat}

\subsection{Gray presentations}
\label{ssec:gray-presentation}
\input{gray-pres}

%%% Local Variables:
%%% mode: latex
%%% TeX-master: "article"
%%% End:

%  LocalWords:  associativity

%% file: gray-tens.tex
We recall here the definitions of the Gray tensor products on $2$\categories, in
its lax and pseudo variants. We refer the reader
to~\cite[Sec.~I,4]{gray2006formal} for details.

A (strict) \emph{$2$\category} is a $2$\precategory~$C$ such that, for all
$\phi,\psi \in C_2$ with $\ctgt_0(\phi) = \csrc_0(\psi)$,
\[
  (\phi \comp_0 \csrc_1(\psi)) \comp_1 (\ctgt_1(\phi) \comp_0 \psi)
  =
  (\csrc_1(\phi) \comp_0 \psi) \comp_1 (\phi \comp_0 \ctgt_1(\psi)).
\]
We denote $\nCat 2$ the full subcategory of $\nPCat 2$ whose objects are
$2$\categories. We write $\termcat$ for the terminal $2$-category and we write
${\ast}$ for its unique $0$-cell.
% We recall that $\nCat 2$ is equivalent to the category of categories enriched
% in $\Cat$ equipped with the cartesian product as tensor product.

Given $C$ and $D$ two $2$\categories, we write $C \graylax D$ for the
$2$\category which is presented as follows:
\begin{itemize}
\item the $0$-cells of $C \graylax D$ are the pairs $(x,y)$ where $x \in C_0$ and $y
  \in D_0$,
\item the $1$-cells of $C \graylax D$ are generated by $1$-cells
  \[
    (f,y)\co (x,y) \to (x',y) \qtand (x,g)\co (x,y) \to (x,y'),
  \]
  for $f\co x \to x' \in C_1$ and $g\co y \to y' \in C_2$,
\item the $2$-cells of $C \graylax D$ are generated by the $2$-cells
  \[
    (\phi,y)\co (f,y) \to (f',y) \qtand (x,\psi)\co (x,g) \to (x,g')
  \]
  for $\phi\co f\To f' \in C_2$, $\psi\co g\To g' \in C_2$ and $x,y \in C_0$,
  and by the $2$-cells
  % \[
    % (f,g)\co (f,y) \comp_0 (x',g) \to (x,g) \comp_0 (f,y')
  % \]
  \[
    \begin{tikzcd}
      (x,y)\ar[d,"{(x,g)}"']\ar[r,"{(f,y)}"]\ar[phantom,dr,"{\Downarrow\scriptstyle(f,g)}"]&(x',y)\ar[d,"{(x',g)}"]\\
      (x,y')\ar[r,"{(f,y')}"']&(x',y')
    \end{tikzcd}
  \]
  for $f\co x \to x' \in C_1$ and $g\co y \to y' \in C_1$,
\end{itemize}
under the conditions that
\begin{enumerate}[label=(\roman*),ref=(\roman*)]
\item the $1$-generators are compatible with $0$-composition, meaning that
  \begin{align*}
    (\unit x,y) &= (x,\unit y) = \unit{(x,y)}  \\
    (f \comp_0 f',y) &= (f,y) \comp_0 (f',y)\\
    (x,g \comp_0 g') &= (x,g) \comp_0 (x,g')
  \end{align*}
  for all $x \in C_0$, $y \in D_0$, $0$-composable $f,f' \in C_1$ and
  $0$-composable $g,g' \in D_1$,
\item the $2$-generators are compatible with $0$-composition, meaning that
  \begin{align*}
    (\unitp{2}{x},y) &= (x,\unitp{2}{y}) = \unitp{2}{(x,y)} \\
    (\phi_1 \comp_0 \phi_2, y) & = (\phi_1,y) \comp_0 (\phi_2,y) \\
    (x,\psi_1 \comp_0 \psi_2) &= (x,\psi_1) \comp_0 (x,\psi_2)
  \end{align*}
  for all $x \in C_0$, $y \in D_0$, $0$-composable $\phi_1,\phi_2 \in C_2$ and
  $0$-composable $\psi_1,\psi_2 \in D_2$, \ie graphically,
  \begingroup
  \displayskipforlongtable
  \begin{longtable}{C}
    \begin{spacebelowenv}
      \begin{tikzcd}[column sep=4em,cramped,ampersand replacement=\&]
        {(x,y)} \arrow[r, bend left=31,""{auto=false,name=src},"{(\unit x,y)}"]
        \arrow[r, bend right=31,""{auto=false,name=tgt},"{(\unit x,y)}"']
        \ar[r,phantom,"\Downarrow\scriptstyle{(\unitp 2 x,y)}"] \& {(x,y)}
        % \ar[from=src,to=tgt,phantom,"{\Downarrow\!\!(\phi_1 \comp_0
        % \phi_2,y)}"]
      \end{tikzcd}
    \end{spacebelowenv}
    =
    \begin{tikzcd}[column sep=4em,cramped,ampersand replacement=\&]
      {(x,y)} \arrow[r, bend left=31,""{auto=false,name=src},"{(x,\unit y)}"] \arrow[r, bend
      right=31,""{auto=false,name=tgt},"{(x,\unit y)}"']
      \ar[r,phantom,"\Downarrow\scriptstyle{(x,\unitp 2 y)}"]
      \& {(x,y)}
      % \ar[from=src,to=tgt,phantom,"{\Downarrow\!\!(\phi_1 \comp_0 \phi_2,y)}"]
    \end{tikzcd}
    =
    \begin{tikzcd}[column sep=4em,cramped,ampersand replacement=\&]
      {(x,y)} \arrow[r, bend left=31,""{auto=false,name=src},"{\unit{(x,y)}}"] \arrow[r, bend
      right=31,""{auto=false,name=tgt},"{\unit{(x,y)}}"']
      \ar[r,phantom,"\Downarrow\scriptstyle\unitp 2 {(x,y)}"]
      \& {(x,y)}
      % \ar[from=src,to=tgt,phantom,"{\Downarrow\!\!(\phi_1 \comp_0 \phi_2,y)}"]
    \end{tikzcd}
    \\
    \begin{spacebelowenv}
      \begin{tikzcd}[column sep=5em,cramped,ampersand replacement=\&]
        {(x_0,y)} \arrow[r, bend left=31,""{auto=false,name=src},"{(f_1 \comp_0
          f_2,y)}"] \arrow[r, bend right=31,""{auto=false,name=tgt},"{(f'_1
          \comp_0 f'_2,y)}"'] \ar[r,phantom,"{\Downarrow\scriptstyle(\phi_1 {\comp_0}
          \phi_2,y)}"] \& {(x_2,y)}
        % \ar[from=src,to=tgt,phantom,"{\Downarrow\!\!(\phi_1 \comp_0
        % \phi_2,y)}"]
      \end{tikzcd}
    \end{spacebelowenv}
    =
    \begin{tikzcd}[column sep=4em,cramped,ampersand replacement=\&]
      {(x_0,y)} \arrow[r, bend left=31,"{(f_1,y)}"] \arrow[r, bend right=31,"{(f'_1,y)}"']
      \ar[r,phantom,"{\Downarrow\scriptstyle(\phi_1,y)}"]
      \& {(x_1,y)} \arrow[r, bend left=31,"{(f_2,y)}"] \arrow[r, bend right=31,"{(f'_2,y)}"']
      \ar[r,phantom,"{\Downarrow\scriptstyle(\phi_2,y)}"]
      \& {(x_2,y)}
    \end{tikzcd}
    \\
    %\cr\noalign{\penalty-10000}%
    \begin{tikzcd}[column sep=5em,cramped,ampersand replacement=\&]
      {(x,y_0)} \arrow[r, bend left=31,""{auto=false,name=src},"{(x,g_1 \comp_0 g_2)}"] \arrow[r, bend
      right=31,""{auto=false,name=tgt},"{(x,g'_1 \comp_0 g'_2)}"']
      \ar[r,phantom,"{\Downarrow\scriptstyle(x,\psi_1 {\comp_0} \psi_2)}"]
      \& {(x,y_2)}
      % \ar[from=src,to=tgt,phantom,"{\Downarrow\!\!(\phi_1 \comp_0 \phi_2,y)}"]
    \end{tikzcd}
    =
    \begin{tikzcd}[column sep=4em,cramped,ampersand replacement=\&]
      {(x,y_0)} \arrow[r, bend left=31,"{(x,g_1)}"] \arrow[r, bend right=31,"{(x,g'_1)}"']
      \ar[r,phantom,"{\Downarrow\scriptstyle(x,\psi_1)}"]
      \& {(x,y_1)} \arrow[r, bend left=31,"{(x,g_2)}"] \arrow[r, bend right=31,"{(x,g'_2)}"']
      \ar[r,phantom,"{\Downarrow\scriptstyle(x,\psi_2)}"]
      \& {(x,y_2)}
    \end{tikzcd}
  \end{longtable}\todo{CHAQUE FOIS: vérifier l'espacement avec le bas de page}
  % sans la pénalité, dépasse sur le numéro de page en bas
  % avec la pénalité, peut faire des gros espaces verticaux si modifs au-dessus
  \endgroup
\item the $2$-generators are compatible with $1$-composition, meaning that
  \begin{align*}
    (\unit{f},y) &= \unit{(f,y)} \\
    (\phi_1 \comp_1 \phi_2, y) & = (\phi_1,y) \comp_1 (\phi_2,y) \\
    (x,\unit{g}) &= \unit{(x,g)}\\
    (x,\psi_1 \comp_1 \psi_2) &= (x,\psi_1) \comp_1 (x,\psi_2)
  \end{align*}
  for all $\phi_i\co f_{i-1} \To f_i \co x \to x'$ and $\psi_i\co g_{i-1} \To
  g_i \co y \to y'$ for $i \in \set{1,2}$ and $f\co x \to x'$ and $g\co y \to
  y'$, \ie graphically,
  \begingroup
  \displayskipforlongtable
  \begin{longtable}{CCC}
    \begin{spacebelowenv}
      \begin{tikzcd}[column sep=5em,cramped,ampersand replacement=\&]
        {(x,y)} \arrow[r, "{(f,y)}", bend left=31] \arrow[r, "{(f,y)}"', bend
        right=31] \ar[r,phantom,"\Downarrow\scriptstyle{(\unit f,y)}"] \& {(x',y)}
      \end{tikzcd}
    \end{spacebelowenv}
    &=&
    \begin{tikzcd}[column sep=5em,cramped,ampersand replacement=\&]
      {(x,y)} \arrow[r, "{(f,y)}", bend left=31] \arrow[r, "{(f,y)}"', bend
      right=31]
      \ar[r,phantom,"\Downarrow\scriptstyle\unit{(f,y)}"]
      \& {(x',y)}
    \end{tikzcd}
    \\
    \begin{spacebelowenv}
      \begin{tikzcd}[column sep=5em,cramped,ampersand replacement=\&]
        {(x,y)}
        % \arrow[r, "{(f_1,y)}" description,""{auto=false,name=arrx}]
        \arrow[r, "{(f_0,y)}",""{auto=false,name=arr}, bend left=31] \arrow[r,
        "{(f_2,y)}"',""{auto=false,name=arrxx},bend right=31]
        \ar[r,"{\Downarrow\scriptstyle(\phi_1 {\comp_1} \phi_2,y)}",phantom] \&
        {(x',y)}
      \end{tikzcd}
    \end{spacebelowenv}
    &=&
    \begin{tikzcd}[column sep=5em,cramped,ampersand replacement=\&]
      {(x,y)} \arrow[r, "{(f_1,y)}" description,""{auto=false,name=arrx}]
      \arrow[r, "{(f_0,y)}",""{auto=false,name=arr}, bend left=50] \arrow[r,
      "{(f_2,y)}"',""{auto=false,name=arrxx},bend right=50] \& {(x',y)}
      \ar[from=arr,to=arrx,phantom,"{\Downarrow\scriptstyle(\phi_1,y)}"]
      \ar[from=arrx,to=arrxx,phantom,"{\Downarrow\scriptstyle(\phi_2,y)}"]
    \end{tikzcd}
    \\
    \begin{spacebelowenv}
      \begin{tikzcd}[column sep=5em,cramped,ampersand replacement=\&]
        {(x,y)} \arrow[r, "{(x,g)}", bend left=31] \arrow[r, "{(x,g)}"', bend
        right=31] \ar[r,phantom,"\Downarrow\scriptstyle{(x,\unit g)}"] \& {(x,y')}
      \end{tikzcd}
    \end{spacebelowenv}
    &=&
    \begin{tikzcd}[column sep=5em,cramped,ampersand replacement=\&]
      {(x,y)} \arrow[r, "{(x,g)}", bend left=31] \arrow[r, "{(x,g)}"', bend
      right=31]
      \ar[r,phantom,"\Downarrow\scriptstyle\unit{(x,g)}"]
      \& {(x,y')}
    \end{tikzcd}
    \\
    \begin{tikzcd}[column sep=5em,cramped,ampersand replacement=\&]
      {(x,y)}
      %\arrow[r, "{(f_1,y)}" description,""{auto=false,name=arrx}]
      \arrow[r, "{(x,g_0)}",""{auto=false,name=arr}, bend left=31] \arrow[r,
      "{(x,g_2)}"',""{auto=false,name=arrxx},bend right=31]
      \ar[r,"{\Downarrow\scriptstyle(x,\psi_1 {\comp_1} \psi_2)}",phantom]
      \& {(x,y')}
    \end{tikzcd}
    &=&
    \begin{tikzcd}[column sep=5em,cramped,ampersand replacement=\&]
      {(x,y)} \arrow[r, "{(x,g_1)}" description,""{auto=false,name=arrx}]
      \arrow[r, "{(x,g_0)}",""{auto=false,name=arr}, bend left=50] \arrow[r,
      "{(x,g_2)}"',""{auto=false,name=arrxx},bend right=50] \& {(x',y)}
      \ar[from=arr,to=arrx,phantom,"{\Downarrow\scriptstyle(x,\psi_1)}"]
      \ar[from=arrx,to=arrxx,phantom,"{\Downarrow\scriptstyle(x,\psi_2)}"]
    \end{tikzcd}
  \end{longtable}
  \endgroup
\item the interchangers are compatible with $0$-composition, meaning that
  \begin{align*}
    (\unit x,g) &= \unit{(x,g)} \\
    (f_1 \comp_0 f_2,g) & = ((f_1,y) \comp_0 (f_2,g)) \comp_1 ((f_1,g) \comp_0 (f_2,y')) \\
    (f,\unit y) &= \unit{(f,y)} \\
    (f,g_1 \comp_0 g_2) & = ((f,g_1) \comp_0 (x',g_2)) \comp_1 ((x,g_1) \comp_0 (f,g_2))
  \end{align*}
  for all $f_i\co x_{i-1} \to x_i$ and $g_i\co y_{i-1} \to y_i$ for $i \in
  \set{1,2}$ and $f\co x \to x'$ and $g\co y \to y'$, \ie graphically,
  \begingroup%
  \displayskipforlongtable%
  \begin{longtable}{C}
    \begin{spacebelowenv}
      \begin{tikzcd}[column sep=4em,cramped,ampersand replacement=\&]
        {(x,y)} \ar[rd,phantom,"{\overset{(\unit x,g)}\Leftarrow}"]
        \arrow[d, "{(x,g)}"'] \arrow[r, "{(\unit x,y)}"] \& {(x,y)} \arrow[d, "{(x,g)}"] \\
        {(x,y')} \arrow[r, "{(\unit x,y)}"'] \& {(x,y')}
      \end{tikzcd}
    \end{spacebelowenv}
    =
    \begin{tikzcd}[row sep=3em,column sep=4em,cramped,ampersand replacement=\&]
      {(x,y)}
      \ar[d,bend left=49,"{(x,g)}"]
      \ar[d,phantom,"{\overset{\unit{(x,g)}}\Leftarrow}"]
      \ar[d,bend right=49,"{(x,g)}"']
      \\
      (x,y')
    \end{tikzcd}
    \\
    \begin{spacebelowenv}
      \begin{tikzcd}[column sep=4em,cramped,ampersand replacement=\&]
        {(x_0,y)} \ar[rd,phantom,"{\Downarrow\scriptstyle(f_1 {\comp_0} f_2,g)}"]\arrow[r, "{(f_1 \comp_0 f_2,y)}"] \arrow[d, "{(x_0,g)}"'] \& {(x_2,y)} \arrow[d, "{(x_2,g)}"]\\
        {(x_0,y')} \arrow[r, "{(f_1 \comp_0 f_2,y')}"'] \& {(x_2,y')}
      \end{tikzcd}
    \end{spacebelowenv}
    =
    \begin{tikzcd}[column sep=4em,cramped,ampersand replacement=\&]
      {(x_0,y)} \arrow[d, "{(x_0,g)}"'] \arrow[r, "{(f_1,y)}"]
      \ar[rd,"{\Downarrow\scriptstyle(f_1,g)}",phantom]
      \& {(x_1,y)} \arrow[d, "{(x_1,g)}" description] \arrow[r, "{(f_2,y)}"]
      \ar[rd,"{\Downarrow\scriptstyle(f_2,g)}",phantom]
      \& {(x_2,y)} \arrow[d, "{(x_2,g)}"] \\
      {(x_0,y')} \arrow[r, "{(f_1,y')}"']                      \& {(x_1,y')} \arrow[r, "{(f_2,y')}"']                                 \& {(x_2,y')}
    \end{tikzcd}
    \\
      \begin{tikzcd}[column sep=4em,cramped,ampersand replacement=\&]
        {(x,y)} \ar[rd,phantom,"{\Downarrow\scriptstyle(f,\unit y)}"]
        \arrow[d, "{(x,\unit y)}"'] \arrow[r, "{(f,y)}"] \& {(x',y)} \arrow[d, "{(x',\unit y)}"] \\
        {(x,y)} \arrow[r, "{(f,y)}"'] \& {(x',y)}
      \end{tikzcd}
    =
    \begin{tikzcd}[column sep=4em,cramped,ampersand replacement=\&]
      {(x,y)}
      \ar[r,bend left=31,"{(f,y)}"]
      \ar[r,phantom,"{\Downarrow\scriptstyle\unit{(f,y)}}"]
      \ar[r,bend right=31,"{(f,y)}"']
      \&
      (x',y)
    \end{tikzcd}
    \\
    \begin{tikzcd}[column sep=4em,cramped,ampersand replacement=\&]
      {(x,y_0)}
      \arrow[d, "{(x,g_1 {\comp_0} g_2)}"'] \arrow[r, "{(f,y_0)}"]
      \ar[rd,phantom,"{\Downarrow\scriptstyle(f,g_1 {\comp_0} g_2)}"]
      \& {(x',y_0)} \arrow[d, "{(x',g_1{\comp_0} g_2)}"] \\
      {(x,y_2)} \arrow[r, "{(f,y_2)}"']                                    \& {(x',y_2)}
    \end{tikzcd}
    =
    \begin{tikzcd}[column sep=4em,cramped,ampersand replacement=\&]
      {(x,y_0)}
      \arrow[r, "{(f,y_0)}"] \arrow[d, "{(x,g_1)}"']
      \ar[rd,phantom,"{\Downarrow\scriptstyle(f,g_1)}"]
      \& {(x',y_0)} \arrow[d, "{(x',g_1)}"] \\
      {(x,y_1)} \arrow[r, "{(f,y_1)}" description] \arrow[d, "{(x,g_2)}"']
      \ar[rd,phantom,"{\Downarrow\scriptstyle(f,g_2)}"]
      \& {(x',y_1)} \arrow[d, "{(x',g_2)}"] \\
      {(x,y_2)} \arrow[r, "{(f,y_2)}"']                                    \& {(x',y_2)}
    \end{tikzcd}
  \end{longtable}
  \endgroup
\item the interchangers commute with the $2$-generators, meaning that
  \begin{align*}
    ((f,g) \comp_1 ((x,g) \comp_0 (\phi,y'))) &= (( (\phi,y) \comp_0 (x',g)) \comp_1 (f',g)) \\
    ((f,g) \comp_1 ((x,\psi) \comp_0 (f,y'))) &= (( (f,y) \comp_0 (x',\psi)) \comp_1 (f,g'))
  \end{align*}
  for $\phi\co f\To f'\co x \to x'$ and $\psi\co g\To g'\co y \to y'$, \ie graphically,
  \begingroup
  \displayskipforlongtable
  \begin{longtable}{C}
    \begin{spacebelowenv}
    \begin{tikzcd}[column sep=4em,row sep=normal,cramped]
      {(x,y)} \arrow[d, "{(x,g)}"'] \arrow[r, "{(f,y)}", bend left=31,""{name=src}]                    & {(x',y)} \arrow[d, "{(x',g)}"] \\
      {(x,y')} \arrow[r, "{(f,y')}", bend left=31,""{name=tgt}]
      \ar[r,phantom,"{\Downarrow\scriptstyle(\phi,y')}"]
      \arrow[r, "{(f',y')}"', bend right=31,""{name=tgtx}] & {(x',y')}                     
      \ar[from=src,to=tgt,phantom,"{\Downarrow\scriptstyle(f,g)}"]
    \end{tikzcd}
    \end{spacebelowenv}
    =
    \begin{tikzcd}[column sep=4em,row sep=normal,cramped]
      {(x,y)}
      \ar[r,phantom,"{\Downarrow\scriptstyle(\phi,y)}"]
      \arrow[d, "{(x,g)}"']
      \arrow[r, "{(f,y)}", bend left=31]
      \arrow[r, "{(f',y)}"', bend right=31,""{name=src}] & {(x',y)} \arrow[d, "{(x',g)}"] \\
      {(x,y')} \arrow[r, "{(f',y')}"', bend right=31,""{name=tgt}]                                                        & {(x',y')}                     
      \ar[from=src,to=tgt,phantom,"{\Downarrow\scriptstyle(f',g)}",pos=0.65]
    \end{tikzcd}
    \\
    \begin{tikzcd}[column sep=2.5em,row sep=2.6em,cramped]
      {(x,y)}
      \ar[d,phantom,"{\overset{(x,\psi)}{\Leftarrow}}"]
      \arrow[d, "{(x,g')}"', bend right=49] \arrow[r, "{(f,y)}"] \arrow[d, "{(x,g)}", bend left=49,""{name=src}] & {(x',y)} \arrow[d, "{(x',g)}", bend left=49,""{name=tgt}] \\
      {(x,y')} \arrow[r, "{(f,y')}"']                                                                       & {(x',y')}                                    
      \ar[from=src,to=tgt,phantom,"{\overset{(f,g)}\Leftarrow}",pos=0.6]
    \end{tikzcd}
    =
    \begin{tikzcd}[column sep=2.5em,row sep=2.6em,cramped]
      {(x,y)}
      \arrow[d, "{(x,g')}"', bend right=49,""{name=src}] \arrow[r, "{(f,y)}"]
      & {(x',y)} \arrow[d, "{(x',g)}", bend left=49]
      \arrow[d, "{(x',g')}"', bend right=49,""{name=tgt}]
      \ar[d,phantom,"{\overset{(x',\psi)}\Leftarrow}"]
      \\
      {(x,y')} \arrow[r, "{(f,y')}"']                                    & {(x',y')}                                                                          
      \ar[from=src,to=tgt,phantom,"{\overset{(f,g')}\Leftarrow}",pos=0.27]
    \end{tikzcd}\pbox.
  \end{longtable}
  \endgroup
\end{enumerate}
The construction extends to a bifunctor $\nCat 2 \times \nCat 2 \to \nCat 2$ by
defining, for $F\co C \to C'$ and $G\co D \to D'$, $F\graylax G$ as the unique
functor mapping
\begin{align*}
  (\phi,y) &\mapsto (F(\phi),G(y)) \\
  (x,\psi) &\mapsto (F(x),G(\psi)) \\
  (f,g) &\mapsto (F(f),G(g))
\end{align*}
for all $x \in C_0$, $y \in D_0$, $\phi \in C_2$, $\psi \in D_2$, $f \in C_1$
and $g \in D_1$.

For $C,D,E \in \nCat 2$, there is a $2$-functor
\[
  \grayasslax_{C,D,E}\co (C \graylax D) \graylax E \xto\sim C \graylax (D \graylax E)
\]
which is an isomorphism natural in $C,D,E$ and uniquely defined by the following
mappings on generators
\begin{align*}
    ((\phi,y),z) & \mapsto (\phi,(y,z)) &
    ((f,g),z) & \mapsto (f,(g,z)) \\
    ((x,\psi),z) & \mapsto (x,(\psi,z)) &
    ((f,y),h) & \mapsto (f,(y,h)) \\
    ((x,y),\gamma) & \mapsto (x,(y,\gamma)) &
    ((x,g),h) & \mapsto (x,(g,h))
\end{align*}
for $\phi\co f \To f'\co x \to x' \in C_2$, $\psi\co g \To g'\co y \to y' \in
D_2$ and $\gamma\co h\To h'\co z \to z' \in E_2$.

For $C \in \nCat 2$, there are two $2$-functors
\[
  \grayllax_C\co \termcat \graylax C \to C \qtand \grayrlax_C\co C \graylax \termcat \to C
\]
which are isomorphisms, natural in $C$, and uniquely defined by the mappings
\begin{align*}
  \grayllax((\ast,\psi)) = \psi \qtand \grayrlax((\psi,\ast)) = \psi
\end{align*}
for $\psi \in C_2$.
By checking coherence conditions between $\grayasslax$, $\grayllax$ and
$\grayrlax$, we get that:
\begin{prop}
  \label{prop:gray-tensor-product}
  The bifunctor $\graylax$ together with the unit~$\termcat$ and the natural
  isomorphisms $\grayasslax$, $\grayllax$ and $\grayrlax$ equip $\nCat 2$ with a structure
  of a monoidal category.
\end{prop}
\noindent The monoidal structure $(\nCat 2,\graylax,\termcat,\grayasslax,\grayllax,\grayrlax)$ is
called the \emph{lax Gray tensor product}.

A variant of the Gray tensor is
called the \emph{pseudo Gray tensor product} is the monoidal structure
$(\nCat 2,\grayy,\termcat,\grayass,\grayl,\grayr)$ where, given $C,D \in \nCat 2$,
$C\grayy D$ is defined the same way as ${C \graylax D}$, except that we moreover
require that the $2$-cells $(f,g)$ of $C \grayy D$ be invertible. The natural
isomorphisms $\grayass,\grayl,\grayr$ are uniquely defined by similar mappings
than those defining $\grayasslax,\grayllax,\grayrlax$, and we have:
\begin{prop}
  \label{prop:gray-tensor-product-pseudo}
  The bifunctor $\grayy$ together with the unit~$\termcat$ and the natural isomorphisms
  $\grayass$, $\grayl$, $\grayr$ equip $\nCat 2$ with a structure of a monoidal
  category.
\end{prop}

%% file: gray-cat.tex
For each of the two variants of Gray tensor product defined in the previous
section, there is an associated notion of $3$-dimensional category as we now
describe here.

A \emph{lax Gray category}~\cite[I,4.25]{gray2006formal} is a category
enriched in the category of $2$-categories equipped with the lax Gray tensor
product. A more explicit definition using generators and
relations can be given as follows. A \emph{Gray category} is a $3$\precategory~$C$ together with, for
all every $0$\composable pair of $2$-cells $\phi\co f \To f'\co x \to y$ and $\psi\co g \To g'\co y \to z$,
a $3$-cell
\[
  \begin{tabularx}{\linewidth}{XR@{}C@{\ }C@{\ }CX}
    &X_{\phi,\psi}\co
    &(\phi \comp_0 g) \comp_1 (f' \comp_0 \psi)
    &\TO
    & (f \comp_0
      \psi) \comp_1 (\phi \comp_0 g')
    &
    % \\
    % &\begin{tikzcd}[ampersand replacement=\&,sep={1.5cm,between origins}]
    %   x
    %   \ar[r,bend left=25,phantom,"\Downarrow\!\phi"description]
    %   \ar[r,bend left=50,"f"]
    %   \ar[r,"{f'}"']
    %   \&
    %   y
    %   \ar[r,bend right=25,phantom,"\Downarrow\!\psi"description]
    %   \ar[r,"g"]
    %   \ar[r,bend right=50,"{g'}"']
    %   \&
    %   z
    % \end{tikzcd} 
    % &\TO
    % &
    % \begin{tikzcd}[ampersand replacement=\&,sep={1.5cm,between origins}]
    %   x
    %   \ar[r,bend right=25,phantom,"\Downarrow\!\phi"description]
    %   \ar[r,"f"]
    %   \ar[r,bend right=50,"u'"']
    %   \&
    %   y
    %   \ar[r,bend left=25,phantom,"\Downarrow\!\psi"description]
    %   \ar[r,bend left=50,"g"]
    %   \ar[r,"g'"']
    %   \&
    %   z
    % \end{tikzcd}     
    \\
    \noalign{\vskip\belowdisplayskip\hbox{which can be represented using string diagrams by}\vskip\abovedisplayskip}
    &
    &
      % \tikzset{3-cell/.style={draw,thick,rounded corners=2,text width=1em,align=center}}
      % \begin{tikzpicture}[baseline=(current bounding box.center),xscale=1.5,yscale=0.8]
        % \node (TL) at (0,2) {$f\mathstrut$};
        % \node (TR) at (1,2) {$g\mathstrut$};
        % \node[3-cell] (L) at (0,1) {$\phi$};
        % \node[3-cell] (R) at (1,0) {$\psi$};
        % \node (BL) at (0,-1) {$f'\mathstrut$};
        % \node (BR) at (1,-1) {$g'\mathstrut$};
        % \draw[thick] (TL) -- (L) -- (BL);
        % \draw[thick] (TR) -- (R) -- (BR);
      % \end{tikzpicture}
    % &
      % \TO
    % &
      % \tikzset{3-cell/.style={draw,thick,rounded corners=2,text width=1em,align=center}}
      % \begin{tikzpicture}[baseline=(current bounding box.center),xscale=1.5,yscale=0.8]
        % \node (TL) at (0,2) {$f\mathstrut$};
        % \node (TR) at (1,2) {$g\mathstrut$};
        % \node[3-cell] (L) at (0,0) {$\phi$};
        % \node[3-cell] (R) at (1,1) {$\psi$};
        % \node (BL) at (0,-1) {$f'\mathstrut$};
        % \node (BR) at (1,-1) {$g'\mathstrut$};
        % \draw[thick] (TL) -- (L) -- (BL);
        % \draw[thick] (TR) -- (R) -- (BR);
        % \end{tikzpicture}
    \satex[scale=1.5]{phi-psi}
    &
    \TO
    &
    \satex[scale=1.5]{psi-phi}
  \end{tabularx}
\]
\noindent
called \emph{interchanger} and satisfying the following sets of axioms:
\begin{enumerate}[label=(\roman*),ref=(\roman*)]
\item compatibility with compositions and identities: for $\phi\co f \To f'$,
  $\phi'\co f' \To f''$, $\psi\co g\To g'$, $\psi'\co g'\To g''$ in $C_2$ and
  $e$, $h$ in $C_1$ such
  that $e$, $\phi$, $\psi$ and $h$ are $0$\composable, we have
  \begingroup
  \begin{align*}
    X_{\unit{f},\psi} &= \unit{f \comp_0 \psi} &
                                                 X_{\phi \comp_1 \phi',\psi} &= ((\phi \comp_0 g) \comp_1 X_{\phi',\psi})
                                                                               \comp_2 (X_{\phi,\psi} \comp_1 (\phi' \comp_0 g')) \\
    X_{\phi,\unit{g}} &= \unit{\phi \comp_0 g} &
                                                 X_{\phi,\psi \comp_1 \psi'} &= (X_{\phi,\psi} \comp_1 (f' \comp_0 \psi'))
                                                 \comp_2 ((f \comp_0 \psi) \comp_1 X_{\phi,\psi'})
  \end{align*}
  \endgroup
  and
  \begin{align*}
    X_{e \comp_0 \phi,\psi} &= e \comp_0 X_{\phi,\psi} 
    & X_{\phi, \psi \comp_0 h} &= X_{\phi,\psi} \comp_0 h.
  \end{align*}
  Moreover, given $\phi,\psi \in C_2$ and $f \in C_1$ such that $\phi$, $f$ and
  $\psi$ are $0$\composable, we have
  \[
    X_{\phi \comp_0 f, \psi} = X_{\phi, f \comp_0 \psi}
  \]

\item exchange law for $3$-cells: for all $A\co \phi \TO \psi \in C_3$ and $B\co
  \psi \TO \psi' \in C_3$ such that $A$ and $B$ are $1$\composable, we have
  \[ 
    (A \comp_1 \psi) \comp_2 (\phi' \comp_1 B) = 
    (\phi \comp_1 B) \comp_2 (A \comp_1 \psi')
  \]
\item compatibility between interchangers and $3$-cells:
  given
  \[
    A \co \phi \TO \phi'\co u\To u' \in C_3 \qtand B\co\psi\TO \psi'\co v \To v'
    \in C_3,
  \]
  such that $A$, $B$ are $0$\composable, we have
  \begin{align*}
    ((A \comp_0 v) \comp_1 (u' \comp_0
    \psi)) \comp_2 X_{\phi',\psi} &= X_{\phi,\psi} \comp_2 ((u \comp_0 \psi)
    \comp_1 (A \comp_0 v')) \\
    ((\phi \comp_0 v) \comp_1 (u' \comp_0 B))
    \comp_2 X_{\phi,\psi'} &= X_{\phi,\psi} \comp_2 ((u \comp_0 B) \comp_1 (\phi
    \comp_0 v')).
  \end{align*}
\end{enumerate}
A \emph{morphism between two lax Gray categories}~$C$ and~$D$ is a $3$-prefunctor~$F\co
C \to D$ such that $F(X_{\phi,\psi}) = X_{F(\phi),F(\psi)}$.

We similarly have a notion of \emph{pseudo Gray category} which is a category
enriched in the category of $2$-categories equipped with the pseudo Gray tensor
product. In terms of generators and relations, a pseudo Gray category is a lax
Gray category~$C$ where the $3$-cell $X_{\phi,\psi}$ is invertible for every
$0$-composable $2$-cells
$\phi,\psi \in C_2$. A morphism between two pseudo Gray categories $C,D$ is a
morphism of lax Gray categories between $C$ and $D$.

In the following, a \emph{$(3,2)$-Gray category} is a lax Gray category whose
underlying $3$\precategory is a $(3,2)$\precategory. Note that it is then also a
pseudo Gray category. As one can expect, a localization of a lax Gray category
gives a $(3,2)$-Gray category:
\begin{prop}
  \label{prop:gray-induces-3-2-gray}
  If $C$ is a lax Gray category, then $\freeinvf{C}$ is canonically a
  $(3,2)$-Gray category.
\end{prop}
\begin{proof}
  Given $1$-composable $3$-cells $F\co \phi \TO \phi'$ and $G\co \psi \TO \psi' \in C_3$, by the
  exchange law for $3$-cells, we have, in $\freeinvf{C}_3$,
  \[ 
    (F \comp_1 \psi) \comp_2 (\phi' \comp_1 G) = 
    (\phi \comp_1 G) \comp_2 (F \comp_1 \psi').
  \]
  By inverting $F \comp_1 \psi$ and $F \comp_1 \psi'$, we obtain
  \[
    (\phi' \comp_1 G) \comp_2 (\finv{F} \comp_1 \psi') = 
    (\finv{F} \comp_1 \psi) \comp_2 (\phi \comp_1 G).
  \]
  Similarly, 
  \[ 
    (\phi \comp_1 \finv{G}) \comp_2 (F \comp_1 \psi) = 
    (F \comp_1 \psi') \comp_2 (\phi' \comp_1 \finv{G})
  \]
  and
  \[ 
    (\finv{F} \comp_1 \psi') \comp_2 (\phi \comp_1 \finv{G}) = 
    (\phi' \comp_1 \finv{G}) \comp_2 (\finv{F} \comp_1 \psi).
  \]
  Now, given general $1$-composable $F\co \phi \TO \phi', G\co \psi \TO \psi' \in
  \freeinvf{C}_3$, we have that
  \[
    F = F_1 \comp_2 \finv{F_2} \comp_2 \cdots \comp_2 F_{2k-1} \comp_2 \finv{F_{2k}}
  \]
  and
  \[
    G = G_1 \comp_2 \finv{G_2} \comp_2 \cdots \comp_2 G_{2l-1} \comp_2 \finv{G_{2l}}
  \]
  for some $k,l \ge 1$ and $F_i,G_j \in C_3$ for $1 \le i \le 2k$ and
  $1 \le j \le 2l$. By applying the formulas above $4kl$ times to exchange the
  $F_i$'s with the $G_j$'s, we get
  \[
    (F \comp_1 \psi) \comp_2 (\phi' \comp_1 G) = 
    (\phi \comp_1 G) \comp_2 (F \comp_1 \psi').
  \]
  A similar argument gives the compatibility between interchangers and $3$-cells
  of $\freeinvf{C}$. Thus, $\freeinvf{C}$ is a $(3,2)$-Gray category.
\end{proof}

%% file: gray-pres.tex
Starting from a $3$\prepolygraph~$\P$, such as the one of
\Exr{pseudo-monoid-gray-pres}, we want to add $3$\generators to~$\P$ and
relations on the $3$\cells of~$\freecat\P_3$ in order to obtain a presentation of
a lax Gray category. This can of course be achieved naively by adding, for each
pair of $0$\composable $2$\cells $\phi,\psi$ in $\freecat\P_2$, a $3$\generator
corresponding to the interchanger ``$X_{\phi,\psi}$'', together with the
relevant relations, but the resulting presentation has a large number of
generators, and we detail below a more economical way of proceeding in order to
present lax Gray categories.

A \emph{Gray presentation} is a $4$\prepolygraph~$\P$ containing the following distinguished
generators:
\begin{enumerate}[label=(\roman*)]
% \item for every pair of $3$-generators~$A_1$ and~$A_2$, as well as morphisms as
%   on the left of~???, \todo{référence}% \eqref{eq:peiffer},
%   there is a relation as on the right of~???%\eqref{eq:peiffer}
%   called a \emph{Peiffer generator},
  
\item for $0$-composable $\alpha,g,\beta$ with $\alpha\co f \To f',\beta\co g
  \To g' \in \P_2$, $g \in \freecat\P_1$, a $3$-generator $X_{\alpha,g,\beta}\in
  \P_3$ called \emph{interchange generator}, which is of type
  \begingroup
  \vskip\abovedisplayskip
  \tabskip=0pt plus 1fill\halign
  to\linewidth{\tabskip=0pt\relax\tabskip=0pt\relax
    $#$\hfill&\hfill $#$\hfill&\hfill $#$\hfill&\hfill $#$\hfill\tabskip=0pt plus 1fill\cr
      X_{\alpha,g,\beta}\co
      &
        (\alpha\comp_0 g\comp_0 h)\comp_1(f'\comp_0 g\comp_0\beta)
      &\ \TO\ 
      &(f\comp_0 g\comp_0\beta)\comp_1(\alpha\comp_0 g\comp_0 h')
      % \\
      % &
      % \begin{tikzcd}[ampersand replacement=\&]
      %   {x}
      %   \ar[r,bend left=25,phantom,"{\Downarrow\!\alpha}"description]
      %   \ar[r,bend left=50,"{f}",pos=0.58]
      %   \ar[r,"{f'}"']
      %   \&
      %   {x'}
      %   \ar[r,"{g}"]
      %   \&
      %   {y'}
      %   \ar[r,bend right=25,phantom,"{\Downarrow\!\beta}"description]
      %   \ar[r,"{h}"]
      %   \ar[r,bend right=50,"{h'}"',pos=0.48]
      %   \&
      %   {z}
      % \end{tikzcd} 
      % &\TO
      % &
      % \begin{tikzcd}[ampersand replacement=\&]
      %   x
      %   \ar[r,bend right=25,phantom,"{\Downarrow\!\alpha}"description]
      %   \ar[r,"{f}"]
      %   \ar[r,bend right=50,"{f'}"']
      %   \&
      %   {x'}
      %   \ar[r,"{g}"]
      %   \&
      %   {y'}
      %   \ar[r,bend left=25,phantom,"{\Downarrow\!\beta}"description]
      %   \ar[r,bend left=50,"{h}",pos=0.43]
      %   \ar[r,"{h'}"']
      %   \&
      %   z
      % \end{tikzcd}     
      \cr
      \noalign{\goodbreak\vskip\belowdisplayskip\hbox{\hskip\leftmargin which can be represented using string diagrams by}\vskip\abovedisplayskip}
      &
          % \tikzset{3-cell/.style={draw,thick,rounded corners=2,text width=1cm,align=center,align=center}}
          % \begin{tikzpicture}[baseline=(current bounding box.center),xscale=0.7,yscale=0.8]
            % \coordinate (topline) at (0,2);
            % \coordinate (botline) at (0,0);
            % \node[3-cell] (L) at (0.5,1) {\strut$\alpha$};
            % \node[3-cell] (R) at (4.5,0) {\strut$\beta$};
            % \draw[thick] (0,2) -- (0,2 |- L.north);
            % \draw[thick] (1,2) -- (1,2 |- L.north);
            % \draw[thick] (4,2) -- (4,2 |- R.north);
            % \draw[thick] (5,2) -- (5,2 |- R.north);
            % \draw[thick] (2,2) -- (2,-1);
            % \draw[thick] (3,2) -- (3,-1);
            % \draw[thick] (0,2 |- L.south) -- (0,-1);
            % \draw[thick] (1,2 |- L.south) -- (1,-1);
            % \draw[thick] (4,2 |- R.south) -- (4,-1);
            % \draw[thick] (5,2 |- R.south) -- (5,-1);
            % \draw[thick] (4,2) -- (4,2 |- R.north);
            % \draw[thick] (5,2) -- (5,2 |- R.north);
            % \node at (0.52,1.75) {$\cdots$};
            % \node at (0.52,-0.75) {$\cdots$};
            % \node at (2.52,1.75) {$\cdots$};
            % \node at (2.52,-0.75) {$\cdots$};
            % \node at (4.52,1.75) {$\cdots$};
            % \node at (4.52,-0.75) {$\cdots$};
            % \node[anchor=south] at (0.5,2) {$f\mathstrut$};
            % \node[anchor=north] at (0.5,-1) {$f'\mathstrut$};
            % \node[anchor=south] at (2.5,2) {$g\mathstrut$};
            % \node[anchor=north] at (2.5,-1) {$g\mathstrut$};
            % \node[anchor=south] at (4.5,2) {$h\mathstrut$};
            % \node[anchor=north] at (4.5,-1) {$h'\mathstrut$};
            % \end{tikzpicture}
      \satex[scale=1.5]{a-b}
      &
      \TO
      &
      \satex[scale=1.5]{b-a}
        % \tikzset{3-cell/.style={draw,thick,rounded corners=2,text width=1cm,align=center,align=center}}
        % \begin{tikzpicture}[baseline=(current bounding box.center),xscale=0.7,yscale=0.8]
          % \coordinate (topline) at (0,2);
          % \coordinate (botline) at (0,0);
          % \node[3-cell] (L) at (0.5,0) {\strut$\alpha$};
          % \node[3-cell] (R) at (4.5,1) {\strut$\beta$};
          % \draw[thick] (0,2) -- (0,2 |- L.north);
          % \draw[thick] (1,2) -- (1,2 |- L.north);
          % \draw[thick] (4,2) -- (4,2 |- R.north);
          % \draw[thick] (5,2) -- (5,2 |- R.north);
          % \draw[thick] (2,2) -- (2,-1);
          % \draw[thick] (3,2) -- (3,-1);
          % \draw[thick] (0,2 |- L.south) -- (0,-1);
          % \draw[thick] (1,2 |- L.south) -- (1,-1);
          % \draw[thick] (4,2 |- R.south) -- (4,-1);
          % \draw[thick] (5,2 |- R.south) -- (5,-1);
          % \draw[thick] (4,2) -- (4,2 |- R.north);
          % \draw[thick] (5,2) -- (5,2 |- R.north);
          % \node at (0.52,1.75) {$\cdots$};
          % \node at (0.52,-0.75) {$\cdots$};
          % \node at (2.52,1.75) {$\cdots$};
          % \node at (2.52,-0.75) {$\cdots$};
          % \node at (4.52,1.75) {$\cdots$};
          % \node at (4.52,-0.75) {$\cdots$};
          % \node[anchor=south] at (0.5,2) {$f\mathstrut$};
          % \node[anchor=north] at (0.5,-1) {$f'\mathstrut$};
          % \node[anchor=south] at (2.5,2) {$g\mathstrut$};
          % \node[anchor=north] at (2.5,-1) {$g\mathstrut$};
          % \node[anchor=south] at (4.5,2) {$h\mathstrut$};
          % \node[anchor=north] at (4.5,-1) {$h'\mathstrut$};
          % \end{tikzpicture}
        \cr
    }
  \endgroup
  %%%%%%%%%%%%
  \item for every pair of $3$-generators~$A,B \in \P_3$ and $e,e',h,h' \in
  \freecat\P_1$ and $\chi \in \freecat\P_2$ as in
  \begin{equation}
    \label{eq:peiffer-gen}
    \begin{tikzcd}[column sep=10ex,row sep=3ex]
      &
      x
      \ar[r,bend left,"f"]
      \ar[r,phantom,"{\scriptstyle\phi}\Downarrow\overset{A}{\underset{\phantom{A}}\TO}\Downarrow{\scriptstyle\phi'}"]
      \ar[r,bend right,"g"']
      &
      y
      \ar[dr,"h"description]
      \\
      w
      % \ar[rrr,bend left=80,"v_1"]
      % \ar[rrr,bend right=80,"v_2"']
      % \ar[rrr,bend left=50,phantom,"\hspace{10ex}\Downarrow\!\chi_1"]
      % \ar[rrr,bend right=50,phantom,"\hspace{10ex}\Downarrow\!\chi_2"]
      \ar[ur,"e"description]
      \ar[rrr,phantom,"\phantom{\scriptstyle\chi\!}\Downarrow{\!\scriptstyle\chi}"]
      \ar[dr,"e'"description]
      &&&
      z
      \\
      &
      x'
      \ar[r,bend left,"f'"]
      \ar[r,phantom,"{\scriptstyle\psi}\Downarrow\overset{B}{\underset{\phantom{B}}\TO}\Downarrow{\scriptstyle\psi'}"]
      \ar[r,bend right,"g'"']
      &
      y'
      \ar[ur,"h'"description]
    \end{tikzcd}
  \end{equation}
  a $4$-generator of type $\Gamma \TOO \Delta$, called \emph{independence
    generator},
  where
  \[
    \Gamma = ((e \comp_0 A \comp_0 h) \comp_1 \chi \comp_1 (e' \comp_0 \psi \comp_0
    h')) \comp_2 ((e \comp_0 \phi' \comp_0 h) \comp_1 \chi \comp_1 (e' \comp_0 B \comp_0
    h'))
  \]
  and
  \[
    \Delta = ((e \comp_0 \phi \comp_0 h) \comp_1 \chi \comp_1 (e' \comp_0 B \comp_0
    h')) \comp_2 ((e \comp_0 A \comp_0 h) \comp_1 \chi \comp_1 (e' \comp_0 \psi' \comp_0
    h'))
  \]
  and can be pictured as
  \[
    \begin{tikzcd}[column sep=8ex,row sep=5ex]
      \satex{peiffer}
      \tar[r,"B"]
      \tar[d,"A"']
      \ar[rd,phantom,"\TOO"]
      &
      \satex{peiffer-r}
      \tar[d,"A"]
      \\
      \satex{peiffer-l}
      \tar[r,"B"']
      &
      \satex{peiffer-e}
    \end{tikzcd}
  \]
\item for all $0$\composable $A,g,\beta$ with $A\in\P_3$, $g \in \freecat\P_1$
  and $\beta\in\P_2$, and respectively, $0$\composable $\alpha,g',B$ with
  $\alpha\in\P_2$, $g'\in\freecat\P_1$ and $B\in\P_3$ as on the first, \resp
  second line below
  \begin{equation}
    \label{eq:inat-gen}
    \begin{split}
    \begin{tikzcd}[column sep=8ex,ampersand replacement=\&]
      x
      \ar[r,phantom,"{\scriptstyle\phi}\Downarrow\overset{A}{\underset{\phantom{A}}\TO}\Downarrow\scriptstyle\phi'"]
      \ar[r,bend left=40,"f"]
      \ar[r,bend right=40,"f'"']
      \&
      x'
      \ar[r,"g"]
      \&
      y'
      \ar[r,phantom,"\Downarrow\scriptstyle\beta"]
      \ar[r,bend left=40,"h"]
      \ar[r,bend right=40,"h'"']
      \&
      y
    \end{tikzcd}
    \\
    \begin{tikzcd}[column sep=8ex,ampersand replacement=\&]
      x
      \ar[r,phantom,"\Downarrow\scriptstyle\alpha"]
      \ar[r,bend left=40,"f"]
      \ar[r,bend right=40,"f'"']
      \&
      x'
      \ar[r,"g'"]
      \&
      y'
      \ar[r,phantom,"{\scriptstyle\psi}\Downarrow\overset{B}{\underset{\phantom{B}}\TO}\Downarrow\scriptstyle\psi'"]
      \ar[r,bend left=40,"h"]
      \ar[r,bend right=40,"h'"']
      \&
      y
    \end{tikzcd}
    \end{split}
  \end{equation}
  a $4$-generator, called \emph{interchange naturality
    generator}, respectively of type
  \begin{align*}
    ((A\comp_0 g\comp_0 h)\comp_1(f'\comp_0 g\comp_0\beta))
    \comp_2
    X_{\phi',g \comp_0 \beta}
    &\TOO
    X_{\phi,g \comp_0 \beta}
    \comp_2
    ((f\comp_0 g\comp_0\beta)\comp_1(A\comp_0 g\comp_0 h'))
    \shortintertext{and}
    ((\alpha\comp_0 g'\comp_0 h)\comp_1(f'\comp_0 g'\comp_0 B))
    \comp_2
    X_{\alpha \comp_0 g',\psi'}
    &\TOO
    X_{\alpha \comp_0 g',\psi}
    \comp_2
    ((f\comp_0 g'\comp_0 B)\comp_1(\alpha\comp_0 g'\comp_0 h'))
  \end{align*}
  where the $X_{\chi_1,\chi_2}$'s appearing in the sources and targets will be
  defined below for any $0$-composable $\chi_1,\chi_2 \in \freecat\P_2$; the
  first kind of interchange naturality generator can be pictured by
  % (see \Cref{fig:inter-nat-gen-satex} for an example).
  % \begin{figure}
    % \centering
    \[
      \begin{tikzcd}[column sep=8ex,row sep=5ex]
        \satex{inter-nat-gen}
        \tar[r,"X"]
        \tar[d,"A"']
        \ar[rd,phantom,"\TOO"]
        &
        \satex{inter-nat-gen-r}
        \tar[d,"A"]
        \\
        \satex{inter-nat-gen-l}
        \tar[r,"X"']
        &
        \satex{inter-nat-gen-e}
      \end{tikzcd}
    \]
    % \caption{The first type of interchange naturality generators}
    % \label{fig:inter-nat-gen-satex}
  % \end{figure}
  % for all $p \in \freecat S_{|\psi|,1}$ and $p' \in \freecat S_{|\phi|,1}$, and
  % respectively $p \in \freecat S_{1,|\psi|}$ and $p' \in \freecat S_{1,|\phi|}$.
  % are suitable composites of interchanger generators defined as follows: if
  % $\phi = \unit{f}$ (resp. $\psi = \unit{g}$) then $Y_{\phi,\psi} =
  % \unit{f \comp_0 \psi}$ (resp. $\unit{\phi \comp_0 g}$), otherwise we can
  % write $\phi$ $\phi' \comp_1 (u, \alpha, u')$ and $\psi$ as $(v,\beta,v')
  % \comp_1 \psi'$ and we define $Y_{\phi,\psi}$ inductively as
  % \begin{align*}
  %     Y_{\phi,\psi} &= ((\phi' \comp_0 \csrc_1(\psi))
  %     \comp_1 (u \comp_0 \intg_{\alpha \comp_0 u',v \comp_0 \beta} \comp_0 v') \comp_1 (\ctgt_1(\phi) \comp_0 \psi')) \\
  %     & \comp_2 ((\phi' \comp_0 \csrc_1(\psi)) \comp_1 (\ctgt_1(\phi') \comp_0 v
  %     \comp_0 \beta \comp_0 v') \comp_1 Y_{(u,\alpha,u'),\psi'}) \\
  %     & \comp_2 (Y_{\phi',\psi} \comp_1 (u \comp_0 \alpha \comp_0 u' \comp_0 \ctgt_1(\psi)))
  % \end{align*}
\end{enumerate}
The $3$\cells $X_{\phi,\psi} \in \freecat\P_3$, which are used in the above
definition, generalize interchange generators to any pair of $0$-composable
$2$-cells $\phi$ and $\psi$. Their definition consists in a suitable composite
of the generators $X_{\alpha,u,\beta}$ and is detailed below. Let us give an
idea of the definition of those $3$\cells on an example. Consider a Gray
presentation~$\Q$ with
\[
\Q_0
= \set{x},\quad \Q_1 = \set{\bar1\co x \to x} \qtand \Q_2 = \set{\tau\co
  \bar1\To\bar1}
\]
where $\tau$ is pictured by $\satex{pearl}$. Then, the
following sequence of ``moves'' is an admissible definition
for~$X_{\tau\comp_1\tau,\tau\comp_1\tau}$:
\begin{equation}
  \label{eq:ex-move-sequence}
  \satex{exch-ex-b} \qTO \satex{exch-ex-l-1} \qTO \satex{exch-ex-l-2}
  \qTO \satex{exch-ex-l-3} \qTO \satex{exch-ex-e}\pbox.
\end{equation}
Each ``move'' above is a $3$\cell of the form $\phi\comp_1 X_{\tau,\unit
  x,\tau}\comp_1 \psi$ for some $\phi,\psi \in \freecat\Q_2$ and where~$X_{\tau,\unit x,\tau}$ is an interchange generator provided by the
definition of Gray presentation. Another admissible sequence of moves is the
following:
\begin{equation*}
  \satex{exch-ex-b} \qTO \satex{exch-ex-r-1} \qTO \satex{exch-ex-r-2}
  \qTO \satex{exch-ex-r-3} \qTO \satex{exch-ex-e}
\end{equation*}
We see that there are multiple ways one can define the $3$\cells~$X_{\phi,\psi}$
based on the interchange generators of a Gray presentation~$\P$. We will show in
\Propr{interchange-coherence} that, in the end, the choice does not matter,
because all the possible definitions give rise to the same $3$\cell
in~$\prespcat\P$. Still, we need to introduce a particular structure that allows
us to represent all the possible definitions of the $3$\cells~$X_{\phi,\psi}$
and reason about them. This structure consists in a graph~$\phi \shuffle \psi$
associated to each pair of $0$\composable $2$-cells~$\phi$ and $\psi$
in~$\freecat\P_2$: intuitively, a vertex in this graph will correspond to an
interleaving of the $2$\generators of~$\phi$ and~$\psi$, and an edge will
correspond to a ``move'' as above, \ie an interchange
generator~$X_{\alpha,g,\beta}$ in context that exchanges two
$2$\generators~$\alpha$ from~$\phi$ and~$\beta$ from~$\psi$, which appear
consecutively in an interleaving of~$\phi$ and~$\psi$.
Given two $0$\composable $2$\cells
\[
  \phi = \phi_1 \pcomp_1 \cdots \pcomp_1 \phi_k \in \freecat\P_2
  \qquad\text{and}\qquad
  \psi = \psi_1\pcomp_1 \cdots \pcomp_1 \psi_{k'} \in \freecat\P_2
\]
with~$\phi_i = f_i \pcomp_0 \alpha_i \pcomp_0 g_i$ and~$\psi_j = f'_j \pcomp_0
\alpha'_j \pcomp_0 g'_j$ for some~$f_i,g_i,f'_j,g'_j \in \freecat\P_1$
and~$\alpha_i,\alpha'_j \in \P_2$, we define the
graph~$\phi \shuffle \psi$
\begin{itemize}
\item whose vertices are the \emph{shuffles} of the words~$\letter {l}_1 \ldots \letter
  l_k$ and~$\letter r_1 \ldots \letter r_{k'}$ on the alphabet
  \[
    \Sigma_{\phi,\psi} = \set{\letter l_1,\ldots,\letter l_k,\letter
      r_1,\ldots,\letter r_{k'}},
  \]
  \ie words of length~$k+k'$ which are order-preserving interleavings of the words~$l_1\ldots l_k$
  and~$\letter r_1 \ldots \letter r_{k'}$,
\item whose edges are of the form~$\wtrans_{w,w'}\co w \letter
  l_i \letter r_j w' \to w \letter r_j \letter l_i w'$ for some~$i \in
  \N^*_k$,~$j \in \N^*_{k'}$ and some words~$w,w' \in
  \freecat\Sigma_{\phi,\psi}$ such that~$w \letter l_i \letter r_j w' \in (\phi
  \shuffle \psi)_0$, intuitively representing the local \eq{swaps} one can do to move a
  letter~$\letter l_i$ to the right of a letter~$\letter r_j$ in a word.
\end{itemize}
Given~$i,j,p,q\in \N$ with~$0 \le i\le k$,~$0\le j \le k'$,~$0 \le p \le k-i+1$,~$0
\le q \le k'-j+1$, and a shuffle~$u$ of the words
\[
  \letter l_i \ldots \letter l_{i+p-1}
  \qtand
  \letter
  r_j \ldots \letter r_{j+q-1}\zbox,
\]
we define~$\winterp u^{i,j}_{\phi,\psi} \in \freecat\P_2$ (or simply $\winterp
u^{i,j}$) by induction on~$p$ and~$q$:
\[
  \winterp u^{i,j} =
  \begin{cases}
    (\phi_i \pcomp_0 \ctgt_1(\psi_j)) \pcomp_1 \winterp {u'}^{i+1,j} & \text{if~$u
      = \letter l_i  u'$,} \\
    (\ctgt_1(\phi_i) \pcomp_0 \psi_j) \pcomp_1 \winterp {u'}^{i,j+1} & \text{if~$u
      = \letter r_j  u'$,} \\
    \ctgt_1(\phi_i) \pcomp_0 \ctgt_1(\psi_j) & \text{if~$u$ is the empty word,}
  \end{cases}
\]
where, by
convention,~$\ctgt_1(\phi_0) = \csrc_1(\phi_1)$ and~$\ctgt_1(\psi_0) = \csrc_1(\psi_1)$.
Note that the indices of~$\winterp{u}^{i,j}$ are uniquely determined if~$u$ has at least
an~$\letter l$ letter and an~$\letter r$ letter.
Intuitively, the letters~$\letter l_i$ and~$\letter r_j$ correspond to the
$2$\cells~$\phi_i \pcomp_0 (-)$ and~$(-) \pcomp_0 \psi_j$ where the
$1$\cells~$(-)$ are most of the time uniquely determined by the context, so
that~$\winterp{u}^{1,1}$ for~$u \in (\phi \shuffle \psi)_0$ is an interleaving
of the~$\phi_i \pcomp_0 (-)$ and~$(-) \pcomp_0 \psi_j$. Now,
given
\begin{gather*}
  \wtrans_{u,v}\co u \letter l_i\letter r_j v \to u \letter r_j \letter l_iv 
  \shortintertext{in $(\phi\shuffle\psi)_1$, we define the $3$\cell}
  \winterp{\wtrans_{u,v}}_{\phi,\psi}\co
  \winterp{u \letter l_i\letter r_j v}^{1,1}_{\phi,\psi} \TO \winterp{u \letter r_j \letter l_i v}^{1,1}_{\phi,\psi}
  \shortintertext{in $\freecat\P_3$ by}
  \winterp{\wtrans_{u,v}}_{\phi,\psi} = \winterp{u}^{1,1}_{\phi,\psi} \pcomp_1 (f_i \pcomp_0 X_{\alpha_i,g_i \pcomp_0
    f'_j,\alpha'_j} \pcomp_0 g'_j) \pcomp_1 \winterp{v}^{i+1,j+1}_{\phi,\psi}\zbox.
\end{gather*}
We thus obtain a functor
\[
  \winterp{-}_{\phi,\psi}\co\freecat{(\phi \shuffle \psi)} \to
  \freecat\P(\csrc_1(\phi) \pcomp_0 \csrc_1(\psi),\ctgt_1(\phi) \pcomp_0
  \ctgt_1(\psi))
\]
where~$\freecat{(\phi \shuffle \psi)}$ is the free $1$\category
on~$\phi\shuffle\psi$ considered as a $1$\polygraph, and
where~$\winterp{-}_{\phi,\psi}$ is defined by the
mappings
%\vspace*{-\abovedisplayskip}\vspace{\abovedisplayshortskip}
\begin{align*}
  u \in (\phi \shuffle \psi)_0 &\mapsto \winterp{u}^{1,1}_{\phi,\psi} \in \freecat\P_2 \\
  \wtrans_{u,v} \in (\phi \shuffle \psi)_1 &\mapsto \winterp{\wtrans_{u,v}}_{\phi,\psi} \in \freecat\P_3\zbox.
\end{align*}
For example, for~$\Q$ defined as above and~$\phi = \psi = \tau
\pcomp_1 \tau$,~$\winterp{\letter l_1 \letter l_2 \letter r_1 \letter
  r_2}_{\phi,\psi}$ and~$\winterp{\letter l_1 \letter r_1 \letter l_2 \letter
  r_2}_{\phi,\psi}$ are respectively the $2$\cells of~$\freecat\Q_2$
\[
  \begin{tabularx}{0.7\linewidth}{>{\hfil}Xc>{\hfil}X}
    $\satex{exch-ex-b}$ &and& $\satex{exch-ex-l-1}$
  \end{tabularx}
\]
and~$\winterp{\wtrans_{\letter l_1,\letter r_2}}_{\phi,\psi}$ and~$\winterp{\wtrans_{\letter l_1\letter r_1,\eps}}_{\phi,\psi}$ are respectively
the $3$\cells of~$\freecat\Q_3$
\[
  \begin{tabularx}{0.7\linewidth}{>{\hfil}Xc>{\hfil}X}
    $\satex{exch-ex-b} \qTO \satex{exch-ex-l-1}$
    &and&
          $\satex{exch-ex-l-1}\qTO \satex{exch-ex-l-2}$\pbox.
  \end{tabularx}
\]

\noindent We write~$\stdpath_{\phi,\psi}$ for the path
\[
   \wtrans_{u_1,v_1} \pcomp_1 \cdots \pcomp_1 \wtrans_{u_{kk'},v_{kk'}}\in
   {\freecat{(\phi \shuffle \psi)}}(\letter l_1 \ldots \letter l_k\letter r_1
   \ldots\letter r_{k'},\letter r_1 \ldots\letter r_{k'} \letter l_1
  \ldots \letter l_k)
\]
defined by induction by
\[
  u_1 = \letter l_1 \ldots \letter l_{k-1} \qquad\qtand\qquad v_1 =\letter r_2 \ldots
  \letter r_{k'}
\]
and where~$u_{i+1},v_{i+1}$ are the unique words of~$\freecat\Sigma_{\phi,\psi}$ such that
\[
  \ctgt_0(\wtrans_{u_i,v_i}) = u_{i+1} \letter l_p\letter r_q v_{i+1}
  \qquad\text{with}\qquad v_{i+1} = \letter r_{q+1} \ldots\letter r_{k'} \letter
  l_{p+1} \ldots \letter l_k
\]
 for some~$p,q \in \N$. We can finally end the definition of Gray presentations
 by putting
 \[
   X_{\phi,\psi} = [\stdpath_{\phi,\psi}]_{\phi,\psi}.
 \]
 For example, for~$\Q$ defined as above,~$X_{\tau\pcomp_1\tau,\tau\pcomp_1\tau}$
 is the composite of $3$\cells of~$\freecat\Q_3$ given
 by~\eqref{eq:ex-move-sequence}.

\begin{example}
  \label{ex:pseudo-monoid-gray-pres}
  We define the \emph{Gray presentation of pseudomonoids} as the
  $4$\prepolygraph obtained by extending the $3$\prepolygraph for
  pseudomonoids~$\P$ seen in \Exr{pseudo-monoid-3-pol}. First, we add to $\P_3$
  the $3$\nbd-generators
  \[
    \hss
    \begin{array}{rccc@{\hspace{2em}}rccc}
      X_{\mu,\bar n,\mu}\co & \satex{mon-ich-mu-mu-l} &\TO &
                                                             \satex{mon-ich-mu-mu-r}
      & X_{\mu,\bar n,\eta}\co & \satex{mon-ich-mu-eta-l} &\TO &
                                                                 \spacebelowsatex{mon-ich-mu-eta-r} \\
      X_{\eta,\bar n,\mu}\co&\satex{mon-ich-eta-mu-l} &\TO&
                                                            \satex{mon-ich-eta-mu-r} &
                                                                                       X_{\eta,\bar n,\eta}\co&\satex{mon-ich-eta-eta-l} &\TO&
                                                                                                                                               \satex{mon-ich-eta-eta-r}
    \end{array}
    \hss
  \]
  for $n\in \N$. Second, we define $\P_4$ as a minimal set of $4$-generators
  such that, given a configuration of cells of $\freecat{(\restrictcat\P 3)}$ as
  in \eqref{eq:peiffer-gen}, there is a corresponding independence generator in
  $\P_4$, and given a configuration of cells of $\freecat{(\restrictcat\P 3)}$
  as in the first or the second line of \eqref{eq:inat-gen}, there is a
  corresponding interchange naturality generator in $\P_4$.
\end{example}

\noindent Our notion of Gray presentation is correct, in the sense that:
\begin{restatable}{theo}{graypresgraycatthm}
  \label{thm:gray-pres-gray-cat}
  % Given a Gray presentation~$\P$, the 3-precategory $\loc{\prescat\P}{\P_X}$,
  % obtained by localizing the presented category under interchange generators, is
  % canonically a Gray category. Similarly, $\loc{\prescat\P}{\P_3}$ is a Gray
  % $(3,2)$-category.
  Given a Gray presentation~$\P$, the presented precategory~$\prespcat\P$ is
  canonically a lax Gray category.
\end{restatable}
\begin{proof}
  See \Appr{gray-pres-gray-cat}.
  % This is a consequence of \Lemr{compat-X-one-cells}, \Lemr{compat-X-one-comp},
  % \Lemr{prespcat-peiffer} and \Lemr{prespcat-exch}.
\end{proof}
\begin{coro}
  \label{coro:gray-pres-gray-gpd}
  Given a Gray presentation $\P$, $\freeinvf{\prespcat\P}$ is canonically a
  $(3,2)$-Gray category.
\end{coro}
\begin{proof}
  By \Thmr{gray-pres-gray-cat} and \Propr{gray-induces-3-2-gray}.
\end{proof}
%%% Local Variables:
%%% mode: latex
%%% TeX-master: "proxy-gray-pres"
%%% End:

%% file: rewriting.tex
In this section, we get to the heart of the matter and introduce our tools in
order to show coherence results for presented Gray categories. These are
obtained as generalizations of techniques developed in rewriting theory by
rewriting morphisms in free precategories, and having a relation~$\sequiv$ on
pairs of parallel rewriting $3$\cells which plays the role of witness for
confluence.
% where we work with a relation~$\sequiv$ between pairs of parallel rewriting
% paths that witnesses the fact that the ``holes'' created by these parallel
% rewriting paths can be filled using elementary tiles.
We first define coherence and show how coherence can be
obtained from a property of confluence on $3$\precategories
(\Ssecr{coherence}). Then, we adapt the elementary notions of rewriting to the
setting of~$3$\prepolygraphs (\Ssecr{rewriting-on-3-pol}) together with
classical results: a criterion for termination based on reduction orders
(\Ssecr{termination}), a critical pair lemma together with a finiteness property
on the number of critical branchings (\Ssecr{critical-branchings}). Our main
result of this section is a coherence theorem for Gray presentations
(\Cref{thm:gray-coherence}), together with an associated coherence criterion
(\Cref{thm:squier}) that will be our main tool for the examples of the next section.
% In this section, we adapt the rewriting theory to give a criterion for the
% coherence of Gray categories. In this adaptation, we are interested in rewriting
% systems that are confluent relatively to congruence on the rewriting paths,
% where this congruence should be understood as equality. We prove adaptations of
% several classical results in rewriting theory to our context, like Newman's
% lemma, the critical branching lemma, the finiteness property on the number of
% critical branchings\todo{cette propriété existe, hein ?} and tools to show
% termination. We moreover show a coherence criterion for coherent
% $3$\precategory, \ie $3$\precategories arising from confluent rewriting systems.

\subsection{Coherence in Gray categories}
\label{ssec:coherence}
The aim of this article is to provide tools to study the coherence of presented
Gray categories, by which we mean the following. A $3$\precategory~$C$ is
\emph{coherent} when, for every pair of parallel $3$\cells
$F_1,F_2\co \phi\TO\psi\in C_3$, we have $F_1=F_2$. By extension, a Gray presentation~$\P$ is
\emph{coherent} when the underlying $(3,2)$\precategory of the $(3,2)$\nbd-Gray
category~$\freeinvf{\prespcat\P}$ is coherent (remember that $\prespcat\P$ is a
lax Gray category by \Cref{thm:gray-pres-gray-cat}, which implies that
$\freeinvf{\prespcat\P}$ is a $(3,2)$\nbd-Gray category by
\Cref{prop:gray-induces-3-2-gray}). Gray presentations~$\P$ with no other
$4$\generators than the independence generators and the interchange naturality
generators are usually not coherent. For example, in the Gray
presentation~$\P$ of pseudomonoids given in \Exr{pseudo-monoid-gray-pres}, we do
not expect the following parallel $3$\cells
\begin{equation}
  \label{eq:parallel-not-equal}
  \begin{tikzcd}[row sep={2.5em,between origins}]
    & \satex{mon-cp1-r} \tar[r]& \satex{mon-cp1-r-2} \tar[rd] \\
    \satex{mon-cp1} \tar[ru]\tar[rd] & & & \satex{mon-cp1-e} \\
    & \satex{mon-cp1-l} \tar[r]& \satex{mon-cp1-l-2} \tar[ru]
  \end{tikzcd}
\end{equation}
to be equal in~$\freeinvf{\prespcat\P}$. For coherence, we need to add ``tiles''
in~$\P_4$ to fill the ``holes'' created by parallel $3$\cells as the ones above.
A trivial way to do this is to add a $4$\generator~$R\co F_1 \TOO F_2$ for every
pair of parallel $3$\cells $F_1$ and $F_2$ of~$\freecat\P$. However, this method
gives quite big presentations, whereas we aim at small ones, so that the number
of axioms to verify in concrete instances is as little as possible. We expose a
better method in \Ssecr{critical-branchings}, in the form of \Cref{thm:squier}: we
will see that it is enough to add a tile of the form
\[
  \begin{tikzcd}[sep=small,cramped]
    & \phi\tar[ld,"S_1"'] \tar[rd,"S_2"] \\
    \phi_1 \tar[rd,"F_1"']&\sequiv& \phi_2\tar[ld,"F_2"] \\
    & \psi
  \end{tikzcd}
\]
for every critical branching $(S_1,S_2)$ of~$\P$ for which we chose $3$\cells
$F_1,F_2$ that make the branching $(S_1,S_2)$ joinable (definitions are
introduced below).

We now show how the coherence property can be obtained
starting from a $3$\precategory whose $3$\cells satisfy a property of confluence,
motivating the adaptation of rewriting theory to $3$\prepolygraphs in later
sections in order to study the coherence of Gray presentations. In fact, we
can already prove an analogous of the Church-Rosser property coming from
rewriting theory in the context of confluent categories.

A $3$\precategory~$C$ is \emph{confluent} when, for
$2$\cells~$\phi,\phi_1,\phi_2 \in C_2$ and $3$\cells
\[
  F_1\co \phi \TO \phi_1
  \qtand
  F_2\co \phi \TO \phi_2
\]
of~$C$, there exist a $2$\cell $\psi \in C_2$ and $3$\cells
\[
  G_1\co \phi_1 \TO \psi \in C_3
  \qtand
  G_2\co \phi_2 \TO \psi \in C_3
\]
of~$C$ such that~$F_1 \pcomp_2 G_1 = F_2 \pcomp_2 G_2$:
\[
  \begin{tikzcd}[sep=small,cramped]
    & \phi\tar[ld,"F_1"'] \tar[rd,"F_1"] \\
    \phi_1 \tar[rd,dotted,"G_1"']&& \phi_2\tar[ld,dotted,"G_2"] \\
    & \psi
  \end{tikzcd}
\]
The $3$-cells of a
$(3,2)$\precategory associated to a confluent $3$\precategory admits a simple
form, as in:
\begin{prop}
  \label{prop:confluent-cr}
  Given a confluent $3$\precategory~$C$, every $3$-cell $F\co \phi \TO \phi' \in \freeinvf
  C$ can be written $F = G \comp_2 \finv H$ for some $G\co \phi \TO \psi \in
  C_3$ and $H\co \phi'\TO \psi \in C_3$.
\end{prop}
\noindent The above property says that confluent categories satisfy a
``Church-Rosser property'' (\cite[Def.~2.1.3]{baader1999term}, for example), and
is analogous to the classical result stating that confluent rewriting systems
are Church-Rosser (\cite[Thm.~2.1.5]{baader1999term}, for example).
\begin{proof}
  By the definition of $\freeinvf C$, a $3$-cell $F\co \phi \TO \phi' \in \freeinvf C$
  can be written
  \[
    F = \finv G_1 \comp_2 H_1 \comp_2 \cdots \comp_2 \finv{G_k} \comp_2 H_k
  \]
  for some $k \ge 0$, $G_i\co \chi_i \TO \phi_{i-1}$ and $H_i\co
  \chi_i \TO \phi_i$ for $1 \le i \le k$ with $\phi_0 = \phi$ and $\phi_k =
  \phi'$, as in
  \[
    \begin{tikzcd}[sep=small,cramped]
      & \chi_1 \tar[ld,"G_1"'] \tar[rd,"H_1"]& & \cdots \tar[ld,"G_2"']
      \tar[rd,"H_{k-1}"] & & \chi_k \tar[ld,"G_k"'] \tar[rd,"H_k"] & \\
      \phi_0 & & \phi_1 & \cdots & \phi_{k-1} & & \phi_k
    \end{tikzcd}
    \pbox.
  \]
  We prove the property by induction on $k$. If $k = 0$, $F$ is an identity and
  the result follows. Otherwise, since $C$ is confluent, there exists $\psi_k$,
  $G_k'\co \phi_{k-1} \to \psi_{k}$ and $H_{k}'\co \phi_{k} \to \psi_{k}$ with
  \[
    \begin{tikzcd}[sep=small,cramped]
      & \chi_k \tar[ld,"G_k"'] \tar[rd,"H_k"] & \\
      {\phantom{\phi_k}\mathmakebox[0pt]{\phi_{k-1}}} \tar[rd,"G_{k}'"'] & = & {\phi_k} \tar[ld,"H_k'"] \\
      & {\psi_{k}}
    \end{tikzcd}
    \pbox.
  \]
  By induction, the morphism
  \[
    \finv G_1 \comp_2 H_1 \comp_2 \cdots \comp_2 \finv{G_{k-2}} \comp_2 H_{k-2}
    \comp_2 \finv{G_{k-1}} \comp_2 (H_{k-1} \comp_2 G_{k}')
  \]
  can be written $G \comp_2 \finv{H}$ for some $\psi$ in $C_2$ and
  $G\co \phi_0 \TO \psi$, $H \co \psi_k \TO \psi$ in $C_3$. Since $G_k \comp_2
  G_k' = H_k \comp_2 H_k'$, we have $\finv{G_k} \comp_2 H_k = G_k'
  \comp_2 \finv{H_k'}$. Hence, 
  \[
    F = G \comp_2 \finv{H} \comp_2 \finv{H_k'} = G
    \comp_2 \finv{(H_k' \comp_2 H)}
  \]
  which is of the wanted form.
\end{proof}
\noindent Starting from a confluent $3$\precategory, we have the following
simple criterion to deduce the coherence of the associated $(3,2)$\precategory:
\begin{prop}
  \label{prop:confluent-impl-coherence}
  Let $C$ be a confluent $3$\precategory which moreover satisfies that, for every
  $F_1,F_2\co \phi \TO \phi' \in C_3$, we have $F_1 = F_2$ in the
  localization~$\freeinvf C$. Then, $\freeinvf C$ is coherent. In particular, if
  $C$ is a confluent $3$\precategory satisfying that, for every $F_1,F_2\co \phi
  \TO \phi'\in C_3$, there is $G\co \phi'\TO \phi''\in C_3$ such that $F_1 \comp_2 G = F_2
  \comp_2 G$ in~$C_3$, then $\freeinvf C$ is coherent.
\end{prop}
\begin{proof}
  Let $F_1,F_2\co \phi \TO \phi' \in \freeinvf{C}_3$. By \Cref{prop:confluent-cr},
  for $i \in \set{1,2}$, we have $F_i = G_i \comp_2 \finv{H_i}$ for some $\psi_i
  \in C_2$, $G_i \co \phi \TO \psi_i \in C_3$ and $H_i \co \phi' \TO \psi_i \in
  C_3$, as in
  \[
    \begin{tikzcd}[column sep={between origins,3.2em},row sep={between origins,2.5em},cramped]
      & \psi_1 & \\
      \phi \tar[ur,"G_1"] \tar[dr,"G_2"']& & \phi' \tar[ul,"H_1"'] \tar[dl,"H_2"] \\
      & \psi_2
    \end{tikzcd}
    \pbox.
  \]
  By confluence, there are $\psi \in C_2$ and $K_i\co \psi_i \TO \psi \in C_3$ for $i \in
  \set{1,2}$, such that $G_1 \comp_2 K_1 = G_2 \comp_2 K_2$. By the second
  hypothesis, we have $H_1 \comp_2 K_1 = H_2 \comp_2 K_2$ so that
  \begin{align*}
    G_1 \comp_2 \finv{H_1} &= G_1 \comp_2 K_1 \comp_2 \finv{(H_1 \comp_2 K_1)} \\
    &= G_2 \comp_2 K_2 \comp_2 \finv{(H_2 \comp_2 K_2)} \\
    &= G_2 \comp_2 \finv{H_2}.
  \end{align*}
  Hence, $F_1 = F_2$. For the last part, note that if $F_1 \comp_2 G = F_2
  \comp_2 G$, then $\eta(F_1) = \eta(F_2)$, where~$\eta$ is the canonical
  $3$\prefunctor~$C \to \freeinvf C$.
\end{proof}

\subsection[Rewriting on \texorpdfstring{$3$}{3}-prepolygraphs]{Rewriting on $\bm 3$-prepolygraphs}
\label{ssec:rewriting-on-3-pol}
As we have seen in the previous section, coherence can be deduced from a
confluence property on the $3$\cells of $3$\precategories. Since confluence of
classical rewriting systems is usually shown using tools coming from rewriting theory, it
motivates an adaptation of it in the context of~$3$\prepolygraphs
for the aim of studying the coherence of Gray presentations.

Given a $3$\prepolygraph~$\P$, a \emph{rewriting step of~$\P$} is a $3$-cell~$S
\in \freecat\P_3$ of the form
\[
  \lambda \comp_1 (l \comp_0 A \comp_0 r) \comp_1
  \rho
\]
for some $l,r \in \freecat\P_1$, $\lambda,\rho \in \freecat\P_2$ and $A \in
\P_3$, with $l,A,r$ $0$-composable and $\lambda,l \comp_0 A \comp_0 r,\rho$
$1$-composable. For such $S$, we say that $A$ is the \emph{inner $3$\generator}
of~$S$. A \emph{rewriting path} is a $3$-cell $F\co \phi \TO \phi'$
in~$\freecat\P_3$. Such a rewriting path has a \emph{length} $\len F \in \N$
which is defined as in \Cref{ssec:cell-nf}. Also, by \Cref{thm:precat-nf}, it
can be uniquely written as a composite of rewriting steps $S_1\comp_2 \cdots
\comp_2 S_{\len F}$, since rewriting steps are exactly $3$\dimensional whiskers.
Given $\phi,\psi \in \freecat\P_2$, \emph{$\phi$ rewrites to $\psi$} when there
exists a rewriting path $F\co \phi \TO \psi$. A~\emph{normal form} is a $2$-cell
$\phi \in \freecat\P_2$ such that for all $\psi \in \freecat\P_2$ and $F\co
\phi\TO \psi$, we have $F = \unit \phi$.\penalty-50{} $\P$ is \emph{terminating}
when there does not exist an infinite sequence of rewriting steps $F_i\co \phi_i
\TO \phi_{i+1}$ for~$i \ge 0$;

A \emph{branching} is a pair
rewriting paths $F_1\co \phi\TO \phi_1$ and $F_2\co \phi\TO \phi_2$ with the same source. The
\emph{symmetric branching} of a branching $(F_1,F_2)$ is $(F_2,F_1)$. A
branching $(F_1,F_2)$ is \emph{local} when both $F_1$ and $F_2$ are rewriting
steps. A branching $(F_1,F_2)$ is \emph{joinable} when there exist rewriting paths $G_1\co \phi_1 \TO
\psi$ and $G_2\co \phi_2 \TO \psi$; moreover, given a congruence $\sequiv$ on
$\freecat\P$, if we have that $F_1 \comp_2 G_1 \sequiv F_2 \comp_2
G_2$, as in
\[
  \begin{tikzcd}[sep=small,cramped]
    &\phi\tar[dl,"F_1"']\tar[dr,"F_2"]&\\
    \phi_1\tar[dr,dotted,"G_1"']&\sequiv&\tar[dl,dotted,"G_2"]\phi_2\\
    &\psi
  \end{tikzcd}
  % \qquad\qquad\qquad\qquad
  % \begin{tikzcd}[sep=2ex]
  %   &&\\
  %   x_1\ar[dr,dotted,"q_1"']\ar[rr,bend left=50,"p"]&\sequiv&\ar[dl,dotted,"q_2"]x_2\\
  %   &y 
  % \end{tikzcd}
\]
we say that the branching is \emph{confluent (for $\sequiv$)}.

A \emph{rewriting system~$(\P,\sequiv)$} is the data of a $3$\prepolygraph $\P$
together with a congruence~$\sequiv$ on~$\freecat\P$. $(\P,\sequiv)$ is
(\emph{locally}) \emph{confluent} when every (local) branching is confluent. It
is \emph{convergent} when it is locally confluent and $\P$ is terminating. Given
a $4$\prepolygraph~$\P$, there is a canonical rewriting system $(\restrictcat \P
3,\stdcong^\P)$ (recall the definition of~$\stdcong^\P$ given in
\Ssecr{presentations}) where $\stdcong^\P$ intuitively witnesses that the
``space'' between two parallel $3$\cells can be filled with elementary tiles
that are the elements of~$\P_4$. In the following, most of the concrete
rewriting systems we study are of this form.

% Note that our notion of rewriting system differs from an abstract rewriting
% system where the objects are the $2$\cells of $\freecat\P$ and the rewrite
% relation~$\to$ is given by the rewriting steps of~$\P$. Indeed, in our
% formalism, rewriting paths are not defined by the transitive closure of~$\to$,
% but are sequences of concrete rewriting steps, so that two parallel rewriting
% paths are not necessarily equal.
% Commentaire Sam: ce paragraphe n'est pas très clair: on peut avoir deux chemins parallèles distincts dans un ARS + je pense qu'on a un ARS sous-jacent non ?
% \item \emph{coherent} when, for all $p,q\co x \to y \in \freecat{\gengpd G}_1$, $p
%   \freegpd\sequiv q$.

The analogues of several well-known properties of abstract
rewriting systems can be proved in our context. In particular, the classical
proof by well-founded induction of Newman's lemma
(\cite[Lem.~2.7.2]{baader1999term}, for example), can be directly adapted in
order to show that:
\begin{theo}
  \label{thm:newman-modulo}
  A rewriting system which is convergent is confluent.
\end{theo}
\begin{proof}
  Let $(\P,\sequiv)$ be a rewriting system which is convergent. Let $\TO^+
  \subseteq \freecat\P_2 \times \freecat\P_2$ be the partial order such that
  $\phi \TO^+ \psi$ if there exists a rewriting path $F\co \phi \TO \psi \in
  \freecat\P_3$ with $\len{F} > 0$. Since the underlying rewriting system is
  terminating, $\TO^+$ is well-founded. Thus, we can prove the theorem by
  induction on~$\TO^+$. Suppose given a branching $F_1\co \phi\TO \phi_1 \in
  \freecat \P_3$ and $F_2\co \phi\TO \phi_2 \in \freecat \P_3$. If $\len{F_1}=0$
  or $\len{F_2}=0$, then the branching is confluent. Otherwise, $F_i = S_i
  \comp_2 F_i'$ with $S_i\co \phi \TO \phi_i'$ a rewriting step and $F_i'\co
  \phi_i' \TO \phi_i$ a rewriting path for $i \in \set{1,2}$. Since the
  rewriting system is locally confluent, there are $\psi \in \freecat\P_2$ and
  rewriting paths $G_i\co \phi_i' \TO \psi$ for $i \in \set{1,2}$ such that $S_1
  \comp_2 G_1 \sequiv S_2 \comp_2 G_2$. Since the rewriting system is
  terminating and $\sequiv$ is stable by composition, by composing the $G_i$'s
  with a path $G\co \psi \TO \psi'$ where $\psi'$ is a normal form, we can
  suppose that $\psi$ is a normal form. By induction on $\phi_1'$ and $\phi_2'$,
  there are rewriting paths $H_i \co \phi_i \TO \psi_i'$ and $F_i''\co \psi \TO
  \psi_i'$ such that $F_i' \comp_2 H_i \sequiv G_i \comp_2 F_i''$ for $i \in
  \set{1,2}$. Since $\psi$ is in normal form, $F_i'' = \unit \psi$ and we have
  $H_i\co \phi_i \TO \psi$ for $i \in \set{1,2}$ as in
  \[
    \begin{tikzcd}[sep=1.5em,cramped]%[sep={4em,between origins}]
      & & \phi \ar[dd,phantom,"\sequiv"]\tar[dl,"S_1"'] \tar[dr,"S_2"] & &\\
      & |[alias=Lp]| \phi_1' \tar[dl,"F_1'"']\tar[dr,"G_1"] & & |[alias=Rp]| \phi_2' \tar[dl,"G_2"'] \tar[dr,"F_2'"] \\
      \phi_1 \tar[rr,"H_1"'{name=L}]& & \psi & & \phi_2 \tar[ll,"H_2"{name=R}]
      \ar[phantom,from=Lp,to=L,"\sequiv"]
      \ar[phantom,from=Rp,to=R,"\sequiv"]
    \end{tikzcd}
    \pbox.
  \]
  Moreover,
  \begin{align*}
    F_1 \comp_2 H_1 &\sequiv S_1 \comp_2 (F_1' \comp_2 H_1) \\
                    &\sequiv S_1 \comp_2 G_1 \\
                    &\sequiv S_2 \comp_2 G_2 \\
                    &\sequiv S_2 \comp_2 (F_2' \comp_2 H_2) \\
                    &\sequiv F_2 \comp_2 H_2.
                      \tag*{\qedhere}
  \end{align*}
\end{proof}
\noindent \Cref{thm:newman-modulo} implies that, up to post-composition, all the
parallel paths of a convergent rewriting system are equivalent. Later, this will
allow us to apply \Cref{prop:confluent-impl-coherence} for showing the coherence
of Gray presentations.
% \begin{lem}
%   \label{lem:equiv-on-nf}
%   In a rewriting system $(\P,\sequiv)$ which is convergent, given two rewriting
%   paths $F_1,F_2\co \phi \TO \phi'$, there is $G\co \phi' \TO \psi$ such that
%   $F_1 \comp_2 G \sequiv F_2 \comp_2 G$.
% \end{lem}
\begin{lem}
  \label{lem:equiv-on-nf}
  Given a convergent rewriting system $(\P,\sequiv)$ and rewriting
  paths $F_1,F_2\co \phi \TO \phi' \in \freecat\P_3$ as in
  \[
    \begin{tikzcd}[cramped]
      \phi
      \tar[d,bend right=49,"F_1"']
      \tar[d,bend left=49,"F_2"]
      \\
      \phi'
    \end{tikzcd}
  \]
  there exists $G\co \phi' \TO \psi \in \freecat\P_3$ such that
  $F_1 \comp_2 G \sequiv F_2 \comp_2 G$, \ie
  \[
    \begin{tikzcd}[sep=small,baseline=(\tikzcdmatrixname-2-3.base),cramped]
      &
      \phi
      \tar[ld,"F_1"']
      \tar[rd,"F_2"]
      & \\
      \phi'
      \tar[rd,"G"']
      & \sequiv &
      \phi'
      \tar[ld,"G"]
      \\
      & \psi
    \end{tikzcd}
    \pbox.
  \]
\end{lem}
\begin{proof}
  Given $F_1,F_2$ as above, since the rewriting system is terminating, there is
  a rewriting path $G\co \phi' \TO \psi$ where $\psi$ is a normal form. By
  confluence, there exist $G_1\co \psi \TO \psi'$ and $G_2\co \psi \TO \psi'$
  such that $F_1 \comp_2 G \comp_2 G_1 \sequiv F_2 \comp_2 G \comp_2 G_2$. Since
  $\psi$ is a normal form, we have $G_1 = G_2 = \unit {\psi}$. Hence, $F_1
  \comp_2 G \sequiv F_2 \comp_2 G$.
\end{proof}
\noindent Note that, in \Lemr{equiv-on-nf}, we do not necessarily have
\[
  \begin{tikzcd}[cramped]
    \phi
    \tar[d,bend right=49,"F_1"']
    \tar[d,bend left=49,"F_2"]
    \ar[d,phantom,"\sequiv"]
    \\
    \phi'
  \end{tikzcd}
\]
which explains why the method we develop in this section for showing coherence
will only apply to $(3,2)$\precategories, but not to general $3$\precategories.

\subsection{Termination}
\label{ssec:termination}

Here, we show a termination criterion for rewriting systems $(\P,\sequiv)$ based
on a generalization of the notion of reduction order in classical rewriting
theory where we require a compatibility between the order and the composition
operations of cells.

A \emph{reduction order} for a $3$\prepolygraph~$\P$ is a well-founded partial
order~$<$ on~$\freecat\P_2$ such that:
\begin{itemize}
  
\item given~$A\co \phi\TO\phi' \in \P_3$, we have~$\phi > \phi'$,
\item given~$l,r \in \freecat\P_1$ and parallel~$\phi,\phi' \in \freecat\P_2$ such
  that~$l,\phi,r$ are $0$\composable and~$\phi > \phi'$, we have
  \[
    l
    \pcomp_0 \phi \pcomp_0 r > l \pcomp_0 \phi' \pcomp_0 r\zbox,
  \]
  
\item given $1$\composable~$\lambda,\phi,\rho \in \freecat\P_2$, and~$\phi' \in
  \freecat\P_2$ parallel to~$\phi$ such that~$\phi > \phi'$, we have
  \[
    \lambda \pcomp_1 \phi
    \pcomp_1 \rho > \lambda \pcomp_1 \phi' \pcomp_1 \rho\zbox.
  \]
\end{itemize}
\noindent The termination criterion is then:
\begin{prop}
  \label{prop:term-order-implies-terminating}
  If $(\P,\sequiv)$ is a rewriting system such that there exists a reduction
  order for $\P$, then $(\P,\sequiv)$ is terminating.
\end{prop}
\begin{proof}
  The definition of a reduction order implies that, given a rewriting step
  $\lambda \comp_1 (l \comp_0 A \comp_0 r) \comp_1 \rho$ with $l,r \in
  \freecat\P_1$, $\lambda,\rho \in \freecat\P_2$ and $A\co \phi \TO \phi' \in
  \P_3$ suitably composable, we have
  \[
    \lambda \comp_1 (l \comp_0 \phi
    \comp_0 r) \comp_1 \rho > \lambda \comp_1 (l \comp_0 \phi' \comp_0 r) \comp_1
    \rho. 
  \]
  So, given a sequence of $2$-composable rewriting steps $(F_i)_{i < k}$,
  where $k \in \N \cup \set{\infty}$, $F_i\co \phi_i \TO \phi_{i+1} \in \P_3$
  for $i < k$, we have $\phi_i > \phi_{i+1}$ for $i < k$. Since $>$ is
  well-founded, it implies that $k \in \N$ Hence, the rewriting sytem
  $(\P,\sequiv)$ is terminating.
\end{proof}

In order to build a reduction order for a Gray presentation~$\P$, we have to
build in particular a reduction order for the subset of~$\P_3$ made of
interchange generators. We introduce below a sufficient criterion for the
existence of such a reduction order. The idea is to consider the lengths of the
$1$\cells of the whiskers in the decompositions of $2$\cells and show that they
are decreasing in some way when an interchange generator is applied.
% for some~$k \in \N$ with~$1 < k \le n$ and
%~$\eps \in \set{-,+}$ if~$\len{\csrctgt\eps(g)} > 0$ for all~$g \in \P_k$.

Let~$\Nfseq$ be the set of finite sequences of elements of~$\N$. We
order~$\Nfseq$ by~$\seqord$ where
\[
  (a_1,\ldots,a_k) \seqord (b_1,\ldots,b_l)
\]
when~$k = l$ and there exists~$i \in \N$ with $1 \le i \le k$ such that~$a_j =
b_j$ for some~$j < i$ and~$a_i < b_i$. Note that~$\seqord$ is well-founded.
Given a $2$\prepolygraph~$\P$, there is a function~$\intnorm\co \freecat\P_2 \to
\Nfseq$ such that, given~$\phi \in \freecat\P_2$, decomposed uniquely (using
\Cref{thm:precat-nf}) as
\[
  \phi = (l_1 \pcomp_0 \alpha_1 \pcomp_0 r_1) \pcomp_1 \cdots \pcomp_1 (l_k
  \pcomp_0 \alpha_k \pcomp_0 r_k)
\]
for some~$k \in \N$,~$l_i,r_i \in \freecat\P_1$ and~$\alpha_i \in \P_2$ for~$i
\in \set{1,\ldots,k}$,~$\intnorm(\phi)$ is defined by
\[
  \intnorm(\phi) = (\len{l_k},\len{l_{k-1}},\ldots,\len{l_1}). 
\]
Then,~$\intnorm$ induces a partial order~$\intord$ on~$\freecat\P_2$ by
putting~$\phi \intord \psi$ when~$\csrctgt\eps_1(\phi) = \csrctgt\eps_1(\psi)$
for~$\eps \in \set{-,+}$ and~$\intnorm(\phi) \seqord \intnorm(\psi)$
for~$\phi,\psi \in \freecat\P_2$.

Given a Gray presentation~$\P$, we say that~$\P$ is \emph{positive}
when~$\len{\ctgt_1(\alpha)} > 0$ for all~$\alpha \in \P_2$.
Under positiveness,
the order~$\intord$ can be considered as a reduction order for the subset of
$3$\generators of a Gray presentation made of interchangers, as in
\begin{prop}
  \label{prop:term-criterion-interchanger}
  Let~$\P$ be a positive Gray presentation. The partial
  order~$\intord$ has the following properties:
  \begin{enumerate}[label=(\roman*),ref=(\roman*)]
  \item \label{prop:term-criterion-interchanger:X}  for every~$\alpha,\beta \in \P_2$
    and~$f \in \freecat\P_1$ such that~$\alpha,f,\beta$ are $0$\composable,
    \[
      \csrc_2(X_{\alpha,f,\beta})
      \intordgt \ctgt_2(X_{\alpha,f,\beta})\zbox,
    \]
  \item \label{prop:term-criterion-interchanger:0-comp} for~$\phi,\phi' \in
    \freecat\P_2$ and~$l,r \in \freecat\P_1$ such that~$l,\phi,r$ are
    $0$\composable, if~$\phi \intordgt \phi'$, then
    \[
      l \pcomp_0 \phi \pcomp_0 r \intordgt l \pcomp_0 \phi' \pcomp_0 r\zbox,
    \]
  \item \label{prop:term-criterion-interchanger:1-comp}
    for~$\phi,\phi',\lambda,\rho \in \freecat\P_2$ such that~$\lambda,\phi,\rho$
    are $1$\composable, if~$\phi \intordgt \phi'$, then
    \[
      \lambda \pcomp_1 \phi \pcomp_1 \rho \intordgt \lambda \pcomp_1 \phi'
      \pcomp_1 \rho\zbox.
    \]
  \end{enumerate}
\end{prop}
\begin{proof}
  % $\intord$ is a partial order since $\intord = {\intnorm^{-1}}(<_\omega)$.
  Given $\alpha,\beta \in \P_2$ and $f \in \freecat\P_1$ with
  $\alpha,f,\beta$ are $0$-composable, recall that
  $X_{\alpha,f,\beta}$ is such that
  \[
    X_{\alpha,f,\beta}\co (\alpha \comp_0 f \comp_0 \csrc_1(\beta)) \comp_1
    (\ctgt_1(\alpha) \comp_0 f \comp_0 \beta) \TO (\csrc_1(\alpha) \comp_0 f
    \comp_0 \beta) \comp_1 (\alpha \comp_0 f \comp_0 \ctgt_1(\beta))
  \]
  Then, we have
  \[
    \intnorm(\csrc_2(X_{\alpha,f,\beta})) = (\len{\ctgt_1(\alpha)} + \len{f},0) \qqtand
    \intnorm(\ctgt_2(X_{\alpha,f,\beta})) = (0,\len{\csrc_1(\alpha)} + \len{f}).
  \]
  Since $\P$ is positive, we have $\len{\ctgt_1(\alpha)} > 0$ so that
  $\intnorm(\csrc_2(X_{\alpha,f,\beta})) \intordgt \intnorm(\ctgt_2(X_{\alpha,f,\beta}))$. Now,
  \ref{prop:term-criterion-interchanger:0-comp}
  and~\ref{prop:term-criterion-interchanger:1-comp} can readily be obtained by
  considering the whisker representations of $\phi$ and $\phi'$ and observing the
  action of $l \comp_0 - \comp_0 r$ and $\lambda \comp_1 - \comp_1 \rho$ on
  these representations and the definition of~$\intnorm$.
\end{proof}
\noindent The positiveness condition is required to prevent $2$\cells with
``floating components'', since Gray presentations with such $2$\cells might not
terminate. For example, given a Gray presentation $\P$ where $\P_0$ and $\P_1$
have one element and $\P_2$ has two $2$\generators $\satex{cup}$ and
$\satex{cap}$, there are $2$\cells of~${\freecat\P}$ with ``floating bubbles''
which induce infinite reduction sequence with interchange generators as the
following one:
\[
  \satex{capcup-cycle1}
  \qTO
  \satex{capcup-cycle2}
  \qTO
  \satex{capcup-cycle3}
  \qTO
  \satex{capcup-cycle4}
  \qTO
  \satex{capcup-cycle1}
  \qTO
  \cdots
\]

\subsection{Critical branchings}
\label{ssec:critical-branchings}
In term rewriting systems, a classical result called the ``critical pair lemma''
states that local confluence is a consequence of the confluence of a subset of
local branchings, called \emph{critical branchings}. The latter can be described
as pairs of rewrite rules that are minimally overlapping, see
\cite[Sec.~6.2]{baader1999term} for details. Note that we used this result
earlier in the proof of \Lemr{precat-locally-confluent}.

Here, we show a similar result for rewriting on Gray presentations (introduced
in \Cref{ssec:gray-presentation}). For this purpose, we give a definition of
critical branchings which is similar to term rewriting systems, \ie as minimally
overlapping local branchings, where we moreover filter out some branchings that
involve interchange generators and that are automatically confluent by our
definition of Gray presentation. Then, we give a coherence theorem for Gray
presentation based on the analysis critical branchings together with an
associated coherence criterion, and we finish the section by stating a
finiteness property on the critical branchings.

Let $\P$ be a $3$\prepolygraph. Given a local branching $(S_1\co \phi
\TO \phi_1,S_2\co \phi \TO \phi_2)$ of~$\P$, we say that the branching
$(S_1,S_2)$ is
\begin{itemize}
\item \emph{trivial} when $S_1=S_2$,
\item \index{minimal branching}\emph{minimal} when for all other local branching~$(S'_1,S'_2)$ such
  that
  \[
    S_i=\lambda\pcomp_1(l\pcomp_0 S'_i\pcomp_0 r)\pcomp_1\rho
  \]
  for~$i \in \set{1,2}$ for some $1$\cells~$l,r$ and $2$\cells~$\lambda,\rho$,
  we have that~$l,r,\lambda,\rho$ are all identities,
\item \emph{independent} when
  \begin{align*}
    S_1&=((l_1\comp_0 A_1\comp_0 r_1)\comp_1\chi\comp_1(l_2\comp_0\phi_2\comp_0 r_2))
    \\
    \shortintertext{and}
    S_2&=((l_1\comp_0\phi_1\comp_0 r_1)\comp_1\chi\comp_1(l_2\comp_0 A_2\comp_0 r_2))
  \end{align*}
  for some $l_i,r_i \in \freecat\P_1$ and $A_i\co \phi_i \TO \phi'_i \in \P_3$
  for $i \in \set{1,2}$ and $\chi \in \freecat\P_2$.
\end{itemize}
If moreover $\P = \restrictcat{\Q} 3$, where $\Q$ is a Gray presentation, we
say that the the branching $(S_1,S_2)$ is
\begin{itemize}
\item \emph{natural} when 
  \[
    S_1=((A\comp_0 g\comp_0 h)\comp_1(f'\comp_0 g\comp_0\psi))
    % \qquad\qquad
    % S_2=X_{\phi,v\comp\alpha}
  \]
  for some $A\co \phi\TO\phi' \co f \To f' \in \P_3$, $\psi\co h \To h'\in
  \freecat\P_2$ and $g \in \freecat\P_1$, and
  \[
    S_2 = \winterp{\wtrans_{u,v}}_{\phi,g \comp_0 \psi} \qtext{with}
    u = \letter l_1 \ldots
    \letter l_{\len{\phi} - 1} \qtand
    v = \letter r_2 \ldots \letter
    r_{\len{\psi}}
  \]
  and similarly for the situation on the second line
  of~\eqref{eq:inat-gen},
\item \emph{critical} when it is minimal, and both its symmetrical branching and
  it are neither trivial nor independent nor natural.
\end{itemize}
In the following, we suppose given a Gray presentation $\Q$ and we write
$(\P,\sequiv)$ for $(\restrictcat{\Q} 3,\stdcong^{\Q})$. Our next goal is to show an adapted
version of the critical pair lemma. We start by two technical lemmas:
\begin{lem}
  \label{lem:existence-min-branching}
  For every local branching $(S_1,S_2)$ of~$\P$, there is a minimal branching
  $(S'_1,S'_2)$ and $1$-cells $l,r \in \freecat\P_1$ and $2$-cells $\lambda,\rho
  \in \freecat\P_2$ such that $S_i = \lambda \comp_1 (l \comp_0 S'_i \comp_0 r)
  \comp_1 \rho$ for $i \in \set{1,2}$.
\end{lem}

\begin{proof}
  We show this by induction on $N(S_1)$ where $N(S_1) = \len{\csrc_2(S_1)} +
  \len{\csrc_1(S_1)}$. Suppose that the property is true for all local
  branchings $(S'_1,S'_2)$ with $N(S'_1) < N(S_1)$. If $(S_1,S_2)$ is not
  minimal, then there are rewriting steps $S'_1,S'_2 \in \freecat\P_3$, $l,r \in
  \freecat\P_1$ and $\lambda,\rho \in \freecat\P_2$ such that $S_i = \lambda
  \comp_1 (l \comp_0 S'_i \comp_0 r) \comp_1 \rho$ for $i \in \set{1,2}$, such
  that $l,r,\lambda,\rho$ are not all identities. Since
  \[
    \len{\csrc_1(S_1)} = \len{l} + \len{\csrc_1(S'_1)} + \len{r} \qtand
    \len{\csrc_2(S_1)} = \len{\lambda}
    + \len{\csrc_2(S'_1)} + \len{\rho},
  \]
  we have $N(S'_1) < N(S_1)$ so there is a minimal branching $(S''_1,S''_2)$ and
  $l',r' \in \freecat\P_1$, $\lambda',\rho' \in \freecat\P_2$ such that $S'_i =
  \lambda' \comp_1 (l' \comp_0 S''_i \comp_0 r') \comp_1 \rho'$ for $i \in
  \set{1,2}$. By composing with $\lambda,\rho,l,r$ and normalizing as in
  \Cref{thm:precat-nf}, we obtain the conclusion of the lemma.
\end{proof}
\begin{lem}
  \label{lem:triv-indep-natural-confluent}
  A local branching of~$\P$ which is either trivial or independent or natural is
  confluent.
\end{lem}
\begin{proof}
  A trivial branching is, of course, confluent. Independent and natural
  branching are confluent thanks respectively to the independence generators and interchange
  naturality generators of a Gray presentation.
\end{proof}

\goodbreak\noindent The critical pair lemma adapted to our context is then:
\begin{theo}[Adapted critical pair lemma]
  \label{thm:cp}
  The rewriting system $(\P,\sequiv)$ is locally confluent if and only if every
  critical branching is confluent.
\end{theo}
\begin{proof}
  If the rewriting system is locally confluent, then, in particular, every
  critical branching is confluent. For the converse implication, by
  \Lemr{existence-min-branching}, to check that all local branchings are
  confluent, it is enough to check that all minimal local branchings are
  confluent. Among them, by \Lemr{triv-indep-natural-confluent}, it is enough to
  check the confluence of the critical branchings.
\end{proof}
\noindent We now state the main result of this section, namely a coherence
theorem for Gray presentations based
on the analysis of the critical branchings:
\begin{theo}[Coherence]
  \label{thm:gray-coherence}
  Let $\Q$ be a Gray presentation and $(\P,\sequiv) = (\restrictcat{\Q}
  3,\stdcong^{\Q})$ be the associated rewriting system. If $\P$ is terminating and all
  the critical branchings of $(\P,\sequiv)$ are confluent, then $\Q$ is a
  coherent Gray presentation.
  \end{theo}
\begin{proof}
  By \Cref{thm:cp}, the rewriting system $(\P,\sequiv)$ is locally confluent, and by
  \Cref{thm:newman-modulo} it is confluent. Since $\prespcat{\Q} =
  \freecat\P/_{\sequiv}$, it implies that $\prespcat{\Q}$ is a confluent
  $3$\precategory. To conclude, it is sufficient to show that the criterion in
  the last part of \Cref{prop:confluent-impl-coherence} is satisfied. But the latter
  is a consequence of \Lemr{equiv-on-nf}.
\end{proof}
\noindent Note that \Cref{thm:gray-coherence} requires the rewriting system
$(\P,\sequiv)$ to be confluent. If it is not the case, one can try to first
apply a modified version of the classical Knuth-Bendix completion
procedure~\cite{knuth1970simple} (see also \cite[Sec.~7]{baader1999term}) which,
in addition to adding new $3$\generators in order to make the system confluent,
also adds $4$\generators in order to make it confluent up to $\sequiv$, in order
to hopefully obtain a confluent Gray presentation. Such a procedure is detailed
in the closely related setting of coherent presentations of monoids
in~\cite{guiraud2013homotopical}, where it is called the Knuth-Bendix-Squier
completion procedure.

Our coherence theorem implies a coherence criterion similar to the ones shown
by Squier, Otto and Kobayashi~\cite[Thm.~5.2]{squier1994finiteness} and
Guiraud and Malbos~\cite[Prop.~4.3.4]{guiraud2009higher}, which states that adding a tile for
each critical branching is enough to ensure coherence:

\begin{theo}
  \label{thm:squier}
  Let $\Q$ be a Gray presentation, such that $\restrictcat\Q 3$ is terminating
  and, for every critical branching $(S_1\co\phi\TO\phi_1,S_2\co\phi\TO\phi_2)$
  of~$\restrictcat\Q 3$, there exist $\psi \in \freecat\Q_2$, $F_i\co\phi_i \TO
  \psi\in\freecat\Q_3$ for $i \in \set{1,2}$ and $G\co S_1 \comp_2 F_1 \TOO S_2
  \comp_2 F_2 \in \Q_4$. Then, $\Q$ is a coherent Gray presentation.
\end{theo}
\begin{proof}
  The definition of~$\Q_4$ ensures that all the critical
  branchings are confluent, so that \Cref{thm:gray-coherence} applies.
\end{proof}%
\noindent Note that, in \Cref{thm:squier}, we do not need to add a $4$\generator~$G$
as in the statement for a critical branching $(S_1,S_2)$ if there is already a
generator $G'$ for the symmetrical branching $(S_2,S_1)$, so that a stronger
statement holds.

To finish this section, we mention a finiteness property for critical
branchings of Gray presentations. This property contrasts with the case of
strict $n$\categories, where finite presentations can have an infinite number of
critical branchings~\cite{lafont2003towards, guiraud2009higher}.

\begin{restatable}{theo}{grayfinitecriticalbranchings}%
  \label{thm:finite-cp}
  Given a Gray presentation $\Q$ where $\Q_2$ and $\Q_3$ are finite and
  $\len{\csrc_2(A)} > 0$ for every $A \in \Q_3$, there is a finite number of
  local branchings $(S_1,S_2)$ with rewriting steps $S_1,S_2 \in \freecat\Q_3$
  such that $(S_1,S_2)$ is a critical branching.
\end{restatable}
\begin{proof}
  See \Appr{finiteness-cp}.
\end{proof}
\noindent The proof of \Cref{thm:finite-cp} happens to be constructive, so that we
can extract an algorithm to compute the critical branchings for such Gray
presentations. An implementation of this algorithm was used to compute the
critical branchings of the examples of the next section.
% \noindent We classify the local branchings by their inner $3$-generators. Given
% a local branching $(S_1,S_2)$ with $S_i = \lambda_i \comp_1 (l_i \comp_0 A_i
% \comp_0 r_i) \comp_1 \rho_i$ with $l_i,r_i \in \freecat\Q_1$, $\lambda_i,\rho_i
% \in \freecat\Q_2$ and $A_i \in \Q_3$, we say that the local branching
% $(S_1,S_2)$ is
% \begin{itemize}
% \item \emph{operational-operational} if neither $A_1$ nor $A_2$ are interchange generators,
  
% \item \emph{structural-operational} if exactly one among $A_1$ and $A_2$ is an
%   interchange generator,
  
% \item \emph{structural-structural} if both $A_1$ and $A_2$ are interchange generators.
% \end{itemize}
% \todo{classification pas ou peu utilisée. faire qqchose}

%%% Local Variables:
%%% mode: latex
%%% TeX-master: "proxy-rewriting"
%%% End:

%% file: applications.tex
We now illustrate the techniques of the previous section and show the coherence
of Gray presentations corresponding to several well-known algebraic structures. For
each structure, we introduce a Gray presentation and study the confluence of the
critical branchings of the associated rewriting system. Then, when the rewriting
system is terminating, we can directly apply \Cref{thm:squier} to deduce the
coherence of the presentation. This will be the case for pseudomonoids,
pseudoadjunctions and Frobenius pseudomonoids. We moreover study the example of
self-dualities, where the associated rewriting system is not terminating, for
which we use specific techniques in order to prove a weak coherence result.

\subsection{Pseudomonoids}
\label{ssec:app-pseudomonoid}
% We define the \emph{$3$\prepolygraph for pseudo-monoid} as the
% $3$\prepolygraph~$\P$ such that
% \[
%   \P_0 = \set{\ast} \qtand \P_1 = \set{f} \qtand \P_2 = \set{ \mu\co \bar
%   2 \To \bar 1, \eta\co \bar 0 \To \bar 1}
% \]
% where we denote $\bar n$ for $\underbrace{f \comp_0 \cdots \comp_0 f}_n$ for $n
% \in \N$, and such that $\P_3$ has the following three elements
% \[
%   \begin{array}{rccc}
%     \monA\co &(\mu \comp_0 \bar 1) \comp_1 \mu &\TO& (\bar 1 \comp_0 \mu) \comp_1 \mu \\
%     \monL\co &(\eta \comp_0 \bar 1) \comp_1 \mu &\TO& \unit{\bar 1} \\
%     \monR\co &(\bar 1 \comp_0 \eta) \comp_1 \mu &\TO& \unit{\bar 1}
%   \end{array}
% \]
% We can use a graphical representation to represent the elements of
% $\freecat\P_2$: firstly, we represent the $2$-generators $\mu$ and $\eta$ by
% $\satex{mu}$ and $\satex{eta}$ respectively. Secondly, we represent whiskers
% $\bar m \comp_0 \alpha \comp_0 \bar n$ for $m,n\in\N$ and $\alpha \in \P_2$ by
% adding $m$ wires on the left and $n$ wires on the right of the representation of
% $\alpha$. For example, $2 \comp_0 \mu \comp_0 3$ is represented by
% $\;\satex{whisk}\;$. The representation of a sequence of whiskers $w_1 \comp_1
% \cdots \comp_1 w_k$ is then obtained by stacking the representation of each
% whisker. For example, the $4$-generator $\monA$, $\monL$, and $\monR$ can be pictured as
% \[
%   \begin{array}{rccc}
%     \monA\co &\satex{mon-assoc-l} &\TO& \satex{mon-assoc-r} \\
%     \monL\co &\satex{mon-unit-l} &\TO& \satex{mon-unit-c} \\
%     \monR\co &\satex{mon-unit-r} &\TO& \satex{mon-unit-c}
%   \end{array}
% \]

In \Exr{pseudo-monoid-gray-pres}, we introduced a Gray presentation~$\P$ for the
theory of pseudomonoids. The set~$\P_4$ of $4$-generators contains only the
required ones in a Gray presentation, so that we do not expect~$\P$ to be
coherent (see~\eqref{eq:parallel-not-equal} for an example). We will show that the rewriting system is terminating and thus,
\Thmr{squier}, adding a $4$-generator corresponding to each
critical branching will turn the presentation into a coherent one. Those
branchings can be computed as in the proof of \Thmr{finite-cp}, which is
constructive: we obtain, up to symmetrical branchings, five critical branchings:
\[
      \begin{tikzcd}[ampersand replacement=\&,sep={5em,between origins},cramped,baseline=(\tikzcdmatrixname-1-1.center)]
        \satex{mon-cp1} \tar[r] \tar[d] \& \satex{mon-cp1-r}
        \\
        \satex{mon-cp1-l}
      \end{tikzcd}
  \qquad\qquad
      \begin{tikzcd}[ampersand replacement=\&,sep={5em,between origins},cramped,baseline=(\tikzcdmatrixname-1-1.center)]
        \satex{mon-cp2} \tar[r,""{name=src}] \tar[d] \&\satex{mon-cp2-r}  \\
        \satex{mon-cp2-l}
      \end{tikzcd}
  \qquad\qquad
      \begin{tikzcd}[ampersand replacement=\&,sep={5em,between origins},baseline=(\tikzcdmatrixname-1-1.center),cramped]
        \satex{mon-cp4} \tar[d] \tar[r,""{name=srcp}] \& \satex{mon-cp4-l}
        \\
        \satex{mon-cp4-r} \&
      \end{tikzcd}%
\]
\[
  \begin{tikzcd}[ampersand replacement=\&,sep={5em,between origins},cramped,baseline=(\tikzcdmatrixname-1-1.center)]
    \satex{mon-cp3} \tar[r] \tar[d] \&\satex{mon-cp3-r}
    \\
    \satex{mon-cp3-l}
  \end{tikzcd}
  \qquad
  \qquad
  \qquad
  \begin{tikzcd}[ampersand replacement=\&,sep={5em,between origins},baseline=(\tikzcdmatrixname-1-1.center),cramped]
    \satex{mon-cp5} \tar[r] \tar[d] \&\satex{mon-cp5-l}
    \\
    \satex{mon-cp5-r}
  \end{tikzcd}
\]
We observe that each of these branchings is joinable, and we define formal new
$4$-generators $R_1, R_2, R_3, R_4, R_5$ that fill the holes:
\[
  \begin{tikzcd}[ampersand replacement=\&,sep={5.5em,between origins},baseline=(\tikzcdmatrixname-2-1.center)]
    \satex{mon-cp1} \tar[r] \tar[d] \& \satex{mon-cp1-r} \ar[phantom,d,"\overset{R_1}\LLeftarrow"] \tar[r] \&
    \satex{mon-cp1-r-2} \tar[d,""{auto=false,name=src}] \\
    \satex{mon-cp1-l} \tar[r] \& \satex{mon-cp1-l-2} \tar[r] \&
    \satex{mon-cp1-e}
  \end{tikzcd}
  \quad
  \begin{tikzcd}[ampersand replacement=\&,sep={5.5em,between origins},baseline=(\tikzcdmatrixname-2-1.center)]
    \satex{mon-cp2} \tar[r,""{name=src,auto=false}] \tar[d] \&\satex{mon-cp2-r} \tar[d] \\
    \satex{mon-cp2-l}  \&\satex{mon-cp2-r-2} \tar[l,""{name=tgt,auto=false}]
    \ar[from=src,to=tgt,phantom,"\overset{R_2}\LLeftarrow"]
  \end{tikzcd}
  \quad
  \begin{tikzcd}[ampersand replacement=\&,sep={5.5em,between origins},baseline=(\tikzcdmatrixname-2-1.center)]
    \satex{mon-cp4} \tar[r,""{auto=false,name=src}] \tar[d] \& \satex{mon-cp4-l}
    \tar[d] \\
    \satex{mon-cp4-r} \&\satex{mon-cp4-l-2} \tar[l,""{auto=false,name=tgt}]
    \ar[phantom,from=src,to=tgt,"\overset{R_3}\LLeftarrow"]
  \end{tikzcd}
\]%
\[
  \begin{tikzcd}[ampersand replacement=\&,sep={5.5em,between origins},baseline=(\tikzcdmatrixname-2-1.center)]
    \satex{mon-cp3} \tar[d] \tar[r,""{auto=false,name=src}]  \&
    \satex{mon-cp3-r} \tar[d]
    \\
    \satex{mon-cp3-l} \ar[r,equal,""{auto=false,name=tgt}] \& \satex{mon-cp3-l}
    \ar[phantom,"\overset{R_4}\LLeftarrow",from=src,to=tgt]
  \end{tikzcd}
  \qquad
  \qquad
  \begin{tikzcd}[ampersand replacement=\&,sep={5.5em,between origins},baseline=(\tikzcdmatrixname-2-1.center)]
    \satex{mon-cp5}
    \tar[d] \tar[r,""{auto=false,name=src}] \& \satex{mon-cp5-l} \tar[d] \\
    \satex{mon-cp5-r} \ar[r,equal,""{auto=false,name=tgt}] \& \satex{mon-cp5-r}
    \ar[phantom,from=src,to=tgt,"\overset{R_5}\LLeftarrow"]
  \end{tikzcd}
\]
We then define~$\PMon$ as the Gray presentation obtained from~$\P$ of
\Exr{pseudo-monoid-gray-pres} by adding $R_1,\ldots,R_5$ to~$\P_4$.

As claimed above, in order to deduce coherence, we need to show the termination
of~$\PMon$. For this purpose, we use the tools of \Secr{rewriting} and build a
reduction order. We split the task in two and define a first order that
handles the termination of the $\monA,\monL,\monR$ generators, and then a second one that handles
the termination of interchange generators. For the first task, we use a technique similar
to the one used in~\cite{lafont1992penrose}. Given~$n \in \N$, we
write~$\ltex^1$ for the partial order on~$\N^n$ such that, given $a,b \in \N^n$,
$a \ltex^1 b$ when $a_i \le b_i$ for all $i \in \set{1,\ldots,n}$ and there exists
$j \in \set{1,\ldots,n}$ such that $a_{j} < b_{j}$. Let~$\Monex$ be the
$2$\precategory
\begin{itemize}
\item which has only one $0$-cell: $\Monex_0 = \set{\ast}$,
\item whose $1$-cells are the natural numbers: $\Monex_1 = \N$,
\item whose $2$-cells $m \To n$ for $m,n\in \N$ are the strictly monotone
  functions
  \[
    \phi\co (\N^m,\ltex^1) \to (\N^n,\ltex^1).
  \]
\end{itemize}
Moreover, $\unit\ast = 0$ and composition of $1$-cells is given by addition.
Given $m \in \Monex_1$, $\unit m$ is the identity function on~$\N^m$, and given
$m,n,k,k' \in \N$ and $\chi \co k \to k' \in \Monex_2$, the $2$-cell
\[
  m \comp_0 \chi
  \comp_0 n \co {m+k+n} \To {m+k'+n}
\]
is the function $\chi'\co \N^{m+k+n} \to \N^{m+k'+n}$ such that, for~$x = (x_1,\ldots,x_{m+k+n})\in
\N^{m+k+n}$, for $i \in \set{1,\ldots,m+k'+n}$,
\[
  \chi'(x)_i =
  \begin{cases}
   x_i & \text{if $i \le m$} \\
   \chi(x_{m+1},\ldots,x_{m+k})_{i - m} & \text{if $m < i \le m + k'$} \\
   x_{i - k' + k} & \text{if $i > m + k'$}
  \end{cases}
\]
and, given $m,n,p \in \N$, $\phi \co m \To n \in \Monex_2$ and $\psi \co n \To p
\in \Monex_2$, $\phi \comp_1 \psi$ is defined as $\psi \circ \phi$ and one shows
readily that these operations indeed give strictly monotone functions. One
easily checks that $\Monex$ is a strict $2$\category.
% The $2$\precategory~$\Monex$ has in fact more structure:
% \begin{prop}
%   \label{prop:monex-2cat}
%   $\Monex$ is a $2$-category.
% \end{prop}
% \begin{proof}
%   Given $m,m',n,n' \in \N$ and $\phi\co m \To m' \in \Monex_2$ and $\psi \in
%   \Monex_2$, we check by pointwise computation that
%   \[
%     (\phi \comp_0 n) \comp_1 (m' \comp_0 \psi) = (m \comp_0 \psi) \comp_1 (\phi
%     \comp_0 n')
%   \]
%   holds, from which the conclusion follows.
% \end{proof}
Given $m,m',n,n' \in \N$ and $\phi\co m\To n,\psi\co m'\To n' \in
\Monex$, we write $\phi \ltex^2 \psi$ when $m = m'$, $n = n'$ and $\phi(x)
\ltex^1 \psi(x)$ for all~$x \in \N^m$. We have that:
\begin{prop}
  $\ltex^2$ is well-founded on $\Monex_2$.
\end{prop}
\begin{proof}
  We define a function $N\co \Monex_2 \to \N$ by
  \begin{center}
    $N(\phi) = \phi(z)_1 + \cdots + \phi(z)_n$ \qquad for $\phi\co m \To n \in
    \Monex_2$
  \end{center}
  where $z = (0,\ldots,0)$. Now, if $\psi\co m \To n \in \Monex_2$ is such that
  $\psi \ltex^2 \phi$, then $\psi(z) \ltex^1 \phi(z)$ so that
  $N(\psi)<N(\phi)$. Thus, $\ltex^2$ on $\Monex_2$ is well-founded.
\end{proof}
\noindent We observe that the order $\ltex^2$ is compatible with the structure of $\Monex$:
\begin{prop}
  \label{prop:lessexists-stable}
  Given $m,n,m',n',k,k' \in \N$, $\mu\co m' \To m$, $\nu\co n \To n'$,
  and $\phi,\phi'\co k\To k' \in \Monex_2$ such that $\phi \gtex^2 \phi'$, we
  have
  \begin{enumerate}[label=(\roman*),ref=(\roman*)]
  \item \label{prop:lessexists-stable:0comp}$m \comp_0 \phi \comp_0 n \gtex^2 m \comp_0 \phi' \comp_0 n$,

  \item \label{prop:lessexists-stable:1comp}$\mu \comp_1 \phi \comp_1 \nu \gtex^2 \mu \comp_1 \phi' \comp_1 \nu$.
  \end{enumerate}
\end{prop}
\begin{proof}
  Given~$x \in \N^{m+k+n}$, we have~$\phi(x_{m+1},\ldots,x_{m+k}) \gtex^1
  \phi'(x_{m+1},\ldots,x_{m+k})$ so
  \[
    (m \pcomp_0 \phi \pcomp_0 n)(x) \gtex^1 (m
    \pcomp_0 \phi' \pcomp_0 n)(x)\zbox.
  \]
  Thus,~\ref{prop:lessexists-stable:0comp} holds. Moreover, given~$y \in
  \N^{m'}$, we have~$\phi(\mu(y)) \gtex^1 \phi'(\mu(y))$. Since~$\nu$ is monotone,
  we have~$\nu(\phi(\mu(y))) \gtex^1 \nu(\phi'(\mu(y)))$.
  Thus,~\ref{prop:lessexists-stable:1comp} holds.
\end{proof}
\noindent We define a $2$\prefunctor $F\co \freecat{\PMon}_2 \to \Monex$ by
the universal property of the $2$\polygraph $\restrictcat\PMon{2}$, \ie $F$ is
the unique functor such that $F(\ast) = \ast$, $F(\bar 1) = 1$, $F(\mu) = f_\mu$
and $F(\eta) = f_\eta$ where
\[
  f_\mu\co \N^2 \to \N^1 \qquad\qquad f_\eta \co \N^0 \to \N^1
\]
are defined by $f_\mu(x,y) = 2x + y + 1$ for all $x,y \in \N$ and $f_\eta() =
1$. The interpretation exhibits the $3$\generators $\monA$, $\monL$ and $\monR$ of~$\PMon$
as decreasing operations:
\begin{prop}
  \label{prop:mon-interp-gen-decreasing}
  The followings hold:
  \begin{enumerate}[label=(\roman*),ref=(\roman*)]
  \item \label{prop:mon-interp-gen-decreasing:A}$F(\csrc_2(\monA)) \gtex^2 F(\ctgt_2(\monA))$,

  \item \label{prop:mon-interp-gen-decreasing:L} $F(\csrc_2(\monL)) \gtex^2 F(\ctgt_2(\monL))$,
  \item \label{prop:mon-interp-gen-decreasing:R} $F(\csrc_2(\monR)) \gtex^2 F(\ctgt_2(\monR))$,
  \item \label{prop:mon-interp-gen-decreasing:X} $F(\ctgt_2(X_{\alpha,m,\beta})) = F(\csrc_2(X_{\alpha,m,\beta}))$ for $\alpha,\beta \in \PMon_2$ and
    $m \in \N$.
  \end{enumerate}
\end{prop}
\begin{proof}
  Let $\phi = F(\csrc_2(\monA))$ and $\psi = F(\ctgt_2(\monA))$. By calculations, we get that
  \[
    \phi(x,y,z) = (4x + 2y + z + 3) \qqtand \psi(x,y,z) = (2x + 2y + z + 1)
  \]
  for
  $x,y,z \in \N$,
  so $\phi(x,y,z) \gtex^1 \psi(x,y,z)$ for all $x,y,z \in \N$.
  The cases~\ref{prop:mon-interp-gen-decreasing:L} and~\ref{prop:mon-interp-gen-decreasing:R} are shown similarly.
  \ref{prop:mon-interp-gen-decreasing:X} is a consequence of the fact that
  $\Monex$ is a strict $2$\category.
\end{proof}
\noindent
% Let $W = \Monex_2 \times \Nfseq$. We define a partial order $<_W$ on
% $W$ by putting
% \[
%   (\phi,a) < (\psi,b) \text{ when } \phi \ltex \psi \text{ or } [\phi =
%   \psi \text{ and } a <_\omega b].
% \]
We define a partial order~$<$ on~$\freecat{\PMon}_2$ by
putting, for $\phi,\psi \in \freecat{\PMon}_2$,
\begin{center}
  $\phi < \psi$ when $F(\phi) \ltex^2 F(\psi)$ or [$F(\phi) =
  F(\psi)$ and $\intnorm(\phi) <_\omega \intnorm(\psi)$].
\end{center}
\begin{prop}
  \label{prop:pseudomon-term-order}
  The partial order~$<$ on~$\freecat{\PMon}_2$ is a reduction order for~$\PMon$.
\end{prop}
\begin{proof}
  Let $G \in \PMon_3$. If~$G \in \set{\monA,\monL,\monR}$, then, by
  \Propr{mon-interp-gen-decreasing}, $\ctgt_2(G) < \csrc_2(G)$. Otherwise, if
  ${G = X_{\alpha,u,\beta}}$ for some $\alpha,\beta \in \PMon_2$ and $u \in
  \freecat{\PMon}_1$, then, by
  \Propr{mon-interp-gen-decreasing}\ref{prop:mon-interp-gen-decreasing:X},
  \[
    F(\ctgt_2(G)) = F(\csrc_2(G)) \quad\qtand\quad \intnorm(\ctgt_2(G)) <_\omega
    \intnorm(\csrc_2(G)).
  \]
  So $\ctgt_2(G) < \csrc_2(G)$. The other requirements for~$<$
  to be a reduction order are consequences of \Propr{lessexists-stable} and
  \Propr{term-criterion-interchanger}\ref{prop:term-criterion-interchanger:0-comp}\ref{prop:term-criterion-interchanger:1-comp}.
\end{proof}
\noindent Finally, we can use our coherence criterion to show that:
\begin{theo}
  \label{thm:pseudomon-coherent}
  $\PMon$ is a coherent Gray presentation.
\end{theo}
\begin{proof}
  By~\Propr{pseudomon-term-order}, $\PMon$ has a reduction order, so the
  rewriting system~$\PMon$ is terminating by
  \Propr{term-order-implies-terminating}. Since $R_1,\ldots,R_5 \in \PMon_4$, by
  \Thmr{squier}, $\freeinvf{\prespcat{\PMon}}$ is a coherent
  $(3,2)$-Gray category.
\end{proof}

\subsection{Pseudoadjunctions}
\label{ssec:app-adjunction}
We now show the coherence of the Gray presentation of pseudoadjunctions
introduced below. The way we do this is again by using
\Thmr{squier}. However, we need a specific argument to show
the termination of the interchange generators on the associated rewriting
system. For this, we introduce a notion of ``connected'' diagrams and we use a
result of~\cite{delpeuch2018normalization} stating that interchange generators
terminate on such connected diagrams.

We define the $3$\prepolygraph for pseudoadjunctions as the $3$\prepolygraph~$\P$ such
that
\[
  \P_0 = \set{\appfont x,\appfont y} \qtand \P_1 = \set{\appfont f\co \appfont x
    \to \appfont y, \appfont g \co \appfont y \to \appfont x} \qtand \P_2 =
  \set{\mathsf\eta\co \unit{\appfont x} \To \appfont f \comp_0 \appfont g,
    \varepsilon\co \appfont g \comp_0 \appfont f \To \unit{\appfont y}}
\]
where~$\eta$ and~$\varepsilon$ are pictured as~$\satex{cap}$ and~$\satex{cup}$
respectively, and~$\P_3$ is defined by $\P_3 = \set{\adjN,\adjNinv}$, where
\[
  \adjN\co (\eta \comp_0 \appfont f) \comp_1 (\appfont f \comp_0 \varepsilon) \TO \unit
  {\appfont f}
  \qtand
  \adjNinv\co (\appfont g \comp_0 \eta) \comp_1 (\varepsilon \comp_0 \appfont b) \TO
  \unit {\appfont g}
\]
which can be represented by
\[
  \begin{tikzcd}
    \satex{adj2-l}\tar[r,"\adjN"]&\satex{adj2-r}
  \end{tikzcd}
  \qtand
  \begin{tikzcd}
    \satex{adj1-l}\tar[r,"\adjNinv"]&\satex{adj1-r}
  \end{tikzcd}
  \pbox.
\]
We then extend~$\P$ to a Gray presentation by adding $3$\generators
corresponding to interchange generators and $4$\generators corresponding to
independence generator and interchange naturality generator, just like we did
for pseudomonoids in \Exr{pseudo-monoid-gray-pres}. For coherence, we need to
add other $4$\generators to
$\P_4$. Provided that~$\P$ is terminating, by \Thmr{squier},
adding $4$\generators that fill the holes created by critical branchings is
enough, just like for pseudomonoids.

Using the constructive proof of \Thmr{finite-cp}, we compute all the critical
branchings of~$\P$. We then obtain, up to symmetrical branchings, two critical branchings:
\[
  \begin{tikzcd}[column sep=3ex,baseline=(\tikzcdmatrixname-1-1.center)]
    \satex{adj-cp1-l}\tar[dr]\tar[rr]&&\satex{adj-cp1-c}\\
    &\satex{cup}
  \end{tikzcd}
  \qquad\quad
  \begin{tikzcd}[column sep=3ex,baseline=(\tikzcdmatrixname-1-1.center)]
    \satex{adj-cp2-l}\tar[dr]\tar[rr]& &\satex{adj-cp2-c}\\
    &\satex{cap}
  \end{tikzcd}
\]
We observe that each of these branchings is joinable, and we define formal new
$4$\generators $R_1,R_2$ that fill the holes:
\[
  \begin{tikzcd}[column sep=3ex,baseline=(\tikzcdmatrixname-1-1.center)]
    \satex{adj-cp1-l}\tar[dr]\tar[rr]&\ar[d,phantom,pos=0.3,"\overset{R_1}\LLeftarrow"]&\tar[dl]\satex{adj-cp1-c}\\
    &\satex{cup}
  \end{tikzcd}
  \qquad\quad
  \begin{tikzcd}[column sep=3ex,baseline=(\tikzcdmatrixname-1-1.center)]
    \satex{adj-cp2-l}\tar[dr]\tar[rr]&\ar[d,phantom,pos=0.3,"\overset{R_2}\LLeftarrow"description]&\tar[dl]\satex{adj-cp2-c}\\
    &\satex{cap}
  \end{tikzcd}
\]
We then define~$\PAdj$ as the Gray presentation obtained from~$\P$ by adding
$R_1$ and~$R_2$ to~$\P_4$.

We aim at showing that this rewriting system is terminating by exhibiting a
reduction order. However, we cannot use \Propr{term-criterion-interchanger}
to handle interchangers (as for the case of pseudomonoids) since $\P$ is not
positive. Instead, we invoke the result of
\cite{delpeuch2018normalization} which states the termination of interchangers on
``connected diagrams''. Given a $2$\prepolygraph~$\Q$, a $2$-cell of
$\freecat\Q_2$ is connected when, intuitively, each $2$\generator on its
graphical representation is accessible by a path starting from a top or bottom
input. For example, given $\Q$ such that $\Q_0 = \set{\ast}$, $\Q_1 = \set{\bar
  1}$ and $\Q_2 = \set{\satex{cap}\co \bar 0 \To \bar 2, \satex{cup}\co \bar 2
  \To \bar 0}$, we can build the following two $2$-cells of $\freecat\Q_2$:
\[
  \satex{adj-ex-conn} \hspace*{8em} \satex{adj-ex-not-conn}
\]
where the one on the left is connected whereas the one on the right is not,
since the two generators of the ``bubble'' cannot be accessed from the top or
bottom border.

A more formal definition can be
obtained by computing the ``connected components'' of the diagram, together with
a map between the top and bottom inputs of the diagram to the associated
connected components. This is adequatly represented by cospans of~$\Set$. Based
on this idea, we define a $2$\precategory that allows us to compute the connected
components of a $2$-cell of~$\freecat\Q$. Let~$\N_m$ be the set
$\set{1,\ldots,m}$ for $m \ge 0$.

We define the $2$\precategory~$\cospancat$ as
the $2$\precategory such that:
\begin{itemize}
\item it has a unique $0$-cell, denoted $\ast$,
  
\item the $1$-cells are the natural numbers, with $0$ as unit and addition as composition,
  
\item the $2$-cells $m \To n$ are the classes of equivalent cospans
  $\tikzcdin{\N_m \ar[r,"f"]\& S \& \N_n \ar[l,"g"']}$ in~$\Set$,
\end{itemize}
where two cospans $
\begin{tikzcd}[cramped,column sep=small]
  A \ar[r,"f"] & S & \ar[l,"g"'] B
\end{tikzcd}
$ and
$
\begin{tikzcd}[cramped,column sep=small]
  A \ar[r,"f'"] & S' & \ar[l,"g'"'] B
\end{tikzcd}
$ are said \emph{equivalent} when there exists an isomorphism $h\co S \to S' \in
\Set$ such that $f' = h \circ f$ and $g' = h \circ g$. The unit of $m \in
\cospancat_1$ is the cospan $\tikzcdin{\N_m \ar[r,"\id {\N_m}"] \& \N_m \& \N_m
  \ar[l,"\id {\N_m}"']}$, and, given $\phi \co m_1 \To m_2 \in \cospancat_2$ and
$\psi \co m_2 \To m_3 \in \cospancat_2$, represented by the cospans
\[
  \tikzcdin{\N_{m_1} \ar[r,"f"] \& S \& \N_{m_2} \ar[l,"g"']}
  \qtand
  \tikzcdin{\N_{m_2} \ar[r,"f'"] \& S' \& \N_{m_3} \ar[l,"g'"']}
\]
respectively, their composite is represented by the cospan
\[
  \begin{tikzcd}[sep=small]%[sep={4em,between origins}]
    & & S'' \ar[dd,phantom,"{\dcorner}",very near start] & &\\
    & S \ar[ru,"h",dotted] & & S' \ar[lu,"h'"',dotted] \\
    \N_{m_1} \ar[ru,"f"] & & \N_{m_2} \ar[lu,"g"'] \ar[ru,"{f'}"] & & \N_{m_3} \ar[lu,"{g'}"']
  \end{tikzcd}
\]
where the middle square is a pushout. Given $\phi\co m \To n \in \cospancat_2$
represented by
\[
  \tikzcdin{\N_{m} \ar[r,"f"] \& S \& \N_{n} \ar[l,"g"']}
\]
and $p,q \in \cospancat_1$, the $2$-cell $p \comp_0 \phi \comp_0 q$ is
represented by the cospan
\[
  \begin{tikzcd}[sep=small]%[sep={4em,between origins}]
    & \N_p \sqcup S \sqcup \N_q &  \\
    \N_{p + m + q} \ar[ru,pos=0.3,"(\id {\N_p} \sqcup f \sqcup \id {\N_q}) \circ
      \theta_{p,m,q}"] & & \N_{p + n + q} \ar[lu,pos=0.3,"(\id {\N_p} \sqcup g \sqcup \id
      {\N_q}) \circ \theta_{p,n,q}"']
  \end{tikzcd}
\]
where $\theta_{p,r,q}\co \N_{p + r + q} \to \N_{p} \sqcup \N_r \sqcup \N_q$, for
$r \in \N$, is the obvious bijection.
One easily verifies that $\cospancat$ is in fact a $2$\category (fact that will be
useful when dealing with interchange generators later).

Given a $2$\prepolygraph~$\Q$, by the universal property of $2$\prepolygraphs, we
define a $2$\prefunctor $\coninterp_\Q\co \freecat\Q \to \cospancat$ such that
\begin{itemize}
\item the image of $x \in \Q_0$ is $\ast$,
\item the image of $a \in \Q_1$ is $1$,
\item the image of $\alpha\co f \To g \in \Q_2$ is represented by the unique
  cospan $\tikzcdin{\N_{\len f} \ar[r,"\ast"] \& \set{\ast} \& \N_{\len g}
    \ar[l,"\ast"']}$
\end{itemize}
We can now give our definition for connectedness: a $2$-cell $\phi \in \freecat\Q_2$ is
\emph{connected} when $\coninterp_\Q(\phi)$ is represented by a cospan
$\tikzcdin{\N_{m} \ar[r,"f"] \& S \& \N_{n} \ar[l,"g"']}$, with $m =
\len{\csrc_1(\phi)}$ and $n = \len{\ctgt_1(\phi)}$, such that $f,g$ are jointly
epimorphic. Since the latter property is invariant by equivalences of cospan, if
$\phi$ is connected, then for every representative $\tikzcdin{\N_{m} \ar[r,"f"] \& S
  \& \N_{n} \ar[l,"g"']}$ of $\coninterp_\Q(\phi)$, $f,g$ are jointly epimorphic.

\bigskip \noindent In the case of~$\PAdj$, as one can expect, the $3$\generators $\adjN$
and $\adjNinv$ do not change connexity:
\begin{lem}
  \label{lem:adj-connectivity-3gen}
  We have
  \[
    \coninterp_{\PAdj}((\eta \comp_0 \appfont f) \comp_1 (\appfont f \comp_0
    \varepsilon)) = \coninterp_{\PAdj}(\unit{\appfont f})
  \]
  and
  \[
    \coninterp_{\PAdj}( (\appfont g \comp_0 \eta)
    \comp_1 (\varepsilon \comp_0 \appfont g)) = \coninterp_{\PAdj}(\unit{\appfont g}).
  \]
\end{lem}
\begin{proof}
  By calculations, we verify that
  \[
    \begin{tikzcd}[sep=small]%[sep={3em,between origins}]
      & \set{\ast} & \\
      \N_{1} \ar[ru,"\ast"] & & \N_{1} \ar[lu,"\ast"']
    \end{tikzcd}
  \]
  is a representative of both $\coninterp_{\PAdj}((\eta \comp_0 \appfont f)
  \comp_1 (\appfont f \comp_0 \varepsilon))$ and $\coninterp_{\PAdj}(\unit{\appfont
    f})$, so that
  \[
    \coninterp_{\PAdj}((\eta \comp_0 \appfont f)
  \comp_1 (\appfont f \comp_0 \varepsilon)) = \coninterp_{\PAdj}(\unit{\appfont
    f})
  \]
  and similarly,
  \[
    \coninterp_{\PAdj}( (\appfont g \comp_0 \eta)
    \comp_1 (\varepsilon \comp_0 \appfont g)) = \coninterp_{\PAdj}(\unit{\appfont g}). \qedhere
  \]
\end{proof}

\noindent Moreover, connexity is preserved by interchangers in general:
\begin{lem}
  \label{lem:connectivity-int}
  Let $\P$ be a $2$\prepolygraph. Let $\alpha,\beta \in \P_2$ and $g \in
  \freecat\P_1$ such that $\alpha,g,\beta$ are $0$\composable. Then,
  \[
    \coninterp_\P((\alpha \comp_0 g \comp_0 \csrc_1(\beta)) \comp_1
    (\ctgt_1(\alpha) \comp_0 g \comp_0 \beta))
    =
    \coninterp_\P((\csrc_1(\alpha) \comp_0 g \comp_0 \beta) \comp_1
    (\alpha \comp_0 g \comp_0 \ctgt_1(\beta)))
  \]
\end{lem}

\begin{proof}
  This is a direct consequence of the fact that $\cospancat$ is a strict $2$\category.
\end{proof}

\noindent We now prove a technical lemma that we will use to show the connexity of the
$2$-cells in~$\freecat{\PAdj}_2$:
\begin{lem}
  \label{lem:epi-cospan-connectivity}
  Let $\P$ be a $2$\prepolygraph and $\phi,\phi' \in \freecat\P_2$ and
  $\tikzcdin{\N_{n_1} \ar[r,"f"] \& S \& \N_{n_2} \ar[l,"g"']}$ be a
  representative of $\coninterp_\P(\phi)$ for some $n_1,n_2 \in \N$ such that
  $\phi,\phi'$ are $1$\composable and $f$ is surjective. Then, $\phi \comp_1
  \phi'$ is connected if and only if $\phi'$ is connected.
\end{lem}

\begin{proof}
  Let $\tikzcdin{\N_{n_2} \ar[r,"f'"] \& S' \& \N_{n_3} \ar[l,"g'"']}$ be a
  representative of~$\coninterp_\P(\phi')$ for some $n_2,n_3 \in \N$. Then,
  $\coninterp_\P(\phi \comp_1 \phi')$ is represented by $\tikzcdin[sep=2.5em]{\N_{n_1}
    \ar[r,"f'' \circ f"] \& S'' \& \N_{n_3} \ar[l,"g'' \circ g'"']}$ where
  $S''$, $f''$ and~$g''$ are defined by the pushout of~$g$ and~$f'$ as in
  \[
    \begin{tikzcd}[sep=small]%[row sep={4em,between origins},column sep={4em,between origins}]
      & & S'' \ar[dd,phantom,very near
      start,"{\dcorner}"] & &\\
      & S \ar[ru,"f''",dotted]& & S' \ar[lu,"g''"',dotted] \\
      \N_{n_1} \ar[ru,"f"] & & \N_{n_2} \ar[lu,"g"'] \ar[ru,"f'"] & & \N_{n_3} \ar[lu,"g'"']
    \end{tikzcd}
    \pbox.
  \]
  Suppose that $\phi'$ is connected, \ie $f'$ and~$g'$ are jointly
  surjective. Since $f$ is surjective by hypothesis and ~$f''$ and~$g''$ are jointly
  surjective (by the universal property of pushout), we have that $f'' \circ f,
  g'' \circ f', g'' \circ g'$ are jointly surjective. Moreover,
  \[
    g'' \circ f' = f'' \circ g = f'' \circ f \circ h
  \]
  where $h$ is a factorization of~$g$ through~$f$ (that exists, since $f$ is
  supposed surjective). Thus, we conclude that $f'' \circ f, g'' \circ g'$ are
  jointly surjective.
  
  Conversely, suppose that $f'' \circ f$ and $g'' \circ g'$ are jointly
  surjective and let $y \in S'$. We have to show that $y$ is in the image
  of~$f'$ or~$g'$. Recall that
  \[
    S'' \cong (S \sqcup S')/\sim
  \]
  where $\sim$ is the equivalence relation induced by $g(x) \sim f'(x)$ for $x
  \in \N_{n_2}$: either $y$ is in the image of~$f'$, or we have both that $y$ is the only preimage of
  $g''(y)$ by $g''$ and $g''(y)$ is not in the image of~$f''$. In the former case,
  we conclude directly, and in the latter, since $f'' \circ f$ and $g'' \circ g'$
  are jointly surjective, there is $x \in \N_{n_3}$ such that $g'' \circ
  g' (x) = g''(y)$, so that $g'(x) = y$, which is what we wanted. Thus, $f'$ and~$g'$ are jointly surjective, \ie $\phi'$ is connected.
\end{proof}

\noindent We can now prove our connectedness result for pseudoadjunctions:
\begin{prop}
  \label{prop:adj-connex}
  For every $\phi \in \freecat{\PAdj}_2$, $\phi$ is connected.
\end{prop}
\begin{proof}
  Assume by contradiction that it is not true and let $N \in \N$ be the smallest
  natural number such that the set $S = \set{\phi \in \freecat{\PAdj}_2 \mid
    \len\phi = N \text{ and } \phi \text{ is not connected}}$ is not empty.
  Given $\phi \in S$, let
  \[
    (f_1 \comp_0 \alpha_1 \comp_0 h_1) \comp_1 \cdots \comp_1 (f_N
    \comp_0 \alpha_N \comp_0 h_N)
  \]
  be a decomposition of~$\phi$.

  Note that there is at least one $i \in \set{1,\ldots,N}$ such that $\alpha_i =
  \varepsilon$. Indeed, given $f,h \in \freecat{\PAdj}_1$ such that $f,\eta,h$ are
  $0$\composable, a representative $\tikzcdin{\N_{m} \ar[r,"u"] \& T \& \N_n
    \ar[l,"v"']}$ of $\coninterp_\Q(f \comp_0 \eta \comp_0 h)$ has the property
  that $v$ is an epimorphism. Since epimorphisms are stable by pushouts, given
  $\phi' \in \freecat{\PAdj}_2$ such that $\phi' = (f'_1 \comp_0 \eta \comp_0
  h'_1) \comp_1 \cdots \comp_1 (f'_k \comp_0 \eta \comp_0 h'_k)$ with $f'_i,h'_i
  \in \freecat{\PAdj}_1$ for $i \in \set{1,\ldots,k}$, a representative
  $\tikzcdin{\N_{m'} \ar[r,"u'"] \& T' \& \N_{n'} \ar[l,"v'"']}$ of
  $\coninterp_{\PAdj}(\phi')$ has the property that $v'$ is an epimorphism (by
  induction on~$k$), and in particular, $\phi'$ is connected. Consider the minimal index $i_0$
  such that there is $\phi \in S$ with $\alpha_{i_0} = \varepsilon$.

  Suppose first that $i_0 = 1$. Then, given a representative $\tikzcdin{\N_{m_1}
    \ar[r,"u_1"] \& T_1 \& \N_{m_2} \ar[l,"v_1"']}$ of $\coninterp_{\PAdj}(f_1
  \comp_0 \alpha_1 \comp_0 h_1)$, we easily check that $u_1$ is an epimorphism.
  By \Lemr{epi-cospan-connectivity}, we deduce that
  \[
    (f_2 \comp_0 \alpha_2
    \comp_0 h_2) \comp_1 \cdots \comp_1 (f_k \comp_0 \alpha_k \comp_0 h_k)
  \]
    is not
  connected, contradicting the minimality of~$N$.

  Suppose $i_0 > 1$. By the definition of $i_0$, we have $\alpha_{i_0 - 1} =
  \eta$. There are different cases depending on~$\len{f_{i_0-1}}$ (see~\Cref{fig:adj-conn}):
  \begin{figure}
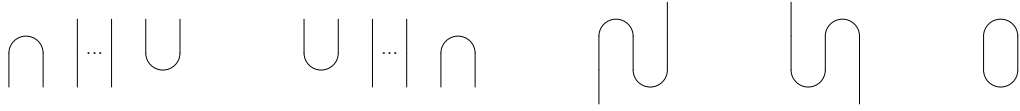

    \centering
    \[
      \satex{adj-conn-fig-1}
      \qquad
      \qquad
      \satex{adj-conn-fig-2}
      \qquad
      \qquad
      \satex{adj-conn-fig-3}
      \qquad
      \qquad
      \satex{adj-conn-fig-4}
      \qquad
      \qquad
      \satex{adj-conn-fig-5}
    \]
    \caption{The different cases}
    \label{fig:adj-conn}
  \end{figure}
  \begin{itemize}
  \item if $\len{f_{i_0-1}} \le \len{f_{i_0}} - 2$, then, since $\ctgt_1(f_{i_0
      - 1} \comp_0 \alpha_{i_0 - 1} \comp_0 h_{i_0 - 1}) = \csrc_1(f_{i_0}
    \comp_0 \alpha_{i_0} \comp_0 h_{i_0})$, we have
    \[
      f_{i_0} = f_{i_0 - 1}
      \comp_0 \ctgt_1(\eta) \comp_0 g \qtand h_{i_0 - 1} = g \comp_0 \csrc_1(\varepsilon)
      \comp_0 h_{i_0}
    \]
    for some $g \in \freecat{\PAdj}_1$.
    By \Lemr{connectivity-int}, we have
    \[
      \coninterp_{\PAdj}( (\eta \comp_0 g \comp_0 \csrc_1(\varepsilon)) \comp_1
      (\ctgt_1(\eta) \comp_0 g \comp_0 \varepsilon)) = 
      \coninterp((\csrc_1(\eta) \comp_0 g \comp_0 \varepsilon) \comp_1
      (\eta \comp_0 g \comp_0 \ctgt_1(\varepsilon)))
    \]
    thus, by functoriality of $\coninterp_{\PAdj}$, the morphism $\phi'$ defined by
    \begin{multline*}
      \phi' = (f_1 \comp_0 \alpha_1 \comp_0 h_1) \comp_1 \cdots \comp_1 (f_{i_0 - 2}
      \comp_0 \alpha_{i_0 - 2} \comp_0 h_{i_0 - 2}) \\
      \comp_1 (f_{i_0 - 1} \comp_0 g \comp_0 \varepsilon \comp_0 h_{i_0}) \comp_1
      (f_{i_0 - 1} \comp_0 \eta \comp_0 g \comp_0 h_{i_0}) \\
      \comp_1 (f_{i_0 + 1} \comp_0 \alpha_{i_0 + 1} \comp_0 h_{i_0+1}) \comp_1
      \cdots \comp_1 (f_k \comp_0 \alpha_k \comp_0 h_k)
    \end{multline*}
    satisfies that $\coninterp_{\PAdj}(\phi) = \coninterp_{\PAdj}(\phi')$. So
    $\phi'$ is not connected, and the $(i_0 {-} 1)$-th $2$\generator in the
    decomposition of $\phi'$ is $\varepsilon$, contradicting the minimality of~$i_0$;
    
  \item if $\len{f_{i_0-1}} \ge \len{f_{i_0}} + 2$, then the case is similar to the
    previous one;
    
  \item if $\len{f_{i_0 - 1}} = \len{f_{i_0}} - 1$, then, since
    $\coninterp_{\PAdj}((\eta \comp_0 \appfont f) \comp_1 (\appfont f \comp_0
    \varepsilon)) = \coninterp_{\PAdj}(\unit{\appfont f})$ by
    \Lemr{adj-connectivity-3gen}, the $2$-cell $\phi'$ defined by
    \begin{multline*}
      \phi' = (f_1 \comp_0 \alpha_1 \comp_0 h_1) \comp_1 \cdots \comp_1 (f_{i_0 - 2}
      \comp_0 \alpha_{i_0 - 2} \comp_0 h_{i_0 - 2}) \\
      \comp_1 (f_{i_0 + 1} \comp_0 \alpha_{i_0 + 1} \comp_0 h_{i_0+1}) \comp_1
      \cdots \comp_1 (f_k \comp_0 \alpha_k \comp_0 h_k)
    \end{multline*}
    satisfies $\coninterp_{\PAdj}(\phi) = \coninterp_{\PAdj}(\phi')$ (by
    functoriality of~$\coninterp_{\PAdj}$), so that $\phi'$ is not connected, contradicting the minimality of~$N$;
  
  \item if $\len{f_{i_0-1}} = \len{f_{i_0}} + 1$, then the situation is similar
    to the previous one, since, by \Lemr{adj-connectivity-3gen},
    \[
      \coninterp_{\PAdj}( (\appfont g \comp_0 \eta)
      \comp_1 (\varepsilon \comp_0 \appfont g)) = \coninterp_{\PAdj}(\unit{\appfont g});
    \]
    
  \item finally, the case $\len{f_{i_0-1}} = \len{f_{i_0}}$ is impossible since
    \[
      f_{i_0 - 1} \comp_0 \ctgt_1(\alpha_{i_0 - 1}) \comp_0 h_{i_0 - 1} =
      f_{i_0} \comp_0 \csrc_1(\alpha_{i_0}) \comp_0 h_{i_0}
    \]
    and
    \[
      \ctgt_1(\alpha_{i_0 - 1}) = \appfont f \comp_0 \appfont g \neq \appfont g
      \comp_0 \appfont f = \csrc_1(\alpha_{i_0}).\qedhere
    \]

  \end{itemize}
  % Moreover,
  % given a representative $\tikzcdin{\N_{m_2} \ar[r,"u_2"] \& T_2 \& \N_{m_3}
  %   \ar[l,"v_2"']}$ of $\coninterp_{\PAdj}((f_2 \comp_0 \alpha_2 \comp_0 h_2)
  % \comp_1 \cdots \comp_1 (f_k \comp_0 \alpha_k \comp_0 h_k))$
\end{proof}

\noindent We are now able to prove termination:
\begin{prop}
  \label{prop:adj-terminating}
  The rewriting system $\PAdj$ is terminating.
\end{prop}
\begin{proof}
  Suppose by contradiction that there is an infinite sequence $S_i\co \phi_i \TO
  \phi_{i+1}$ for $i \ge 0$ with $S_i$ a rewriting step in~$\freecat{\PAdj}_3$.
  Since
  \[
    \len{\csrc_2(\adjN)} = \len{\csrc_2(\adjNinv)} = 2 \qtand
    \len{\ctgt_2(\adjN)} = \len{\ctgt_2(\adjNinv)} = 0,
  \]
  if the inner $3$\generator of~$S_i$ is~$\adjN$ or~$\adjNinv$, for some $i \ge
  0$, then $\len{\phi_{i+1}} = \len{\phi_i} - 2$. Since
  \[
    \csrc_2(X_{\alpha,f,\beta}) = \ctgt_2(X_{\alpha,f,\beta}) = 2
  \]
  for $0$\composable $\alpha \in \PAdj_2$, $f \in \freecat{\PAdj}_1$, $\beta
  \in \PAdj_2$, it means that there is $i_0 \ge 0$ such that for $i \ge i_0$,
  the inner generator of~$S_i$ is an interchanger. By
  \cite[Thm.~16]{delpeuch2018normalization}, there is no infinite sequence of
  rewriting steps made of interchangers. Thus, by \Propr{adj-connex}, there is
  no infinite sequence of rewriting steps whose inner $3$\generator is an
  interchanger of~$\PAdj$, contradicting the existence of $(S_i)_{i \ge 0}$.
  Thus, $\PAdj$ is terminating.
\end{proof}

\noindent Finally, we can apply our coherence criterion and show that:
\begin{theo}
  \label{thm:adj-coherent}
  $\PAdj$ is a coherent Gray presentation.
\end{theo}
\begin{proof}
  By \Propr{adj-terminating}, $\restrictcat\PAdj 3$ is terminating. Since $R_1,R_2 \in
  \PAdj_4$, by \Thmr{squier}, the conclusion follows.
\end{proof}

\subsection{Self-dualities}
\label{ssec:app-untyped-adjunction}
\input{app-untyped}

\subsection{Frobenius pseudomonoids}
\label{ssec:app-frobenius-monoid}

We now consider the example of Frobenius
pseudomonoids~\cite{street2004frobenius}, which categorifies the classical notion
of Frobenius monoids. Sadly, it is only a partial example since we were not able
to handle the units of the structure (if we add them, the critical branchings are not
confluent) and to show that our presentation is terminating, even though we
believe that the latter is true. We nevertheless give the computation
of critical branchings for this example, hoping that a termination argument will
be found later.

We define the $3$\prepolygraph~$\P$ for (non-unitary) Frobenius pseudomonoids as
follows. We put
\[
  \P_0 = \set{\ast} \qtand \P_1 = \set{\bar 1} \qtand \P_2 = \set{\mu\co \bar 2
    \to \bar 1,
    \delta\co \bar 1 \to \bar 2}
\]
where we denote $\bar n$ by $\underbrace{\bar 1 \comp_0 \cdots \comp_0 \bar
  1}_n$ for $n \in \N$. We picture $\mu$ and $\delta$ by $\satex{mu}$ and
$\satex{delta}$ respectively, and we define $\P_3$ by $\P_3 =
\set{\frobN,\frobNinv,\frobA,\frobAco,\frobM,\frobMco}$ where
\begin{align*}
  \satex{frob-l}&\xTO{\mathsf{N}}\satex{frob-c}
  &
  \satex{frob-assoc-l}&\xTO{\frobA}\satex{frob-assoc-r}
  &
  \satex{frob-bmu-l}&\xTO{\frobM}\satex{frob-bmu-r}
  \\
  \satex{frob-r}&\xTO{\frobNinv}\satex{frob-c}
  &
  \satex{frob-coassoc-l}&\xTO{\frobAco}\satex{frob-coassoc-r}
  &
  \satex{frob-bdelta-l}&\xTO{\frobMco}\satex{frob-bdelta-r}
\end{align*}
As before, we then extend $\P$ to a Gray presentation by adding $3$\generators
corresponding to interchange generators and $4$\generators corresponding to
independence generators and interchange naturality generators.

Using the constructive proof of~\Thmr{finite-cp}, we find 19 critical
branchings, and we use them to define a set of nineteen $4$-generators
$R_1,\ldots,R_{19}$ that we add to~$\P_4$. These critical branchings are shown
in~\Cref{fig:pseudofrobenius-cps}.
\begin{figure}[p]\ContinuedFloat*
  \centering
  \begingroup
  \def\mycolsep{5em}
  \openup15pt
  \begin{gather*}
    \mathmakebox[\linewidth][c]{\hfil\begin{tikzcd}[row sep={5.8em,between origins},column sep={\mycolsep,between origins},ampersand replacement=\&,baseline=(\tikzcdmatrixname-2-1.center)]
        \satexnoscale[scale=0.8]{ass-zig1} \ar[d,equal]\tar[r]\ar[rd,phantom,"\DDownarrow R_1"description] \& \satexnoscale[scale=0.8]{ass-zig2} \tar[d]
        \\
        \satexnoscale[scale=0.8]{ass-zig1} \tar[r] \&  \satexnoscale[scale=0.8]{ass-zig3}
      \end{tikzcd}
      \hfil
      \begin{tikzcd}[row sep={5.8em,between origins},column sep={\mycolsep,between origins},ampersand replacement=\&,baseline=(\tikzcdmatrixname-2-1.center)]
        \satexnoscale[scale=0.8]{coass-zig1}
        \ar[rd,phantom,"\DDownarrow R_{2}",start anchor=center,end anchor=center]
        \tar[r]
        \ar[d,equal]
        \& \satexnoscale[scale=0.8]{coass-zig2} \tar[d]
        \\
        \satexnoscale[scale=0.8]{coass-zig1} \tar[r] \& \satexnoscale[scale=0.8]{coass-zig3}
      \end{tikzcd}
      \hfil
      \begin{tikzcd}[row sep={5.8em,between origins},column sep={\mycolsep,between origins},ampersand replacement=\&,baseline=(\tikzcdmatrixname-2-1.center)]
        \satexnoscale[scale=0.8]{cotrans-ass1} \ar[d,equal] \tar[r] \ar[rd,start anchor=center,end
        anchor=center,phantom,"\DDownarrow R_3"]
        \& \satexnoscale[scale=0.8]{cotrans-ass2} \tar[d]
        \\
        \satexnoscale[scale=0.8]{cotrans-ass1}  \tar[r] \& \satexnoscale[scale=0.8]{cotrans-ass3}
      \end{tikzcd}
      \hfil
      \begin{tikzcd}[row sep={5.8em,between origins},column sep={\mycolsep,between origins},ampersand replacement=\&,baseline=(\tikzcdmatrixname-2-1.center)]
        \satexnoscale[scale=0.8]{coass-trans1} \tar[r] \ar[d,equal]\ar[rd,"\DDownarrow R_4",start
        anchor=center,end anchor=center,phantom] \& \satexnoscale[scale=0.8]{coass-trans2} \tar[d] \\
        \satexnoscale[scale=0.8]{coass-trans1} \tar[r] \& \satexnoscale[scale=0.8]{coass-trans3}
      \end{tikzcd}\hfil}
    \\
    \mathmakebox[\linewidth][c]{\hfil\begin{tikzcd}[row sep={4.9em,between origins},column sep={\mycolsep,between origins},ampersand replacement=\&,baseline=(\tikzcdmatrixname-2-1.center)]
        \satexnoscale[scale=0.8]{mon-hex1}\tar[d]\tar[r]\&\satexnoscale[scale=0.8]{mon-hex2}\tar[r]\ar[d,phantom,"\DDownarrow
        R_5"description]\&\satexnoscale[scale=0.8]{mon-hex3}\tar[d]\\
        \satexnoscale[scale=0.8]{mon-hex6}\tar[r]\&\satexnoscale[scale=0.8]{mon-hex5}\tar[r]\&\satexnoscale[scale=0.8]{mon-hex4}
      \end{tikzcd}
      \hfil
      \begin{tikzcd}[row sep={4.9em,between origins},column sep={\mycolsep,between origins},ampersand replacement=\&,baseline=(\tikzcdmatrixname-2-1.center)]
        \satexnoscale[scale=0.8]{comon-hex1} \tar[d] \tar[r] \& \satexnoscale[scale=0.8]{comon-hex2} \tar[r]
        \ar[d,phantom,"\DDownarrow R_6"description]\&
        \satexnoscale[scale=0.8]{comon-hex3} \tar[d] \\
        \satexnoscale[scale=0.8]{comon-hex4} \tar[r] \& \satexnoscale[scale=0.8]{comon-hex5} \tar[r] \&
        \satexnoscale[scale=0.8]{comon-hex6}
      \end{tikzcd}\hfil}
    \\
    \mathmakebox[\linewidth][c]{\hfil\begin{tikzcd}[row sep={4.9em,between origins},column sep={\mycolsep,between origins},ampersand replacement=\&,baseline=(\tikzcdmatrixname-2-1.center)]
        \satexnoscale[scale=0.8]{zag-ass1} \tar[r] \tar[d] \& \satexnoscale[scale=0.8]{zag-ass2}
        \tar[r]\ar[d,phantom,"\DDownarrow R_7"description] \&
        \satexnoscale[scale=0.8]{zag-ass3} \tar[d] \\
        \satexnoscale[scale=0.8]{zag-ass4} \tar[r] \& \satexnoscale[scale=0.8]{zag-ass5} \tar[r] \&
        \satexnoscale[scale=0.8]{zag-ass6}
      \end{tikzcd}
      \hfil
      \begin{tikzcd}[row sep={4.9em,between origins},column sep={\mycolsep,between origins},ampersand replacement=\&,baseline=(\tikzcdmatrixname-2-1.center)]
        \satexnoscale[scale=0.8]{coass-zag1} \tar[r] \tar[d] \& \satexnoscale[scale=0.8]{coass-zag4} \tar[r]
        \ar[d,phantom,"\DDownarrow R_8"description] \&
        \satexnoscale[scale=0.8]{coass-zag5} \tar[d] \\
        \satexnoscale[scale=0.8]{coass-zag2} \tar[r] \& \satexnoscale[scale=0.8]{coass-zag3} \tar[r] \&
        \satexnoscale[scale=0.8]{coass-zag6}
      \end{tikzcd}\hfil}
    \\
    \mathmakebox[\linewidth][c]{\hfil\begin{tikzcd}[ampersand replacement=\&,row sep={4.9em,between origins},column sep={\mycolsep,between origins},baseline=(\tikzcdmatrixname-2-1.center)]
        \satexnoscale[scale=0.8]{zig-echmu1} \tar[d] \tar[r] \& \satexnoscale[scale=0.8]{zig-echmu2} \tar[r]
        \ar[d,phantom,"\DDownarrow R_{9}"description]\&
        \satexnoscale[scale=0.8]{zig-echmu3} \tar[d] \\
        \satexnoscale[scale=0.8]{zig-echmu4} \tar[r] \& \satexnoscale[scale=0.8]{zig-echmu5} \tar[r] \&
        \satexnoscale[scale=0.8]{zig-echmu6}
      \end{tikzcd}
      \hfil
      \begin{tikzcd}[row sep={4.9em,between origins},column sep={\mycolsep,between origins},ampersand replacement=\&,baseline=(\tikzcdmatrixname-2-1.center)]
        \satexnoscale[scale=0.8]{echde-zig1} \tar[d] \tar[r] \& \satexnoscale[scale=0.8]{echde-zig2} \tar[r]
        \ar[d,phantom,"\DDownarrow R_{10}"description]\&
        \satexnoscale[scale=0.8]{echde-zig3} \tar[d] \\
        \satexnoscale[scale=0.8]{echde-zig4} \tar[r] \& \satexnoscale[scale=0.8]{echde-zig5} \tar[r] \&
        \satexnoscale[scale=0.8]{echde-zig6}
      \end{tikzcd}\hfil}
  \end{gather*}
  \endgroup
  \caption{The critical branchings for Frobenius pseudomonoids}
  \label{fig:pseudofrobenius-cps}
\end{figure}%
\begin{figure}[p!]\ContinuedFloat
  \centering
  \begingroup
  \def\mycolsep{5em}
  \openup15pt
  \begin{gather*}
    \mathmakebox[\linewidth][c]{\hfil\begin{tikzcd}[row sep={6.2em,between origins},column sep={\mycolsep,between origins},ampersand replacement=\&,baseline=(\tikzcdmatrixname-2-1.center)]
        \satexnoscale[scale=0.8]{echmu-zig4} \tar[r] \& \satexnoscale[scale=0.8]{echmu-zig5} \tar[r]
        \ar[d,phantom,"\DDownarrow R_{11}"description] \&
        \satexnoscale[scale=0.8]{echmu-zig6} \tar[d] \\
        \satexnoscale[scale=0.8]{echmu-zig1} \tar[u] \tar[r] \& \satexnoscale[scale=0.8]{echmu-zig2} \tar[r] \&
        \satexnoscale[scale=0.8]{echmu-zig3}
      \end{tikzcd}
      \hfil
      \begin{tikzcd}[ampersand replacement=\&,row sep={6.2em,between origins},column sep={\mycolsep,between origins},baseline=(\tikzcdmatrixname-2-1.center)]
        \satexnoscale[scale=0.8]{zig-echde4} \tar[r] \& \satexnoscale[scale=0.8]{zig-echde5}
        \tar[r]\ar[d,phantom,"\DDownarrow R_{12}"description] \&
        \satexnoscale[scale=0.8]{zig-echde6} \tar[d] \\
        \satexnoscale[scale=0.8]{zig-echde1} \tar[u] \tar[r] \& \satexnoscale[scale=0.8]{zig-echde2} \tar[r] \&
        \satexnoscale[scale=0.8]{zig-echde3}
      \end{tikzcd}\hfil}
    \\
    \begin{tikzcd}[ampersand replacement=\&,row sep={6.2em,between origins},column sep={\mycolsep,between origins},baseline=(\tikzcdmatrixname-2-1.center)]
      \satexnoscale[scale=0.8]{cotrans-trans1} \tar[r] \tar[d] \& \satexnoscale[scale=0.8]{cotrans-trans2}
      \ar[d,phantom,"\DDownarrow R_{13}"]\tar[r,""{name=arrb}] \&
      \satexnoscale[scale=0.8]{cotrans-trans3} \tar[d] \\
      \satexnoscale[scale=0.8]{cotrans-trans5} \tar[r] \& \satexnoscale[scale=0.8]{cotrans-trans6} \tar[r,""{name=arre}] \&
      \satexnoscale[scale=0.8]{cotrans-trans4}
    \end{tikzcd}
    \\
    \begin{tikzcd}[row sep={6.2em,between origins},column sep={\mycolsep,between origins},ampersand replacement=\&,baseline=(\tikzcdmatrixname-2-1.center)]
      \satexnoscale[scale=0.8]{trans-ass1} \tar[r] \tar[d] \& \satexnoscale[scale=0.8]{trans-ass2} \tar[r,""{name=arrb,auto=false}] \&
      \satexnoscale[scale=0.8]{trans-ass3} \tar[r] \& \satexnoscale[scale=0.8]{trans-ass4} \tar[d] \\
      \satexnoscale[scale=0.8]{trans-ass5} \tar[r] \& \satexnoscale[scale=0.8]{trans-ass6} \tar[r,""{name=arre,auto=false}] \&
      \satexnoscale[scale=0.8]{trans-ass7} \tar[r]
      \& \satexnoscale[scale=0.8]{trans-ass8} \ar[from=arrb,to=arre,"\DDownarrow R_{14}",phantom]
    \end{tikzcd}
    \qquad
    \begin{tikzcd}[row sep={6.2em,between origins},column sep={\mycolsep,between origins},ampersand replacement=\&,baseline=(\tikzcdmatrixname-2-1.center)]
      \satexnoscale[scale=0.8]{coass-cotrans1} \tar[r] \tar[d] \& \satexnoscale[scale=0.8]{coass-cotrans2} \tar[r,""{name=arrb,auto=false}] \&
      \satexnoscale[scale=0.8]{coass-cotrans3} \tar[r] \& \satexnoscale[scale=0.8]{coass-cotrans4} \tar[d] \\
      \satexnoscale[scale=0.8]{coass-cotrans5} \tar[r] \& \satexnoscale[scale=0.8]{coass-cotrans6} \tar[r,""{name=arre,auto=false}] \&
      \satexnoscale[scale=0.8]{coass-cotrans7} \tar[r]
      \&
      \satexnoscale[scale=0.8]{coass-cotrans8}\ar[from=arrb,to=arre,phantom,"\DDownarrow R_{15}"description]
    \end{tikzcd}
  \end{gather*}
  \endgroup
  \caption{The critical branchings for Frobenius pseudomonoids}
  \label{fig:pseudofrobenius-cpsbis}
\end{figure}%
\begin{figure}[p!]\ContinuedFloat
  \centering
  \begingroup
  \def\mycolsep{5em}
  \openup15pt
  \begin{gather*}
    \begin{tikzcd}[ampersand replacement=\&,row sep={6.2em,between origins},column sep={\mycolsep,between origins},baseline=(\tikzcdmatrixname-2-1.center)]
      \satexnoscale[scale=0.8]{echde-trans1} \tar[r] \tar[d] \& \satexnoscale[scale=0.8]{echde-trans2} \tar[r,""{name=arrb,auto=false}] \&
      \satexnoscale[scale=0.8]{echde-trans3} \tar[r] \& \satexnoscale[scale=0.8]{echde-trans4} \tar[d] \\
      \satexnoscale[scale=0.8]{echde-trans5} \tar[r] \& \satexnoscale[scale=0.8]{echde-trans6} \tar[r,""{name=arre,auto=false}] \&
      \satexnoscale[scale=0.8]{echde-trans7} \tar[r]
      \& \satexnoscale[scale=0.8]{echde-trans8}\ar[from=arrb,to=arre,phantom,"\DDownarrow R_{16}"]
    \end{tikzcd}
    \qquad
    \begin{tikzcd}[ampersand replacement=\&,row sep={6.2em,between origins},column sep={\mycolsep,between origins},baseline=(\tikzcdmatrixname-2-1.center)]
      \satexnoscale[scale=0.8]{cotrans-echmu1} \tar[r] \tar[d] \& \satexnoscale[scale=0.8]{cotrans-echmu2} \tar[r,""{name=arrb,auto=false}] \&
      \satexnoscale[scale=0.8]{cotrans-echmu3} \tar[r] \& \satexnoscale[scale=0.8]{cotrans-echmu4} \tar[d] \\
      \satexnoscale[scale=0.8]{cotrans-echmu5} \tar[r] \& \satexnoscale[scale=0.8]{cotrans-echmu6} \tar[r,""{name=arre,auto=false}] \&
      \satexnoscale[scale=0.8]{cotrans-echmu7} \tar[r]
      \& \satexnoscale[scale=0.8]{cotrans-echmu8}
      \ar[from=arrb,to=arre,phantom,"\DDownarrow R_{17}"]
    \end{tikzcd}
    \\
    \begin{tikzcd}[ampersand replacement=\&,row sep={6.2em,between origins},column sep={\mycolsep,between origins},baseline=(\tikzcdmatrixname-2-1.center)]
      \satexnoscale[scale=0.8]{echmu-trans4} \tar[r]
      \ar[rrrd,phantom,"\DDownarrow R_{18}",start anchor=center,end anchor=center]
      \& \satexnoscale[scale=0.8]{echmu-trans5} \tar[r]
      \&
      \satexnoscale[scale=0.8]{echmu-trans6} \tar[r] \& \satexnoscale[scale=0.8]{echmu-trans7} \tar[d] \\
      \satexnoscale[scale=0.8]{echmu-trans1} \tar[r] \tar[u] \& \satexnoscale[scale=0.8]{echmu-trans2} \tar[r] \&
      \satexnoscale[scale=0.8]{echmu-trans3} \ar[r,equal] \& \satexnoscale[scale=0.8]{echmu-trans3}
    \end{tikzcd}
    \qquad
    \begin{tikzcd}[ampersand replacement=\&,row sep={6.2em,between origins},column sep={\mycolsep,between origins},baseline=(\tikzcdmatrixname-2-1.center)]
      \satexnoscale[scale=0.8]{cotrans-echde4} \tar[r]
      \ar[rrrd,phantom,"\DDownarrow R_{19}",start anchor=center,end anchor=center]
      \& \satexnoscale[scale=0.8]{cotrans-echde5} \tar[r]
      \&
      \satexnoscale[scale=0.8]{cotrans-echde6} \tar[r] \& \satexnoscale[scale=0.8]{cotrans-echde7} \tar[d] \\
      \satexnoscale[scale=0.8]{cotrans-echde1} \tar[r] \tar[u] \& \satexnoscale[scale=0.8]{cotrans-echde2} \tar[r] \&
      \satexnoscale[scale=0.8]{cotrans-echde3}
      \ar[r,equal]
      \&
      \satexnoscale[scale=0.8]{cotrans-echde3}
    \end{tikzcd}
    % \pbox{\raisebox{2em}.}
  \end{gather*}
  \endgroup
  \caption{The critical branchings for Frobenius pseudomonoids}
  \label{fig:pseudofrobenius-cpsbisbis}
\end{figure}
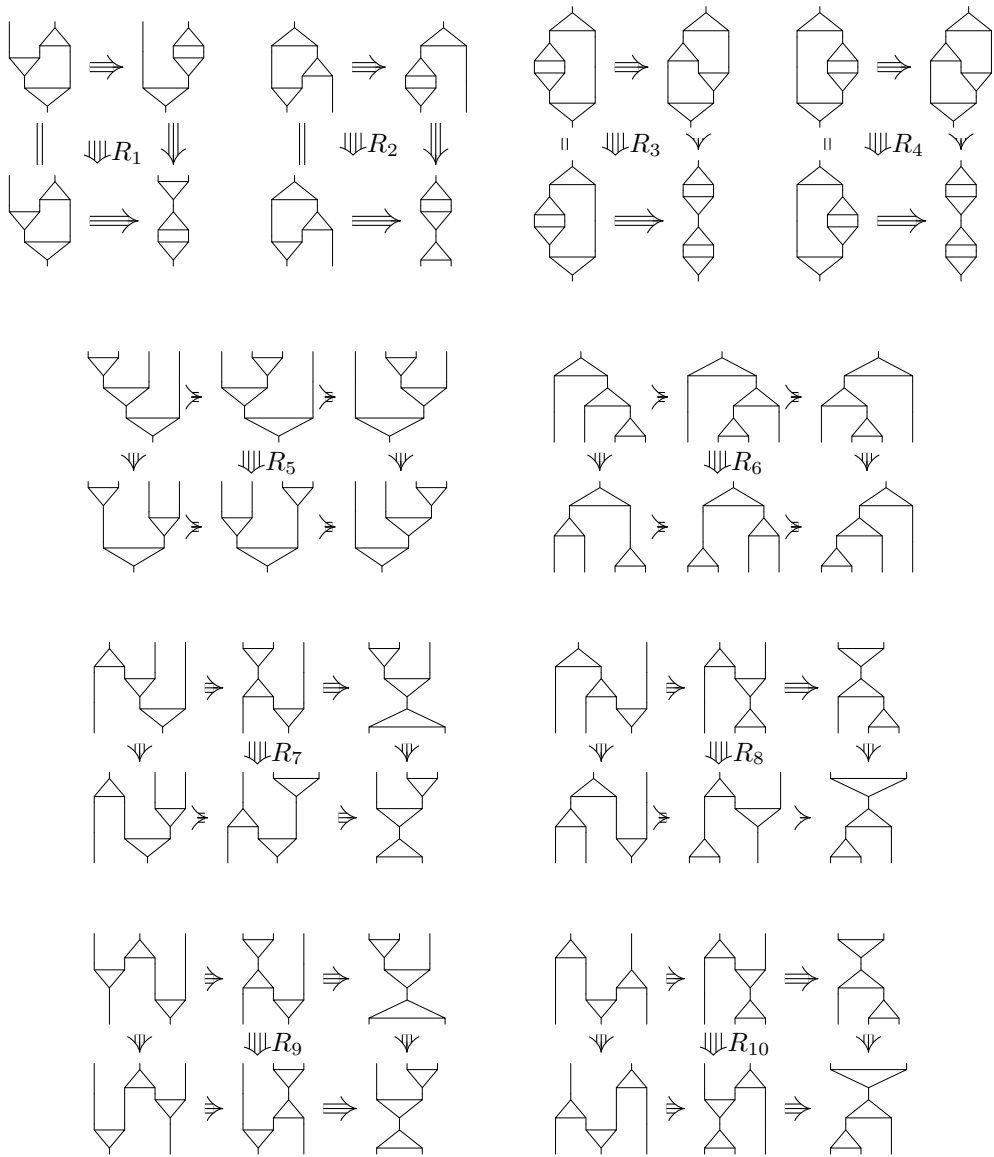%
We then define $\PFrob$ as the Gray presentation obtained from~$\P$ by adding
the $4$\generators $R_1,\ldots,R_{19}$ from above. Since we were not able to
show termination, we conjecture it:
\begin{conj}
 $\PFrob$ is terminating. 
\end{conj}
\noindent From this assumption, we deduce that:
\begin{theo}
  \label{thm:frob-coherent}
  If $\PFrob$ is terminating, then $\PFrob$ is a coherent Gray presentation.
\end{theo}

\begin{proof}
  This is a consequence of \Thmr{squier}.
\end{proof}

%%% Local Variables:
%%% mode: latex
%%% TeX-master: "proxy-applications"
%%% End:

%% file: app-untyped.tex
We consider a variant of the preceding example, by considering the theory corresponding to pseudoadjunctions between an endofunctor and itself.
This new example requires a special treatment since the underlying rewriting
system is not terminating, and, more fundamentally, the induced $(3,2)$-Gray
category is not expected to be fully coherent. We show instead a partial
coherence result.

We define the $3$\prepolygraph
for self-dualities as the $3$\prepolygraph~$\P$ such that
\[
  \P_0 = \set{\ast} \qtand \P_1 = \set{\bar 1\co \ast \to \ast} \qtand \P_2 =
  \set{\eta\co \unit{\ast} \To \bar 2, \varepsilon\co
    \bar 2 \To \unit{\ast}}
\]
where we write $\bar n$ for $\underbrace{\bar 1 \comp_0 \cdots \comp_0 \bar
  1}_n$ for $n \in \N$. The $2$\generators $\eta$ and $\varepsilon$ are pictured as
$\satex{cap}$ and $\satex{cup}$ respectively, and $\P_3$ is defined by $\P_3 =
\set{\adjN,\adjNinv}$ where
\[
  \adjN\co (\eta \comp_0 \bar 1) \comp_1 (\bar 1 \comp_0 \varepsilon) \TO \unit
  {\bar 1}
  \qtand
  \adjNinv\co (\bar 1 \comp_0 \eta) \comp_1 (\varepsilon \comp_0 \bar 1) \TO
  \unit {\bar 1}
\]
which is pictured again by
\[
  \begin{tikzcd}
    \satex{adj2-l}\tar[r,"\adjN"]&\satex{adj2-r}        
  \end{tikzcd}
  \qtand
  \begin{tikzcd}
    \satex{adj1-l}\tar[r,"\adjNinv"]&\satex{adj1-r}
  \end{tikzcd}
\]
As before, we then extend $\P$ to a Gray presentation by adding $3$\generators
corresponding to interchange generators and $4$\generators corresponding to
independence generators and interchange naturality generators. We also add the
same $4$\generators that we added for pseudoadjunctions
\[
  \begin{tikzcd}[column sep=3ex,baseline=(\tikzcdmatrixname-1-1.center)]
    \satex{adj-cp1-l}\tar[dr]\tar[rr]&\ar[d,phantom,pos=0.3,"\overset{R_1}\LLeftarrow"]&\tar[dl]\satex{adj-cp1-c}\\
    &\satex{cup}
  \end{tikzcd}
  \qquad\quad
  \begin{tikzcd}[column sep=3ex,baseline=(\tikzcdmatrixname-1-1.center)]
    \satex{adj-cp2-l}\tar[dr]\tar[rr]&\ar[d,phantom,pos=0.3,"\overset{R_2}\LLeftarrow"description]&\tar[dl]\satex{adj-cp2-c}\\
    &\satex{cap}
  \end{tikzcd}
\]
to $\P$ and we denote $\PUAdj$ the resulting Gray presentation. Here, it is not
possible to apply \Thmr{squier} to obtain a coherence result, as in previous
section. Indeed, $\PUAdj$ is not terminating, since we have the reduction
\[
  \satex{capcup-cycle1}
  \qTO
  \satex{capcup-cycle2}
  \qTO
  \satex{capcup-cycle3}
  \qTO
  \satex{capcup-cycle4}
  \qTO
  \satex{capcup-cycle1}
\]
Moreover, this endomorphism $3$-cell is not expected to be an identity,
discarding hopes for the presentation to be coherent.
Following~\cite{dunn2016coherence}, we can still aim at showing a partial
coherence result by restricting to $2$-cells which are connected, in the sense
of the previous section. In this case, termination can actually be shown
by using the same arguments as for
pseudoadjunctions. %initialement écrit: \secr{adj}
However, the critical pairs are not joinable either since, for instance, we have
\[
  \satex{dual-nj-l}
  \quad\OT\quad
  \satex{dual-nj-c}
  \quad\TO\quad
  \satex{dual-nj-r}
\]
for which there is little hope that a Knuth-Bendix completion will provide a
reasonably small presentation. However, one can obtain a rewriting system,
introduced below, which
is terminating on connected $2$-cells and confluent by orienting the
interchangers. Using this rewriting system, we are able to show a partial
coherence result.

We define an alternate rewriting system $\Q$ where
\begin{center}
 $\Q_i = \P_i$ for $i \in
\set{0,1,2}$ \quad and\quad $\Q_3 = \set{\adjN,\adjNinv} \sqcup \intsub\Q_3$ 
\end{center}
 where
$\intsub\Q_3$ contains the following $3$-generators, called
\emph{$\Q$-interchange generators}:
\[
  \hss
  \begin{array}{rccc@{\hspace{3em}}rccc}
    X'_{\eta,\bar n,\eta}\co& \satex{dual-capcap-l} &\TO& \satex{dual-capcap-r}
    &
      X'_{\eta,\bar n,\varepsilon}\co& \satex{dual-capcup-l} &\TO& \spacebelowsatex{dual-capcup-r} \\
    X'_{\varepsilon,\bar n,\eta}\co& \satex{dual-cupcap-l} &\TO& \satex{dual-cupcap-r}
    &
      X'_{\varepsilon,\bar n,\varepsilon}\co& \satex{dual-cupcup-l} &\TO& \satex{dual-cupcup-r}
  \end{array}
  \hss
\]
for $n \in \N$.

There is a morphism of $3$\precategories $\Gamma\co \freecat\Q \to
\freeinvf{\prespcat\P}$ uniquely defined by $\Gamma(u) = u$ for $u \in
\freecat\Q_i$ with $i \in \set{0,1,2}$ and mapping the $3$-generators as follows:
\begin{align*}
  \adjN &\mapsto \adjN & \adjNinv & \mapsto \adjNinv\\
  X'_{\eta,\bar n,\eta} &\mapsto \finv X_{\eta,\bar n,\eta} &
  X'_{\eta,\bar n,\varepsilon} &\mapsto X_{\eta,\bar n,\varepsilon} \\
  X'_{\varepsilon,\bar n,\eta} &\mapsto \finv X_{\varepsilon,\bar n,\varepsilon} &
  X'_{\varepsilon,\bar n,\varepsilon} &\mapsto X_{\varepsilon,\bar n,\varepsilon}
\end{align*}
for $n \in \N$. We get a rewriting system $(\Q,\sequiv)$ by putting $F \sequiv
F'$ if and only if $\Gamma(F) = \Gamma(F')$ for parallel $F,F' \in
\freecat\Q_3$. By inspecting the $3$-generators of $\Q_3$, we can show that,
given $F\co \phi\TO \phi' \in \freecat\Q_3$, $\phi$ is connected if and only if
$\phi'$ is connected. Indeed, one easily checks that for every $A \in \Q_3$, we
have $\coninterp_\Q(\csrc_2(A)) = \coninterp_\Q(\ctgt_2(A))$, so that
$\coninterp_\Q(\phi) = \coninterp_\Q(\phi')$.

We first show a weak termination property for~$\Q$, stating that it is terminating on
connected $2$-cells:
\begin{prop}
  \label{prop:untyped-adj-terminating}
  Given a connected $2$\cell~$\phi$ in~$\freecat\Q_2$, there is no infinite
  sequence~$F_i\co \phi_i\TO \phi_{i+1}$ of rewriting steps
  where~$\phi_0 = \phi$.
\end{prop}
\begin{proof}
  Since any rewriting step whose inner $3$\generator is~$\adjN$ or~$\adjNinv$
  decreases by two the number of $2$\generators in a diagram, it is enough to
  show that there is no infinite sequence of composable rewriting steps made of
  elements of~$\intsub\Q_3$. For this purpose, we combine several counting
  functions: a function~$N_1$ which counts the potential number of
  rules~$X'_{\eta,-,\eps}$ and~$X'_{\eps,-,\eta}$ which can be applied, and
  functions~$N_2^\eta$ and~$N_2^\eps$ which counts the potential number of
  rules~$X'_{\eta,-,\eta}$ and~$X'_{\eps,-,\eps}$ which can be applied respectively. Given a
  $2$\cell
  \[
    \phi = (\bar m_1 \pcomp_0 \alpha_1
    \pcomp_0 \bar n_1) \pcomp_1 \cdots \pcomp_1 (\bar m_k \pcomp_0 \alpha_k \pcomp_0
    \bar n_k)
  \]
  of~$\freecat\Q_2$, with~$\alpha_i \in \Q_2$ and~$m_i,n_i \in \N$
  for~$i \in \set{1,\ldots,k}$, we define~$N_1(\phi) \in \N$ by
  \[
    N_1(\phi) =  \setsize{\set{ (i,j) \in \N^2 \mid 1\le i < j \le k \text{ and } \alpha_i =
      \eta \text{ and } \alpha_j = \varepsilon}}.
  \]
  Moreover, if we write~$p,q \in \set{0,\ldots,k}$ and~$i_1,\ldots,i_p,j_1,\ldots,j_q \in
  \N$ for the unique integers such that
  \[
    1\le i_1 < \cdots < i_p\le k \qquad 1\le j_1 < \cdots < j_q\le k \qquad \set{i_1,\ldots,i_p,j_1,\ldots,j_q} = \set{1,\ldots,k}
  \]
  and~$\alpha_{i_r} = \eta$ and~$\alpha_{j_s} = \varepsilon$ for~$r \in \set{1,\ldots,p}$
  and~$s \in \set{1,\ldots,q}$, we define~$N^\eta_2(\phi) \in \N^p$ and~$N^\varepsilon_2(\phi)
  \in \N^q$ by
  \[
    N^\eta_2(\phi) = (n_{i_p},\ldots,n_{i_1})
    \qquad\qtand\qquad
    N^\varepsilon_2(\phi)
    = (n_{j_1},\ldots,n_{j_q}).
  \]
  Finally, we define~$N(\phi) \in \N^{1+p+q}$ by
  \[
    N(\phi) = (N_1(\phi),N^\eta_2(\phi),N^\varepsilon_2(\phi))
  \]
  and we equip~$\N^p$,~$\N^q$ and~$\N^{1+p+q}$ with the lexicographical
  ordering~$\ltlex$. Now, keeping~$\phi$ as above, let
  \[
    \lambda \pcomp_1 (l \pcomp_0 A \pcomp_0 r) \pcomp_1
    \rho \co \phi \To \phi'\in \freecat\Q_3
  \]
  be a rewriting step for some~$l,r \in
  \freecat\Q_1$,~$\lambda,\rho,\phi' \in \freecat\Q_2$ and~$A \in \Q_3$ with
  \[
    \phi' =
    (\bar{m}'_1 \pcomp_0 \alpha'_1 \pcomp_0 \bar{n}'_1) \pcomp_1 \cdots \pcomp_1
    (\bar{m}'_k \pcomp_0 \alpha'_k \pcomp_0 \bar{n}'_k) 
  \]
  for some~$\alpha'_i \in \Q_2$
  and~$m'_i,n'_i \in \N$ for~$i \in \set{1,\ldots,k}$. We distinguish the three following cases.
  \begin{itemize}
  \item 
  If~$A = X'_{\eta,\bar
    u,\varepsilon}$ or~$A = X'_{\varepsilon,\bar u,\eta}$ for some~$u \in \N$, then~$N_1(\phi') = N_1(\phi) - 1$.
\item Otherwise, if~$A = X_{\eta,\bar u,\eta}$ for some~$u \in \N$, then we
  have~$N_1(\phi) = N_1(\phi')$ and, writing~$r$ for~$\len\lambda + 1$, we
  have~$n_{s} = n'_{s}$ for~$s \in \set{1,\ldots,k} \setminus\set{r,r+1}$.
  Moreover, we have~$n'_{r+1} \le n_{r+1} - 2$, so that~$N^\eta_2(\phi') \ltlex
  N^\eta_2(\phi)$. For example, consider the application of $X'_{\eta,\bar
    n,\eta}$ on $\csrc_2(X'_{\eta,\bar n,\eta})$ without additional whiskering.
  In this case, we have that $N^\eta_2(\csrc_2(X'_{\eta,\bar n,\eta})) =
  (n+2,0)$ while $N^\eta_2(\ctgt_2(X'_{\eta,\bar n,\eta})) = (0,n)$, so that the
  value of $N^\eta_2(-)$ is decreased with respect to $\ltlex$ by the
  application of $X'_{\eta,\bar n,\eta}$.

  \item
  Otherwise,~$A = X'_{\varepsilon,\bar u,\varepsilon}$ for some~$u \in \N$.
  Then~$N^\eta_2(\phi) = N^\eta_2(\phi')$ and, by a similar argument as
  before,~$N^\varepsilon_2(\phi') \ltlex N^\varepsilon_2(\phi)$. 
  \end{itemize}
  In any case, we get
  that~$N(\phi) \ltlex N(\phi')$. Since~$\ltlex$ is well-founded, we conclude
  that there is no infinite sequence of rewriting steps~$R_i\co \phi_i \TO
  \phi_{i+1}$ for~$i \in \N$ with~$\phi_0$ connected.
\end{proof}

We now aim at showing the confluence of the branchings of~$\Q$. The idea is to
use a critical pair lemma and a Newman's lemma adapted to the specific setting
of~$\Q$ where the notion of critical branching is different and where we only
consider connected $2$\cells as sources.
We say that a branching $(S_1,S_2)$ of $\Q$ is \emph{connected} when
$\csrc_2(S_1)$ is connected. We say that it is \emph{$\Q$\critical} when it is
local, minimal, not trivial and not independent. We first state adapted versions of the critical pair lemma and
Newman's lemma to the setting of~$\Q$:%
\begin{lem}
  \label{lem:untyped-adj-cp}
  If all connected $\Q$\critical branchings $(S_1,S_2)$ of $(\Q,\sequiv)$ are
  confluent, then all connected local branchings of $(\Q,\sequiv)$ are
  confluent.
\end{lem}
\begin{proof}
  By a direct adaptation of the proof of \Thmr{cp} to connected $2$-cells and
  rewriting steps between connected $2$-cells.
\end{proof}

\begin{lem}
  \label{lem:untyped-adj-newman}
  If all connected local branchings of $(\Q,\sequiv)$ are confluent, then
  all connected branchings of $(\Q,\sequiv)$ are confluent.
\end{lem}

\begin{proof}
  By a direct adaptation of \Thmr{newman-modulo} to connected $2$-cells and
  rewriting steps between connected $2$-cells, using
  \Propr{untyped-adj-terminating}.
\end{proof}
\noindent By the above properties, in order to deduce the confluence of the
branchings of~$\Q$, it is enough to check that the critical branchings of~$\Q$
are confluent, fact that we verify in the following property:%
\begin{lem}
  \label{lem:untyped-adj-cp-confluent}
  The connected $\Q$\critical branchings of $(\Q,\sequiv)$ are confluent.
\end{lem}
\begin{proof}
  We first consider the $\Q$\critical branchings~$(S_1,S_2)$ that are
  \emph{structural-structural}, \ie such that the inner $3$\generators of $S_1$
  and $S_2$ are $\Q$-interchange generators.
  We classify them as \emph{separated} and \emph{half-separated} and \emph{non-separated}. There are eight kinds of separated structural-structural
  $\Q$\critical branchings listed below:
  \begingroup
  \vskip\abovedisplayskip
  \tabskip=0pt plus 1fill\halign
  to\linewidth{\tabskip=0pt plus 1fill\relax\hfill #\hfill&\hfill
    $#$\hfill&\hfill $#$\hfill&\hfill $#$\hfill&\hfill $#$\hfill&\hfill $#$\hfill\tabskip=0pt plus 1fill\cr
    (1)&\satex[scale=.88]{untyped-adj/untypedadj-st-st-1-l-1}
    &\OT&\satex[scale=.88]{untyped-adj/untypedadj-st-st-1}
    &\TO&\spacebelowsatex[scale=.88]{untyped-adj/untypedadj-st-st-1-r-1} 
    \cr
    (2)&\satex[scale=.88]{untyped-adj/untypedadj-st-st-2-l-1}
    &\OT&\satex[scale=.88]{untyped-adj/untypedadj-st-st-2}
    &\TO&\spacebelowsatex[scale=.88]{untyped-adj/untypedadj-st-st-2-r-1} 
    \cr
    (3)&\satex[scale=.88]{untyped-adj/untypedadj-st-st-3-l-1}
    &\OT&\satex[scale=.88]{untyped-adj/untypedadj-st-st-3}
    &\TO&\spacebelowsatex[scale=.88]{untyped-adj/untypedadj-st-st-3-r-1} 
    \cr
    (4)&\satex[scale=.88]{untyped-adj/untypedadj-st-st-4-l-1}
    &\OT&\satex[scale=.88]{untyped-adj/untypedadj-st-st-4}
    &\TO&\spacebelowsatex[scale=.88]{untyped-adj/untypedadj-st-st-4-r-1} 
    \cr
    (5)&\satex[scale=.88]{untyped-adj/untypedadj-st-st-5-l-1}
    &\OT&\satex[scale=.88]{untyped-adj/untypedadj-st-st-5}
    &\TO&\spacebelowsatex[scale=.88]{untyped-adj/untypedadj-st-st-5-r-1} 
    \cr
    (6)&\satex[scale=.88]{untyped-adj/untypedadj-st-st-6-l-1}
    &\OT&\satex[scale=.88]{untyped-adj/untypedadj-st-st-6}
    &\TO&\spacebelowsatex[scale=.88]{untyped-adj/untypedadj-st-st-6-r-1} 
    \cr
    (7)&\satex[scale=.88]{untyped-adj/untypedadj-st-st-7-l-1}
    &\OT&\satex[scale=.88]{untyped-adj/untypedadj-st-st-7}
    &\TO&\spacebelowsatex[scale=.88]{untyped-adj/untypedadj-st-st-7-r-1} 
    \cr
    (8)&\satex[scale=.88]{untyped-adj/untypedadj-st-st-8-l-1}
    &\OT&\satex[scale=.88]{untyped-adj/untypedadj-st-st-8}
    &\TO&\satex[scale=.88]{untyped-adj/untypedadj-st-st-8-r-1}\pbox.\cr}
  \vskip\belowdisplayskip
  \endgroup
  \noindent Each one can be shown confluent for $\sequiv$ by considering the confluence of
  a natural branching in~${(\PUAdj,\stdcong^{\PUAdj})}$. For example, (5) is
  joinable as follows:
  \[
    \begin{tikzcd}[column sep=2.5ex]
      \satex[scale=.86]{untyped-adj/untypedadj-st-st-5} \tar[r]\tar[d]&
      \satex[scale=.86]{untyped-adj/untypedadj-st-st-5-l-1} \tar[r]&
      \satex[scale=.86]{untyped-adj/untypedadj-st-st-5-l-2} \tar[d]\\
      \satex[scale=.86]{untyped-adj/untypedadj-st-st-5-r-1} \tar[r]&
      \satex[scale=.86]{untyped-adj/untypedadj-st-st-5-r-2} \tar[r]&
      \satex[scale=.86]{untyped-adj/untypedadj-st-st-5-e}
    \end{tikzcd}
  \]
  Up to inverses, it corresponds to the following confluent natural branching of $(\PUAdj,\stdcong^{\PUAdj})$:
  \[
    \begin{tikzcd}[column sep=2.5ex]
      \satex[scale=.86]{untyped-adj/untypedadj-st-st-5} \tar[d,dotted]&
      \satex[scale=.86]{untyped-adj/untypedadj-st-st-5-l-1} \tar[l,dotted]
      \ar[d,phantom,"\stdcong^{\PUAdj}"]
      &
      \satex[scale=.86]{untyped-adj/untypedadj-st-st-5-l-2} \tar[l]\tar[d]\\
      \satex[scale=.86]{untyped-adj/untypedadj-st-st-5-r-1} &
      \satex[scale=.86]{untyped-adj/untypedadj-st-st-5-r-2} \tar[l,dotted]&
      \satex[scale=.86]{untyped-adj/untypedadj-st-st-5-e} \tar[l,dotted]
    \end{tikzcd}
  \]
  By the definition of $\sequiv$, (5) is confluent for $\sequiv$. The other
  kinds of separated structural-structural $\Q$\critical branchings are
  confluent by similar arguments.

  \medskip \noindent There are four kinds of half-separated structural-structural $\Q$\critical
  branchings listed below
  %\begingroup
  %\displayskipforlongtable
  \begingroup
  \vskip\abovedisplayskip
  \tabskip=0pt plus 1fill\halign
  to\linewidth{\tabskip=0pt\relax\tabskip=0pt\relax\hfill #\hfill\quad&\hfill
    $#$\hfill\quad&\hfill $#$\hfill\quad&\hfill $#$\hfill\quad&\hfill $#$\hfill\quad&\hfill $#$\hfill\tabskip=0pt plus 1fill\cr
      (1)&
           \satex{untyped-adj/untypedadj-half-1-l}
      &\OT&\;
            \satex{untyped-adj/untypedadj-half-1}\;
      &\TO&
            \spacebelowsatex{untyped-adj/untypedadj-half-1-r}
            % \leavevmode\vtop{\hbox{$\satex{untyped-adj/untypedadj-half-1-r}$}%
            % \hrule width0pt height0pt depth6pt
            % }
      \cr
      (2)&
           \satex{untyped-adj/untypedadj-half-2-l}
      &\OT&\;\satex{untyped-adj/untypedadj-half-2}\;
      &\TO&\spacebelowsatex{untyped-adj/untypedadj-half-2-r} \cr
      (3)&
           \satex{untyped-adj/untypedadj-half-3-l}
      &\OT&\;\satex{untyped-adj/untypedadj-half-3}\;
      &\TO&\spacebelowsatex{untyped-adj/untypedadj-half-3-r} \cr
      (4)&
           \satex{untyped-adj/untypedadj-half-4-l}
      &\OT&\;\satex{untyped-adj/untypedadj-half-4}\;
      &\TO&\satex{untyped-adj/untypedadj-half-4-r}\pbox.\cr}
    \vskip\belowdisplayskip
  \endgroup
  \noindent Each one can be shown confluent for~$\sequiv$ by considering the confluence of a natural
  branching in~${(\PUAdj,\stdcong^{\PUAdj})}$. For example, (1) is joinable as follows
  \[
    \begin{tikzcd}
      \satex{untyped-adj/untypedadj-half-1}\tar[d]\tar[r]&\satex{untyped-adj/untypedadj-half-1-l}\tar[d]\\
      \satex{untyped-adj/untypedadj-half-1-r}\tar[r]&\satex{untyped-adj/untypedadj-half-1-e}
    \end{tikzcd}
  \]
  Up to inverses, it corresponds to the following confluent natural branching of $(\PUAdj,\stdcong^{\PUAdj})$:
  \[
    \begin{tikzcd}
      \satex{untyped-adj/untypedadj-half-1}\tar[r,dotted,""{name=src,auto=false}]&\satex{untyped-adj/untypedadj-half-1-l}\tar[d,dotted]\\
      \satex{untyped-adj/untypedadj-half-1-r}\tar[u]\tar[r,""{name=tgt,auto=false}]&\satex{untyped-adj/untypedadj-half-1-e}
      \ar[from=src,to=tgt,phantom,"\stdcong^{\PUAdj}"]
    \end{tikzcd}
  \]
  By definition of $\sequiv$, it implies that (1) is confluent for $\sequiv$.

  There are two kinds of non-separated structural-structural $\Q$\critical
  branchings listed below:
  \[
    \begin{array}{r@{\quad}c@{\quad}c@{\quad}c@{\quad}c@{\quad}c}
      (1)&\satex{untyped-adj/untyped-sep-1-l}
      &\OT&\;\satex{untyped-adj/untyped-sep-1}\;
      &\TO&\spacebelowsatex{untyped-adj/untyped-sep-1-r} 
      \\
      (2)&
           \satex{untyped-adj/untyped-sep-2-l}
      &\OT&\;\satex{untyped-adj/untyped-sep-2}\;
      &\TO&\satex{untyped-adj/untyped-sep-2-r}
    \end{array}
  \]
  They are not confluent but they are not connected branchings.

  We now consider \emph{structural-operational} $\Q$\critical branchings, \ie
  those $\Q$\critical branchings~$(S_1,S_2)$ such that the inner $3$\generator
  of $S_1$ is a $\Q$-interchange generator and the inner $3$\generator of $S_2$
  is $\adjN$ or $\adjNinv$. We classify them as \emph{separated} and
  \emph{half-separated}. There are four kinds of separated
  structural-operational $\Q$\critical branchings listed below:
  \begingroup
  \vskip\abovedisplayskip
  \tabskip=0pt plus 1fill\halign
  to\linewidth{\tabskip=0pt\relax\tabskip=0pt\relax\hfill #\hfill\quad&\hfill
    $#$\hfill\quad&\hfill $#$\hfill\quad&\hfill $#$\hfill\quad&\hfill $#$\hfill\quad&\hfill $#$\hfill\tabskip=0pt plus 1fill\cr
      (1)&\satex{untyped-adj/untyped-st-op-sep-1-l}
      &\OT&\;\spacebelowsatex{untyped-adj/untyped-st-op-sep-1}\;
      &\TO&\satex{untyped-adj/untyped-st-op-sep-1-r} 
      \cr
      (2)&\satex{untyped-adj/untyped-st-op-sep-2-l}
      &\OT&\;\spacebelowsatex{untyped-adj/untyped-st-op-sep-2}\;
      &\TO&\satex{untyped-adj/untyped-st-op-sep-2-r}
      \cr
      (3)&\satex{untyped-adj/untyped-st-op-sep-3-l}
      &\OT&\;\spacebelowsatex{untyped-adj/untyped-st-op-sep-3}\;
      &\TO&\satex{untyped-adj/untyped-st-op-sep-3-r} 
      \cr
      (4)&\satex{untyped-adj/untyped-st-op-sep-4-l}
      &\OT&\;\satex{untyped-adj/untyped-st-op-sep-4}\;
      &\TO&\satex{untyped-adj/untyped-st-op-sep-4-r}\pbox.\cr}
    \vskip\belowdisplayskip
    \endgroup
  As above, each one can be shown confluent by considering a natural branching of
  $(\PUAdj,\stdcong^{\PUAdj})$.

  There are two kinds of half-separated structural-operational $\Q$\critical
  branchings listed below:
  \[
    \begin{array}{r@{\quad}c@{\quad}c@{\quad}c@{\quad}c@{\quad}c}
      (1)&\satex{adj-cp1-c}
      &\OT&\;\spacebelowsatex{adj-cp1-l}\;
      &\TO&\satex{adj-cp1-r} 
      \\
      (2)&\satex{adj-cp2-l}
      &\OT&\;\satex{adj-cp2-c}\;
      &\TO&\satex{adj-cp2-r}
    \end{array}
  \]
  As above, each one of them can be proved confluent by considering the
  associated critical branching in $(\PUAdj,\stdcong^{\PUAdj})$.

  Note that there are no \emph{operational-operational} $\Q$\critical branching,
  \ie $\Q$\critical branchings $(S_1,S_2)$ where the inner $3$\generators of
  both~$S_1$ and~$S_1$ are in $\set{\adjN,\adjNinv}$. Hence, all connected
  $\Q$\critical branchings are confluent.
\end{proof}

\goodbreak\noindent We can now show our weak confluence property:
\begin{prop}
  \label{prop:untyped-adj-confluent}
  All the connected branchings of $(\Q,\sequiv)$ are confluent.
\end{prop}
\begin{proof}
  By \Lemr{untyped-adj-newman}, \Lemr{untyped-adj-cp} and \Lemr{untyped-adj-cp-confluent}.
\end{proof}

In order to obtain a weak coherence property for~$\PUAdj$, we first adapt
several properties stated in \Ssecr{coherence}.
\begin{lem}
  \label{lem:uadj-(3,2)-Q-canonical-form}
  Given $F\co \phi\TO \phi' \in \freeinvf{\prespcat\Q}$ where either $\phi$ or
  $\phi'$ is connected, we have $F = G \comp_2 \finv{H}$ for some $G\co \phi \TO \psi$
  and $H\co \phi' \TO \psi$.
\end{lem}
\begin{proof}
  By a direct adaptation of \Propr{confluent-cr} involving connected $2$-cells
  only, and using \Propr{untyped-adj-confluent}.
\end{proof}

\begin{lem}
  \label{lem:untyped-adj-small-coherence}
  Given $F_1,F_2\co \phi\TO \phi' \in \prespcat\Q_3$, if $\phi$ is
  connected, then $F_1 = F_2$ in $\freeinvf{\prespcat\Q}_3$.
\end{lem}
\begin{proof}
  Since $\phi$ is connected, $\phi'$ is connected. By
  \Propr{untyped-adj-terminating}, there is $G\co \phi' \TO \psi \in
  \prespcat\Q_3$ such that $\psi$ is a normal form for $\Q$. By
  \Propr{untyped-adj-confluent}, there is $H_1,H_2\co \psi \TO \psi' \in
  \prespcat\Q_3$ such that $F_1 \comp_2 G \comp_2 H_1 = F_1 \comp_2 G \comp_2
  H_2$. Since $\psi$ is a normal form, $H_1 = H_2 = \unit \psi$. So $F_1 \comp_2
  G = F_2 \comp_2 G$, thus $F_1 = F_2$ in $\freeinvf{\prespcat\Q}_3$.
\end{proof}

\begin{lem}
  \label{lem:uadj-parallel-implies-eq}
  Given $F_1,F_2\co \phi\TO \phi' \in \freeinvf{\prespcat\Q}_3$, if $\phi$ is
  connected, then $F_1 = F_2$ in $\freeinvf{\prespcat\Q}_3$.
\end{lem}
\begin{proof}
  By a direct adaptation of the proof of \Propr{confluent-impl-coherence}, 
  using \Lemr{uadj-(3,2)-Q-canonical-form} and \Lemr{untyped-adj-small-coherence}.
\end{proof}
\noindent We can now conclude with the weak coherence property for~$\PUAdj$:
\begin{theo}
  Given~$F_1,F_2\co \phi \TO \phi' \in \smash{\freeinvf{\prespcat{\PUAdj}}_3}$
  with~$\phi$ or~$\phi'$ connected, we have~$F_1 = F_2$.
\end{theo}
\begin{proof}
  Let~$\Gamma'\co \smash{\freeinvf{\prespcat\Q} \to
    \freeinvf{\prespcat{\PUAdj}}}$ be the $3$\prefunctor which is the
  factorization of~$\Gamma$ through the canonical
  $3$\prefunctor~$\freecat{(\restrict 3 \Q)} \to
  \smash{\freeinvf{\prespcat\Q}}$. By definition
  of~$\smash{\freeinvf{\prespcat\PUAdj}}$, for~$i \in \set{1,2}$, we have
  \[
    F_i = G_{i,1} \pcomp_2 \finv H_{i,1} \pcomp_2
    \cdots \pcomp_2 G_{i,k_i} \pcomp_2 \finv H_{i,k_i}
  \]
  for some~$k_i \in \N$,
  $2$\cells~$\phi_{i,0},\ldots,\phi_{i,k_i},\psi_{i,1},\ldots,\psi_{i,k_i} \in
  \freecat\Q_2$ such that~$\phi_{i,0} = \phi$ and~$\phi_{i,k_i} = \phi'$, and,
  for~$j \in \set{1,\ldots,k_i}$, $3$\cells
  \[
    G_{i,j}\co \phi_{i,j-1} \TO \psi_{i,j}
    \qtand
    H_{i,j} \co \phi_{i,j} \TO \psi_{i,j}
  \]
  of~$\prespcat\PUAdj_3$. Since either~$\phi$ or~$\phi'$ is connected, we have
  that all the~$\phi_{i,j}$'s and the~$\psi_{i,j}$'s are connected. Moreover,
  all the~$G_{i,j}$'s and the~$H_{i,j}$'s are in the image of~$\Gamma'$. So,
  for~$i \in \set{1,2}$,~$F_i = \Gamma'(F'_i)$ for some~$F'_i \co \phi \TO \phi'
  \in \smash{\freeinvf{\prespcat\Q}}$. By \Lemr{uadj-parallel-implies-eq}, we
  have~$F'_1 = F'_2$, so that~$F_1 = F_2$.
\end{proof}

%%% Local Variables:
%%% mode: latex
%%% TeX-master: "proxy-app-untyped"
%%% End:

%% file: appendix.tex
\allowdisplaybreaks

\section{Equivalence between definitions of precategories}
\label{sec:precat-equivalent-definitions}
\noindent We prove the equivalence between the equational and the enriched
definition of precategories:%
\ifx\precatequivalentdefsprop\undefined
\begin{prop}
  There is an equivalence of categories between $\nPCat{n+1}$ and
  categories enriched in $\nPCat n$ with the funny tensor product.
\end{prop}
\else
\precatequivalentdefsprop*
\fi
% \begin{propapp}
%   \label{prop:enriched}
%   There is an equivalence of categories between $\nPCat{n+1}$ and
%   categories enriched in $\nPCat n$ with the funny tensor product.
% \end{propapp}
\begin{proof}
  \renewcommand\tcomp{c}%
  \renewcommand\tunit{i}%
  Given~$C \in \nPCat {n+1}$, we define an associated object~$D \in
  \enrCatpar{\nPCat n}$ as follows. We put
  \[
    D_0 = C_0 \qtand D(x,y) = C_{\uparrow(x,y)}
  \]
  where~$C_{\uparrow(x,y)}$ is the $n$\precategory such that
  \[
    (C_{\uparrow(x,y)})_i = \set{ u \in C_{i+1} \mid \csrc_0(u) = x \text{ and }
      \ctgt_0(u) = y}
  \]
  for~$i \in \set{0,\ldots,n}$ and whose composition operation~$\pcomp_{k,l}$ is the
  operation~$\pcomp_{k+1,l+1}$ on~$C$ for~$k,l \in \set{1,\ldots,n}$. Given~$x \in D_0$,
  we define the identity morphism
  \[
    \tunit_{x}\co \termcat \to D(x,x)
  \]
  as the morphism which maps the unique $0$\cell~$\ast$ of~$\termcat$
  to~$\unit {x} \in C_1$. Given~$x,y,z \in C_0$, we define the composition
  morphism
  \[
    \tcomp_{x,y,z}\co D(x,y) \funny D(y,z) \to D(x,z) \in \nPCat {n}
  \]
  as the unique morphism such that~$l_{x,y,z} = \tcomp_{x,y,z} \circ
  \lincfun{D(x,y),D(y,z)}$ is the composite
  \[
    D(x,y) \times \catsk 0 {D(y,z)} \cong \coprod_{g \in D(y,z)_0} D(x,y)
    \xto{[(-)\pcomp_0 g]_{g \in D(y,z)_0}} D(x,z)
  \]
  and~$r_{x,y,z} = \tcomp_{x,y,z} \circ \rincfun{D(x,y),D(y,z)}$ is the
  composite
  \[
    \catsk 0 {D(x,y)} \times D(y,z) \cong \coprod_{f\in D(x,y)_0} D(y,z)
    \xto{[f\pcomp_0 (-)]_{f \in D(x,y)_0}} D(x,z).
  \]
  We verify that the composition morphism is left unital, \ie given~$x,y
  \in D_0$, the diagram
  \[
    \begin{tikzcd}[column sep=small]
      \termcat \funny D(x,y) \ar[rr,"{\tunit_x \funny D(x,y)}"] \ar[rd,"\funnyl_{D(x,y)}"'] & &
      D(x,x) \funny D(x,y) \ar[ld,"{\tcomp_{x,x,y}}"] \\
      &  D(x,y)
    \end{tikzcd}
  \]
  commutes. We compute that
  \begin{align*}
    \tcomp_{x,x,y} \circ (\tunit_x \funny D(x,y)) \circ \lincfun{\termcat,D(x,y)}
    &= \tcomp_{x,x,y} \circ \lincfun{D(x,x),D(x,y)} \circ (\tunit_x \times \catsk 0 {D(x,y)})
    \\
    &
      & \makebox[1cm][r]{(by definition of~$\funny$)} \\
    &= l_{x,x,y} \circ (\tunit_x \times \catsk 0 {D(x,y)}) \\
    &= \cucatsk {D(x,y)} \circ \pi_2
      &
      \makebox[1cm][r]{\hfill (by unitality of~$\unit x$)}
      \\
    &= \funnyl_{D(x,y)} \circ \lincfun{\termcat,D(x,y)}
      \shortintertext{and}
      \tcomp_{x,x,y} \circ (\tunit_x \funny D(x,y)) \circ \rincfun{\termcat,D(x,y)}
    &=  \tcomp_{x,x,y} \circ \rincfun{D(x,x),D(x,y)} \circ (\catsk 0 {(\tunit_x)} \times D(x,y)) \\
    &
      & \makebox[1cm][r]{(by definition of~$\funny$)} \\
    &=  r_{x,x,y} \circ (\catsk 0 {(\tunit_x)} \times D(x,y)) \\
    &=  \pi_2
      &
      \makebox[1cm][r]{\hfill (by unitality of~$\unit x$)}
    \\
    &= \funnyl_{D(x,y)} \circ \rincfun{\termcat,D(x,y)}
  \end{align*}
  Thus, by the colimit definition of~$\termcat \funny D(x,y)$, the above
  triangle commutes. Similarly, the triangle
  \[
    \begin{tikzcd}[column sep=small]
      D(x,y) \funny \termcat \ar[rr,"{D(x,y) \funny \tunit_y}"] \ar[rd,"\funnyr_{D(x,y)}"'] & &
      D(x,y) \funny D(y,y) \ar[ld,"{\tcomp_{x,y,y}}"] \\
      &  D(x,y)
    \end{tikzcd}
  \]
  commutes, so that the composition morphism is right unital. We now verify that
  it is associative, \ie given~$w,x,y,z \in D_0$, that the diagram
  \begingroup
  \makeatletter
  \renewcommand{\maketag@@@}[1]{\hbox to 0.000008pt{\hss\m@th\normalsize\normalfont#1}}%
  \makeatother
  \begin{equation}
    \label{eq:precat-enr-assoc}
    \begin{tikzpicture}[commutative diagrams/every diagram,xscale=2.4,yscale=0.6]
      \node at (0:2.7cm) {};
      \node at (180:2.7cm) {};
      \node (PB) at (90+72:2cm) {$\mathmakebox[3cm][c]{\hspace*{1cm}(D(w,x) \funny D(x,y)) \funny D(y,z)}$};
      \node (PL) at (90:2cm) {$D(w,y) \funny D(y,z)$};
      \node (PR1) at (90+2*72:2cm) {$D(w,x) \funny (D(x,y) \funny D(y,z))$};
      \node (PR2) at (90+3*72:2cm) {$D(w,x) \funny D(x,z)$};
      \node (PE) at (90+4*72:2cm) {$D(w,z)$};
      % \node at (0,0) {$=$};
      \path[commutative diagrams/.cd,every arrow,every label]
      (PB) edge node {${c_{w,x,y} \funny D(y,z)}$} (PL)
      (PL) edge node {${c_{w,y,z}}$} (PE)
      (PB) edge node[swap,pos=0.3] {$\mathstrut\smash{\funnyass_{D(w,x),D(x,y),D(y,z)}}$} (PR1)
      (PR1) edge node[swap] {${D(w,x) \funny c_{x,y,z}}$} (PR2)
      (PR2) edge node[swap,pos=0.7] {$\mathstrut\smash{c_{w,x,z}}$} (PE);
    \end{tikzpicture}
    % \begin{tikzcd}[column sep=7em]
    %   (D(w,x) \funny D(x,y)) \funny D(y,z) \ar[r,"\funnyass_{D(w,x),D(x,y),D(y,z)}"] \ar[d,"{\tcomp_{w,x,y} \funny D(y,z)}"'] & D(w,x) \funny (D(x,y) \funny
    %   D(y,z)) \ar[d,"{D(w,x) \funny \tcomp_{x,y,z}}"]\\
    %   D(w,y) \funny D(y,z) \ar[dr,"{\tcomp_{w,y,z}}"'] & D(w,x) \funny D(x,z) \ar[d,"{\tcomp_{w,x,z}}"] \\
    %   & D(w,z)
    % \end{tikzcd}
  \end{equation}
  \endgroup
  commutes. By a colimit definition analogous to~\eqref{eq:funny-3-pushout}, it
  is enough to show the commutation of the diagram when precomposing with the
  morphisms~$\varpi_1,\varpi_2,\varpi_3$ where
  \begin{align*}
    \varpi_1 &= \lincfun{D(w,x)\funny D(x,y),D(y,z)} \circ (\lincfun{D(w,x),D(x,y)}
      \times \catsk 0 {D(y,z)}), \\
    \varpi_2 &= \lincfun{D(w,x)\funny D(x,y),D(y,z)} \circ (\rincfun{D(w,x),D(x,y)}
      \times \catsk 0 {D(y,z)}), \\
    \varpi_3 &= \rincfun{D(w,x)\funny D(x,y),D(y,z)}\zbox.
  \end{align*}
  Writing~$D^1,D^2,D^3$ for~$D(w,x),D(x,y),D(y,z)$, we compute that
  \begin{align*}
    &\phantom{\;=\;} \tcomp_{w,x,z} \circ (D^1\funny\tcomp_{x,y,z}) \circ \funnyass_{D^1,D^2,D^3} \circ \varpi_1 \\
    &= \tcomp_{w,x,z} \circ (D^1\funny\tcomp_{x,y,z}) \circ \funnyass_{D^1,D^2,D^3} \circ  \lincfun{D^1\funny D^2,D^3} \circ (\lincfun{D^1,D^2}
      \times \catsk 0 {D^3})\\
    &= \tcomp_{w,x,z} \circ (D^1\funny\tcomp_{x,y,z}) \circ \lincfun{D^1,D^2\funny D^3} \circ  \alpha_{D^1,\catsk 0 D^2,\catsk 0 D^3}\\
    &= \tcomp_{w,x,z} \circ \lincfun{D^1,D(x,z)} \circ (D^1\times ((-)\pcomp_0(-))) \circ  \alpha_{D^1,\catsk 0 D^2,\catsk 0 D^3}\\
    &= ((-)\pcomp_0(-)) \circ (D^1\times ((-)\pcomp_0(-))) \circ  \alpha_{D^1,\catsk 0 D^2,\catsk 0 D^3}\\
    &= ((-)\pcomp_0(-)) \circ (((-)\pcomp_0(-))\times \catsk 0 D^3) & \makebox[1cm][r]{(by associativity of~$\pcomp_0$)}\\
    &= \tcomp_{w,y,z} \circ \lincfun{D(w,y),D^3} \circ (((-)\pcomp_0(-))\times \catsk 0 D^3)\\
    &= \tcomp_{w,y,z} \circ \lincfun{D(w,y),D^3} \circ (\tcomp_{w,x,y}\times \catsk 0 D^3) \circ (\lincfun{D^1,D^2}
      \times \catsk 0 {D^3})\\
    &= \tcomp_{w,y,z} \circ (\tcomp_{w,x,y}\funny D^3) \circ \lincfun{D^1\funny D^2,D^3} \circ (\lincfun{D^1,D^2}
      \times \catsk 0 {D^3})\\
    &= \tcomp_{w,y,z} \circ (\tcomp_{w,x,y}\funny D^3) \circ \varpi_1
  \end{align*}
  so that the diagram~\eqref{eq:precat-enr-assoc} commutes when precomposed with~$\varpi_1$ and, similarly, it commutes when precomposed with~$\varpi_2$
  and~$\varpi_3$. Thus, \eqref{eq:precat-enr-assoc} commutes. Hence,~$D$ is a
  category enriched in $n$\precategories. The operation~$C \mapsto D$ can easily
  be extended to morphisms of $(n{+}1)$\precategories, giving a functor
  \[
    F\co \nPCat{n+1} \to \enrCatpar{\nPCat n}.
  \]
  \medskip \noindent Conversely, given~$C \in \enrCatpar{\nPCat n}$, we define
  an associated object~$D \in \nPCat {n+1}$. We put
  \[
    D_0 = C_0 \qtand D_{i+1} = \coprod_{x,y \in
      C_0}C(x,y)_i
  \]
  for~$i \in \set{0,\ldots,n}$. In the following, given $x,y \in C_0$, we write
  $\iota_{x,y} \co C(x,y)_i \to D_{i+1}$ for the canonical coprojection.
  Given~$k \in \N$ with~$k\le n$,~$\iota_{x,y}(u)\in D_{k+1}$ and~$\eps \in
  \set{-,+}$, we put
  \[
    \csrctgt\eps_k(\iota_{x,y}(u)) =
    \begin{cases}
      x & \text{if~$k = 0$ and~$\eps = -$,} \\
      y & \text{if~$k = 0$ and~$\eps = +$,} \\
      \iota_{x,y}(\csrctgt\eps_{k-1}(u)) & \text{if~$k > 0$,}
    \end{cases}
  \]
  so that the operations~$\csrc,\ctgt$ equips~$D$ with a structure of
  $(n{+}1)$\globular set. Given~$x \in D_0$, we put
  \[
    \unit {x} =
    \iota_{x,x}(\tunit_x(\ast))
  \]
  and, given~$k \in \N$ with~$k \le n-1$ and~$\iota_{x,y}(u) \in D_{k+1}$, we put
  \[
    \unitp
    {k+2}{\iota_{x,y}(u)} = \iota_{x,y}(\unitp {k+1}{u})\zbox.
  \]
  Given~$i,{k_1},{k_2} \in \set{0,\ldots,n}$ with~$i = \min({k_1},{k_2}) - 1$, and~$u =
  \iota_{x,y}(\tilde u) \in D_{k_1},v = \iota_{x',y'}(\tilde v) \in D_{k_2}$
  that are $i$\composable, we put
  \[
    u \pcomp_i v =
    \begin{cases}
      \iota_{x,y}({\tilde u} \pcomp_i {\tilde v}) & \text{if~$i > 0$} \\
      \iota_{x,y'}(l_{x,y,y'}({\tilde u},\unitp {k_1 - 1} {\tilde v})) & \text{if~$i = 0$ and~${k_2} = 1$} \\
      \iota_{x,y'}(r_{x,y,y'}(\unitp {k_2 - 1} {\tilde u},{\tilde v})) & \text{if~$i = 0$ and~${k_1} = 1$}
    \end{cases}
  \]
  where~$l_{x,y,z}$ is the composite
  \[
    C(x,y) \times \catsk 0 {C(y,z)} \xto{\lincfun{C(x,y),C(y,z)}} C(x,y) \funny C(y,z)
    \xto{\tcomp_{x,y,z}} C(x,z)
  \]
  and~$r_{x,y,z}$ is the composite
  \[
    \catsk 0 {C(x,y)} \times C(y,z) \xto{\rincfun{C(x,y),C(y,z)}} C(x,y) \funny C(y,z)
    \xto{\tcomp_{x,y,z}} C(x,z).
  \]
  We now have to show that the axioms of $(n{+}1)$\precategories are satisfied.
  Note that, by the definition of~$D$, it is enough to prove the axioms for
  the~$\unitp 1{}$ and~$\pcomp_0$ operations. Given~$x \in D_0$ and~$\eps \in
  \set{-,+}$, we have
  \[
    \csrctgt\eps_0(\unit x) = \csrctgt\eps_0(\iota_{x,x}(\tunit_x(\ast))) = x
  \]
  so that~\Axr{precat:src-tgt-unit} holds. For~$k \in \set{1,\ldots,n+1}$, given~$u =
  \iota_{x,y}(\tilde u) \in D_k$ and~$v = \iota_{y,z}({\tilde v}) \in D_1$ such
  that~$u,v$ are $0$\composable, if~$k = 1$, then
  \[
    \csrc_0(u \pcomp_0 v) = \csrc_0(\iota_{x,z}(l_{x,y,z}(\tilde u,{\tilde v}))) = x,
  \]
  and, similarly,~$\ctgt_0(u \pcomp_0 v) = z$. Otherwise, if~$k > 1$, then, for~$\eps \in \set{-,+}$,
  \begin{align*}
    \csrctgt\eps_{k-1}(u \pcomp_0 v)
    &= \csrctgt\eps_{k-1}(\iota_{x,z}(l_{x,y,z}(\tilde u,\unitp{k-1} {\tilde v}))) \\
    &= \iota_{x,z}(\csrctgt\eps_{k-2}(l_{x,y,z}(\tilde u,\unitp {k-1}{\tilde v}))) \\
    &= \iota_{x,z}(l_{x,y,z}(\csrctgt\eps_{k-2}(\tilde u),\unitp {k-2}{\tilde v})) \\
    &= \iota_{x,y}(\csrctgt\eps_{k-2}(\tilde u)) \pcomp_0 \iota_{y,z}(\tilde v) \\
    &= \csrctgt\eps_{k-1}(u) \pcomp_0 v.
  \end{align*}
  Analogous equalities are satisfied for $0$\composable~$u \in D_1$ and~$v \in
  D_k$, so that~\Axr{precat:csrc-tgt} holds. Given~$k \in \set{1,\ldots,n+1}$ and~$u =
  \iota_{x,y}(\tilde u) \in D_k$, we have
  \begin{align*}
    u \pcomp_0 \unit y
    &= \iota_{x,y}(l_{x,y,y}(\tilde u,\unitp {k-1} {\tunit_y(\ast)})) \\
    &= \iota_{x,y}(c_{x,y,y} \circ (C(x,y) \funny \tunit_y) \circ \lincfun{C(x,y),\termcat}(\tilde u,\unitp {k-1} {\ast})) \\
    &= \iota_{x,y}(\funnyr_{C(x,y)} \circ \lincfun{C(x,y),\termcat}(\tilde u,\unitp {k-1} {\ast}))
      & \makebox[4cm][r]{(by the axioms of enriched categories)}
    \\
    &= \iota_{x,y}(\pi_1(\tilde u,\unitp {k-1} {\ast}))
    & \makebox[4cm][r]{(by definition of~$\funnyr$)} \\
    &= u\zbox.
  \end{align*}
  Moreover, given~$k \in \set{1,\ldots,n}$ and $0$\composable~$u = \iota_{x,y}(\tilde u)
  \in D_1$ and~$v = \iota_{y,z}(\tilde v) \in D_k$, we have
  \begin{align*}
    u \pcomp_0 \unitp{k+1} v
    &= \iota_{x,z}(r_{x,y,z}(\unitp k {\tilde u},\unitp k {\tilde v})) \\
    &= \iota_{x,z}(\unitp k {} (r_{x,y,z}(\unitp {k-1} {\tilde u},\tilde v))) \\
    &= \unitp {k+1}{}(\iota_{x,z}(r_{x,y,z}(\unitp {k-1} {\tilde u},\tilde v))) \\
    &= \unitp {k+1} {u \pcomp_0 v}\zbox.
  \end{align*}
  Analogous equalities hold when composing with identities on the left, so
  that~\Axr{precat:compat-id-comp} holds. Given~$k \in \set{1,\ldots,n+1}$ and
  $0$\composable~$u_1 = \iota_{w,x}(\tilde u_1) \in D_k$,~$u_2 =
  \iota_{x,y}(\tilde u_2) \in D_1$ and~$u_3 = \iota_{y,z}(\tilde u_3) \in D_1$,
  we have
  \begin{align*}
    (u_1 \pcomp_0 u_2) \pcomp_0 u_3
    &= \iota_{w,z}(l_{w,y,z}(l_{w,x,y}(\tilde u_1,\unitp{k-1}{\tilde u_2}),\unitp {k-1}{\tilde u_3}))\zbox.
  \end{align*}
  Writing~$C^1,C^2,C^3$ for~$C(w,x),C(x,y),C(y,z)$, we compute that
  \begin{align*}
    &\phantom{\;=\;}l_{w,y,z} \circ (l_{w,x,y} \times \catsk 0 {C^3}) \\
    &= \tcomp_{w,y,z} \circ \lincfun{C(w,y),C^3} \circ (\tcomp_{w,x,y} \times \catsk 0 {C^3}) \circ (\lincfun{C^1,C^2} \times \catsk 0 {C^3}) \\
    &= \tcomp_{w,y,z} \circ (\tcomp_{w,x,y} \funny C^3) \circ \lincfun{C^1\funny C^2,C^3} \circ (\lincfun{C^1,C^2} \times \catsk 0 {C^3})
      \shortintertextnobreakabove{\hfill(by definition of~$\funny$)}
    &= \tcomp_{w,x,z} \circ (C^1 \funny \tcomp_{x,y,z}) \circ \funnyass_{C^1,C^2,C^3} \circ \lincfun{C^1\funny C^2,C^3} \circ (\lincfun{C^1,C^2} \times \catsk 0 {C^3})
    \shortintertextnobreakabove{\hfill(by the axioms of enriched categories)}
    &= \tcomp_{w,x,z} \circ (C^1 \funny \tcomp_{x,y,z}) \circ \lincfun{C^1,C^2 \funny C^3} \circ \alpha_{C^1,\catsk 0 C^2,\catsk 0 C^3}
      \shortintertextnobreakabove{\hfill(by definition of~$\funnyass$)}
    &= \tcomp_{w,x,z} \circ \lincfun{C^1,C(x,z)} \circ (C^1 \times \catsk 0 {(\tcomp_{x,y,z})}) \circ \alpha_{C^1,\catsk 0 C^2,\catsk 0 C^3} \\
    &= l_{w,x,z} \circ (C^1 \times \catsk 0 {(l_{x,y,z})}) \circ \alpha_{C^1,\catsk 0 C^2,\catsk 0 C^3}\zbox.
  \end{align*}
  Thus,
  \begin{align*}
    (u_1 \pcomp_0 u_2) \pcomp_0 u_3
    &= \iota_{w,z}(l_{w,x,z}(\tilde u_1,\catsk 0 {(l_{x,y,z})}(\unitp{k-1}{\tilde u_2},\unitp{k-1}{\tilde u_3}))) \\
    &= \iota_{w,z}(l_{w,x,z}(\tilde u_1,\unitp{k-1}{\catsk 0 {(l_{x,y,z})}({\tilde u_2},{\tilde u_3})})) \\
    &= u_1 \pcomp_0 \iota_{x,z}(\catsk 0 {(l_{x,y,z})}({\tilde u_2},{\tilde u_3}))\\
    &= u_1 \pcomp_0 \iota_{x,z}(l_{x,y,z}({\tilde u_2},{\tilde u_3}))\\
    &= u_1 \pcomp_0 (u_2 \pcomp_0 u_3)
  \end{align*}
  and similar equalities can be shown for~$(u_1,u_2,u_3) \in (D_1 \times_0 D_k
  \times_0 D_1) \sqcup (D_1 \times_0 D_1 \times_0 D_k)$, so that
  \Axr{precat:assoc} holds. Finally, for~$i,k_1,k_2,k \in \set{1,\ldots,n+1}$ such that~$i = \min(k_1,k_2) - 1$,~$k = \max(k_1,k_2)$, given~$u = \iota_{x,y}(\tilde u)
  \in D_1$ and $i$\composable~$v_1 = \iota_{y,z}(\tilde v_1) \in D_{k_1}, v_2 =
  \iota_{y,z}(\tilde v_2) \in D_{k_2}$, we have
  \begin{align*}
    u \pcomp_0 (v_1 \pcomp_i v_2) 
    &= u \pcomp_0 \iota_{y,z}(\tilde v_1 \pcomp_{i-1} \tilde v_2) \\
    &= \iota_{x,z}(r_{x,y,z}(\unitp {k-1} u,\tilde v_1 \pcomp_{i-1} \tilde v_2)) \\
    &= \iota_{x,z}(r_{x,y,z}(\unitp {k_1-1} u \pcomp_{i-1} \unitp {k_2-1} u,\tilde v_1 \pcomp_{i-1} \tilde v_2)) \\
    &= \iota_{x,z}(r_{x,y,z}(\unitp {k_1-1} u,\tilde v_1) \pcomp_{i-1}r_{x,y,z}(\unitp {k_2-1} u,\tilde v_2)) \\
    &= \iota_{x,z}(r_{x,y,z}(\unitp {k_1-1} u,\tilde v_1)) \pcomp_{i} \iota_{x,z}(r_{x,y,z}(\unitp {k_2-1} u,\tilde v_2)) \\
    &= (u \pcomp_0 v_1) \pcomp_{i} (u \pcomp_0 v_2)
  \end{align*}
  and an analogous equality can be shown for~$((u_1,u_2),v) \in ((D_{k_1}
  \times_i D_{k_2}) \times_0 D_1)$, so that \Axr{precat:distrib} holds.
  Hence,~$D$ is an $(n{+}1)$\precategory. The construction~$C \mapsto D$ extends
  naturally to enriched functors, giving a functor~$G\co \enrCatpar{\nPCat n} \to
  \nPCat{n+1}$.

  Given~$C \in \nPCat {n+1}$ and~$C' = G \circ F(C)$, there is a
  morphism~$\alpha_C \co C \to C'$ which is the identity between~$C_0$
  and~$C'_0$ and, for~$k \in \N$ with~$k\le n$, maps~$u \in C_{k+1}$
  to~$\iota_{x,y}(u)$ where~$x = \csrc_0(u)$ and~$y = \ctgt_0(u)$, and one can
  verify that it is an isomorphism which is natural in~$C$.

  Conversely, given~$C \in \enrCatpar{\nPCat n}$ and~$C' = F \circ G (C)$, there
  is a morphism~$\beta\co C \to C'$ which is the identity between~$C_0$ and~$C'_0$, and, for~$x,y \in C_0$, maps~$u \in C(x,y)$ to~$\iota_{x,y}(u) \in
  C'(x,y)$, and one can verify that it is an isomorphism which is natural
  in~$C$. Hence,~$F$ is an equivalence of categories. 
\end{proof}

% \newpage
\section{Gray presentations induce Gray categories}
\label{sec:gray-pres-gray-cat}

Until the end of this section, we
suppose fixed a Gray presentation~$\P$. Our goal is to prove
\Thmr{gray-pres-gray-cat}, \ie that $\prespcat\P$
is a lax Gray category. We start by the exchange law for $3$-cells that we prove
first on rewriting steps:
\begin{lemapp}
  \label{lem:prespcat-peiffer-ctxt}
  % Given $A_1 = \lambda_1 \comp_0 (l_1 \comp_0 R_1^{\eps_1} \comp_0 r_1) \comp_1 \rho_1$
  % and $A_2 = \lambda_2 \comp_0 (l_2 \comp_0 R_2^{\eps_2} \comp_0 r_2) \comp_2 \rho_2$
  % and $u_1,u_2,v_1,v_2 \in
  % \prespcat\P_1$ and $\chi \in \prespcat\P_2$ as in
  % Given rewriting steps $R_i \in \freecat\P_3$ and $l_i,r_i \in \freecat\P_1$,
  % $\lambda_i,\rho_i \in \freecat\P_2$, $A_i \in \P_3$ such that $R_i =
  % \lambda_i \comp_0 (l_i \comp_0 A_i \comp_0 r_i) \comp_i \rho_i$ for $i \in
  % \set{1,2}$ and $\ctgt_1(\rho_1) = \csrc_1(\lambda_2)$, we have, in
  % $\prespcat\P_3$,
  % \[
  %   (R_1 \comp_1  \phi_2) \comp_2 (\phi_1' \comp_1  R_2) = (\phi_1 \comp_1
  %   R_2) \comp_2 (R_1 \comp_1 \phi_2')
  % \]
  % where $\phi_i = \csrc_2(R_i)$ and $\phi_i' = \ctgt_2(R_i)$ for $i \in \set{1,2}$.
  Given rewriting steps $R_i\co \phi_i \TO \phi'_i \in \freecat\P_3$ for $i \in
  \set{1,2}$, such that $R_1,R_2$ are $1$\composable, we have, in
  $\prespcat\P_3$,
  \[
    (R_1 \comp_1  \phi_2) \comp_2 (\phi_1' \comp_1  R_2) = (\phi_1 \comp_1
    R_2) \comp_2 (R_1 \comp_1 \phi_2').
  \]
\end{lemapp}
\begin{proof}
  Let $l_i,r_i \in \prespcat\P_1$, $\lambda_i,\rho_i \in \prespcat\P_2$, $A_i \in
  \P_3$ such that $R_i = \lambda_i \comp_0 (l_i \comp_0 A_i \comp_0 r_i) \comp_i
  \rho_i$ for $i \in \set{1,2}$, and $\mu_i,\mu'_i\in \prespcat\P_2$ such that
  $A_i\co \mu_i \TO \mu_i'$ for $i \in \set{1,2}$. In $\prespcat\P_3$, we have
  \begin{align*}
& (R_1 \comp_1  \phi_2) \comp_2 (\phi_1' \comp_1  R_2) \\
    =\;&\lambda_1  \\
    & \comp_1 [((l_1 \comp_0 A_1 \comp_0 r_1) \comp_1 \rho_1 \comp_1 \lambda_2 \comp_1 ( l_2 \comp_0 \mu_2 \comp_0 r_2) ) \\
    & \phantom{{}\comp_1{}}\comp_2 ((l_1 \comp_0 \mu'_1 \comp_0 r_1) \comp_1 \rho_1 \comp_1 \lambda_2 \comp_1 ( l_2 \comp_0 A_2 \comp_0 r_2 ))] \\
    & \comp_1 \rho_2 && \text{(by the axioms of precategories)}\\
    =\;&\lambda_1  \\
    & \comp_1 [((l_1 \comp_0 \mu_1 \comp_0 r_1) \comp_1 \rho_1 \comp_1 \lambda_2 \comp_1 ( l_2 \comp_0 A_2 \comp_0 r_2 ) ) \\
    & \phantom{{}\comp_1{}}\comp_2 ((l_1 \comp_0 A_1 \comp_0 r_1) \comp_1 \rho_1 \comp_1 \lambda_2 \comp_1 ( l_2 \comp_0 \mu_2' \comp_0 r_2))] \\
    & \comp_1 \rho_2 && \text{(by independence generator)}\\
    &= (\phi_1 \comp_1  R_2) \comp_2 (R_1 \comp_1  \phi_2') &&\qedhere
  \end{align*}
\end{proof}
\noindent We can now conclude that the exchange law for $3$-cells holds:
\begin{lemapp}
  \label{lem:prespcat-peiffer}
  % Given $A_1 = \lambda_1 \comp_0 (l_1 \comp_0 R_1^{\eps_1} \comp_0 r_1) \comp_1 \rho_1$
  % and $A_2 = \lambda_2 \comp_0 (l_2 \comp_0 R_2^{\eps_2} \comp_0 r_2) \comp_2 \rho_2$
  % and $u_1,u_2,v_1,v_2 \in
  % \prespcat\P_1$ and $\chi \in \prespcat\P_2$ as in
  Given $F_i\co \phi_i \TO \phi_i' \in \prespcat\P_3$ for $i \in \set{1,2}$
  such that $F_1,F_2$ are $1$\composable,
  we have, in $\prespcat\P_3$,
  \[
    (F_1 \comp_1  \phi_2) \comp_2 (\phi_1' \comp_1  F_2) = (\phi_1 \comp_1
    F_2) \comp_2 (F_1 \comp_1 \phi_2').
  \]
\end{lemapp}
\begin{proof}
  As an element of $\prespcat\P_3$, $F_i$ can be written $F_i = R_{i,1} \comp_2
  \cdots \comp_2 R_{i,k_i}$ where
  \[
    R_{i,j} = \lambda_{i,j} \comp_1 (l_{i,j}
    \comp_0 A_{i,j} \comp_0 r_{i,j}) \comp_1 \rho_{i,j} 
  \]
  for some
  $k_{i} \in \N$, $\lambda_{i,j},\rho_{i,j} \in \prespcat\P_2$, $l_{i,j},r_{i,j}
  \in \prespcat\P_1$, $A_{i,j} \in \P_3$ for $1 \le j
  \le k_i$, for $i \in \set{1,2}$.
  Note that
  \[
    F_1 \comp_1 \phi_2 = (R_{1,1} \comp_1 \phi_2) \comp_2 \cdots \comp_2
    (R_{1,k_1} \comp_1 \phi_2)
  \] 
  and
  \[
    \phi_1' \comp_1 F_2 = (\phi_1' \comp_1 R_{2,1}) \comp_2 \cdots \comp_2
    (\phi_1' \comp_1 R_{2,k_2}).
  \]
  Then, by using \Lemr{prespcat-peiffer-ctxt}
  $k_1 k_2$ times as expected to reorder the $R_{1,j_1}$'s after the
  $R_{2,j_2}$'s for $1 \le j_i \le k_i$ for $i \in \set{1,2}$,
  we obtain that
  \[
    (F_1 \comp_1  \phi_2) \comp_2 (\phi_1' \comp_1  F_2) = (\phi_1 \comp_1
    F_2) \comp_2 (F_1 \comp_1 \phi_2').
    \qedhere
  \]
\end{proof}
\noindent We now prove the various conditions on~$X_{-,-}$. First, a technical lemma:
\begin{propapp}
  \label{prop:compat-Y-one-cells}
  Given $f\in \freecat\P_1$, $\phi,\psi \in \freecat\P_2$ with $f,\phi,\psi$
  $0$-composable, there is a canonical isomorphism $(f \comp_0 \phi) \shuffle
  \psi \cong \phi \shuffle \psi$ and for all $p \in \freecat{(\phi \shuffle
    \psi)}_1$, we have
  \[
    \winterp{p}_{f \comp_0 \phi,\psi} = f \comp_0 \winterp{p}_{\phi,\psi}
  \]
  Similarly, given $\phi, \psi \in \freecat\P_2$ and $h\in \freecat\P_1$ with
  $\phi,\psi,h$ $0$-composable, we have a canonical isomorphism $\phi \shuffle
  (\psi \comp_0 h) \cong \phi \shuffle \psi$ and for all $p \in \freecat{(\phi
    \shuffle (\psi \comp_0 h))}_1$, we have
  \[
    \winterp{p}_{\phi,\psi \comp_0 h} = \winterp{p}_{\phi,\psi} \comp_0 h.
  \]
  Finally, given $\phi,\psi \in \freecat\P_2$ and $g \in \freecat\P_1$ with
  $\phi,g,\psi$ $0$-composable, we have a canonical isomorphism $(\phi \comp_0
  g) \shuffle \psi \cong \phi \shuffle (g \comp_0 \psi)$ and for all $p \in
  \freecat {((\phi \comp_0 g) \shuffle \psi)}_1$, we have
  \[
    \winterp{p}_{\phi \comp_0 g,\psi} = \winterp{p}_{\phi,g \comp_0 \psi}.
  \]
\end{propapp}
\begin{proof}
  Let $f\in \freecat\P_1$, $\phi,\psi \in \freecat\P_2$ with $f,\phi,\psi$
  $0$-composable and let $r,s \ge 0$, $f_i,g_i \in \freecat\P_1$, $\alpha_i \in
  \P_2$ for $i \in \set{1,\ldots,r}$ and $f'_j,g'_j \in \freecat\P_1$,
  $\alpha'_j \in \P_2$ for $j \in \set{1,\ldots,s}$ such that
  \begin{align*}
    \phi &= (f_1 \comp_0 \alpha_1 \comp_0 g_1) \comp_1 \cdots \comp_1 (f_r
    \comp_0 \alpha_r \comp_0 g_r)
    \\
    \shortintertext{and}
    \psi &= (f'_1 \comp_0 \alpha'_1 \comp_0 g'_1) \comp_1 \cdots \comp_1 (f'_r
    \comp_0 \alpha'_r \comp_0 g'_r).
  \end{align*}
  By contemplating the definitions of $(f \comp_0 \phi) \shuffle \psi$ and $\phi
  \shuffle \psi$, we deduce a canonical isomorphism between them. Under this
  isomorphism, we easily verify that we have $\winterp{w}_{f \comp_0 \phi,\psi}
  = f \comp_0 \winterp{w}_{\phi,\psi}$ for $w \in ((f \comp_0 \phi) \shuffle
  \psi)_0$. Now, given $u \letter l_i \letter r_j v \in ((f\comp_0 \phi)
  \shuffle \psi)_0$, we have
  \begin{align*}
  \winterp{\wtrans_{u,v}}_{f \comp_0 \phi,\psi} &= \winterp{u}_{f \comp_0
    \phi,\psi} \comp_1 (f \comp_0 f_i \comp_0 X_{\alpha_i,g_i \comp_0
    f_j,\alpha'_j} \comp_0 g_j) \comp_1 \winterp{v}_{f \comp_0 \phi,\psi} \\
    &= f \comp_0 (\winterp{u}_{
    \phi,\psi} \comp_1 (f_i \comp_0 X_{\alpha_i,g_i \comp_0
      f_j,\alpha'_j} \comp_0 g_j) \comp_1 \winterp{v}_{\phi,\psi} ) \\
    &= f \comp_0 \winterp{\wtrans_{u,v}}_{\phi,\psi}.
  \end{align*}
  By functoriality of $\winterp{-}_{f \comp_0 \phi,\psi}$ and
  $\winterp{-}_{\phi,\psi}$, we deduce that, for all $p \in \freecat{(f \comp_0
    \phi) \shuffle \psi}$,
  \[
    \winterp{p}_{f \comp_0 \phi,\psi} = f \comp_0 \winterp{p}_{\phi,\psi}.
  \]
  The two other properties are shown similarly.
\end{proof}
\noindent We can now conclude the most simple properties of~$X_{-,-}$:
\begin{lemapp}
  \label{lem:compat-X-one-cells}
  Given $\phi \co f \To f'\in \prespcat\P_2$ and $\psi\co g \To g' \in
  \prespcat\P_2$, we have the following equalities in~$\prespcat\P_3$:
  \begin{enumerate}[label=(\roman*),ref=(\roman*)]
  \item \label{lem:compat-X-one-cells:units} $X_{\unit f,\psi} = \unit {f \comp_0 \psi}$ and $X_{\phi, \unit g} =
    \unit{\phi \comp_0 g}$ when $\phi,\psi$ are $0$\composable,
    
  \item \label{lem:compat-X-one-cells:left} $X_{l \comp_0 \phi,\psi} = l \comp_0 X_{\phi,\psi}$ for $l \in
    \freecat\P_1$ such that $l,\phi,\psi$ are $0$\composable,
  \item \label{lem:compat-X-one-cells:middle} $X_{\phi \comp_0 m,\psi} = X_{\phi, m \comp_0 \psi}$ for $m \in
    \freecat\P_1$ such that $\phi,m,\psi$ are $0$\composable,
  \item \label{lem:compat-X-one-cells:right} $X_{\phi,\psi \comp_0 r} =
    X_{\phi,\psi} \comp_0 r$ for $r \in \freecat\P_1$ such that $\phi,\psi,r$
    are $0$\composable.
  \end{enumerate}
\end{lemapp}
\begin{proof}
  \ref{lem:compat-X-one-cells:units} is clear, since both $\wtrans_{\unit
    f,\psi}$ and $\wtrans_{\phi,\unit g}$ are identity paths on the unique
  $0$-cells of~$\freecat{(\unit f\shuffle\psi)}$
  and~$\freecat{(\phi\shuffle\unit g)}$ respectively.
  \ref{lem:compat-X-one-cells:left} is a consequence of
  \Propr{compat-Y-one-cells}, since $\wtrans_{f\comp_0\phi,\psi}$ is sent to
  $\wtrans_{\phi,\psi}$ by the canonical isomorphism $(f \comp_0\phi) \shuffle
  \psi \cong \phi \shuffle \psi$. \ref{lem:compat-X-one-cells:middle}
  and~\ref{lem:compat-X-one-cells:right} follow similarly.
\end{proof}
The last required properties on~$X_{-,-}$ are more difficult to prove. In fact,
we need a proper coherence theorem showing that, for $0$\composable $\phi,\psi
\in \prespcat\P_2$, $X_{\phi,\psi} = \winterp{p}_{\phi,\psi}$ for all $p \in
\freecat{(\phi\shuffle\psi)}_1$ parallel to $\wtrans_{\phi,\psi}$. We
progressively introduce the necessary material to prove this fact below.

Given a
word $w \in (\phi \shuffle \psi)_0$, there is a function
\[
  \lindex_w\co \set{1,\ldots,\len\phi} \to \set{1,\ldots,\len{\phi} + \len{\psi}}
\]
defined such that, for $i \in \set{1,\ldots,\len{\phi}}$, if $w = w'\letter l_i
w''$, then $\lindex_w(i) = \len{w'}+1$. We have that the function~$\lindex$
characterizes the existence of path in $\freecat{(\phi\shuffle\psi)}$, as in:
\begin{lemapp}
  \label{lem:s-path-criterion}
  Given $0$\composable $\phi,\psi \in \freecat\P_2$ and $w,w' \in (\phi \shuffle
  \psi)_0$, there is a path
  \[
    p\co w \to w' \in \freecat{(\phi \shuffle \psi)}_1
  \]
  if and only if $\lindex_w(i) \le \lindex_{w'}(i)$ for $1 \le i \le
  \len{\phi}$.
\end{lemapp}
\begin{proof}
  Given $\wtrans_{u,v}\co u \letter l_r \letter r_s v \to u \letter r_s \letter
  l_r v \in (\phi \shuffle \psi)_1$, it is clear that $\lindex_{u \letter l_r
    \letter r_s v}(i) \le \lindex_{u \letter r_s \letter l_r v}(i)$ for all $1
  \le i \le \len\phi$, so that, given a path $p\co w \to w' \in \freecat{(\phi
    \shuffle \psi)}_1$, by induction on~$p$, we have $\lindex_w(i) \le
  \lindex_{w'}(i)$ for $1 \le i \le \len{\phi}$.

  Conversely, given $w,w' \in (\phi \shuffle \psi)_0$ such that $\lindex_w \le
  \lindex_{w'}$, we show by induction on~$N(w,w')$ defined by
  \[
    N(w,w') = \sum_{1 \le i \le
      \len\phi}\lindex_{w'}(i) - \lindex_w(i)
  \]
  that there is a path $p\co w \to w' \in \freecat{(\phi \shuffle \psi)}_1$. If
  $N(w,w') = 0$, then $w = w'$ and $\id w\co w \to w'$ is a suitable path.
  Otherwise, let $i_{\max}$ be the largest $i \le \len{\phi}$ such that
  $\lindex_{w'}(i) > \lindex_w(i)$. Then, either $i_{\max} = \len{\phi}$ or
  $\lindex_w(i_{\max}) + 1 < \lindex_w(i_{\max} + 1)$ since 
  \begin{align*}
    \lindex_w(i_{\max}) + 1 &\le \lindex_{w'}(i_{\max}) \\
    & < \lindex_{w'}(i_{\max} + 1) \\
                            &= \lindex_w(i_{\max} + 1)
  \end{align*}
  So we can write $w = u \letter l_{i_{\max}} \letter r_j v$ for some words
  $u,v$ and $j \in \set{1,\ldots,\len{\psi}}$. We have a path generator
  $\wtrans_{u,v}\co w \to \tilde w \in (\phi \shuffle \psi)_1$ where $\tilde w =
  u \letter r_j \letter l_{i_{\max}} v$. Then,
  \[
    \lindex_{\tilde w}(i) =
    \begin{cases}
      \lindex_w(i) & \text{if $i \neq i_{\max}$} \\
      \lindex_w(i_{\max}) + 1 & \text{if $i = i_{\max}$}
    \end{cases}
  \]
  so $\lindex{\tilde w} \le \lindex {w'}$ and $N(\tilde w,w') < N(w,w')$. Thus,
  by induction, we get 
  \[
    p'\co \tilde w \to w' \in \freecat{(\phi \shuffle
      \psi)}_1
  \]
  and we build a path $\wtrans_{u,v} \comp_0 p'\co w \to w' \in
  \freecat{(\phi \shuffle \psi)}_1$ as wanted.
\end{proof}
\noindent Given $0$\composable $\phi,\psi \in \freecat\P_2$ and $w = w_1\ldots w_{\len\phi+\len\psi} \in (\phi \shuffle \psi)_0$, we define
$\Inv(w)$ as
\begin{multline*}
  \Inv(w) = \setsize{\set{ (i,j) \mid 1 \le i < j \le \len\phi + \len\psi \text{ and } w_i = \letter r_{i'} \text{ and }
    w_j = \letter l_{j'} \\
    \text{ for some $i' \in \set{1,\ldots,\len\psi}$ and $j'
      \in \set{1,\ldots,\len\phi}$}}}.
\end{multline*}
\noindent We have that $\Inv$ characterizes the length of the paths
of~$\freecat{(\phi\shuffle\psi)}$, as in:
\begin{lemapp}
  \label{lem:s-path-length}
  Given $0$\composable $\phi,\psi \in \freecat\P_2$ and $p\co w \to w' \in \freecat{(\phi
  \shuffle \psi)}_1$, we have 
  \[
    \len{p} = \Inv(w') - \Inv(w). 
  \]
  In particular, given $w,w' \in (\phi \shuffle \psi)_0$, all the paths $p\co w
  \to w' \in \freecat{(\phi \shuffle \psi)}_1$ have the same length.
\end{lemapp}
\begin{proof}
  We show this by induction on the length of~$p$. If $p = \unit w$, then the
  conclusion holds. Otherwise, $p = \wtrans_{u,u'}\comp_0 r$ for some $u,u' \in
  \Sigma_{\phi,\psi}$ and $r\co \tilde w \to w' \in
  \freecat{(\phi\shuffle\psi)}_1$. Then, by induction hypothesis, $\len{r} =
  \Inv(w') - \Inv(\tilde w)$. Note that, by the definition of~$\wtrans_{u,u'}$,
  $w = u \letter l_i \letter r_j u'$ and $\tilde w = u \letter r_j\letter l_i r$
  for some $i \in \set{1,\ldots,\len\phi}$ and $j \in \set{1,\ldots,\len\psi}$.
  Hence,
  \[
    \len{p} = \len{r} + 1 = \Inv(w') - \Inv(\tilde w) + \Inv(\tilde w) -
    \Inv(w) = \Inv(w') - \Inv(w).\qedhere
  \]
\end{proof}
\noindent Given $0$\composable $\phi,\psi \in\freecat\P_2$, we now prove the
following coherence property for~$\freecat{(\phi\shuffle\psi)}$:
\begin{lemapp}
  \label{lem:s-relation}
  Let $\approx$ be a congruence on $\freecat{(\phi\shuffle\psi)}$. Suppose that, for all
  words $u_1,u_2,u_3 \in \Sigma_{\phi,\psi}$, $i,i' \in \set{1,\ldots,\len\phi}$
  and $j,j' \in \set{1,\ldots,\len\psi}$ such that $u_1 \letter l_{i}\letter
  r_{j} u_2 \letter l_{i'}\letter r_{j'} u_3 \in (\phi \shuffle \psi)_0$, we
  have
  \[
    \begin{tikzcd}[column sep={5em,between origins},cramped]
      & u_1 \letter l_{i}\letter r_{j} u_2 \letter l_{i'}\letter r_{j'} u_3
      \arrow[dl,pos=0.6,"\wtrans_{u_1,u_2 \letter l_{i'}\letter r_{j'} u_3}"']
      \arrow[dr,pos=0.6,"\wtrans_{u_1 \letter l_{i}\letter r_{j} u_2,u_3}"] & \\
      \makebox[6ex][c]{$u_1 {\letter r_{j}\letter l_{i}} u_2 \letter l_{i'}\letter r_{j'}
        u_3$}
      \arrow[dr,pos=0.3,"\wtrans_{u_1 \letter r_{j}\letter l_{i} u_2,u_3}"']
      & \approx & 
      \makebox[6ex][c]{$u_1 \letter l_{i}\letter r_{j} u_2\letter
        r_{j'}\letter l_{i'} u_3$}
      \arrow[dl,pos=0.3,"\wtrans_{u_1,u_2\letter r_{j'}\letter l_{i'} u_3}"]
      \\
      & u_1 \letter r_{j}\letter l_{i} u_2 \letter r_{j'}\letter l_{i'} u_3 &
    \end{tikzcd}
  \]%
  then, for all $p_1,p_2\co v \to w \in \freecat{(\phi\shuffle\psi)}_1$, we
  have $p_1 \approx p_2$.
\end{lemapp}
\begin{proof}
  We prove
  this by induction on~$\len{p_1}$. By \Lemr{s-path-length}, we have $\len{p_1}
  = \len{p_2}$. In particular, if $p_1 = \unit v$, then $p_2 = \unit v$.
  Otherwise, $p_i = q_i \comp_0 r_i$ with $q_i\co v \to v_i$ and $r_i\co v_i \to
  w$ and $\len{q_i} = 1$ for $i \in \set{1,2}$. If $q_1 = q_2$, then we conclude
  with the induction hypothesis on $r_1$ and $r_2$. Otherwise, up to symmetry,
  we have $q_1 = \wtrans_{u_1,u_2 \letter l_{i'}\letter r_{j'} u_3}$ and
  $q_2=\wtrans_{u_1 \letter{l}_i\letter{r}_j u_2, u_3}$ for some $u_1,u_2,u_3
  \in \freecat\Sigma_{\phi,\psi}$, $i,i' \in \set{1,\ldots,\len\phi}$ and
  $j,j' \in \set{1,\ldots,\len\psi}$. Let
  \begin{align*}
    q_1' &= \wtrans_{u_1\letter r_j \letter l_i u_2,u_3}, &
    q_2' &=
    \wtrans_{u_1, u_2\letter r_{j'}\letter l_{i'} u_3}, &
    v' &=u_1\letter
    r_{j}\letter l_i u_2 \letter r_{j'}\letter l_{i'}u_3.
  \end{align*}
  Since we have a path $v
  \xto{q_1} v_1 \xto{r_1} w$, by \Lemr{s-path-criterion}, we have $\lindex_v(s)
  \le \lindex_w(s)$ for $s \in \set{1,\ldots,\len\phi}$. Moreover,
  \[
    \lindex_v(i) < \lindex_{v_1}(i) \le \lindex_w(i) \qtand 
    \lindex_v(i') < \lindex_{v_2}(i') \le \lindex_w(i').
  \]
  Also, for $s \in \set{1,\ldots,\len\phi}$,
  \[
    \lindex_{v'}(s) =
    \begin{cases}
      \lindex_v(s) + 1 & \text{if $s \in \set{i,i'}$,} \\
      \lindex_v(s) & \text{otherwise.}
    \end{cases}
  \]
  From the preceding properties, we deduce that $\lindex_{v'}(s) \le
  \lindex_w(s)$ for $s \in \set{1,\ldots,\len\phi}$. Thus, by
  \Lemr{s-path-criterion}, there is a path $r'\co v' \to w\in \freecat{(\phi\shuffle\psi)}_1$ as in
  \[
    \begin{tikzcd}[column sep={4em,between origins}]
      & v_1
      \ar[dr,"{q_1'}"{description}]
      \ar[drrr,"r_1"]
      & &
      \\
      v 
      \ar[ur,"q_1"]
      \ar[dr,"q_2"'] 
      &
      &
      {v'}
      \ar[rr,"{r'}",pos=0.3]
      & & w
      \\
      & v_2 
      \ar[ur,"{q_2'}"'{description}]
      \ar[urrr,"{r_2}"']
    \end{tikzcd}
  \]
  % \[
  %   \begin{tikzcd}
  %     & v \arrow[dl,"q_1"'] \arrow[dr,"q_2"]& \\
  %     v_1 \arrow[dr,"{q_1'}"] \arrow[dddr, "r_1"'] & & v_2 \arrow[dl, "{q_2'}"'] \ar[lddd,"r_2"]\\
  %     & {v'} \ar[dd,"{r'}",pos=0.3] & \\
  %     & & & \\
  %     & w &
  %   \end{tikzcd}
  % \]
  Since $\len{r_i} = \len{p_i} - 1$ for $i \in \set{1,2}$, by induction
  hypothesis, we have $r_i \approx q_i' \comp_0 r'$ for $i \in \set{1,2}$, which
  can be extended to $q_i \comp_0 r_i \approx q_i \comp_0 q_i' \comp_0 r'$,
  since $\approx$ is a congruence. By hypothesis, we have $q_1 \comp_0 q_1'
  \approx q_2 \comp_0 q_2'$, which can be extended to $q_1 \comp_0 q_1' \comp_0
  r' \approx q_2 \comp_0 q_2' \comp_0 r'$. By transitivity of~$\approx$, we get
  that $q_1 \comp_0 r_1 \approx q_2 \comp_0 r_2$, that is, $p_1 \approx p_2$.
\end{proof}
\noindent We then apply this coherence property to~$\winterp{-}_{-,-}$ and get
that ``all exchange methods are equivalent'', as in:
\begin{propapp}
  \label{prop:interchange-coherence}
  Given $0$\composable $\phi,\psi \in \prespcat\P_2$, for all $p_1,p_2\co u \to v \in
  \freecat{(\phi\shuffle\psi)}_1$, we have, in $\prespcat\P_3$,
  \[
    \winterp{p_1}_{\phi,\psi} = \winterp{p_2}_{\phi,\psi}.
  \]
\end{propapp}
\begin{proof}
  By \Lemr{prespcat-peiffer}, for all words $u_1,u_2,u_3 \in
  \Sigma_{\phi,\psi}$, $i,i' \in \set{1,\ldots,\len\phi}$ and $j,j' \in
  \set{1,\ldots,\len\psi}$ such that $u_1 \letter l_{i}\letter r_{j} u_2 \letter
  l_{i'}\letter r_{j'} u_3 \in (\phi \shuffle \psi)_0$, we have
  % \[
  %   \begin{tikzcd}[column sep=1ex]
  %     & \makebox[5ex][c]{$Y_{\phi,\psi}(l \wcomp \letter{a b} \wcomp m \wcomp
  %       \letter{a b} \wcomp r)$} \arrow[dl,pos=0.7,"Y_{\phi,\psi}(l \wcomp \transpo \wcomp m \wcomp
  %     \letter{a
  %     b} \wcomp r)"'] \arrow[dr,pos=0.7,"Y_{\phi,\psi}(l \wcomp \letter{a b} \wcomp m \wcomp \transpo \wcomp r)"] & \\
  %     Y_{\phi,\psi}(l \wcomp \letter{b a} \wcomp m \wcomp \letter{a b} \wcomp r)
  %     \arrow[dr,pos=0.3,"Y_{\phi,\psi}(l \wcomp \letter{b a} \wcomp m \wcomp \transpo \wcomp r)"'] & = &
  %     Y_{\phi,\psi}(l \wcomp \letter{a b} \wcomp m \wcomp \letter{b a} \wcomp r)
  %     \arrow[dl,pos=0.3,"Y_{\phi,\psi}(l \wcomp \transpo \wcomp m \wcomp \letter{b a} \wcomp r)"]
  %     \\
  %     & \makebox[5ex][c]{$Y_{\phi,\psi}(l \wcomp \letter{b a} \wcomp m \wcomp \letter{b a} \wcomp r)$}
  %     &
  %   \end{tikzcd}.
  % \]
  \[
    \begin{tikzcd}[column sep={6em,between origins},cramped]
      & \winterp{u_1 \letter l_{i}\letter r_{j} u_2 \letter l_{i'}\letter r_{j'} u_3}_{\phi,\psi}
      \arrow[dl,pos=0.6,"\winterp{\wtrans_{u_1,u_2 \letter l_{i'}\letter r_{j'} u_3}}_{\phi,\psi}"']
      \arrow[dr,pos=0.6,"\winterp{\wtrans_{u_1 \letter l_{i}\letter r_{j} u_2,u_3}}_{\phi,\psi}"] & \\
      \winterp{u_1 {\letter r_{j}\letter l_{i}} u_2 \letter l_{i'}\letter r_{j'}
        u_3}_{\phi,\psi}
      \arrow[dr,pos=0.3,"\winterp{\wtrans_{u_1 \letter r_{j}\letter l_{i} u_2,u_3}}_{\phi,\psi}"']
      & = & 
      \winterp{u_1 \letter l_{i}\letter r_{j} u_2\letter
        r_{j'}\letter l_{i'} u_3}_{\phi,\psi}
      \arrow[dl,pos=0.3,"\winterp{\wtrans_{u_1,u_2\letter r_{j'}\letter l_{i'} u_3}}_{\phi,\psi}"]
      \\
      & \winterp{u_1 \letter r_{j}\letter l_{i} u_2 \letter r_{j'}\letter l_{i'} u_3}_{\phi,\psi} &
    \end{tikzcd}
  \]
  Moreover, the relation $\approx$ defined on parallel $p_1,p_2 \in
  \freecat{(\phi \shuffle \psi)}_1$ by $p_1 \approx p_2$ when
  $\winterp{p_1}_{\phi,\psi} = \winterp{p_2}_{\phi,\psi}$ is clearly a
  congruence. Hence, by \Lemr{s-relation}, we have that
  $\winterp{p_1}_{\phi,\psi} = \winterp{p_2}_{\phi,\psi}$ for all parallel
  $p_1,p_2 \in \freecat{(\phi\shuffle\psi)}_1$.
\end{proof}
\noindent The preceding property says in particular that $X_{\phi,\psi} =
\winterp{p}_{\phi,\psi}$ for all $0$\composable $\phi,\psi \in \freecat\P_2$
and paths $p\in \freecat{(\phi\shuffle\psi)}_1$ parallel
to~$\wtrans_{\phi,\psi}$. 

Let $\phi,\psi \in \freecat\P_2$ be $0$\composable $2$-cells, and $\phi',\psi'
\in \freecat\P_2$ be $0$\composable $2$-cells such that $\phi,\phi'$ and
$\psi,\psi'$ are $1$\composable. To obtain the last required properties
on~$X_{-,-}$, we need to relate $\phi\shuffle\psi$ and $\phi'\shuffle\psi'$ to
$(\phi\comp_1\phi')\shuffle(\psi\comp_1\psi')$. Given $w \in (\phi \shuffle
\psi)_0$, there is a functor
\[
  w\wcomp(-)\co \freecat{(\phi'\shuffle\psi')} \to
  \freecat{((\phi\comp_1\phi')\shuffle(\psi\comp_1\psi'))}
\]
which is uniquely defined by the mappings
\begin{align*}
  u & \mapsto w\wshiftup (u) \\
  \wtrans_{u_1,u_2} & \mapsto X_{w\wshiftup(u_1),\wshiftup(u_2)}
\end{align*}
for $u \in (\phi' \shuffle \psi')_0$ and $\wtrans_{u_1,u_2} \in (\phi' \shuffle
\psi')_1$ and where, for $v = v_1\ldots v_k \in \freecat\Sigma_{\phi',\psi'}$,
$\wshiftup(v) \in \freecat\Sigma_{\phi\comp_1\phi',\psi\comp_1\psi'}$ is defined
by
\[
  \wshiftup(v)_r =
  \begin{cases}
    \letter l_{\len{\phi} + i} & \text{if $v_r = \letter l_i$ for some $i \in \set{1,\ldots,\len{\phi'}}$} \\
    \letter r_{\len{\psi} + j} & \text{if $v_r = \letter r_j$ for some $j \in \set{1,\ldots,\len{\psi'}}$}
  \end{cases}
\]
for $r \in \set{1,\ldots,k}$. Similarly, given $w \in (\phi'\shuffle\psi')_0$,
there is a functor
\[
  (-)\wcomp w\co \freecat{(\phi\shuffle\psi)} \to \freecat{((\phi\comp_1\phi')\shuffle(\psi\comp_1\psi'))}
\]
which is uniquely defined by the mappings
\begin{align*}
  u & \mapsto u\wshiftup(w) \\
  \wtrans_{u_1,u_2} & \mapsto \wtrans_{u_1,u_2\wshiftup(w)}
\end{align*}
for $u \in (\phi\shuffle\psi)_0$ and $\wtrans_{u_1,u_2} \in
(\phi\shuffle\psi)_1$ and where $\wshiftup(-)$ is defined as above.

\bigskip\noindent The functors $w\wcomp(-)$ and $(-)\wcomp w$ satisfy the following compatibility property:
\begin{lemapp}
  \label{lem:compat-word-waction}
  Let $\phi,\psi \in \freecat\P_2$ be $0$\composable $2$-cells, and
  $\phi',\psi' \in \freecat\P_2$ be $0$\composable $2$-cells such that
  $\phi,\phi'$ and $\psi,\psi'$ are $1$\composable. Given $w \in
  (\phi\shuffle\psi)_0$, we have the following equalities in~$\freecat\P_3$:
  \begin{enumerate}[label=(\roman*),ref=(\roman*)]
  \item \label{lem:compat-word-waction:0-cells} $\winterp{w \wcomp (u)}_{\phi\comp_1\phi',\psi\comp_1\psi'} = \winterp{w}_{\phi,\psi}
    \comp_1 \winterp{u}_{\phi',\psi'}$ for $u \in (\phi'\shuffle\psi')_0$,
  \item \label{lem:compat-word-waction:1-cells} $\winterp{w \wcomp(p)}_{\phi\comp_1\phi',\psi\comp_1\psi'} =
    \winterp{w}_{\phi,\psi} \comp_1 \winterp{p}_{\phi',\psi'}$ for $p \in \freecat{(\phi'\shuffle\psi')}_1$.
  \end{enumerate}
  Similarly, given $w \in (\phi'\shuffle\psi')_0$, we have:
  \begin{enumerate}[label=(\roman*),ref=(\roman*)]
  \item $\winterp{(u)\wcomp w}_{\phi\comp_1\phi',\psi\comp_1\psi'} =
    \winterp{u}_{\phi,\psi} \comp_1 \winterp{w}_{\phi',\psi'}$ for $u \in (\phi\shuffle\psi)_0$,
    
  \item $\winterp{(p) \wcomp w}_{\phi\comp_1\phi',\psi\comp_1\psi'} =
    \winterp{p}_{\phi,\psi} \comp_1 \winterp{w}_{\phi',\psi'}$ for $p \in \freecat{(\phi\shuffle\psi)}_1$.
  \end{enumerate}
\end{lemapp}
\begin{proof}
  \newcommand\bigindices{{\phi\comp_1\phi',\psi\comp_1\psi'}}
  We only prove the first part, since the second part is similar.
  We start by~\ref{lem:compat-word-waction:0-cells}. We have
  $\winterp{w\wcomp(u)}_{\phi\comp_1\phi',\psi\comp_1\psi'} =
  \winterp{w\wshiftup(u)}^{1,1}_{\phi\comp_1\phi',\psi\comp_1\psi'}$. By a simple
  induction on $w$, we obtain
  \begin{equation*}
    \winterp{w\wshiftup(u)}^{1,1}_{\phi\comp_1\phi',\psi\comp_1\psi'} =
    \winterp{w}^{1,1}_{\phi\comp_1\phi',\psi\comp_1\psi'} \comp_1 \winterp{\wshiftup(u)}^{\len{\phi},\len{\psi}}_{\phi\comp_1\phi',\psi\comp_1\psi'}
  \end{equation*}
  and, by other simple inductions on~$w$ and~$u$, we get
  \begin{align*}
    \winterp{w}^{1,1}_{\phi\comp_1\phi',\psi\comp_1\psi'}
    &=
    \winterp{w}^{1,1}_{\phi,\psi} = \winterp{w}_{\phi,\psi}
      &
    \winterp{\wshiftup(u)}^{\len{\phi},\len{\psi}}_{\phi\comp_1\phi',\psi\comp_1\psi'}
    &=
    \winterp{u}^{1,1}_{\phi',\psi'} = \winterp{u}_{\phi,\psi}
  \end{align*}
  so that \ref{lem:compat-word-waction:0-cells} holds.

  For~\ref{lem:compat-word-waction:1-cells}, by induction on~$p$, it is
  sufficient to prove the equality for~$p = \wtrans_{u_1,u_2} \in
  (\phi\shuffle\psi)_1$. Let $m = \len{\phi}$, $n = \len{\psi}$, and
  \[
    (e_1 \comp_0 \alpha_1 \comp_0 f_1) \comp_1 \cdots
    \comp_1 (e_{m} \comp_0 \alpha_{m} \comp_0 f_m) \qquad
    (g_1 \comp_0 \beta_1 \comp_0 h_1) \comp_1 \cdots
    \comp_1 (g_{m} \comp_0 \beta_{m} \comp_0 h_m)
  \]
  be the unique decomposition of $\phi$ and $\psi$ respectively, for some
  $e_i,f_i,g_j,h_j \in \freecat\P_1$ and $\alpha_i,\beta_j \in \P_2$ for $i \in
  \set{1,\ldots,m}$ and $j\in \set{1,\ldots,n}$. We then have
  \begin{align*}
    \winterp{w\wcomp(\wtrans_{u_1,u_2})}_{\phi\comp_1\phi',\psi\comp_1\psi'}
    &= \winterp{\wtrans_{w\wshiftup(u_1),\wshiftup(u_2)}}_{\phi\comp_1\phi',\psi\comp_1\psi'} \\
    &= \winterp{w\wshiftup(u_1)}^{1,1}_{\phi\comp_1\phi',\psi\comp_1\psi'}\comp_1(e_i \comp_0X_{\alpha_i,f_i\comp_0g_j,\beta_j}\comp_0h_j) \comp_1\winterp{\wshiftup(u_2)}^{k_l,k_r}_{\phi\comp_1\phi',\psi\comp_1\psi'}
  \end{align*}
  where $i,j$ are such that $u_1\letter l_i\letter r_ju_2 \in
  (\phi'\shuffle\psi')_0$ and
  \begin{align*}
    k_l &= \len\phi + i + 1 & k_r &= \len\psi + j + 1.
  \end{align*}
  By simple inductions, we obtain
  \begin{align*}
    \winterp{w\wshiftup(u_1)}^{1,1}_\bigindices
    &=\winterp{w}^{1,1}_\bigindices\comp_1\winterp{\wshiftup(u_1)}^{\len{\phi},\len{\psi}}_\bigindices \\
    &=\winterp{w}^{1,1}_{\phi,\psi}\comp_1\winterp{u_1}^{1,1}_{\phi',\psi'} \\
    &=\winterp{w}_{\phi,\psi}\comp_1\winterp{u_1}^{1,1}_{\phi',\psi'} \\
    \shortintertext{and}
    \winterp{\wshiftup(u_2)}^{k_l,k_r}_{\phi\comp_1\phi',\psi\comp_1\psi'}
    &= \winterp{u_2}^{i+1,j+1}_{\phi',\psi'}
  \end{align*}
  so that
  \begin{align*}
    \winterp{w\wcomp(\wtrans_{u_1,u_2})}_{\phi\comp_1\phi',\psi\comp_1\psi'}
    &=\winterp{w}_{\phi,\psi}\comp_1\winterp{u_1}^{1,1}_{\phi',\psi'}
      \comp_1(e_i \comp_0X_{\alpha_i,f_i\comp_0g_j,\beta_j}\comp_0h_j)
      \comp_1\winterp{u_2}^{i+1,j+1}_{\phi',\psi'} \\
    &= \winterp{w}_{\phi,\psi} \comp_1 \winterp{\wtrans_{u_1,u_2}}_{\phi',\psi'}.
      \qedhere
  \end{align*}
\end{proof}
\goodbreak\noindent We can now conclude the last required properties on~$X_{-,-}$:
\begin{lemapp}
  \label{lem:compat-X-one-comp}
  Given $1$-composable $\phi,\phi' \in \prespcat\P_2$, $1$-composable
  $\psi,\psi' \in \prespcat\P_2$ such that $\phi,\psi$ are $0$-composable, we
  have the following equalities in~$\prespcat\P_3$:
  \begin{align*}
    X_{\phi \comp_1 \phi',\psi} &= ((\phi \comp_0 \csrc_1(\psi))\comp_1 X_{\phi',\psi})\comp_2 (X_{\phi,\psi}\comp_1(\phi' \comp_0 \ctgt_1(\psi))) \\
    \shortintertext{and}
    X_{\phi,\psi \comp_1 \psi'} &= (X_{\phi,\psi} \comp_1 (\ctgt_1(\phi) \comp_0 \psi'))
    \comp_2 ((\csrc_1(\phi) \comp_0 \psi) \comp_1 X_{\phi,\psi'}).
  \end{align*}
\end{lemapp}
\begin{proof}
  \newcommand\phii{\phi\comp_1\phi'}%
  We only prove the first equality, since the second one is similar.
  By definition of~$X_{\phii,\psi}$, we have $X_{\phii,\psi} = \winterp{\wtrans_{\phii,\psi}}_{\phii,\psi}$.
  Moreover, by~\Propr{interchange-coherence},
  $\winterp{\wtrans_{\phii,\psi}}_{\phii,\psi} = \winterp{p}_{\phii,\psi}$
  in~$\prespcat\P_3$ for
  all path $p \in ((\phii)\shuffle\psi)_1$ parallel to~$\wtrans_{\phii,\psi}$.
  In particular,
  \[
    \winterp{\wtrans_{\phii,\psi}}_{\phii,\psi} = \winterp{
      (w\wcomp(\wtrans_{\phi',\psi})) \comp_0 ((\wtrans_{\phi,\psi})\wcomp w')}_{\phii,\psi}
  \]
  where
  {\abovedisplayskip=0pt%
    \begin{align*}
    w &= \letter l_1\ldots\letter l_{\len{\phi}} & w' &= \letter l_1\ldots\letter l_{\len{\phi'}}
  \end{align*}}%
  are the only $0$-cells of $\phi' \shuffle \unit{\csrc(\phi)}$ and $\phi
  \shuffle \unit{\ctgt(\psi)}$ respectively. Thus,
  \vspace{\abovedisplayskip}%
  \par\noindent{%
    \abovedisplayshortskip=0pt%
    \abovedisplayskip=0pt%
    \belowdisplayskip=0pt%
    \belowdisplayshortskip=0pt%
    \newlength\rightwidth
    \settowidth\rightwidth{(by definition of~$\winterp{-}_{-,-}$ and~$X_{-,-}$).}%
    \newdimen\mymargin
    \setlength{\mymargin}{\linewidth-\rightwidth}%
    \begin{align*}
    \winterp{\wtrans_{\phii,\psi}}_{\phii,\psi} &= \winterp{
      (w\wcomp(\wtrans_{\phi',\psi})) \comp_0 ((\wtrans_{\phi,\psi})\wcomp w')}_{\phii,\psi} \\
    &= \winterp{
      (w\wcomp(\wtrans_{\phi',\psi}))}_{\phii,\psi} \comp_2 \winterp{((\wtrans_{\phi,\psi})\wcomp w')}_{\phii,\psi} \\
    \shortintertext{\hspace*{\mymargin}(by functoriality of~$\winterp{-}_{\phii,\psi}$)}
    &= (\winterp{w}_{\phi,\unit{\csrc(\psi)}} \comp_1 
      \winterp{\wtrans_{\phi',\psi}}_{\phi',\psi})
      \comp_2
      (\winterp{\wtrans_{\phi,\psi}}_{\phi,\psi}\comp_1\winterp{w'}_{\phi',\unit{\ctgt(\psi)}}) \\
    \shortintertext{\hspace*{\mymargin}(by \Lemr{compat-word-waction})}
    &= ((\phi \comp_0 \csrc_1(\psi)) \comp_1 
      X_{\phi',\psi})
      \comp_2
      (X_{\phi,\psi}\comp_1(\phi'\comp_0 \ctgt_1(\psi))
  \shortintertext{\hspace*{\mymargin}(by definition of~$\winterp{-}_{-,-}$ and~$X_{-,-}$).}
\end{align*}}%
\vskip-\baselineskip\vskip-\jot\vskip\belowdisplayskip

\noindent Hence,
\[
  X_{\phii,\psi} =
  ((\phi \comp_0 \csrc_1(\psi)) \comp_1 
  X_{\phi',\psi})
  \comp_2
  (X_{\phi,\psi}\comp_1(\phi'\comp_0 \ctgt_1(\psi)).\qedhere
\]
\end{proof}
\noindent We now prove the compatibility between $3$-cells and interchangers. We
start by proving the compatibility with $3$\generators:
\begin{lemapp}
  \label{lem:prespcat-exch-gen}
  Given $A\co \phi \TO \phi'\co f \To f' \in \P_3$ and $\psi\co g \To g' \in
  \prespcat\P_2$ such that $A,\psi$ are $0$\composable, we have, in
  $\prespcat\P_3$,
  \[
    ((A \comp_0 g) \comp_1 (f' \comp_0 \psi)) \comp_2 X_{\phi',\psi} =
    X_{\phi,\psi} \comp_2 ((f \comp_0 \psi) \comp_1 (A \comp_0 g')).
  \]
  Similarly, given $\phi\co f\To f' \in \prespcat\P_2$ and $B\co \psi\TO\psi'\co
  g\To g'$ such that $\phi,B$ are $0$\composable, we have, in $\prespcat\P$,
  \[
    X_{\phi,\psi} \comp_2 ((g \comp_0 B) \comp_1 (\phi \comp_0 f')) =
    ((\phi \comp_0 g) \comp_1 (f \comp_0 B)) \comp_2 X_{\phi,\psi'}.
  \]
\end{lemapp}

\begin{proof}
  We only prove the first part of the property, since the other one is
  symmetric, and we do so by an induction on~$\len\psi$. If $\len\psi = 0$,
  $\psi$ is an identity and the result follows. Otherwise, $\psi = w \comp_1
  \tilde\psi$ where $w = (l \comp_0 \alpha \comp_0 r)$ with $l,r \in
  \prespcat\P_1$, $\alpha\co h \To h' \in \P_2$ and $\tilde\psi \in
  \prespcat\P_2$ with $\len{\tilde\psi} = \len\psi-1$. Let $\tilde g =
  \ctgt_1(w)$. By \Lemr{compat-X-one-comp}, we have
  \begin{align}
    \label{eq:x-phi-psi} X_{\phi,\psi} &= (X_{\phi,w} \comp_1 (f'\comp_0 \tilde\psi)) \comp_2 ((f \comp_0 w) \comp_1 X_{\phi,\tilde\psi})  \\
    \label{eq:x-phi-psip} X_{\phi',\psi} &= (X_{\phi',w} \comp_1 (f'\comp_0 \tilde\psi)) \comp_2 ((f \comp_0 w) \comp_1 X_{\phi',\tilde\psi}).
  \end{align}
  Also, by \Lemr{compat-X-one-cells}\ref{lem:compat-X-one-cells:right}, we have
  \begin{align}
    \label{eq:x-phip?-w} X_{\phi,w} &= X_{\phi,l \comp_0 \alpha} \comp_0 r & X_{\phi',w} &= X_{\phi',l \comp_0 \alpha} \comp_0 r
  \end{align}
  so that
  \begin{equation}
    \label{eq:A-X-w}
  \begin{aligned}[c]
    & ((A \comp_0 g) \comp_1 (f' \comp_0 w)) \comp_2 X_{\phi',w} \\
    =\; & \left[((A \comp_0 l \comp_0 h) \comp_1 (f' \comp_0 l \comp_0 \alpha)) \comp_2 X_{\phi',l \comp_0 \alpha}\right] \comp_0 r \\
    =\; & \left[X_{\phi,l \comp_0 \alpha} \comp_2 ((f \comp_0 l \comp_0 \alpha) \comp_1 (A \comp_0 l \comp_0 h'))\right] \comp_0 r \\
    & \text{\hspace*{15em}(by interchange naturality generator)} \\
    =\; & X_{\phi,w} \comp_2 ((f \comp_0 w) \comp_1 (A \comp_0 g')).
  \end{aligned}
  \end{equation}
  Thus,
  \begin{align*}
    & ((A \comp_0 g) \comp_1 (f' \comp_0 \psi)) \comp_2
    X_{\phi',\psi} \\
    =\; & ((A \comp_0 g) \comp_1 (f' \comp_0 w) \comp_1 (f' \comp_0 \tilde\psi)) \\
    & \hspace*{1em}\comp_2 (X_{\phi',w} \comp_1 (f'\comp_0 \tilde\psi)) \comp_2 ((f \comp_0 w) \comp_1 X_{\phi',\tilde\psi}) && \text{(by \eqref{eq:x-phi-psip})}\\
    =\; & \left[ ( ((A \comp_0 g) \comp_1 (f' \comp_0 w) ) \comp_2 X_{\phi',w}  ) \comp_1 (f'\comp_0 \tilde\psi)\right] \\
    & \hspace*{1em}\comp_2 ((f \comp_0 w) \comp_1 X_{\phi',\tilde\psi}) \\
    =\; & \left[ ( X_{\phi,w} \comp_2 ((f \comp_0 w) \comp_1 (A \comp_0 \tilde g) )  ) \comp_1 (f'\comp_0 \tilde\psi)\right] \\
    & \hspace*{1em}\comp_2 ((f \comp_0 w) \comp_1 X_{\phi',\tilde\psi}) && \text{(by \eqref{eq:A-X-w})}\\
    =\; &  ( X_{\phi,w} \comp_1 (f'\comp_0 \tilde\psi)) \\
    & \hspace*{1em}\comp_2 ((f \comp_0 w) \comp_1 (A \comp_0 \tilde g) \comp_1 (f'\comp_0 \tilde\psi))\comp_2 ((f \comp_0 w) \comp_1 X_{\phi',\tilde\psi}) \\
    =\; &  ( X_{\phi,w} \comp_1 (f'\comp_0 \tilde\psi)) \\
    & \hspace*{1em}\comp_2 \left[  (f \comp_0 w) \comp_1 (((A \comp_0 \tilde g) \comp_1 (f'\comp_0 \tilde\psi))\comp_2  X_{\phi',\tilde\psi}) \right]\\
    =\; &  ( X_{\phi,w} \comp_1 (f'\comp_0 \tilde\psi)) \\
    & \hspace*{1em}\comp_2 \left[  (f \comp_0 w) \comp_1 (X_{\phi',\tilde\psi} \comp_2  ((f\comp_0 \tilde\psi) \comp_1 (A \comp_0 g'))) \right] && \text{(by induction)}\\
    =\; &  ( X_{\phi,w} \comp_1 (f'\comp_0 \tilde\psi)) \comp_2 ((f \comp_0 w) \comp_1 (X_{\phi',\tilde\psi})) \\
     & \hspace*{1em}\comp_2  ((f \comp_0 w) \comp_1 (f\comp_0 \tilde\psi) \comp_1 (A \comp_0 g')) \\
    =\; & X_{\phi,\psi} \comp_2 ((f \comp_0 \psi) \comp_1 (A \comp_0 g')) && \text{(by \eqref{eq:x-phi-psi})}. \qedhere
  \end{align*}
\end{proof}
\noindent Next, we prove the compatibility between interchangers and rewriting steps:
\begin{lemapp}
  \label{lem:prespcat-exch-ctxt}
  Given a rewriting step~$R\co \phi\TO\phi'\co f\To f' \in \freecat\P_3$ with $R
  = \lambda \comp_1 (l \comp_0 A \comp_0 r) \comp_1 \rho$ for some $l,r \in
  \freecat\P_1$, $\lambda,\rho \in \freecat\P_2$, $A\co \mu \TO \mu' \in \P_3$,
  and $\psi\co g\To g'\in \freecat\P_2$ such that $R,\psi$ are $0$\composable,
  we have, in $\prespcat\P_3$,
  \begin{equation}
    \label{eq:prespcat-exch-ctxt:goal}
    ((R \comp_0 g) \comp_1 (f' \comp_0 \psi)) \comp_2
    X_{\phi',\psi} = X_{\phi,\psi} \comp_2 ((f \comp_0 \psi) \comp_1 (R \comp_0 g')).
  \end{equation}
  Similarly, given $\phi \in \freecat\P_2$ and a rewriting step $S\co
  \psi\TO\psi'\co g\To g'\in \freecat\P_3$ with $S =
  \lambda \comp_1 (l \comp_0 B \comp_0 r) \comp_1 \rho$ for some $\lambda,\rho
  \in \freecat\P_2$, $l,r \in \freecat\P_1$, $B\co \nu \TO \nu' \in \P_3$ such that $\phi,S$
  are $0$\composable, we have, in $\prespcat\P_3$,
  \[
    X_{\phi,\psi} \comp_2 ((f \comp_0 B) \comp_1 (\phi \comp_0 g')) =
    ((\phi \comp_0 g) \comp_1 (f' \comp_0 B)) \comp_2 X_{\phi,\psi'}.
  \]
\end{lemapp}
\begin{proof}
  By symmetry, we only prove the first part. Let
  \begin{align*}
    \tilde \mu &= l \comp_0 \mu \comp_0 r
    & h &= \csrc_1(\mu)
    & \tilde h &= \csrc_1(\tilde \mu) \\
    \tilde \mu' &= l \comp_0 \mu' \comp_0 r
    & h' &= \ctgt_1(\mu')
    & \tilde h' &= \ctgt_1(\tilde \mu)
  \end{align*}
  We have
  \begin{align*}
    R \comp_0 g &= (\lambda \comp_0 g) \comp_1 (l \comp_0 A \comp_0 r \comp_0 g) \comp_1 (\rho \comp_0 g)
  \end{align*}
   and, by \Lemr{compat-X-one-comp},
   \begin{equation}
     \label{eq:prespcat-exch-ctxt:X-dec}
     \begin{aligned}
       % X_{\lambda \comp_1 \tilde \mu \comp_1 \rho,\psi} =\;
       X_{\phi,\psi} =\;
       & (((\lambda \comp_1 \tilde \mu) \comp_0 g) \comp_1 X_{\rho,\psi}) \\
       &\hspace*{1em}\comp_2 (((\lambda \comp_0 g) \comp_1 X_{\tilde \mu,\psi} \comp_1 (\rho \comp_0 g'))) \\
       &\hspace*{1em}\comp_2 ((X_{\lambda,\psi} \comp_1 ((\tilde \mu \comp_1 \rho) \comp_0 g')))
     \end{aligned}
   \end{equation}
   \begin{equation}
     \label{eq:prespcat-exch-ctxt:X-decp}
     \begin{aligned}
       % X_{\lambda \comp_1 \tilde \mu' \comp_1 \rho,\psi} =\;
       X_{\phi',\psi} =\;
       & (((\lambda \comp_1 \tilde \mu') \comp_0 g) \comp_1 X_{\rho,\psi}) \\
       &\hspace*{1em}\comp_2 (((\lambda \comp_0 g) \comp_1 X_{\tilde \mu',\psi} \comp_1 (\rho \comp_0 g'))) \\
       &\hspace*{1em}\comp_2 ((X_{\lambda,\psi} \comp_1 ((\tilde \mu' \comp_1 \rho) \comp_0 g'))).
     \end{aligned}
   \end{equation}
   We start the calculation of the left-hand side
   of~\eqref{eq:prespcat-exch-ctxt:goal}, using~\eqref{eq:prespcat-exch-ctxt:X-decp}. We get
  \begin{align*}
    & ((R \comp_0 g) \comp_1 (f' \comp_0 \psi)) \comp_2 (((\lambda \comp_1 \tilde \mu') \comp_0 g) \comp_1 X_{\rho,\psi}) \\
    =\;& (\lambda \comp_0 g) \\
    & \hspace*{1em}\comp_1 \Big[((l \comp_0 A \comp_0 r \comp_0 g) \comp_1 (\rho \comp_0 g) \comp_1 (f' \comp_0 \psi)) 
    \comp_2 ((\mu'\comp_0 g) \comp_1 X_{\rho,\psi} ) \Big] \\
    =\;&(\lambda \comp_0 g) \\
    & \hspace*{1em}\comp_1 \Big[((\mu \comp_0 g) \comp_1 X_{\rho,\psi} )
    \comp_2 ((l \comp_0 A \comp_0 r \comp_0 g) \comp_1 (\tilde h' \comp_0 \psi) \comp_1 (\rho \comp_0 g'))\Big] && \text{(by \Lemr{prespcat-peiffer})} \\
     =\;&((\lambda \comp_0 g) \comp_1 (\tilde \mu \comp_0 g) \comp_1 X_{\rho,\psi} ) \\
    &\hspace*{1em}\comp_2 ((\lambda \comp_0 g) \comp_1 (l \comp_0 A \comp_0 r \comp_0 g) \comp_1 (\tilde h' \comp_0 \psi) \comp_1 (\rho \comp_0 g')).
  \end{align*}
  Symmetrically, we do a step of calculation for the right-hand side
  of~\eqref{eq:prespcat-exch-ctxt:goal},
  using~\eqref{eq:prespcat-exch-ctxt:X-dec}. We get
  \begin{align*}
    & (X_{\lambda,\psi} \comp_1 ((\tilde \mu \comp_1 \rho) \comp_0 g')) \comp_2 ((f \comp_0 \psi) \comp_1 (R \comp_0 g')) \\
    =\; & ((\lambda \comp_0 g) \comp_1 (\tilde h \comp_0 \psi) \comp_1 (l \comp_0 A \comp_0 r \comp_0 g') \comp_1 (\rho \comp_0 g')) \\
    &\hspace*{1em}\comp_2 (X_{\lambda,\psi} \comp_1 (\tilde \mu' \comp_0 g') \comp_1 (\rho \comp_0 g') ).
  \end{align*}
  Finally, we do the last step of calculation between the left-hand side and the
  right-hand side of~\eqref{eq:prespcat-exch-ctxt:goal}. Note that
  \begin{align*}
    &
     ((l \comp_0 A \comp_0 r \comp_0 g) \comp_1 (\tilde h' \comp_0 \psi)) \comp_2 X_{\tilde \mu',\psi}
          \\
    =\; &
           l \comp_0 (((A \comp_0 r \comp_0 g) \comp_1 (h' \comp_0 r \comp_0 \psi)) \comp_2 X_{\mu' \comp_0 r,\psi}) && \text{(by \Lemr{compat-X-one-cells}\ref{lem:compat-X-one-cells:left})}
          \\
    =\; &
           l \comp_0 (((A \comp_0 r \comp_0 g) \comp_1 (h' \comp_0 r \comp_0 \psi)) \comp_2 X_{\mu',r \comp_0 \psi}) && \text{(by \Lemr{compat-X-one-cells}\ref{lem:compat-X-one-cells:middle})}
           \\
    =\; & l \comp_0 (X_{\mu,r \comp_0 \psi} \comp_2 ((h \comp_0 r \comp_0 \psi) \comp_1 (A \comp_0 r \comp_0 g')))
        && \text{(by \Lemr{prespcat-exch-gen})}\\
    =\; &l \comp_0 (X_{\mu \comp_0 r,\psi} \comp_2 ((h \comp_0 r \comp_0 \psi) \comp_1 (A \comp_0 r \comp_0 g'))) && \text{(by \Lemr{compat-X-one-cells}\ref{lem:compat-X-one-cells:middle})}
          \\
    =\; &X_{\tilde \mu,\psi} \comp_2 ((\tilde h \comp_0 \psi) \comp_1 (l \comp_0 A \comp_0 r \comp_0 g')) && \text{(by \Lemr{compat-X-one-cells}\ref{lem:compat-X-one-cells:left})}
  \end{align*}
  so that
  \begin{align*}
    &((\lambda \comp_0 g) \comp_1 (l \comp_0 A \comp_0 r \comp_0 g) \comp_1 (\tilde h' \comp_0 \psi) \comp_1 (\rho \comp_0 g'))
      \comp_2 ((\lambda \comp_0 g) \comp_1 X_{\tilde \mu',\psi} \comp_1 (\rho \comp_0 g')) \\
    =\; &(\lambda \comp_0 g) \comp_1
          \left[  ((l \comp_0 A \comp_0 r \comp_0 g) \comp_1 (\tilde h' \comp_0 \psi)) \comp_2 X_{\tilde \mu',\psi}\right]
          \comp_1 (\rho \comp_0 g') \\
    % =\; &(\lambda \comp_0 g) \comp_1
    %       \left[  l \comp_0 (((A \comp_0 r \comp_0 g) \comp_1 (h' \comp_0 r \comp_0 \psi)) \comp_2 X_{\mu' \comp_0 r,\psi})\right]
    %       \comp_1 (\rho \comp_0 g')  \\
    % =\; &(\lambda \comp_0 g) \comp_1
    %       \left[  l \comp_0 (((A \comp_0 r \comp_0 g) \comp_1 (h' \comp_0 r \comp_0 \psi)) \comp_2 X_{\mu',r \comp_0 \psi})\right]
    %       \comp_1 (\rho \comp_0 g') \\
    % =\; &(\lambda \comp_0 g) \comp_1 \left[  l \comp_0 (X_{\mu,r \comp_0 \psi} \comp_2 ((h \comp_0 r \comp_0 \psi) \comp_1 (A \comp_0 r \comp_0 g')))\right] \comp_1 (\rho \comp_0 g')
    %     && \text{\hspace{-6em}(by \Lemr{prespcat-exch-gen})}\\
    % =\; &(\lambda \comp_0 g) \comp_1 \left[  l \comp_0 (X_{\mu \comp_0 r,\psi} \comp_2 ((h \comp_0 r \comp_0 \psi) \comp_1 (A \comp_0 r \comp_0 g')))\right] \comp_1 (\rho \comp_0 g') \\
    =\; &(\lambda \comp_0 g) \comp_1 \left[  X_{\tilde \mu,\psi} \comp_2 ((\tilde h \comp_0 \psi) \comp_1 (l \comp_0 A \comp_0 r \comp_0 g'))\right] \comp_1 (\rho \comp_0 g') \\
    =\; &((\lambda \comp_0 g) \comp_1  X_{\tilde \mu,\psi} \comp_1 (\rho \comp_0 g'))\comp_2 ((\lambda \comp_0 g) \comp_1 (\tilde h \comp_0 \psi) \comp_1 (l \comp_0 A \comp_0 r \comp_0 g') \comp_1 (\rho \comp_0 g')).
  \end{align*}%
  % \begin{align*}
  %   &((l \comp_0 A \comp_0 r \comp_0 g) \comp_1 (\ctgt_1(\tilde \mu) \comp_0 \psi)) \comp_2 X_{\tilde \mu',\psi} \\
  %   =\;& l \comp_0 (((A \comp_0 r \comp_0 g) \comp_1 (\ctgt_1(\mu \comp_0 r) \comp_0 \psi)) \comp_2 X_{\mu' \comp_0 r,\psi}) \\
  %   =\;& l \comp_0 ((X_{\mu \comp_0 r,\psi} \comp_2 (A \comp_0 r \comp_0 g')))
  % \end{align*}
  By combining the previous equations, we obtain
  \begin{align*}
    & ((R \comp_0 g) \comp_1 (f' \comp_0 \psi)) \comp_2
      X_{\phi',\psi} \\
    =\; & ((\lambda \comp_0 g) \comp_1 (l \comp_0 A \comp_0 r \comp_0 g) \comp_1 (\rho \comp_0 g)\comp_1 (f' \comp_0 \psi)) \\
    & \hspace*{1em}\comp_2
      (((\lambda \comp_1 \tilde \mu') \comp_0 g) \comp_1 X_{\rho,\psi}) \\
       &\hspace*{1em}\comp_2 (((\lambda \comp_0 g) \comp_1 X_{\tilde \mu',\psi} \comp_1 (\rho \comp_0 g'))) \\
       &\hspace*{1em}\comp_2 ((X_{\lambda,\psi} \comp_1 ((\tilde \mu' \comp_1 \rho) \comp_0 g'))) \\
    =\; & (((\lambda \comp_1 \tilde \mu) \comp_0 g) \comp_1 X_{\rho,\psi}) \\
    &\hspace*{1em}\comp_2 (((\lambda \comp_0 g) \comp_1 X_{\tilde \mu,\psi} \comp_1 (\rho \comp_0 g'))) \\
    &\hspace*{1em}\comp_2 ((X_{\lambda,\psi} \comp_1 ((\tilde \mu \comp_1 \rho) \comp_0 g'))) \\
    & \hspace*{1em}\comp_2 ((f \comp_0 \psi) \comp_1 (\lambda \comp_0 g) \comp_1 (l \comp_0 A \comp_0 r \comp_0 g) \comp_1 (\rho \comp_0 g)) \\
    =\; & X_{\phi,\psi} \comp_2 ((f \comp_0 \psi) \comp_1 (R \comp_0 g'))
  \end{align*}
  which is what we wanted.
\end{proof}% \begin{proof}
\noindent We can deduce the complete compatibility between interchangers and $3$-cells:
\begin{lemapp}
  \label{lem:prespcat-nat-exch}
  Given $F\co \phi \TO \phi'\co f \To f' \in \prespcat\P_3$ and $\psi\co g \To
  g' \in \prespcat\P_2$ such that $F,\psi$ are $0$\composable, we have
  \[
    ((F \comp_0 g) \comp_1 (f' \comp_0 \psi)) \comp_2
    X_{\phi',\psi} = X_{\phi,\psi} \comp_2 ((f \comp_0 \psi) \comp_1 (F \comp_0 g')).
  \]
  Similarly, given $\phi\co f \To f' \in\prespcat\P_2$ and $G\co \psi \TO \psi'\co g \To g' \in
  \prespcat\P_3$ such that $\phi,G$ are $0$\composable, we have
  \[
    X_{\phi,\psi} \comp_2 ((f \comp_0 G) \comp_1 (\phi \comp_0 g')) =
    ((\phi \comp_0 g) \comp_1 (f' \comp_0 G)) \comp_2 X_{\phi,\psi'}.
  \]
\end{lemapp}
\begin{proof}
  Remember that each $3$-cell $\prespcat\P$ can be written as a sequence of rewriting
  steps of~$\P$. By induction on the length of such a sequence defining $F$ or $G$ as in
  the statement, we conclude using \Lemr{prespcat-exch-ctxt}.
\end{proof}
\noindent We can conclude that:
\ifx\graypresgraycatthm\undefined
\begin{theo}
  Some theorem.
\end{theo}
\else
\graypresgraycatthm*
\fi
% \begin{theo}
%   \todo{faire le lien avec le théorème énoncé dans le document}
%   \label{thm:gray-pres-gray-cat}
%   Given a Gray presentation $\P$, $\prespcat\P$ is canonically a
%   lax Gray category.
% \end{theo}
\begin{proof}
  The axioms of lax Gray category follow from
  \Lemr{compat-X-one-cells}, \Lemr{compat-X-one-comp},
  \Lemr{prespcat-peiffer} and \Lemr{prespcat-nat-exch}.
\end{proof}

\section{Finiteness of critical branchings}
\label{sec:finiteness-cp}

In this section, we give a proof of \Thmr{finite-cp}, \ie that Gray
presentations, under some reasonable conditions, have a finite number of
critical branchings. Our proof is constructive, so that we can extract a program
to compute the critical branchings of such Gray presentations. First, we aim at
showing that there is no critical branching~$(S_1,S_2)$ of a Gray
presentation~$\P$ where both inner $3$\generators of $S_1$ and $S_2$ are
interchange generators. We begin with a technical lemma for minimal and
independent branchings:
% \begin{lemapp}
%   A minimal local branching $(S_1,S_2)$, with $S_i = \lambda_i \comp_1 (l_i
%   \comp_0 A_i \comp_0 r_i) \comp_1 \rho_i$ and $l_i,r_i \in \freecat\P_1$,
%   $\lambda_i,\rho_i \in \freecat\P_2$, $A_i \in \P_3$ for $i \in \set{1,2}$, is
%   independent if and only if $\len{\lambda_1} + \len{\ctgt_2(S_1)} + \len{\csrc_2(S_2)} +
%   \len{\rho_2} \ge \len{\ctgt_2(S_1)}$ or $\len{\lambda_2} + \len{\csrc_2(S_2)} + \len{\ctgt_2(S_1)}
%   + \len{\rho_1} \ge \len{\ctgt_2(S_1)}$.
% \end{lemapp}

% \begin{proof}
%   Suppose that $(S_1,S_2)$ is independent. By symmetry, we can suppose that
%   $\lambda_1$ and $\rho_2$ are units.
% \end{proof}

\begin{lemapp}
  \label{lem:caract-min-indep}
  Given a minimal local branching $(S_1,S_2)$ of a Gray presentation~$\P$, with
  \[S_i = \lambda_i \comp_1 (l_i \comp_0 A_i \comp_0 r_i) \comp_1 \rho_i\]
  and
  $l_i,r_i \in \freecat\P_1$, $\lambda_i,\rho_i \in \freecat\P_2$, $A_i \in
  \P_3$ for $i \in \set{1,2}$, the followings hold:
  \begin{enumerate}[label=(\roman*),ref=(\roman*)]
  \item \label{lem:caract-min-indep:lambda} either $\lambda_1$ or $\lambda_2$ is
    an identity,
  \item \label{lem:caract-min-indep:rho} either $\rho_1$ or $\rho_2$ is an
    identity,
  \item \label{lem:caract-min-indep:indep} $(S_1,S_2)$ is independent if and
    only if
    \[
      \len{\csrc_2(A_1)} + \len{\csrc_2(A_2)} \le \len{\csrc_2(S_1)} \qqtand
      \len{\lambda_1}\len{\rho_1} = \len{\lambda_2}\len{\rho_2} = 0.
    \]
  \end{enumerate}
  If $(S_1,S_2)$ is moreover not independent:
  \begin{enumerate}[start=4,label=(\roman*),ref=(\roman*)]
  \item \label{lem:caract-min-indep:l} either $l_1$ or $l_2$ is an identity,
  \item \label{lem:caract-min-indep:r} either $r_1$ or $r_2$ is an identity.
  \end{enumerate}
\end{lemapp}
\begin{proof}
  Suppose that neither $\lambda_1$ nor $\lambda_2$ are identities. Then, since
  \[
    \lambda_1 \comp_1 (l_1 \comp_0 \csrc_2(A_1) \comp_0 r_1) \comp_1 \rho_1 =
    \lambda_2 \comp_1 (l_2 \comp_0 \csrc_2(A_2) \comp_0 r_2) \comp_1 \rho_2,
  \]
  we have $\lambda_i = w \comp_1 \lambda_i'$ for some $w \in \freecat\P_2$
  and $\lambda_i' \in \freecat\P_2$ for $i \in \set{1,2}$, such that $\len{w}
  \ge 1$, contradicting the minimality of $(S_1,S_2)$. So either $\lambda_1$ or
  $\lambda_2$ is an identity and similarly for $\rho_1$ and $\rho_2$, which
  concludes~\ref{lem:caract-min-indep:lambda}
  and~\ref{lem:caract-min-indep:rho}.
  
  By the definition of independent branching, the first implication
  of~\ref{lem:caract-min-indep:indep} is trivial. For the converse, suppose that
  $(S_1,S_2)$ is such that
  \[
    \len{\csrc_2(A_1)} + \len{\csrc_2(A_2)} \le \len{\csrc_2(S_1)} \qtand
    \len{\lambda_1}\len{\rho_1} = \len{\lambda_2}\len{\rho_2} = 0.
  \]
  We can suppose by symmetry that $\lambda_1$ is a unit. Since $\len{\csrc_2(S_1)} =
  \len{\lambda_1} + \len{\csrc_2(A_1)} + \len{\rho_1}$, we have that $\len{\csrc_2(A_2)} \le
  \len{\rho_1}$.

  If $\len{\rho_1} = 0$, then
  \[
    S_1 = l_1 \comp_0 A_1 \comp_0 r_1 \qtand \len{\csrc_2(A_2)} = 0,
  \]
  thus, since $\len{\lambda_2}\len{\rho_2} = 0$, we have
  \[
    \text{either} \quad S_2 = \csrc_2(S_1) \comp_1 (l_2 \comp_2 A_2 \comp_2
    r_2) \quad\text{or}\quad S_2 = (l_2 \comp_2 A_2 \comp_2 r_2) \comp_1
    \csrc_2(S_1).
  \]
  In both cases, $(S_1,S_2)$ is independent.

  Otherwise, $\len{\rho_1} > 0$ and,
  by~\ref{lem:caract-min-indep:rho}, we have $\len{\rho_2} = 0$ so that
  \[
    S_1 = (l_1 \comp_0 A_1 \comp_0 r_1) \comp_1 \rho_1 \qtand S_2 = \lambda_2
    \comp_1 (l_2 \comp_0 A_2 \comp_0 r_2).
  \]
  Since $\len{\csrc_2(A_2)} \le \len{\rho_1}$, we have $\rho_1 = \chi
  \comp_1 (l_2 \comp_0 \csrc_2(A_2) \comp_0 r_2)$ for some $\chi \in
  \freecat\P_2$ and, since $\csrc_2(S_1) = \csrc_2(S_2)$, we get
  \[
    (l_1 \comp_0 \csrc_2(A_1) \comp_0 r_1) \comp_1 \chi \comp_1 (l_2 \comp_0
    \csrc_2(A_2) \comp_0 r_2) = \lambda_2 \comp_1 (l_2 \comp_0 \csrc_2(A_2)
    \comp_0 r_2).
  \]
  So $\lambda_2 = (l_1 \comp_0 \csrc_2(A_1) \comp_0
  r_1) \comp_1 \chi$ and hence $(S_1,S_2)$ is an independent branching, which
  concludes the proof of~\ref{lem:caract-min-indep:indep}.
  
  Finally, suppose that $(S_1,S_2)$ is not independent. By~\ref{lem:caract-min-indep:indep}, it implies that
  \[
    \text{either}\quad \len{\csrc_2(A_1)} + \len{\csrc_2(A_2)} >
    \len{\csrc_2(S_1)} \quad\text{or}\quad \len{\lambda_1}\len{\rho_1} > 0
    \quad\text{or}\quad \len{\lambda_2}\len{\rho_2} > 0.
  \]
  If $\len{\lambda_1}\len{\rho_1} > 0$, then
  $\len{\lambda_2} = \len{\rho_2} = 0$ by~\ref{lem:caract-min-indep:lambda} and~\ref{lem:caract-min-indep:rho},
  so that
  \[
    \lambda_1 \comp_1 (l_1 \comp_0 A_1 \comp_0 r_1) \comp_1 \rho_1 = l_2 \comp_0
    A_2 \comp_0 r_2
  \]
  thus there exists $\lambda_1',\rho_1' \in \freecat\P_2$ such that
  \[
    \lambda_1 = l_2 \comp_0 \lambda_1' \comp_0 r_2 \qtand \rho_1 = l_2 \comp_0
    \rho_1' \comp_0 r_2,
  \]
  and we have
  \[
    l_2 \comp_0 \ctgt_1(\lambda_1') \comp_0 r_2 =
    \ctgt_1(\lambda_1) = l_1 \comp_0 \csrc_1(A_1) \comp_0 r_1.
  \]
  Thus, $l_1$ and
  $l_2$ have the same prefix~$l$ of size $k = \min(\len{l_1},\len{l_2})$ and we
  can write
  \begin{align*}
    S_1 &= l \comp_0 S_1' & S_2 &= l \comp_0 S_2'
  \end{align*}
  for some rewriting steps $S_1,S_2 \in \freecat\P_3$. Since $(S_1,S_2)$ is
  minimal, we have $k = 0$, so $\len{l_1}\len{l_2} = 0$. We show similarly that
  $\len{r_1}\len{r_2} = 0$. The case where $\len{\lambda_2}\len{\rho_2} > 0$ is
  handled similarly.

  So suppose that
  \begin{equation}
    \label{eq:not-indep-conditions}
    \len{\lambda_1}\len{\rho_1} = 0 \qtand
    \len{\lambda_2}\len{\rho_2} = 0 \qtand 
    \len{\csrc_1(A_1)} + \len{\csrc_1(A_2)}
    > \len{\csrc_2(S_1)}.
  \end{equation}
  In particular, we get that $\len{\csrc_2(A_i)} > 0$
  for $i \in \set{1,2}$. Let $u_i,v_i \in \freecat\P_1$ and $\alpha_i \in \P_2$
  for $i \in \set{1,\ldots,r}$ with $r = \len{\csrc_2(S_1)}$ such that
  \[
    \csrc_2(S_1) = (u_1 \comp_0 \alpha_1 \comp_0 v_1) \comp_1 \cdots \comp_1 (u_r
    \comp_0 \alpha_r \comp_0 v_r). 
  \]
  The condition last part of \eqref{eq:not-indep-conditions} implies that
  there is $i_0$ such that $l_1$ and $l_2$ are both prefix of $u_{i_0}$. So,
  $l_1$ and $l_2$ have the same prefix $l$ of length $k =
  \min(\len{l_1},\len{l_2})$.

  Now, we prove that $\lambda_1 = l \comp_0 \lambda_1'$ for some $\lambda_1' \in
  \freecat\P_2$. If $\len{\lambda_1} = 0$, then
  \[\lambda_1 = l_1 \comp_0
    \csrc_1(S_1) \comp_0 r_1,\] so $\lambda = l \comp_0 \lambda_1'$ for some
  $\lambda' \in \freecat\P_2$. Otherwise, if $\len{\lambda_1} > 0$, since
  $\len{\lambda_1}\len{\rho_1} = 0$, we have $\len{\rho_1} = 0$ and,
  by~\ref{lem:caract-min-indep:lambda}, $\len{\lambda_2} = 0$. Also, by the last
  part of~\eqref{eq:not-indep-conditions}, we have $\len{\lambda_1} <
  \len{\csrc_2(A_2)}$. Thus,
  \begin{center}
    $\lambda_1$ is a prefix of $l_2 \comp_0 \csrc_2(A_2) \comp_0 r_2$,
  \end{center}
   so $\lambda_1 = l \comp_0 \lambda_1'$ for some
  $\lambda_1 \in \freecat\P_2$. Similarly, there are $\rho_1',
  \lambda_2',\rho_2'\in \freecat\P_2$ such that
  \[
    \rho_1 = l \comp_0 \rho_1' \qtand \lambda_2 = l \comp_0 \lambda_2' \qtand
    \rho_2 = l \comp_0 \lambda_2'.
  \]
  Hence $S_1 = l \comp_0 S_1'$ and $S_2 = l \comp_0 S_2'$ for some rewriting
  steps $S'_1,S'_2 \in \freecat\P_3$. Since $(S_1,S_2)$ is minimal, we have
  $\len{l_1}\len{l_2} = \len{l} = 0$, which proves~\ref{lem:caract-min-indep:l}.
  The proof of~\ref{lem:caract-min-indep:r} is similar.
\end{proof}
\noindent We now have enough material to show that:
\begin{propapp}
  \label{prop:no-st-st-cp}
  Given a Gray presentation $\P$, there are no critical branching~$(S_1,S_2)$
  of~$\P$ such that both the inner $3$\generators of $S_1$ and $S_2$ are
  interchange generators.
\end{propapp}
\begin{proof}
  Let~$(S_1,S_2)$ be a local minimal branching such that, for~$i \in \set{1,2}$,
  \[
    S_i = \lambda_i \pcomp_1 (l_i \pcomp_0 X_{\alpha_i,g_i,\beta_i} \pcomp_0
    r_i) \pcomp_1 \rho_i
  \]
  for some~$l_i,r_i,g_i \in \freecat\P_1$,~$\lambda_i,\rho_i \in \freecat\P_2$
  and~$\alpha_i,\beta_i \in \P_2$, and let~$\phi$ be~$\csrc_2(S_1)$.
  Since~$\len{\csrc_2(X_{\alpha_1,g_1,\beta_1})} = 2$, we have~$\len{\phi} \ge
  2$.

  If~$\len{\phi} = 2$, then~$\len{\lambda_i} = \len{\rho_i} = 0$
  for~$i \in \set{1,2}$. Thus, since~$\csrc_2(S_1) = \csrc_2(S_2)$, we get
  \begin{align*}
    & (l_1 \pcomp_0 \alpha_1 \pcomp_0 g_1 \pcomp_0 \csrc_1(\beta_1) \pcomp_0 r_1) \pcomp_1
      (l_1 \pcomp_0 \ctgt_1(\alpha_1) \pcomp_0 g_1 \pcomp_0 \beta_1 \pcomp_0 r_1) \\
    =\;&
         (l_2 \pcomp_0 \alpha_2 \pcomp_0 g_2 \pcomp_0 \csrc_1(\beta_2) \pcomp_0 r_2) \pcomp_1
         (l_2 \pcomp_0 \ctgt_1(\alpha_2) \pcomp_0 g_2 \pcomp_0 \beta_2 \pcomp_0 r_2).
  \end{align*}
  By the unique decomposition property given by \Thmr{precat-nf}, we obtain
  % ~$l_1 = l_2$,
  % ~$r_1 = r_2$,~$\alpha_1 = \alpha_2$,~$\beta_1 = \beta_2$ and~$g_1 \pcomp_0
  % \csrc_1(\beta_1) \pcomp_0 r_1 = g_2 \pcomp_0 \csrc_1(\beta_2) \pcomp_0 r_2$.
  \[
    l_1 = l_2,\quad
    r_1 = r_2,\quad \alpha_1 = \alpha_2,\quad \beta_1 = \beta_2 \qtand g_1 \pcomp_0
    \csrc_1(\beta_1) \pcomp_0 r_1 = g_2 \pcomp_0 \csrc_1(\beta_2) \pcomp_0 r_2.
  \]
  So~$g_1 \pcomp_0 \csrc_1(\beta_1) \pcomp_0 r_1 = g_2 \pcomp_0 \csrc_1(\beta_1)
  \pcomp_0 r_1$, which implies that~$g_1 = g_2$. Hence,~$(S_1,S_2)$ is trivial.

  If~$\len{\phi} = 3$, then~$\len{\lambda_i} + \len{\rho_i} = 1$
  for~$i \in \set{1,2}$, and, by \Lemr{caract-min-indep},
  \[
    \text{either\quad$\len{\rho_1} = \len{\lambda_2} = 1$\quad or
      \quad$\len{\lambda_1} = \len{\rho_2} = 1$\zbox.}
  \]
  By symmetry, we can suppose that~$\len{\rho_1} = \len{\lambda_2} = 1$, which
  implies that~$\len{\lambda_1} = \len{\rho_2} = 0$. By unique decomposition of
  whiskers, since~$\csrc_2(S_1) = \csrc_2(S_2)$, we have
  \begin{align*}
    l_1 \pcomp_0 \alpha_1 \pcomp_0 g_1 \pcomp_0 \csrc_1(\beta_1) \pcomp_0 r_1 & = \lambda_2 \\
    l_1 \pcomp_0 \ctgt_1(\alpha_1) \pcomp_0 g_1 \pcomp_0 \beta_1 \pcomp_0 r_1 & = l_2 \pcomp_0 \alpha_2 \pcomp_0 g_2 \pcomp_0 \csrc_1(\beta_2) \pcomp_0 r_2 \\
    \rho_1 & = l_2 \pcomp_0 \ctgt_1(\alpha_2) \pcomp_0 g_2 \pcomp_0 \beta_2 \pcomp_0 r_2 
  \end{align*}
  and the second line implies that~$l_1 \pcomp_0 \ctgt_1(\alpha_1) \pcomp_0 g_1 =
  l_2$,~$\beta_1 = \alpha_2$ and~$r_1 = g_2 \pcomp_0 \csrc_1(\beta_2) \pcomp_0 r_2$.
  Since~$(S_1,S_2)$ is minimal, we have~$\len{l_1} = \len{r_2} = 0$. So
  \begin{align*}
    S_1 &= (X_{\alpha_1,g_1,\beta_1} \pcomp_0 g_2 \pcomp_0 \csrc_1(\beta_2)) \pcomp_1 (\ctgt_1(X_{\alpha_1,g_1,\beta_1}) \pcomp_0 g_2 \pcomp_0 \beta_2) \\
    S_2 &= (\alpha_1 \pcomp_0 g_1 \pcomp_0 \csrc_1(\beta_1) \pcomp_0 g_2 \pcomp_0 \csrc_1(\beta_2)) \pcomp_1 (\ctgt_1(\alpha_1) \pcomp_0 g_1 \pcomp_0 X_{\beta_1,g_2,\beta_2})
  \end{align*}
  thus~$(S_1,S_2)$ is a natural branching, hence not a critical one.

  Finally, if~$\len{\phi} \ge 4$, then, since~$\len{\lambda_i} +
  \len{\rho_i} = \len{\phi} - 2 \ge 2$ for~$i \in \set{1,2}$, by
  \Lemr{caract-min-indep}, we have that
  \[
    \text{either\quad$\len{\lambda_1} = \len{\rho_2} = \len\phi - 2$\quad
      or\quad$\len{\rho_1} = \len{\lambda_2} = \len\phi - 2$\zbox.}
  \]
  In either case,
  \[
    \len{\lambda_1}\len{\rho_1} = \len{\lambda_2}\len{\rho_2} = 0
    \qtand
    \len{\csrc_2(X_{\alpha_1,g_1,\beta_1})} + \len{\csrc_2(X_{\alpha_2,g_2,\beta_2})} = 4 \le \len{\phi}
  \]
  so,
  by~\Lemr{caract-min-indep}\ref{lem:caract-min-indep:indep},~$(S_1,S_2)$ is
  independent, hence not critical.
\end{proof}

\noindent Until the end of this section, we denote by $\P$ a Gray presentation
such that $\P_2$ and $\P_3$ are finite and $\len{\csrc_2(A)} > 0$ for every $A
\in \P_3$, \ie a Gray presentation satisfying the hypothesis of
\Thmr{finite-cp}. The next result we prove is a characterization of independent
branchings among minimal ones:
\begin{lemapp}
  \label{lem:definite-indep}
  % Given a $(3{,}-)$-definite Gray presentation $\P$, and a minimal branching
  % $(S_1,S_2)$ of~$\P$ with
  Given a minimal branching $(S_1,S_2)$ of~$\P$ with
  \[
    S_i = \lambda_i \comp_1 (l_i \comp_0 A_i \comp_0 r_i) \comp_1 \rho_i
  \]
  for some $l_i,r_i \in \freecat\P_1$, $\lambda_i,\rho_i \in \freecat\P_2$ and
  $A_i \in \P_3$ for $i \in \set{1,2}$, we have that $(S_1,S_2)$ is independent
  if and only if
  \begin{center}
    either $\len{\lambda_1} \ge \len{\csrc_2(A_2)}$ or $\len{\rho_1} \ge
    \len{\csrc_2(A_2)}$ (\resp $\len{\lambda_2} \ge \len{\csrc_2(A_1)}$ or
    $\len{\rho_2} \ge \len{\csrc_2(A_1)}$).
  \end{center}
\end{lemapp}

\begin{proof}
 If $(S_1,S_2)$ is independent, then, by
 \Lemr{caract-min-indep}\ref{lem:caract-min-indep:indep},
 \[
   \len{\csrc_2(A_1)} + \len{\csrc_2(A_2)} \le \len{\lambda_1} +
   \len{\csrc_2(A_1)} + \len{\rho_1} = \len{\lambda_2} + \len{\csrc_2(A_2)} +
   \len{\rho_2},
 \]
 that is,
 \[
   \len{\csrc_2(A_1)} \le \len{\lambda_2} + \len{\rho_2} \qtand
   \len{\csrc_2(A_2)} \le \len{\lambda_1} + \len{\rho_1}.
 \]
 By hypothesis, we have $\len{\csrc_2(A_1)} > 0$, so that $\len{\lambda_2} +
 \len{\rho_2} > 0$. If $\len{\lambda_2} > 0$, then, by
 \Lemr{caract-min-indep}\ref{lem:caract-min-indep:lambda}, $\len{\lambda_1} = 0$
 so $\len{\csrc_2(A_2)}\le \len{\rho_1}$. Similarly, if $\len{\rho_2} > 0$, then
 $\len{\csrc_2(A_2)} \le \len{\lambda_1}$, which proves the first implication.

 Conversely, if $\len{\lambda_1} \ge \len{\csrc_2(A_2)}$, then, since
 $\csrc_2(A_2) > 0$ by our hypothesis on~$\P$, we have $\len{\lambda_1} > 0$. By
 \Lemr{caract-min-indep}\ref{lem:caract-min-indep:lambda}, we get that
 $\len{\lambda_2} = 0$. Also,
 \[
   \len{\lambda_1}+ \len{\csrc_2(A_1)} +
   \len{\rho_1} = \len{\csrc_2(A_2)} + \len{\rho_2} \le \len{\lambda_1} +
   \len{\rho_2},
 \]
 so $\len{\rho_2} \ge \len{\csrc_2(A_1)} + \len{\rho_1}$, thus
 $\len{\rho_1} < \len{\rho_2}$. By
 \Lemr{caract-min-indep}\ref{lem:caract-min-indep:rho}, we have $\len{\rho_1} =
 0$. Moreover,
 \[
   \len{\csrc_2(A_1)} + \len{\csrc_2(A_2)} \le \len{\csrc_2(A_1)} +
   \len{\lambda_1} = \len{\csrc_2(S_1)}
 \]
 hence, by~\Lemr{caract-min-indep}\ref{lem:caract-min-indep:indep}, $(S_1,S_2)$
 is independent.
\end{proof}
\noindent Then, we prove that minimal non-independent branchings are uniquely
characterized by a small amount of information:
\begin{lemapp}
  \label{lem:definite-branching-data}
  Given a minimal non-independent branching $(S_1,S_2)$ of~$\P$ with
  \[
    S_i = \lambda_i \comp_1 (l_i \comp_0 A_i
    \comp_0 r_i) \comp_1 \rho_i
  \]
  for some $l_i,r_i \in \freecat\P_1$,
  $\lambda_i,\rho_i \in \freecat\P_2$ and $A_i \in \P_3$ for $i \in \set{1,2}$,
  we have that $(S_1,S_2)$ is uniquely determined by $A_1$, $A_2$,
  $\len{\lambda_1}$ and $\len{\lambda_2}$.
\end{lemapp}

\begin{proof}
  Let the unique $k_1,k_2 >0$, $u_i,u'_i,v_i,v'_i \in \freecat\P_1$ and $\alpha_i,\beta_i
  \in \P_2$ such that
  \[
    \csrc_2(A_1) = (u_1 \comp_0 \alpha_1 \comp_0 u'_1) \comp_1 \cdots \comp_1 (u_{k_1} \comp_0 \alpha_{k_1} \comp_0 u'_{k_1})
  \]
  and
  \[
    \csrc_2(A_2) = (v_1 \comp_0 \beta_1 \comp_0 v'_1) \comp_1 \cdots \comp_1 (v_{k_2} \comp_0 \beta_{k_2} \comp_0 v'_{k_2}).
  \]
  Let $i_1 = 1 + \len{\lambda_1}$ and $i_2 = 1 + \len{\lambda_2}$. Since
  \begin{equation}
    \label{eq:branching-src-equal}
    \lambda_1 \comp_1 (l_1 \comp_0 \csrc_2(A_1) \comp_0 r_1) \comp_1 \rho_1 =
    \lambda_2 \comp_1 (l_2 \comp_0 \csrc_2(A_2) \comp_0 r_2) \comp_1 \rho_2,
  \end{equation}
  and, by~\Lemr{definite-indep}, $\len{\lambda_1} < \len{\csrc_2(A_2)}$ and
  $\len{\lambda_2} < \len{\csrc_2(A_1)}$,
  % and, by~\Lemr{caract-min-indep},
  % $\len{\lambda_1}\len{\lambda_2} = 0$,
  we get 
  \[
    l_1 \comp_0 u_{i_2} \comp_0 \alpha_{i_2}
    \comp_0 u'_{i_2} \comp_0 r_1
    =
    l_2 \comp_0 v_{i_1} \comp_0 \beta_{i_1}
    \comp_0 v'_{i_1} \comp_0 r_2
  \]
  so that
  \[
    l_1 \comp_0 u_{i_2} = l_2 \comp_0 v_{i_1}
    \qtand u'_{i_2} \comp_0 r_1 = v'_{i_1}
    \comp_0 r_2.
  \]
  By~\Lemr{caract-min-indep}\ref{lem:caract-min-indep:l}, either
  $l_1$ or $l_2$ is an identity. Thus, if $\len{u_{i_2}} \le \len{v_{i_1}}$,
  then $\len{l_1} \ge \len{l_2}$ so $l_2$ is a unit and $l_2$ is the prefix of
  $u_{i_2}$ of size $\len{u_{i_2}} - \len{v_{i_1}}$. Otherwise, if $\len{u_{i_2}} \le \len{v_{i_1}}$, we
  obtain similarly that $l_1$ is the prefix of $v_{i_1}$ of size $\len{v_{i_1}}
  - \len{u_{i_2}}$ and $l_2$ is a unit. In both cases, $l_1$ and
  $l_2$ are completely determined by $A_1$, $A_2$, $\len{\lambda_1}$ and
  $\len{\lambda_2}$. A similar argument holds for $r_1$ and $r_2$.

  Now, if $\len{\lambda_1} >0$,
  by~\Lemr{caract-min-indep}\ref{lem:caract-min-indep:lambda}, $\len{\lambda_2}
  = 0$. By \eqref{eq:branching-src-equal} and since $\len{\lambda_1} <
  \len{\csrc_2(A_2)}$, $\lambda_1$ is the prefix of $l_2 \comp_0
  \csrc_2(A_2) \comp_0 r_2$ of length~$\len{\lambda_1}$. Otherwise, if
  $\len{\lambda_1} = 0$, then $\lambda_1 = \unit{l_1 \comp_0 \csrc_1(A_1)
    \comp_0 r_1}$. In both cases, $\lambda_1$ is completely determined by $A_1$,
  $A_2$, $\len{\lambda_1}$. A similar argument holds for $\lambda_2$. Note that,
  if we prove that $\len{\rho_1}$ and $\len{\rho_2}$ are completely determined
  by $A_1$, $A_2$, $\len{\lambda_1}$ and $\len{\lambda_2}$, the above argument
  also applies to $\rho_1$ and $\rho_2$ and the lemma is proved. But
\[
  \len{\lambda_1} + \len{\csrc_2(A_1)} + \len{\rho_1} =
  \len{\lambda_2} + \len{\csrc_2(A_2)} + \len{\rho_2}, 
\]
so that if
$\len{\lambda_1} + \len{\csrc_2(A_1)} \ge \len{\lambda_2} + \len{\csrc_2(A_2)}$,
then, by~\Lemr{caract-min-indep}\ref{lem:caract-min-indep:rho},
$\len{\rho_1} = 0$ and
\[
  \len{\rho_2} = \len{\lambda_1} + \len{\csrc_2(A_1)} -
  \len{\lambda_2} - \len{\csrc_2(A_2)}.
\]
Otherwise, if $\len{\lambda_1} +
\len{\csrc_2(A_1)} \le \len{\lambda_2} + \len{\csrc_2(A_2)}$, we get similarly
that
\[
  \len{\rho_1} = \len{\lambda_2} + \len{\csrc_2(A_2)} - \len{\lambda_1} -
  \len{\csrc_2(A_1)}
\]
and $\len{\rho_2} = 0$.
In both cases, $\len{\rho_1}$ and
$\len{\rho_2}$ are completely determined by $A_1$, $A_2$, $\len{\lambda_1}$ and
$\len{\lambda_2}$, which concludes the proof.
\end{proof}
\noindent Given~$A \in \P_3$, we say that~$A$ is an \emph{operational} generator
if it is not an interchange generator. We now prove that an operational
generator can form a critical branching with a finite number of interchange
generators:
\begin{lemapp}
  \label{lem:definite-finite-op-st-branching}
  Given an operational $A_1 \in \P_3$, there are a finite number interchange
  generator $A_2 \in P_3$ so that there is a critical branching~$(S_1,S_2)$
  of~$\P$ with
  \[
    S_i = \lambda_i \comp_1 (l_i \comp_0 A_i \comp_0 r_i) \comp_1 \rho_i
  \]
  for some
  $l_i,r_i \in \freecat\P_1$, $\lambda_i,\rho_i \in \freecat\P_2$ for $i \in
  \set{1,2}$.
\end{lemapp}

\begin{proof}
  % Let an operational $A_1\in\P_3$, $\alpha,\beta \in \P_2$, $u \in
  % \freecat\P_1$, $l_i,r_i \in \freecat\P_1$, $\lambda_i,\rho_i \in \freecat\P_2$
  % for $i \in \set{1,2}$, so that $(S_1,S_2)$ is a critical branching of~$\P$
  % with
  Let~$\alpha,\beta \in \P_2$,~$u \in \freecat\P_1$, $A_2 =
  X_{\alpha,u,\beta}$,~$l_i,r_i \in \freecat\P_1$,~$\lambda_i,\rho_i \in
  \freecat\P_2$ for~$i \in \set{1,2}$, so that~$(S_1,S_2)$ is a critical
  branching of~$\P$ with
  \begin{center}
    $S_i = \lambda_i \comp_1 (l_i \comp_0 A_i \comp_0 r_i) \comp_1 \rho_i$ for
    $i \in \set{1,2}$
  \end{center}
  for~$i \in \set{1,2}$.
  Let the unique $k \ge 2$, $v_i,v'_i \in \freecat\P_1$, $\gamma_i \in \P_2$ for
  $i \in \set{1,\ldots,k}$ such that
  \[
    \csrc_2(A_1) = (v_1 \comp_0 \gamma_1 \comp_0 v'_1) \comp_1 \cdots \comp_1
    (v_k \comp_0 \gamma_k \comp_0 v'_k).
  \]
  By~\Lemr{definite-indep}, since $(S_1,S_2)$ is non-independent,
  \[
    2 = \len{\csrc_2(X_{\alpha,u,\beta})} > \max(\len{\lambda_1},\len{\rho_1}).
  \]
  Note that we cannot have $\len{\lambda_1} = \len{\rho_1} = 1$. Indeed,
  otherwise, by~\Lemr{caract-min-indep}, we would have $\len{\lambda_2} =
  \len{\rho_2} = 0$, so that
  \[
    2 = \len{\csrc_2(X_{\alpha,u,\beta})} = \len{\lambda_1} + \len{\csrc_2(A_1)}
    + \len{\rho_1}.
  \]
  and thus $\len{\csrc_2(A_1)} = 0$, contradicting our hypothesis on the
  $3$\generators of~$\P$. This leaves three cases to handle.

  Suppose that $\len{\lambda_1} = \len{\rho_1} = 0$. Then,
  \[
    l_1 \comp_0 \csrc_2(A_1) \comp_0 r_1 = \lambda_2 \comp_0 (l_2 \comp_0
    \csrc_2(X_{\alpha,u,\beta}) \comp_0 r_2) \comp_1 \rho_2.
  \]
  Thus,
  \begin{align*}
    l_1 \comp_0 v_{1 + \len{\lambda_2}} \comp_0 \gamma_{1 + \len{\lambda_2}}
    \comp_0 v'_{1 + \len{\lambda_2}} \comp_0 r_1 &= l_2 \comp_0 \alpha \comp_0 u
    \comp_0 \csrc_1(\beta) \comp_0 r_2 \\
    l_1 \comp_0 v_{2 + \len{\lambda_2}} \comp_0 \gamma_{2 + \len{\lambda_2}}
    \comp_0 v'_{2 + \len{\lambda_2}} \comp_0 r_1 &= l_2 \comp_0 \ctgt_1(\alpha) \comp_0 u
    \comp_0 \beta \comp_0 r_2
  \end{align*}
  so
  \begin{align*}
    \gamma_{1 + \len{\lambda_2}} &= \alpha,
    &\gamma_{2 + \len{\lambda_2}} &=
                                    \beta, &
    l_2 &= l_1 \comp_0 v_{1 + \len{\lambda_2}},
    & r_2 &= v'_{2 +
            \len{\lambda_2}} \comp_0 r_1
  \end{align*}
    and $u$ is the suffix of $l_1 \comp_0 v_{2 +
    \len{\lambda_2}}$ of length $\len{l_1 \comp_0 v_{2+\len{\lambda_2}}} -
  \len{l_2 \comp_0 \ctgt_1(\alpha)}$. In particular, $X_{\alpha,u,\beta}$ is
  completely determined by $A_1$ and $\len{\lambda_2}$. And since
  \[
    \len{\lambda_2} = \len{\csrc_2(A_1)} - \len{\csrc_2(X_{\alpha,u,\beta})} -
    \len{\rho_2} \in \set{0,\ldots,\len{\csrc_2(A_1)} - 2},
  \]
  there is a finite number of possible
  $X_{\alpha,u,\beta}$ which induce a critical branching~$(S_1,S_2)$.
  % Moreover, $\lambda_2$ is the prefix of $l_1
  % \comp_0 \csrc_2(R_1) \comp_0 r_1$ of length $\len{\lambda_2}$ and $\rho_2$ is
  % the suffix of $l_1 \comp_0 \csrc_2(R_1) \comp_0 r_1$ of length
  % $\len{\csrc_2(R_1)} - 2 - \len{\lambda_2}$. 

  Suppose now that $\len{\lambda_1} = 1$ and $\len{\rho_1} = 0$. Then,
  by~\Lemr{caract-min-indep}, $\len{\lambda_2} = 0$. So
  \[
    \lambda_1 = l_2 \comp_0 \alpha \comp_0 u \comp_0 \csrc_1(\beta) \comp_0 r_2
  \]
  and
  \[
    l_1 \comp_0 v_1 \comp_0 \gamma_1 \comp_0 v'_1 \comp_0 r_1 = l_2 \comp_0
    \ctgt_1(\alpha) \comp_0 u \comp_0 \beta \comp_0 r_2.
  \]
  In particular, we have $\beta = \gamma_1$ and $r_2 = v'_1 \comp_0 r_1$, so
  $\len{r_1} \le \len{r_2}$.
  By~\Lemr{caract-min-indep}\ref{lem:caract-min-indep:r}, we have $\len{r_1} =
  0$ and $r_2 = v'_1$. Note that we have $\len{u} < \len{v_1}$. Indeed,
  otherwise $u = u' \comp_0 v_1$ for some $u'$ and, since
  \[
    \len{l_1} + \len{v_1} = \len{l_2} +
    \len{\ctgt_1(\alpha)} + \len{u},
  \]
  we get that $\len{l_2} \le \len{l_1}$.
  By~\Lemr{caract-min-indep}\ref{lem:caract-min-indep:l}, it implies that
  $\len{l_2} = 0$ and $l_1 = \ctgt_1(\alpha) \comp_0 u'$, which gives
  \[
    S_1 = (\alpha \comp_0 u' \comp_0 \csrc_1(A_1)) \comp_1 (\ctgt_1(\alpha)
    \comp_0 u' \comp_0 A_1)
  \]
  and
  \[
    S_2 = (X_{\alpha,u' \comp_0 v_1,\gamma_1} \comp_0 v'_1) \comp_0
    ((\ctgt_1(\alpha) \comp_0 u') \comp_0 ((v_2 \comp_0 \gamma_2 \comp_0 v'_2)
    \comp_1 \cdots \comp_1 (v_k \comp_0 \gamma_k \comp_0 v'_k)))
  \]
  so that $(S_1,S_2)$ is a natural branching, contradicting the fact that
  $(S_1,S_2)$ is a critical branching.

  Hence, $\len{u} < \len{v_1}$ and $u$ is a strict suffix of $v_1$, thus there are
  $\len{v_1}$ such possible $u$. Moreover, since $\P_2$ is finite,
  there are a finite number of possible $\alpha \in \P_2$. Hence, there are a
  finite number of possible $X_{\alpha,u,\beta} \in \P_2$ that induces a
  critical branching $(S_1,S_2)$ such that $\len{\lambda_1} = 1$ and
  $\len{\rho_1} = 0$.
  The case where~$\len{\lambda_1} = 0$ and~$\len{\rho_1} = 1$ is similarly
  handled, which concludes the proof.
\end{proof}

\noindent
We can now conclude the finiteness property for critical
branchings of Gray presentations:

\ifx\grayfinitecriticalbranchings\undefined
\begin{theo}
  Some theorem.
\end{theo}
\else
\grayfinitecriticalbranchings*
\fi
\begin{proof}
  Let~$S_i = \lambda_i \pcomp_1 (l_i \pcomp_0 A_i \pcomp_0 r_i) \pcomp_1 \rho_i$
  with~$l_i,r_i \in \freecat\Q_1$,~$\lambda_i,\rho_i \in \freecat\Q_2$ and~$A_i
  \in \Q_3$ for~$i \in \set{1,2}$ such that~$(S_1,S_2)$ is a critical branching
  of~$\Q$. By~\Lemr{definite-branching-data}, such a branching is uniquely
  determined by~$A_1$,~$A_2$,~$\len{\lambda_1}$ and~$\len{\lambda_2}$.
  By~\Lemr{definite-indep},
  \[
    \len{\lambda_1} <
    \len{\csrc_2(A_2)}
    \qtand
    \len{\lambda_2} < \len{\csrc_2(A_1)}.
  \]
  Hence, for a given pair~$(A_1,A_2)$, there are a finite number of
  tuples~$(l_1,l_2,r_1,r_2,\lambda_1,\lambda_2,\rho_1,\rho_2)$ such
  that~$(S_1,S_2)$ is a critical branching. Moreover, by \Propr{no-st-st-cp},
  either~$A_1$ or~$A_2$ is an operational generator. By symmetry, we can suppose
  that~$A_1$ is operational. Since~$\Q_3$ is finite, there is a finite number of
  such~$A_1$. Moreover, there are a finite number of pairs~$(A_1,A_2)$
  where~$A_2$ is operational too. If~$A_2$ is an interchange generator, then,
  by~\Cref{lem:definite-finite-op-st-branching}, there are a finite number of
  possible~$A_2$ for a given~$A_1$ such that~$(S_1,S_2)$ is a critical
  branching, which concludes the finiteness analysis.
\end{proof}

%%% Local Variables:
%%% mode: latex
%%% TeX-master: "proxy-appendix"
%%% End:

%% file: article-plain.bbl
\begin{thebibliography}{10}

\bibitem{adamek1994locally}
Jiří Adámek and Jiří Rosický.
\newblock {\em Locally presentable and accessible categories}, volume 189.
\newblock Cambridge University Press, 1994.

\bibitem{baader1999term}
Franz Baader and Tobias Nipkow.
\newblock {\em Term rewriting and all that}.
\newblock Cambridge university press, 1999.

\bibitem{bar2016globular}
Krzysztof Bar, Aleks Kissinger, and Jamie Vicary.
\newblock Globular: an online proof assistant for higher-dimensional rewriting.
\newblock In {\em LIPIcs}, volume~52, pages 34:1--34:11, 2016.

\bibitem{bar2017data}
Krzysztof Bar and Jamie Vicary.
\newblock Data structures for quasistrict higher categories.
\newblock In {\em Logic in Computer Science (LICS), 32nd Annual Symposium on},
  pages 1--12. IEEE, 2017.

\bibitem{barr2005topos}
Michael Barr and Charles Wells.
\newblock Topos, triples and theories, 2005.

\bibitem{batanin1998computads}
Michael~A Batanin.
\newblock Computads for finitary monads on globular sets.
\newblock {\em Contemporary Mathematics}, 230:37--58, 1998.

\bibitem{burroni1993higher}
Albert Burroni.
\newblock Higher-dimensional word problems with applications to equational
  logic.
\newblock {\em Theoretical computer science}, 115(1):43--62, 1993.

\bibitem{dehn1911unendliche}
Max Dehn.
\newblock {\"U}ber unendliche diskontinuierliche gruppen.
\newblock {\em Mathematische Annalen}, 71(1):116--144, 1911.

\bibitem{delpeuch2018normalization}
Antonin Delpeuch and Jamie Vicary.
\newblock Normalization for planar string diagrams and a quadratic equivalence
  algorithm.
\newblock {\em arXiv preprint arXiv:1804.07832}, 2018.

\bibitem{dunn2016coherence}
Lawrence Dunn and Jamie Vicary.
\newblock Coherence for {Frobenius} pseudomonoids and the geometry of linear
  proofs.
\newblock Preprint, 2016.

\bibitem{gordon1995coherence}
Robert Gordon, John Power, and Ross Street.
\newblock {\em Coherence for tricategories}, volume 558.
\newblock American Mathematical Society, 1995.

\bibitem{gray2006formal}
John~W Gray.
\newblock {\em Formal category theory: adjointness for 2-categories}, volume
  391.
\newblock Springer, 1974.

\bibitem{guiraud2009higher}
Yves Guiraud and Philippe Malbos.
\newblock Higher-dimensional categories with finite derivation type.
\newblock {\em Theory and Applications of Categories}, 22(18):420--478, 2009.

\bibitem{guiraud2012coherence}
Yves Guiraud and Philippe Malbos.
\newblock Coherence in monoidal track categories.
\newblock {\em Mathematical Structures in Computer Science}, 22(6):931--969,
  2012.

\bibitem{guiraud2016polygraphs}
Yves Guiraud and Philippe Malbos.
\newblock Polygraphs of finite derivation type.
\newblock {\em Mathematical Structures in Computer Science}, pages 1--47, 2016.

\bibitem{guiraud2013homotopical}
Yves Guiraud, Philippe Malbos, and Samuel Mimram.
\newblock A homotopical completion procedure with applications to coherence of
  monoids.
\newblock In {\em RTA-24th International Conference on Rewriting Techniques and
  Applications}, volume~21, pages 223--238, 2013.

\bibitem{gurski2013coherence}
Nick Gurski.
\newblock {\em Coherence in three-dimensional category theory}, volume 201.
\newblock Cambridge Univ. Press, 2013.

\bibitem{knuth1970simple}
Donald~E Knuth and Peter~B Bendix.
\newblock Simple word problems in universal algebras.
\newblock In {\em Computational problems in abstract algebra}, pages 263--297,
  1970.

\bibitem{lafont1992penrose}
Yves Lafont.
\newblock Penrose diagrams and 2-dimensional rewriting.
\newblock {\em Applications of Categories in Computer Science}, 177:191--201,
  1992.

\bibitem{lafont1995new}
Yves Lafont.
\newblock {A new finiteness condition for monoids presented by complete
  rewriting systems (after Craig C. Squier)}.
\newblock {\em Journal of Pure and Applied Algebra}, 98(3):229--244, 1995.

\bibitem{lafont2003towards}
Yves Lafont.
\newblock Towards an algebraic theory of boolean circuits.
\newblock {\em Journal of Pure and Applied Algebra}, 184(2):257--310, 2003.

\bibitem{maclane1963natural}
Saunders MacLane.
\newblock Natural associativity and commutativity.
\newblock {\em Rice Institute Pamphlet-Rice University Studies}, 49(4), 1963.

\bibitem{makkai2005word}
Michael Makkai.
\newblock The word problem for computads.
\newblock Available on the author's web page
  \url{http://www.math.mcgill.ca/makkai/}, 2005.

\bibitem{mimram2014towards}
Samuel Mimram.
\newblock {Towards 3-Dimensional Rewriting Theory}.
\newblock {\em Logical Methods in Computer Science}, 10(1):1–47, 2014.

\bibitem{metayer2008cofibrant}
François Métayer.
\newblock Cofibrant objects among higher-dimensional categories.
\newblock {\em Homology, Homotopy and Applications}, 10(1):181--203, 2008.

\bibitem{squier1987word}
Craig~C Squier.
\newblock Word problems and a homological finiteness condition for monoids.
\newblock {\em Journal of Pure and Applied Algebra}, 49(1-2):201--217, 1987.

\bibitem{squier1994finiteness}
Craig~C Squier, Friedrich Otto, and Yuji Kobayashi.
\newblock A finiteness condition for rewriting systems.
\newblock {\em Theoretical Computer Science}, 131(2):271--294, 1994.

\bibitem{street1976limits}
Ross Street.
\newblock Limits indexed by category-valued 2-functors.
\newblock {\em Journal of Pure and Applied Algebra}, 8(2):149--181, 1976.

\bibitem{street1996categorical}
Ross Street.
\newblock Categorical structures.
\newblock {\em Handbook of algebra}, 1:529--577, 1996.

\bibitem{street2004frobenius}
Ross Street.
\newblock Frobenius monads and pseudomonoids.
\newblock {\em Journal of Mathematical Physics}, 45(10):3930--3948, 2004.

\bibitem{terese2003term}
Terese.
\newblock {\em {Term Rewriting Systems}}.
\newblock Number~55 in Cambridge Tracts in Theoretical Computer Science.
  Cambridge University Press, 2003.

\bibitem{thue1914probleme}
Axel Thue.
\newblock {\em {Probleme {\"u}ber Ver{\"a}nderungen von Zeichenreihen nach
  gegebenen Regeln}}, volume~10.
\newblock 1914.

\bibitem{weber2009free}
Mark Weber.
\newblock Free products of higher operad algebras.
\newblock {\em arXiv preprint arXiv:0909.4722}, 2009.

\end{thebibliography}
